\documentclass[final,12pt]{colt2023} 

\usepackage{caption}
\RequirePackage[OT1]{fontenc}

\usepackage[english]{babel}

\usepackage{enumitem}
\setlist[itemize]{leftmargin=1.5em}

\usepackage{mathtools}
\usepackage[capitalize]{cleveref}
\usepackage{dsfont}
\usepackage{mathdots}
\usepackage{nccmath}
\usepackage{scalerel}

\newenvironment{proplist}{\begin{enumerate}[leftmargin=2.5em,labelwidth=0em,label=(\roman{enumi}),topsep=.1em,partopsep=0em,itemsep=-.3em]}{\end{enumerate}}

\usepackage{booktabs}
\usepackage{tikz}
\usetikzlibrary{calc}
\usetikzlibrary{decorations.markings,decorations.pathreplacing}
\tikzstyle{mybraces}=[mirrorbrace/.style={
          decoration={brace, mirror},
          decorate},brace/.style={
          decoration={brace},
          decorate}]
          
\newtheorem{assumption}{Assumption}

\title[A High-dimensional Convergence Theorem for U-statistics]{A High-dimensional Convergence Theorem for U-statistics \\ with Applications to Kernel-based Testing}
\usepackage{times}

\coltauthor{%
 \Name{Kevin Han Huang} \Email{han.huang.20@ucl.ac.uk}\\
 \addr Gatsby Unit, University College London
 \AND
 \Name{Xing Liu} \Email{xing.liu16@imperial.ac.uk}\\
 \addr Department of Mathematics, Imperial College London
 \AND
 \Name{Andrew B. Duncan} \Email{a.duncan@imperial.ac.uk}\\
 \addr Department of Mathematics, Imperial College London and Alan Turing Institute
 \AND
 \Name{Axel Gandy} \Email{a.gandy@imperial.ac.uk}\\
 \addr Department of Mathematics, Imperial College London
}

\newcommand{\argdot}{{\,\vcenter{\hbox{\tiny$\bullet$}}\,}}

\newcommand{\tagaligneq}{\refstepcounter{equation}\tag{\theequation}}

\newcommand{\msup}{\sup\nolimits}
\newcommand{\minf}{\inf\nolimits}
\newcommand{\mmax}{\max\nolimits}
\newcommand{\Tr}{\text{\rm Tr}}
\def\diag{\text{\rm diag}}
\newcommand{\ind}{\mathbb{I}}

\def\bI{\mathbf{I}}
\def\bJ{\mathbf{J}}

\def\bU{\mathbf{U}}
\def\bV{\mathbf{V}}
\def\bW{\mathbf{W}}
\def\bX{\mathbf{X}}
\def\bY{\mathbf{Y}}
\def\bZ{\mathbf{Z}}

\def\ba{\mathbf{a}}

\def\bm{\mathbf{m}}

\def\bv{\mathbf{v}}
\def\bw{\mathbf{w}}
\def\bx{\mathbf{x}}
\def\by{\mathbf{y}}
\def\bz{\mathbf{z}}

\def\bmu{\boldsymbol{\mu}}
\def\bbeta{\boldsymbol{\beta}}

\def\bbeta{\mathbf{\beta}}

\def\bmu{\mathbf{\mu}}

\def\bzero{\mathbf{0}}

\def\E{\mathbb{E}}

\def\N{\mathbb{N}}

\def\P{\mathbb{P}}

\def\R{\mathbb{R}}

\def\Z{\mathbb{Z}}

\def\cF{\mathcal{F}}

\def\cH{\mathcal{H}}
\def\cI{\mathcal{I}}

\def\cN{\mathcal{N}}

\def\cP{\mathcal{P}}

\def\cV{\mathcal{V}}

\newcommand{\mean}{\mathbb{E}}
\newcommand{\Var}{\text{\rm Var}}
\newcommand{\Cov}{\text{\rm Cov}}

\DeclareMathOperator{\msum}{\medmath\sum}
\DeclareMathOperator{\mprod}{\medmath\prod}
\DeclareMathOperator{\mint}{\scaleobj{.8}{\int}}

\def\Mfullnu{M_{\rm full; \nu}}
\def\Mcondnu{M_{\rm cond; \nu}}

\def\Mmaxnu{M_{\rm max; \nu}}
\def\Mfullthree{M_{\rm full; 3}}
\def\Mfullfour{M_{\rm full; 4}}

\def\Mcondthree{M_{\rm cond; 3}}
\def\Mmaxthree{M_{\rm max; 3}}
\def\Mfulltwo{M_{\rm full; 2}}
\def\Mcondtwo{M_{\rm cond; 2}}

\def\sigfull{\sigma_{\rm full}}
\def\sigcond{\sigma_{\rm cond}}
\def\sigmax{\sigma_{\rm max}}
\newcommand{\ksd}{D^{\rm KSD}}
\newcommand{\uPksd}{u_P^{\rm KSD}}

\newcommand{\mmd}{D^{\rm MMD}}
\newcommand{\ummd}{u^{\rm MMD}}

\newcommand{\gksd}{g^{\rm KSD}}
\newcommand{\gmmd}{g^{\rm mmd}}

\renewenvironment{proof}[1][Proof]%
{%
\par\noindent{\bfseries\upshape #1\ }%
}%
{\jmlrQED}

\begin{document}

\maketitle

\begin{abstract}%
    We prove a convergence theorem for U-statistics of degree two, where the data dimension $d$ is allowed to scale with sample size $n$. We find that the limiting distribution of a U-statistic undergoes a phase transition from the non-degenerate Gaussian limit to the degenerate limit, regardless of its degeneracy and depending only on a moment ratio. A surprising consequence is that a non-degenerate U-statistic in high dimensions can have a non-Gaussian limit with a larger variance and asymmetric distribution. Our bounds are valid for any finite $n$ and $d$, independent of individual eigenvalues of the underlying function, and dimension-independent under a mild assumption. As an application, we apply our theory to two popular kernel-based distribution tests, MMD and KSD, whose high-dimensional performance has been challenging to study. In a simple empirical setting, our results correctly predict how the test power at a fixed threshold scales with $d$ and the bandwidth.
\end{abstract}


\section{Introduction} 

We consider a one-dimensional U-statistic of degree two built on $n$ i.i.d.~data points in $\R^d$. Numerous estimators can be formulated as a U-statistic: Modern applications include high-dimensional change-point detection \citep{wang2022inference}, sensitivity analysis of algorithms \citep{gamboa2022global} and convergence guarantees for random forests \citep{peng2022rates}. 

\vspace{.2em}

The asymptotic theory of U-statistics is well-established in the classical setting, where $d$ is fixed and small relative to $n$ (e.g. Chapter 5 of \cite{serfling1980approximation}). Classical theory shows that the large-sample asymptotic of a U-statistic depends on its martingale structure and moments: For U-statistics of degree two, this reduces to the notion of \emph{degeneracy}, i.e.~whether the variance of a certain conditional mean is zero. Non-degenerate U-statistics are shown to have a Gaussian limit, whereas degenerate ones converge to an infinite sum of weighted chi-squares. 

\vspace{.2em}

However, these results fail to apply to the modern context of high-dimensional data, where $d$ is of a comparable size to $n$. The key issue is that the moment terms, which determine degeneracy, may scale with $d$. Existing efforts on high-dimensional results either focus on U-statistics of a growing degree \citep{song2019approximating,chen2019randomized} and of growing output dimension \citep{chen2018gaussian} or rely on very specific data structures \citep{chen2010two, yan2021kernel}. In particular, these articles focus on a comparison to some Gaussian limit in high dimensions, and the effect of moments on a departure from Gaussianity has largely been ignored.

\vspace{.2em}

The practical motivation for our work stems from distribution tests, which typically employ U-statistics as a test statistic. In the machine learning community, it has been empirically observed that the power of kernel-based distribution tests can deteriorate in high dimensions, depending on hyperparameter choices and the class of alternatives \citep{reddi2015high, ramdas2015decreasing}. A theoretical analysis in the most general case has not been possible, due to the lack of a general convergence result for high-dimensional U-statistics. In the statistics community, there are similar interests in analysing U-statistics used in mean testing of high-dimensional data (e.g. \cite{chen2010two, wang2015high}). All existing results, to our knowledge, are limited by very specific data assumptions and a focus on obtaining Gaussian limits. 

\vspace{.2em}

\begin{figure}
    \centering \vspace{-.2em}
    \includegraphics[width=0.85\textwidth]{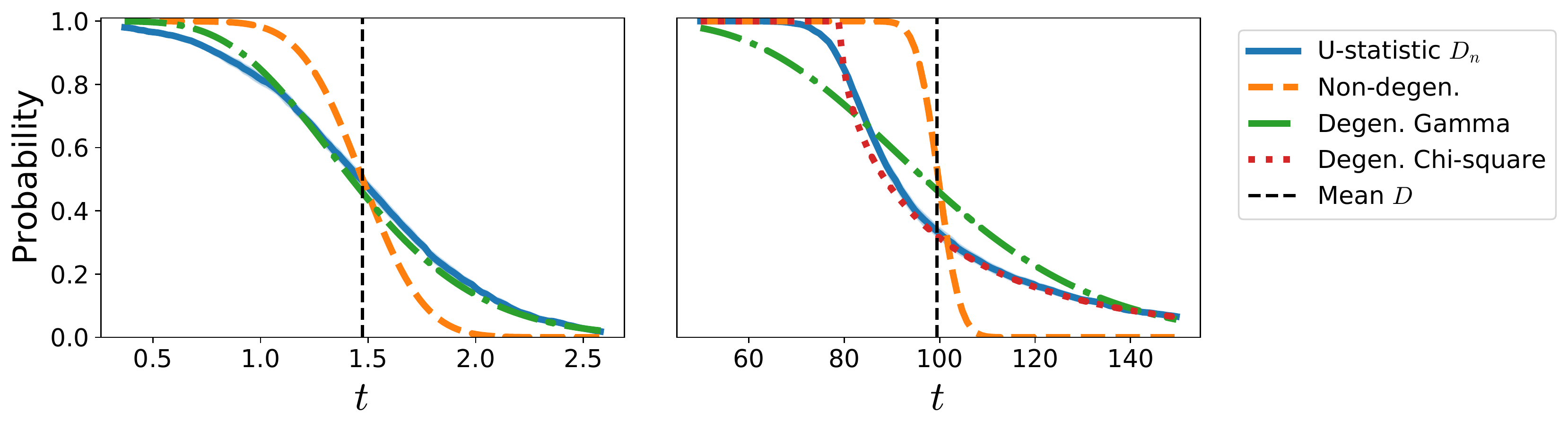} \vspace{-.2em}
    \caption{Behaviour of $\P(X > t)$ for $X=D_n$, a non-degenerate U-statistic, versus $X$ being different theoretical limits. \emph{Left.} KSD with RBF kernel, $n=50$ and $d=2000$. \emph{Right.} MMD with linear kernel, $n=50$ and $d=1000$. The left plot shows that $\P(D_n > t)$ \emph{disagrees} with the non-degenerate limit from known classical results but aligns with the degenerate limit from ours (moment-matched by a Gamma variable -- discussed in \cref{sect:chi:sq}). The right plot is when the limit predicted by our result can be computed exactly as a shifted-and-rescaled chi-square and shows asymmetry, which confirms a departure from Gaussianity. See the last paragraph of \cref{sect:examples} and \cref{appendix:gaussian:mean:shift} for simulation details.
    }
    \label{fig:intro} \vspace{-1em}
\end{figure}

In this paper, we prove a general convergence theorem for U-statistics of degree two, which holds in the high-dimensional setting and under very mild assumptions on the data. We observe a high-dimensional analogue of the classical behaviour: Depending on a moment ratio, the limiting distribution of U-statistics can take either the non-degenerate Gaussian limit, the degenerate limit or an intermediate distribution. Crucially, this happens \emph{regardless of} the statistic's degeneracy, as defined in the classical sense. We provide error bounds that are finite-sample valid and \emph{dimension-independent} under a mild assumption.

\vspace{.2em}

In the context of kernel-based distribution tests, we show that our results hold for \emph{Maximum Mean Discrepancy} (MMD) and for \emph{(Langevin) Kernelized Stein Discrepancy} (KSD) under some natural conditions. We investigate several examples under Gaussian mean-shift -- a setting purposely chosen to be as simple as possible to obtain good intuitions, while already capturing a rich amount of complex behaviours. 
Our theory correctly predicts the high-dimensional behaviour of the test power with a wider variance than classical results and, perhaps surprisingly, potential asymmetry (see \cref{fig:intro} for one such example). Our results enable us to characterise such behaviours based on the size of $d$ and hyperparameter choices.

\subsection{Overview of results}  \label{sec:overview:results}
Given some i.i.d.\,data $\{ \bX_i \}_{i=1}^n$ drawn from a distribution $R$ on $\R^d$ and a symmetric measurable function $u: \R^d \times \R^d \rightarrow \R$, the goal is to estimate the quantity $D \coloneqq \mean[ u(\bX_1,\bX_2)]$. The U-statistic provides an unbiased estimator, defined as \vspace{-.3em}
\begin{align*}
    D_n
    \;\coloneqq\; 
    \mfrac{1}{n(n - 1)} \msum_{1 \leq i \neq j \leq n} u(\bX_i, \bX_j) \;.
    \tagaligneq \label{eqn:general:u:defn} \vspace{-.3em}
\end{align*}

 Our main result is \cref{thm:u:gaussian:quad:general}. Loosely speaking, it says that as $n, d \rightarrow \infty$, the statistic $D_n$ converges in distribution to a quadratic form of Gaussians: \vspace{-.3em}
\begin{align*}
    D_n \;\xrightarrow{d}\;
    W + Z + D\;, \tagaligneq \label{eqn:informal:conv} 
\end{align*}
where $W$ is some infinite sum of weighted and centred chi-squares and $Z$ is some Gaussian. Define
$$
    \rho_d \coloneqq \sigfull \,/\, \sigcond
    \;,
    \text{ where }
    \sigfull 
    \;=\;
    \sqrt{\Var[  u(\bX_1,\bX_2) ]}
    \;\;\text{ and }\;\;
    \sigcond
    \;=\;
    \sqrt{\Var \mean[ u(\bX_1,\bX_2) | \bX_1 ]}
    \;,
$$
and recall that the classical notion of degeneracy is defined by $\sigcond=0$. We next observe that in \eqref{eqn:informal:conv}, $W+D$ is closely related to the classical degenerate limit, whereas $Z+D$ gives exactly the classical non-degenerate limit. It turns out that, up to a mild assumption, the type of asymptotic distribution of $D_n$ is \emph{completely determined} by the ratio $\rho_d$. This is reminiscent of the classical result, where the notion of degeneracy, i.e.~whether $\sigcond=0$, determines the limit of $D_n$. The difference in high dimensions is that $\sigfull$ and $\sigcond$ may scale \emph{differently} with $d$. Even if $\sigcond \neq 0$, $\rho_d$ can grow to infinity as $d$ grows, causing a non-degenerate $D_n$ to behave like a degenerate U-statistic. We show that, depending on $\rho_d$, \eqref{eqn:informal:conv} becomes  \vspace{-.3em}
\begin{align*} 
    D_n \;\xrightarrow{d}&\;
    W + D
    \;\;\;\text{ for } \rho_d = \omega(n^{1/2})\;
    &\text{ and }&&
    D_n \;\xrightarrow{d}&\;
    Z + D
    \;\;\;\text{ for } \rho_d = o(n^{1/2})\;. 
\end{align*}
The second result is the classical Berry-Ess\'een bound for U-statistics, while the first result is new. It recovers the classical degenerate limit as a special case but also applies to very general U-statistics in high dimensions regardless of degeneracy. 

\vspace{.2em}

The paper is organised as follows. \cref{sect:setup} provides definitions and a sketch-of-intuition on the role of moment terms in the limiting behaviour of $D_n$. \cref{sect:general:results} presents the main results along with a proof overview in \cref{sect:proof:overview}. \cref{sec:kernel:general:results} shows that these results apply to MMD and KSD under some natural conditions and \cref{sect:examples} studies the Gaussian mean-shift case in detail. \vspace{-.2em}

\subsection{Related literature} 

\emph{Convergence results for U-statistics.} Existing high-dimensional results focus either on a different setting or on showing asymptotic normality under very specific assumptions on data; some references are provided at the start of this section. The results that resemble our work more closely are finite-sample bounds for classical degenerate U-statistics. Those works focus on providing bounds under conditions on specific eigenvalues of a spectral decomposition of $D_n$, and 
we defer a list of references to Remark \ref{remark:assumption:L_nu}. Among them, \cite{yanushkevichiene2012bounds} provides a rate $O(n^{-1/12})$ under perhaps the least stringent assumption on eigenvalues, but the error is still pre-multiplied by the inverse square-root of the largest eigenvalue. These eigenvalues are intractable and yet depend on $d$ through the data distribution, which make them hard to apply to high-dimensional settings. In the \emph{classical} setting where $d$ is fixed, a recent work by \cite{bhattacharya2022asymptotic} proves a Gaussian-quadratic-form limit similar to ours for a random quadratic polynomial, which includes a simple U-statistic as a special case. However, their results are asymptotic and in particular do not identify a parameter that leads to the phase transition. Our finite-sample results require a very different proof technique and show how a moment ratio governs the transition.

\vspace{.2em}

\noindent
\emph{High-dimensional power analysis for MMD and KSD.} Some recent work has investigated the asymptotic behaviour of $D_n$ for MMD. \cite{yan2021kernel} prove a convergence result under a specific data model and kernel choice, so that $u(\bx,\by) = g(\|\bx-\by\|_2)$ for some function $g: \R \rightarrow \R$ and $\|\argdot\|_2$ being the vector norm. The dimension-independence of $g$ enables a Taylor expansion argument reminiscent of delta method and therefore gives a Gaussian limit. Such structures are not assumed in our work. A related work of \cite{gao2021two} provides a finite-sample bound under more general conditions. The results show asymptotic normality of a studentised version of $D_n$ rather than $D_n$ itself, and the error bound is only valid if a moment ratio, analogous to excess kurtosis, vanishes with $d$ (see their Theorem 13). Interestingly, this effect is obtained as a special case of our results for much more general settings: In \cref{sect:chi:sq}, we point out that the degenerate limit is Gaussian if and only if the excess kurtosis vanishes. Another recent line of work \citep{kim2020dimension,shekhar2022permutation} focuses on a studentised $D_n$ that is modified to exclude half of the terms. They show dimension-agnostic normality results at the cost of not using the full U-statistic $D_n$.

\section{Setup and motivation} \label{sect:setup}

We use the asymptotic notations $o, O, \Theta, \omega, \Omega$ defined in the usual way (see e.g.~Chapter 3 of \cite{cormen2009introduction}) for the limit $n \rightarrow \infty$, where the dimension is allowed to depend on $n$; we make the $n$-dependence explicit in the dimension $d_n$ whenever such asymptotics are considered.

\subsection{Moment terms in high dimensions} \label{sect:moment:terms}

Consider a U-statistic $D_n$ as defined in \eqref{eqn:general:u:defn} with respect to $(R,u)$ with mean $D=\mean[u(\bX_1,\bX_2)]$. For $\nu \geq 1$, denote the $L_\nu$ norms by $\| \argdot \|_{L_\nu} \coloneqq \mean[ |  \argdot |^\nu]^{1/\nu}$. The $\nu$-th central moment of $D_n$ are bounded from above and below in terms of two types of moment terms (see Lemma \ref{lem:u:moments} in the appendix):
$$
    \Mcondnu
    \coloneqq
    \big\| 
    \mean[ u(\bX_1, \bX_2) | \bX_2] - \mean[ u(\bX_1, \bX_2) ] 
    \big\|_{L_\nu},
    \Mfullnu
    \coloneqq
    \big\| 
     u(\bX_1, \bX_2) - \mean[ u(\bX_1, \bX_2)]
    \big\|_{L_\nu}
    . 
$$
In the special case $\nu=2$, the definitions from \cref{sec:overview:results} implies $\sigcond = \Mcondtwo$,  $\sigfull = \Mfulltwo$ and $\rho_d=\sigfull \,/\, \sigcond$. The fact that these moments may scale with $d$ has a significant effect on convergence results: For example, bounds of the form $\frac{\rm moment}{f(n)}$ for some increasing function $f$ of $n$ are no longer guaranteed to be small. This is yet another effect of the ``curse of dimensionality". For U-statistics, the classical Berry-Ess\'een result (see e.g. Theorem 10.3 of \cite{chen2011normal}) says that, if $\sigcond > 0$, then for a normal random variable $Z \sim \cN  ( D, 4 n^{-1} \sigcond^2 )$ and $\nu \in (2,3]$, we have
\begin{align*}
    \msup_{t \in \R} \Big| \P\Big( \mfrac{\sqrt{n}}{\sigcond} D_n < t \Big) - \P\Big( \mfrac{\sqrt{n}}{\sigcond} Z < t \Big) \Big| 
    \;\leq\; 
    \mfrac{6.1 \Mcondnu^\nu}{n^{(\nu-2)/2} \sigcond^\nu}
    +
    \mfrac{(1+\sqrt{2}) \rho_d}{ 2(n-1)^{1/2}} \;. \tagaligneq \label{eqn:berry:esseen}
\end{align*}
Indeed, the error bound in the classical Berry-Ess\'een result
is an increasing function of $n^{-1/2}\rho_d = \sigfull / (n^{1/2}\sigcond)$, which is not guaranteed to be small as $d$ grows.

\vspace{.2em}

The ratio $\Mcondnu / \sigcond$ also appears in classical error bounds. However, we do \emph{not} focus on how this ratio scales, since it appears in Berry-Ess\'een bounds even for sample averages. Error bounds in our main theorem will depend on similar ratios, and for our theorem to imply a convergence theorem, the following assumption is required: \vspace{-.3em}
\begin{assumption} \label{assumption:moment:ratio} There exists some $\nu \in (2,3]$ and some constant $C < \infty$ such that for all $n$ and $d$, we have the uniform bounds $\frac{\Mfullnu}{\sigfull} \leq C$ and $\frac{\Mcondnu}{\sigcond} \leq C$\;.
\end{assumption}

\subsection{Sketch of intuition} \label{sect:intuition}

We motivate our results by noting that the variance of $D_n$ defined in \eqref{eqn:general:u:defn} satisfies
\begin{align*}
    \Var[ D_n] \;=&\;
    O\Big( \mfrac{\mean[ (u(\bX_1,\bX_2) - D)(u(\bX_1,\bX_3) - D)]}{n} + \mfrac{\mean[ (u(\bX_1,\bX_2) - D)(u(\bX_1,\bX_2) - D)]}{n(n-1)} \Big)
    \\
    \;=&\;
    O\big( \mfrac{\sigcond^2}{n} + \mfrac{\sigfull^2}{n(n-1)} \big)\;.
\end{align*}
To study the asymptotic distribution of $D_n$, we need to understand how its asymptotic variance behaves as $n$ and $d$ grow. Suppose we are in the classical \emph{non-degenerate} setting, where $d$ is fixed and $\sigcond > 0$. The dominating term in $\Var[D_n]$ is $O(n^{-1} \sigcond^2)$. The contribution of the $\sigfull^2$ term is small, i.e.~the effect of the variance of each individual summand $u(\bX_1,\bX_2)$ is negligible. In fact, we can approximate $D_n$ by replacing each argument in the summand by an independent copy $\bX'_i$ of $\bX_i$ and applying CLT for an empirical average:
\begin{align*}
    D_n 
    \;=&\;
    D+ \mfrac{1}{n(n - 1)} \msum_{1 \leq i \neq j \leq n} 
    ( u(\bX_i, \bX_j) - D)
    \\
    \;\approx&\; 
    D+ 
    \mfrac{1}{n} \msum_{i=1}^n
    \Big( \mfrac{1}{n-1} \msum_{j \neq i} (u(\bX_i, \bX'_j)-D) \Big)
    +
    \mfrac{1}{n} \msum_{j=1}^n
    \Big( \mfrac{1}{n-1} \msum_{i \neq j} (u(\bX'_i, \bX_j)-D) \Big)
    \\
    \;=&\;
    D +
    \mfrac{2}{n} \msum_{i=1}^n \big( \mean[ u(\bX_i, \bX'_j) | \bX_i] - D \big)
    \;\approx\;
    \cN\big( D \,,\, \mfrac{4\sigcond^2}{n} \big)\;.
\end{align*}
This argument underpins results on CLT for non-degenerate U-statistics. In the classical degenerate setting, however, $d$ is still fixed but $\sigcond=0$, and the above argument fails to apply. Instead, one considers a spectral decomposition $u(\bx,\by) = \sum_{k=1}^\infty \lambda_k \phi_k(\bx) \phi_k(\by)$ for some eigenvalues $\{\lambda_k\}_{k=1}^\infty$ and eigenfunctions $\{\phi_k\}_{k=1}^\infty$, and compares the distribution of $D_n$ to a weighted sum of chi-squares:
\begin{align*}
    D_n 
    \;=&\;
    \mfrac{1}{n(n - 1)} \msum_{1 \leq i \neq j \leq n} \msum_{k=1}^\infty  \lambda_k \phi_k(\bX_i) \phi_k(\bX_j)
    \\
    \;=&\;
    \msum_{k=1}^\infty  \lambda_k \Big( 
    \Big( \mfrac{1}{n}\msum_{i=1}^n  \phi_k(\bX_i) \Big) \Big( \mfrac{1}{n}\msum_{j=1}^n  \phi_k(\bX_j) \Big)
    -
    \mfrac{1}{n^2} \msum_{i=1}^n \phi_k(\bX_i)^2
    \Big)
    \\
    \;\approx&\;
    \mfrac{1}{n}
    \msum_{k=1}^\infty  \lambda_k \, 
    \big( 
    \big(\sqrt{\Var[\phi_k(\bX_1)]} \, \xi_k + \mean[ \phi_k(\bX_1)]\big)^2
    -
    \mean[ \phi_k(\bX_1)^2]
    \big)
    \;,
\end{align*}
where $\xi_k$'s are i.i.d.\,standard normals. The limiting distributions in both settings enable one to construct consistent confidence intervals for $D_n$ and study $\P(D_n > t)$.

\vspace{.2em}

The key takeaway is that the asymptotic distribution of $D_n$ depends on the relative sizes of $\sigcond^2$ and $(n-1)^{-1} \sigfull^2$. This comparison reduces to degeneracy when $d$ is fixed, but is no longer so when $d$ grows. In the high-dimensional setting, $\sigcond$ and $\sigfull$ can scale with $d$ at \emph{different orders}, making it possible for the ratio $\rho_d $ to vary with $d$. In particular, a non-degenerate U-statistic with $\sigcond > 0$ may still satisfy $\rho_d = \omega(n^{1/2})$, i.e. $(n-1)^{-1} \sigfull^2 / \sigcond^2 \rightarrow \infty$ as $n$ and $d$ grow. In this case, the classical argument for a non-degenerate Gaussian limit would fail and a degenerate limit would dominate. This is exactly what we observe in the practical applications in \cref{sect:examples}, and motivates the need for results that explicitly addresses the high-dimensional setting.

\section{Main results} \label{sect:general:results}

The main result presented in this section is a finite-sample bound that compares $D_n$ to a quadratic form of infinitely many Gaussians. The limiting distribution is a sum of the non-degenerate limit and a variant of the degenerate limit, and subject to \cref{assumption:moment:ratio}, the error bound is \emph{independent} of $\rho_d$. In the case $\rho_d = o(n^{1/2})$, the non-degenerate limit dominates and our result agrees with the Gaussian limit given by a Berry-Ess\'een theorem for U-statistics. However when dimension is high such that $\rho_d = \omega(n^{1/2})$, the degenerate limit dominates and implies a \emph{larger asymptotic variance}. We also discuss how to obtain consistent distribution bounds that reflect the effect of a large dimension $d$ on the original statistic $D_n$.

\vspace{.2em}

Our results rest on a functional decomposition assumption. For a sequence of $\R^d \rightarrow \R$ functions $\{\phi_k\}_{k=1}^\infty$ and a sequence of real values $\{\lambda_k\}_{k=1}^\infty$, we define the $L_\nu$ approximation error for $\nu \geq 1$ and a given $K \in \N$ as
\begin{align*}
    \varepsilon_{K;\nu} \;\coloneqq\; \big\| \msum_{k=1}^K \lambda_k \phi_k(\bX_1) \phi_k(\bX_2) 
    -
    u(\bX_1,\bX_2) \big\|_{L_\nu}
    \;.
\end{align*}

\begin{assumption} \label{assumption:L_nu} There exists some $\nu \in (2,3]$ such that, for any fixed $n$ and $d$, as $K \rightarrow \infty$, the $L_\nu$ approximation error $\varepsilon_{K;\nu} \rightarrow 0$ for some choice of $\{\phi_k\}_{k=1}^\infty$ and $\{\lambda_k\}_{k=1}^\infty$.
\end{assumption}

\begin{remark} \label{remark:assumption:L_nu} (i) If \cref{assumption:L_nu} holds for some $\nu > 3$, it certainly holds for $\nu=3$. We restrict our focus to $\nu \in (2,3]$ for simplicity. (ii) \cref{assumption:L_nu} always holds for $\nu=2$ by the spectral decomposition of an operator on $L_2(\R^d, R)$. For degenerate U-statistics with $d$ fixed, the corresopnding orthonormal eigenbasis of functions and eigenvalues are used to prove asymptotic results  (see Section 5.5.2 of \cite{serfling1980approximation}) and finite-sample bounds \citep{bentkus1999optimal, gotze2005asymptotic, yanushkevichiene2012bounds}. In fact, these finite-sample bounds are dependent on the specific $\lambda_k$'s, making the results hard to apply. Instead, we forgo orthonormality at the cost of a convergence slightly stronger than $L_2$. This allows for a much more flexible choice of $\{\phi_k, \lambda_k\}_{k=1}^\infty$ and is particularly well-suited for a kernel-based setting; see Remark \ref{remark:assumption:L_nu:kernel} for a discussion.
\end{remark}

Before stating the results, we introduce some more notations. For every $K \in \N$, we define a diagonal matrix of the first $K$ ``eigenvalues" and a concatenation of the first $K$ ``eigenfunctions" by
\begin{align*}
    &
    \Lambda^K\;\coloneqq\; \diag\{ \lambda_1, \ldots, \lambda_K\}\;\in\R^{K \times K}\;,
    &&
    \phi^K(x)\;\coloneqq\; ( \phi_1(x), \ldots, \phi_K(x) )^\top\;\in\R^K\;. \tagaligneq \label{eqn:truncated:defn}
\end{align*}
We denote the mean and variance of $\phi^K(\bX_1)$ by $\mu^K \coloneqq \mean[ \phi^K(\bX_1)]$ and $\Sigma^K \coloneqq \Cov[\phi^K(\bX_1)]$.

\subsection{Result for the general case} \label{sect:general:thm}

Let $\eta^K_i$, with $i, K \in \N$, be i.i.d.~standard Gaussian vectors in $\R^K$. In the general case, the limiting distribution is given in terms of a quadratic form of Gaussians, defined by
\begin{align*}
    U^K_n
    \coloneqq
    \mfrac{1}{n(n-1)} \msum_{1 \leq i \neq j \leq n} (\eta^K_i)^\top (\Sigma^K)^{1/2} \Lambda^K (\Sigma^K)^{1/2} \eta^K_j
    +
    \mfrac{2}{n} \msum_{i=1}^n 
    (\mu^K)^\top \Lambda^K (\Sigma^K)^{1/2} \eta^K_i
    + D.
\end{align*}
We also denote the dominating moment terms by
\begin{align*}
    &
    \sigmax \;\coloneqq\; \max\{ \sigfull, (n-1)^{1/2} \sigcond\}\;,
    &&
    \Mmaxnu \;\coloneqq\; \max\{ \Mfullnu, (n-1)^{1/2} \Mcondnu\}\;.
\end{align*}
We are ready to state our main result -- a finite-sample error bound that compares $D_n$ to the limiting distribution of $U_n^K$, where the error is given in terms of $n$ and the moment terms.

\begin{theorem} \label{thm:u:gaussian:quad:general} There exists a constant $C>0$ such that, for all $u$, $R$, $d$ and $n$, if $\nu \in (2,3]$ satisfies \cref{assumption:L_nu}, then the following holds:
\begin{align*}
    \msup_{t \in \R}
    \Big| \P\Big( &\mfrac{\sqrt{n(n-1)}}{\sigmax} D_n > t \Big)  - \lim_{K \rightarrow \infty} \P\Big( \mfrac{\sqrt{n(n-1)}}{\sigmax} U_n^K > t\Big) \Big|
    \\ 
    \;\leq&\;
    C \, n^{- \frac{\nu - 2}{4\nu+2}}  \Big(
    \mfrac{ 
    (\Mfullnu)^\nu}
    {\sigmax^{\nu} }
    +
    \mfrac{
     ( (n-1)^{1/2}\,\Mcondnu)^\nu }{\,\sigmax^{\nu}}
    \Big)^{\frac{1}{2\nu+1}}
    \;\leq\;
    2^{\frac{1}{2\nu+1}} C \, n^{- \frac{\nu - 2}{4\nu+2}}  \Big(
    \mfrac{\Mmaxnu}{\sigmax} \Big)^{\frac{\nu}{2\nu+1}}\;.
\end{align*}
\end{theorem}

\begin{remark} If $\nu=3$, the RHS is given by $2^{3/7} C n^{-\frac{1}{14}} \big(\frac{\Mmaxthree}{\sigmax}\big)^{6/7}$. If \cref{assumption:moment:ratio} holds for $\nu$, the RHS can be replaced by $C'n^{- \frac{\nu - 2}{4\nu+2}}$ for some constant $C'$ and is dimension-independent.
\end{remark}

\begin{remark} At first sight, one may be tempted to move $\lim_{K \rightarrow \infty}$ inside $\P$ such that, instead of the cumbersome expression of $W_n^K$ with finite $K$, one may deal with random quantities in a Hilbert space. The reason to stick with $W_n^K$ is that in \cref{assumption:L_nu}, convergence of the infinite sum is required only in $L_\nu$ and not almost surely. This makes verification of the assumption substantially simpler in practice: In \cref{appendix:gaussian:mean:shift}, we illustrate how this assumption holds via a simple Taylor-expansion argument coupled with suitable tail behaviour of the data to control error terms. The same argument is not applicable if we instead require an almost sure convergence.
\end{remark}

\cref{thm:u:gaussian:quad:general} immediately implies a convergence theorem:

\begin{corollary} \label{cor:u:gaussian:quad:general} Let the dimension $d_n$ depend on $n$. Suppose Assumptions \ref{assumption:moment:ratio} and  \ref{assumption:L_nu} hold for some $\nu > 2$ and the sequential distribution limit $\bar U = \lim_{n \rightarrow \infty} \lim_{K \rightarrow \infty} \frac{\sqrt{n(n-1)}}{\sigmax} (U_n^K - D)$ exists. Then
\begin{align*}
    &\mfrac{\sqrt{n(n-1)}}{\sigmax} ( D_n - D) \xrightarrow{d} \bar U
    &&
    \;\text{ as }\; n \rightarrow \infty
    \;.
\end{align*}
\end{corollary}

$U_n^K$ is a quadratic form of Gaussians, which does not admit a closed-form c.d.f.\,in general and whose limiting behaviour depends heavily on $\lambda_k$ and $\phi_k$. Nevertheless, the presence of Gaussianity still allows us to obtain crude bounds that reflect how dimension $d$ affects its distribution. By combining such bounds with \cref{thm:u:gaussian:quad:general}, we can bound the c.d.f.\,of the original U-statistic $D_n$.

\begin{proposition} \label{prop:u:gaussian:quad:distribution bounds} There exists constants $C_1, C_2, C_3 >0$ such that, for all $u$, $R$, $d$, $n$ and $K$, if $\nu \in (2,3]$ satisfies \cref{assumption:L_nu}, then for all $\epsilon > 0$,
\begin{align*}
    \P( | D_n - D| > \epsilon )
    \;\geq&\;\;
    1
    -
    C_1 \Big( \mfrac{\sqrt{n(n-1)}}{\sigmax} \Big)^{1/2} \epsilon^{1/2}
    -
    C_2 \, n^{- \frac{\nu - 2}{4\nu+2}}  \Big(
    \mfrac{\Mmaxnu}{\sigmax} \Big)^{\frac{\nu}{2\nu+1}}\;,
    \\
    \P( | D_n - D| > \epsilon )
    \;\leq&\;\;
    C_3 \epsilon^{-2}
    \Big(\mfrac{\sigmax}{\sqrt{n(n-1)}} \Big)^2\;.
\end{align*}
\end{proposition}

\begin{remark} The second bound is a concentration inequality directly available via Markov's inequality, whereas the first bound is an anti-concentration result. Anti-concentration results are generally available only for random variables from known distribution families, and we obtain such a result by comparing $D_n$ to $U^K_n$. The error bounds are free of any dependence on $K$ and specific choices of $\phi_k$ and $\lambda_k$. The trailing error term involving $\Mmaxnu / \sigmax$ is inherited from \cref{thm:u:gaussian:quad:general} and is negligible, whereas the other error term is directly related to the inverse of the Markov error term.
\end{remark}

Proposition \ref{prop:u:gaussian:quad:distribution bounds} provide two-sided bounds on how likely it is for $D_n$ to be far from $D$. The next corollary provides a more explicit statement.

\begin{corollary} \label{cor:u:gaussian:quad:distribution bounds} Let the  dimension $d_n$ depend on $n$ and fix $\epsilon > 0$. Suppose Assumptions \ref{assumption:moment:ratio} and \ref{assumption:L_nu} hold for some $\nu \in (2,3]$. As $n \rightarrow \infty$, we have that $\P( | D_n - D | > \epsilon) \rightarrow 1$ \,if\, $\sigmax = \omega(n)$ and $\P( | D_n - D | > \epsilon) \rightarrow 0$ \,if\, $\sigmax = o(n)$.
\end{corollary}

Another way of formulating the bounds in Proposition \ref{prop:u:gaussian:quad:distribution bounds} is the following: Similar to the intuition for a Gaussian, when $n$ is large (with $d_n$ depending on $n$), the distribution of $D_n$ is not only concentrated in an interval around $D$ with width being a multiple of $\frac{\sigmax}{n}$, but also ``well spread-out" within the interval. The probability mass gets concentrated around $D$ when $\sigmax = o(n)$, but spreads out along the whole real line when $\sigmax = \omega(n)$; the latter only happens in a high dimensional regime. 

\vspace{.3em}

To have a more precise understanding of the limiting behaviour of $D_n$, we need a better knowledge of $U_n^K$. By a closer examination of $U_n^K$, we see that it is a sum of three terms: A sum of weighted chi-squares with variance of the order $n^{-1}(n-1)^{-1} \sigfull^2$, a Gaussian with variance of the order $n^{-1} \sigcond^2$, and a constant $D$. The first term closely resembles the limit for degenerate U-statistics when $d$ is fixed, while the second term corresponds exactly to the Gaussian limit for non-degenerate U-statistics. It turns out that, unless we are at the boundary case where $\rho_d = \Theta(n^{1/2})$, we can always approximate $U_n^K$ by ignoring either the first or the second term. Ignoring the first term gives exactly the Gaussian limit, where a well-established result has already been provided in \eqref{eqn:berry:esseen}. Ignoring the second term gives an infinite sum of weighted chi-squares, which is discussed next.

\subsection{The case $\rho_d = \omega ( n^{1/2} )$} \label{sect:chi:sq}
Let $\{\xi_k\}_{k=1}^\infty$ be a sequence of i.i.d.~standard Gaussians in 1d, and for $K \in \N$, let $\{\tau_{k;d}\}_{k=1}^K$ be the eigenvalues of $(\Sigma^K)^{1/2} \Lambda^K (\Sigma^K)^{1/2}$. The limiting distribution we consider is given in terms of 
\begin{align*}
    W_n^K  
    \;\coloneqq\;
    \mfrac{1}{\sqrt{n(n-1)}}
    \msum_{k=1}^K \tau_{k;d} (\xi_k^2 - 1)
    +
    D
    \;. \tagaligneq \label{eqn:defn:WnK}
\end{align*}
Note that in this case, $\sigmax = \sigfull$. The next result adapts \cref{thm:u:gaussian:quad:general} by replacing $U_n^K$ with $W_n^K$:

\begin{proposition} \label{prop:u:chi:sq:general} There exists a constant $C>0$ such that, for all $u$, $R$, $d$, $n$ and $K$, if $\nu \in (2,3]$ satisfies \cref{assumption:L_nu}, then the following holds:
\begin{align*}
    \msup_{t \in \R} &\Big| \P\Big( \mfrac{\sqrt{n(n-1)}}{\sigfull} D_n > t \Big) - \lim_{K \rightarrow \infty} \P\Big( \mfrac{\sqrt{n(n-1)}}{\sigfull} W_n^K > t\Big) \Big|
    \\
    \;&\leq\;
    C
    \Big( 
    \mfrac{1}{(n-1)^{1/5}} 
    + 
    \Big( \mfrac{\sqrt{n-1}\, \sigcond}{\sigfull} \Big)^{2/5} 
    +
    \, n^{- \frac{\nu - 2}{4\nu+2}}  \Big(
    \mfrac{ 
    (\Mfullnu)^\nu}
    {\sigfull^{\nu} }
    +
    \mfrac{
     ((n-1)^{1/2}\Mcondnu)^\nu }{\sigfull^{\nu}}
    \Big)^{\frac{1}{2\nu+1}}  \Big)\;.
\end{align*}
\end{proposition}

\begin{remark}
In the case $\nu=3$, the error term above becomes
\begin{align*}
    C
    \Big( 
    \mfrac{1}{(n-1)^{1/5}} 
    + 
    \Big( \mfrac{\sqrt{n-1}\, \sigcond}{\sigfull} \Big)^{2/5} 
    +
    \, n^{- \frac{1}{14}}  \Big(
    \mfrac{ 
    (\Mfullthree)^3}
    {\sigfull^{3} }
    +
    \mfrac{
     \big((n-1)^{1/2} \Mcondthree\big)^3 }{\sigfull^3}
    \Big)^{\frac{1}{7}}  \Big)\;.
\end{align*}
In the case when \cref{assumption:moment:ratio} holds for $\nu$, the error term is $\Theta\big( \big(\frac{n-1}{\rho_d^2}\big)^{1/5} + n^{-\frac{\nu - 2}{4\nu+2}} \big)$.
\end{remark}

\begin{remark} Proposition \ref{prop:u:chi:sq:general} agrees with the classical results for degenerate U-statistics. In those results, $\{\phi_k\}_{k=1}^\infty$ are chosen such that they are orthonormal in $L_2(\R^d,R)$ and $\mean[\phi_k(\bX_1)] = 0$. This corresponds to $\Sigma^K$ being a diagonal matrix and the expression for $\tau_{k;d}$ can be simplified. 
\end{remark}

We seek to obtain a better understanding of the limiting distribution of $D_n$ in the case $\rho_d = \omega ( n^{1/2} )$. Write $W_n \coloneqq \lim_{K \rightarrow \infty} W_n^K$ as the distributional limit of $W_n^K$ as $K \rightarrow \infty$. Provided that $W_n$ exists, Proposition \ref{prop:u:chi:sq:general} gives the convergence of $D_n$ to $W_n$ in the Kolmogorov metric. The next lemma guarantees the existence of $W_n$. 

\begin{proposition} \label{prop:W_n:existence:all:moments} Fix $n,d$. If \cref{assumption:L_nu} holds for some $\nu \geq 2$ and $|D|, \sigfull < \infty$, $W_n$~exists.
\end{proposition}

While $W_n^K$ is a sum of chi-squares, the distributional limit $W_\infty \coloneqq \lim_{n\rightarrow \infty} \lim_{K \rightarrow \infty} W_n^K$ may actually be Gaussian. The crucial subtlety lies in the fact that the weights of $W_n^K$ may depend on $K$ and also on $n$ (through $d\equiv d_n$). In what is well-known in the probability literature as the ``fourth moment phenomenon" \citep{nualart2005central}, the necessary and sufficient condition for Gaussianity of $W_\infty$ is that the limiting excess kurtosis is zero. In our case, the limiting moments can be computed easily when \cref{assumption:L_nu} holds for $\nu \geq 4$, as they depend only on moments of the original function $u$ and \emph{not} on specific values of the intractable weights $\tau_{k;d}$. Lemma \ref{lem:WnK:moments} in the appendix shows that $\mean[W_n^K]=D$ for every $K \in \N$, \; $\lim_{K \rightarrow \infty} \Var[ W_n^K ] = \frac{2}{n(n-1)} \sigfull^2$ and 
\begin{align*}
    \lim_{K \rightarrow \infty} 
    \mean\big[ (W_n^K - D)^4 \big] \;=\; 
    \mfrac{12 (4 \mean[u(\bX_1,\bX_2)u(\bX_2,\bX_3)u(\bX_3,\bX_4)u(\bX_4,\bX_1)] + \sigfull^4)}{n^2(n-1)^2} \;,
\end{align*}
provided that \cref{assumption:L_nu} holds for $\nu \geq 1$, $\nu \geq 2$ and $\nu \geq 4$ respectively. If the excess kurtosis is indeed zero, Gaussian is still the correct limiting distribution for $D_n$, but now with a \emph{larger} variance (characterized by $\sigfull$) than the one naively predicted by the Gaussian CLT limit for non-degenerate U-statistics. Meanwhile, when the excess kurtosis is not zero, the limiting distribution is an infinite sum of weighted chi-squares. A naive example is the following:

\begin{lemma} \label{lem:chisq:limit:not:gaussian} Suppose there exists a finite $K_*$ such that $\lambda_k = 0$ for all $k > K_*$. Then $W_n = W_n^{K^*}$, which is a weighted sum of chi-squares.
\end{lemma}

A weighted sum of chi-squares does not admit a closed-form distribution function. Fortunately in the case when $\tau_{k;d} \geq 0$ for all $k$, many numerical approximation schemes are available and used widely. These methods generally rely on matching the moments of $W_n$, which can be computed easily due to Proposition \ref{prop:W_n:existence:all:moments}. The simplest example is the Welch-Satterthwaite method, which approximates the distribution of $W_n$ by a gamma distribution with the same mean and variance. We refer readers to \cite{bodenham2016comparison} and \cite{duchesne2010computing} for a review of other moment-matching methods. 

\subsection{Proof overview} \label{sect:proof:overview}

The proof for \cref{thm:u:gaussian:quad:general} consists of three main steps:
\begin{proplist}
    \item \textbf{``Spectral" approximation.} We first use \cref{assumption:L_nu} to replace $u(\bX_i,\bX_j)$ with the truncated sum $\sum_{k=1}^K \lambda_k \phi_k(\bX_i)\phi_k(\bX_j)$, which gives a truncation error that vanishes as $K \rightarrow \infty$;
    \item \textbf{Gaussian approximation.} The truncated sum is a simple quadratic form of i.i.d.~vectors in $\R^d$, each of which can be approximated by a Gaussian vector. This is done by following \cite{chatterjee2006generalization}'s adaptataion of Lindeberg's telescoping sum argument. Similar proof ideas have been used to develop new convergence results in statistics and machine learning; examples include empirical risk \citep{montanari2022universality} and bootstrap for non-asymptotically normal estimators \citep{austern2020asymptotics}. This step introduces errors in terms of moment terms of $U_n^K$, which are then related to those of $D_n$;
    \item \textbf{Bound the distribution of $U_n^K$.} Step (ii) introduces errors in terms of the distribution of $U_n^K$, a quadratic form of Gaussians, over a short interval. These errors are then controlled by the distribution bounds from \cite{carbery2001distributional}.
\end{proplist}
The proof for Proposition \ref{prop:u:chi:sq:general} is similar, except that we use an additional Markov-type argument to remove the linear sum from $U_n^K$ and obtain the limit in terms of $W_n^K$. 
\vspace{-0.7em}

\section{Kernel-based testing in high dimensions} \label{sect:ksd:mmd}

Given two probability measures $P$ and $Q$ on $\R^d$, we consider the problem of testing $H_0: P=Q$ against $H_1: P \neq Q$ through some measure of discrepancy between $P$ and $Q$. We focus on  \emph{Maximum Mean Discrepancy} (MMD) and \emph{(Langevin) Kernelized Stein Discrepancy} (KSD), two kernel-based methods that use a U-statistic $D_n$ as the test statistic. It is well-known that $\sigcond=0$ under $H_0$ and the limit of $D_n$ is a weighted sum of chi-squares (see \cite{gretton2012kernel} for MMD and \cite{liu2016kernelized} for KSD). Instead, we are interested in quantifying the power of $D_n$ given as $\P_{H_1}(D_n > t)$. The test threshold $t$ is often chosen adaptively in practice, but we assume $t$ to be fixed for simplicity of analysis. The results in \cref{sect:general:results} offer two key insights to this problem:
\begin{proplist}
    \item $D_n$ may have different limiting distributions depending on $\rho_d$. In the non-Gaussian case, the confidence interval and thereby the distribution curve can be wider than what a Berry-Ess\'een bound predicts, and there may be potential asymmetry;
    \item We can completely characterise the high-dimensional behaviour of the power in terms of $\rho_d$, which in turn depends on the hyperparameters and the set of alternatives considered. 
\end{proplist} 
In this section, we first show that our results naturally apply to MMD and KSD. We then investigate their high-dimensional behaviours in an example of Gaussian mean-shift under simple kernels. Throughout, $\| \argdot \|_2$ denotes the vector Euclidean norm, which is not to be confused with $\|\argdot\|_{L_2}$. 
\vspace{-0.5em}
\subsection{Notations}
We follow the kernel definition from \cite{steinwart2012mercer} as below: \vspace{-0.3em}
\begin{definition} A function $\kappa: \R^d \times \R^d \rightarrow \R$ is called a \emph{kernel} on $\R^d$ if there exists a Hilbert space $(\cH, \langle \argdot, \argdot \rangle_{\cH})$ and a map $\phi: \R^d \rightarrow \cH$ such that $\kappa(\bx,\bx') = \langle \phi(\bx), \phi(\bx') \rangle_{\cH}$ for all $\bx,\bx' \in \cH$.
\end{definition} \vspace{-0.3em}
We give the minimal definitions of MMD and KSD, and refer interested readers to \cite{gretton2012kernel} and \cite{gorham2017measuring} for further reading.  Throughout, we let $\{\bY_j\}_{j=1}^n$ be i.i.d.~samples from $P$ and $\{\bX_i\}_{i=1}^n$ be i.i.d.~samples from $Q$. We also write $\bZ_i \coloneqq (\bX_i, \bY_i)$ and assume that $\kappa$ is measurable. MMD with respect to $\kappa$ is defined by \vspace{-0.5em}
$$
    \mmd(Q,P)
    \;\coloneqq\; \mean_{\bY, \bY' \sim P}[\kappa(\bY,\bY')] - 2 \mean_{\bY \sim P, \bX \sim Q}[\kappa(\bY,\bX)] + \mean_{\bX, \bX' \sim Q}[\kappa(\bX,\bX')] 
    \;. \vspace{-0.5em}
$$ 
A popular unbiased estimator for $\mmd$ is exactly a U-statistic: \vspace{-0.5em}
$$
    \mmd_n
    \;\coloneqq\; \mfrac{1}{n(n-1)} \msum_{1 \leq i \neq j \leq n} \ummd(\bZ_i,\bZ_j)
    \;, \vspace{-0.5em} 
$$
 where the summand is given by $ \ummd\big( (\bx,\by) , (\bx',\by') \big) \coloneqq \kappa(\bx,\bx') + \kappa(\by,\by') - \kappa(\bx,\by') - \kappa(\bx',\by)$. To define KSD, we assume that $\kappa$ is continuously differentiable with respect to both arguments, and $P$ admits a continuously differentiable, positive Lebesgue density $p$. The following formulation of KSD is due to Theorem 2.1 of \citet{chwialkowski2016kernel}: \vspace{-0.5em}
$$
    \ksd(Q, P) \;\coloneqq\; \mean_{\bX, \bX' \sim Q}[ \uPksd(\bX, \bX') ]\;, \vspace{-0.5em}
$$
where we assume $\mean_{\bX \sim Q}[ \uPksd(\bX, \bX) ] < \infty$ and the function $\uPksd: \R^d \times \R^d \rightarrow \R$ is given by
\begin{align*}
    \uPksd(\bx, \bx')
    =&\; \big( \nabla \log p(\bx) \big)^\top \big( \nabla \log p(\bx') \big) \kappa(\bx, \bx') \;+\; \big( \nabla \log p(\bx) \big)^\top  \nabla_2 \kappa(\bx, \bx') \\
    & + \big( \nabla \log p(\bx') \big)^\top  \nabla_1 \kappa(\bx, \bx') \;+\; \Tr(\nabla_1 \nabla_2 \kappa(\bx, \bx')) \;.
\end{align*}
$\nabla_1$ and $\nabla_2$ are the differential operators with respect to the first and second arguments of $\kappa$ respectively. The estimator is again a U-statistic, given by $\ksd_n \coloneqq \frac{1}{n(n-1)} \sum_{1 \leq i \neq j \leq n} \uPksd(\bX_i, \bX_j)$.

\subsection{General results} \label{sec:kernel:general:results}

We show that a kernel structure allows \cref{assumption:L_nu} to be fulfilled under some natural conditions. Let $\bV_1,\bV_2 \overset{i.i.d.}{\sim} R$ for some probability measure $R$ on $\R^b$ and $\kappa^*$ be a measurable kernel on $\R^b$. A sequence of functions $\{\phi_k\}_{k=1}^\infty$ in $L_2(\R^b, R)$ and a sequence of non-negative values $\{\lambda_k\}_{k=1}^\infty$ with $\lim_{k\rightarrow \infty}\lambda_k=0$ is called a \emph{weak Mercer representation} if
$$
\big| \msum_{k=1}^K \lambda_k \phi_k(\bV_1) \phi_k(\bV_2) - \kappa^*(\bV_1,\bV_2) \big| \rightarrow 0
\;\;\text{ almost surely }\;\;
\qquad\text{ as } K \rightarrow \infty 
\;.
$$
\cite{steinwart2012mercer} show that such a representation exists if $\mean[ \kappa^*(\bV_1,\bV_1) ] < \infty$, whose result is summarised in Lemma \ref{lem:mercer} in the appendix. To deduce from this the $L_\nu$ convergence of \cref{assumption:L_nu}, we need the following assumptions on the kernel $\kappa^*$:
\begin{assumption} \label{assumption:moment:bounded} Fix $\nu > 2$. Assume $\mean[\kappa^*(\bV_1,\bV_1)] < \infty$ and let  $\{\lambda_k\}_{k=1}^\infty$ and $\{\phi_k\}_{k=1}^\infty$ be a weak Mercer representation of $\kappa^*$ under $R$. Also assume that for some $\nu^* > \nu$, $\| \kappa^*(\bV_1,\bV_2) \|_{L_{\nu^*}} < \infty$ and $\sup_{K \geq 1} \| \sum_{k=1}^K \lambda_k \phi_k(\bV_1) \phi_k(\bV_2) \|_{L_{\nu^*}} < \infty$ . 
\end{assumption}

For MMD, we can use the weak Mercer representation of $\ummd$ to show that our results apply:
\begin{lemma} \label{lem:mmd:general} $\ummd$ defines a kernel on $\R^{2d}$. Moreover, if \cref{assumption:moment:bounded} holds for $\kappa^*=\ummd$ under $P \otimes Q$ for some $\nu > 2$, then \cref{assumption:L_nu} holds for $\min\{\nu,3\}$ with $u=\ummd$ and $R=P \otimes Q$.
\end{lemma}

In the case of KSD, we use the representation of $\kappa$ directly. We require some additional assumptions for the score function $\nabla \log p(\bx)$ to be well-behaved and the differential operation on $\kappa$ to behave well under the representation.
\begin{assumption} \label{assumption:ksd} Fix $n$, $d$ and $\nu > 2$. Assume that \cref{assumption:moment:bounded} holds with $\nu$ for $\kappa$ under $Q$, with $\{\lambda_k\}_{k=1}^\infty$ and $\{\phi_k\}_{k=1}^\infty$ as the weak Mercer representation of $\kappa$ under $Q$ and $\nu^*$ being defined as in \cref{assumption:moment:bounded}. Further assume that (i) $\| \| \nabla \log p(\bX_1) \|_2  \|_{L_{2\nu^{**}}} < \infty$ for $\nu^{**}  = \frac{\nu(\nu+\nu^*)}{\nu^* - \nu}$ ; (ii) $\sup_{k \in \N} \| \phi_k(\bX_1) \|_{L_{2\nu}} < \infty$; (iii) $\phi_k$'s are differentiable with $\sup_{k \in \N} \| \|\nabla \phi_k(\bX_1)\|_2 \|_{L_{\nu}} < \infty$; (iv) As $K\rightarrow \infty$, we have the convergence $ \big\| \big\| \sum_{k=1}^K \lambda_k (\nabla \phi_k(\bX_1)) \phi_k(\bX_2) - \nabla_1 \kappa(\bX_1,\bX_2) \big\|_2  \big\|_{L_{2\nu}} \rightarrow 0$ as well as the convergence
$ \big\| \sum_{k=1}^K \lambda_k (\nabla \phi_k(\bX_1))^\top (\nabla \phi_k(\bX_2)) - \Tr(\nabla_1 \nabla_2 \kappa(\bX_1,\bX_2)) \big\|_{L_{\nu}}  \rightarrow 0$.
\end{assumption}
We can now form a decomposition of $\uPksd$. Given $\{\lambda_k\}_{k=1}^\infty$ and $\{\phi_k\}_{k=1}^\infty$ from \cref{assumption:ksd} and any fixed $d \in \N$, define the sequences $\{\alpha_k\}_{k=1}^\infty$ and $\{\psi_k\}_{k=1}^\infty$ as, for $1 \leq l \leq d$ and $k' \in \N$, \vspace{-.3em}
\begin{align*} 
    \alpha_{(k'-1) d + l} \;\coloneqq&\; \lambda_{k'}\;
    &\text{ and }&&
    \psi_{(k'-1) d + l}(\bx)
    \;\coloneqq&\; (\partial_{x_l} \log p(\bx)) \phi_{k'}(\bx) + \partial_{x_l} \phi_{k'}(\bx)
    \;.
    \tagaligneq \label{eqn:ksd:spectral}
\end{align*} \vspace{-2.2em}
\begin{lemma} \label{lem:ksd:general} If \cref{assumption:ksd} holds for some $\nu > 2$, then \cref{assumption:L_nu} holds for $\min\{\nu,3\}$ with $u=\uPksd$, $R=Q$, $\lambda_k=\alpha_k$ and $\phi_k=\psi_k$\;.
\end{lemma} \vspace{-1em}
\begin{remark} \label{remark:assumption:L_nu:kernel} The benefits of formulating our results in terms of \cref{assumption:L_nu} are now clear: By forgoing orthonormality, we can choose a functional decomposition e.g.~in terms of the Mercer representation of a kernel, which is already widely considered in this literature. The non-negative eigenvalues from Lemma \ref{lem:mercer} also allow
moment-matching methods discussed in \cref{sect:chi:sq} to be considered. In fact, a Mercer representation is not even necessary: In \cref{appendix:rbf:decompose}, we construct a simple decomposition for the setup in \cref{sect:examples} such that \cref{assumption:L_nu} can be verified easily.
\end{remark}

\subsection{Gaussian mean-shift examples} \label{sect:examples}

We study KSD and MMD under Gaussian mean-shift, where $P = \cN(0, \Sigma)$ and $Q = \cN(\bmu, \Sigma)$ with mean $\mu \in \R^d$ and covariance $\Sigma \in \R^{d \times d}$ to be specified. Two simple kernels are considered in this section, namely the \emph{RBF} kernel and the \emph{linear} kernel.


\paragraph{RBF kernel.} We consider the RBF kernel $\kappa(\bx, \bx')=\exp(-\| \bx - \bx' \|_2^2/(2 \gamma))$, where $\gamma = \gamma(d)$ is a bandwidth potentially depending on $d$. A common strategy to choose $\gamma$ is the \emph{median heuristic}: \vspace{-.5em}
$$
    \gamma_{\textrm{med}}
    \;\coloneqq\;
    \textrm{Median}\left\{ \| \bV - \bV' \|_2^2: \bV, \bV' \in \cV\,,\; \bV \neq \bV' \right\} \;,
    \vspace{-.5em}
$$
where the samples $\cV = \{\bX_i\}_{i=1}^n$ for KSD and $\cV = \{\bX_i\}_{i=1}^n \cup \{\bY_i\}_{i=1}^n$ for MMD. In \cref{appendix:gaussian:mean:shift}, we include a further discussion of this setup as well as verification of \Cref{assumption:moment:ratio} and \Cref{assumption:L_nu}.

\vspace{.2em}

We focus on 
$\Sigma = I_d$, where the $d$-dependence of the moment ratio $\rho_d$ can be explicitly studied for both KSD and MMD. Importantly, we give bounds in terms of the bandwidth $\gamma$ and the scale of mean shift $\|\mu\|_2^2$, which reveal their effects on $\rho_d$ and thereby on the behaviour of the test power. The assumptions on $\gamma$ and $\|\mu\|_2^2$ in both propositions are for simplicity rather than necessity. \vspace{-.2em}

\begin{proposition}[KSD-RBF moment ratio]
\label{prop:ksd:var:ratio}
    Assume $\gamma = \omega(1)$ and $\| \mu \|_2^2 = \Omega(1)$. Under the Gaussian mean-shift setup with $\Sigma = I_d$, the KSD U-statistic satisfies that 
    \begin{proplist}
        \item If $\gamma = o(d^{1/2})$, then $\rho_d = \exp\big( \frac{3d}{4\gamma^2} + o\big( \frac{d}{\gamma^2} \big) \big) \, \Theta\Big( \frac{d}{\gamma \| \bmu \|_2^2} + \frac{d^{1/2}}{\gamma^{1/2} \| \bmu \|_2} + 1 \Big)$\;;
        \item If $\gamma = \omega(d^{1/2})$, then $\rho_d = \Theta\Big( \frac{d^{1/2} (1 + \gamma^{-1/2} \|\mu\|_2)}{\|\mu\|_2 \, ( 1 + \gamma^{-1} d^{1/2} \|\mu\|_2)} + 1  \Big)$\;;
        \item If $\gamma = \Theta(d^{1/2})$, then $\rho_d = \Theta\Big( \frac{d^{1/2}}{\| \bmu \|_2^2}
        + \frac{d^{1/4}}{ \| \bmu \|_2}
        + 1 \Big)$\;.
    \end{proplist} 
\end{proposition} \vspace{-1.3em}
\begin{proposition}[MMD-RBF moment ratio]
\label{prop:MMD:variance:ratio} Consider the Gaussian mean-shift setup with $\Sigma = I_d$ and assume $\gamma = \omega(1)$ and $\| \mu \|_2^2 = \Omega(1)$. For the MMD U-statistic, if $\gamma=o(\|\mu\|_2^2)$ and $\gamma = o(d^{1/2})$, then $\rho_d = \Theta\big( \exp\big(\frac{3d}{4\gamma^2} + o\big( \frac{d}{\gamma^2} \big) \big) \big)$. If instead $\gamma = \omega(\| \mu \|_2^2)$, then
    \begin{proplist}
        \item For $\gamma = o(d^{1/2})$, we have $\rho_d = \Theta\Big( \frac{\gamma}{ \|\mu\|_2^2} \exp\Big(
        \frac{3d}{4\gamma^2} + o\big(\frac{d}{ \gamma^{2}}\big)
        \Big)  \Big)$\;;
        \item For $\gamma = \omega(d^{1/2})$, we have $\rho_d = \Theta\Big( \frac{ \|\mu\|_2 \,+\, d^{1/2} }{ \|\mu\|_2 +  \gamma^{-1}  d^{1/2} \|\mu\|_2^2 } \Big)$\;;
        \item For $\gamma = \Theta(d^{1/2})$, we have $\rho_d = O\Big( \,  \frac{ d^{1/2} }{\|\mu\|_2^2 } \, \Big)$ \;.
    \end{proplist}
\end{proposition} \vspace{-.5em} \renewcommand{\thefootnote}{\fnsymbol{footnote}} 
The case $\|\mu\|_2 = \Omega(\|\Sigma\|_2) = \Omega(d^{1/2})$ is not very interesting, as it means that the signal-to-noise ratio (SNR) is high and can even increase with~$d$. WLOG we focus on a low SNR setting with $\| \mu \|_2 = \Theta(1)$. In this case, it has been shown that the median-heuristic bandwith scales as $\gamma_{\textrm{med}}=\Theta(d)$ \citep{reddi2015high, ramdas2015decreasing, wynne2022kernel}. While Propositions  \ref{prop:ksd:var:ratio} and \ref{prop:MMD:variance:ratio} do not directly address the case $\gamma = \gamma_{\textrm{med}}$ due to its data dependence, they do show that $\rho_d = \Theta(d^{1/2})$ for both KSD and MMD with a data-\emph{independent} bandwidth $\gamma=\Theta(d)$\footnote[2]{In our experiments, the data-independent choice $\gamma=d$ and the data-dependent $\gamma=\gamma_{\rm med}$ yield almost identical plots.}. In this case, the asymptotic distributions of $\ksd_n$ and $\mmd_n$ are \emph{(i)} the non-degenerate Gaussian limit predicted by \eqref{eqn:berry:esseen} when $d = o(n)$ and \emph{(ii)} the degenerate limit from Proposition \ref{prop:u:chi:sq:general} when $d = \omega(n)$.

\begin{figure} 
    \centering   
    \begin{minipage}[t]{0.485\textwidth}
        \centering
        \includegraphics[width=1.\textwidth]{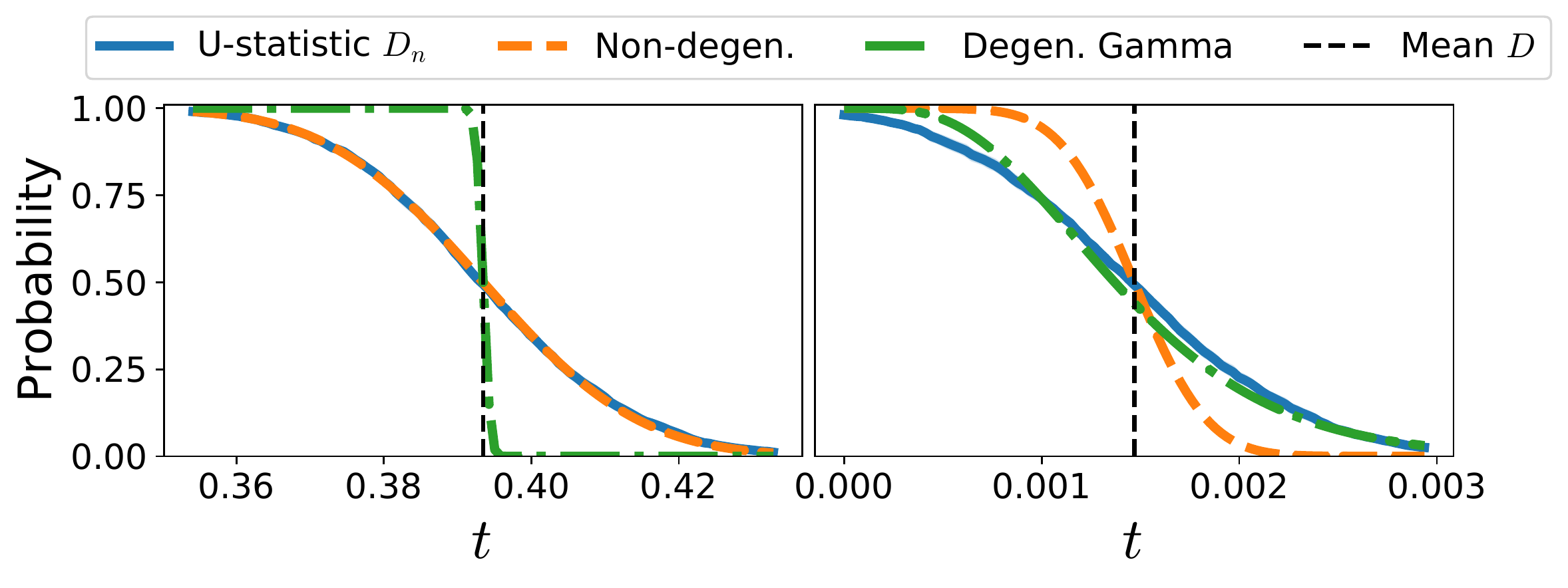}
        \vspace{-2.em}
        \caption{Behaviour of $\P(X > t)$ for $X=\mmd_n$ with the RBF kernel versus $X$ being the theoretical limits. \emph{Left}. $n = 1000$ and $d = 2$. \emph{Right}. $n = 50$ and $d = 1000$.}
        \label{fig:power:mmd:rbf}
    \end{minipage}
    \hfill
    \begin{minipage}[t]{0.485\textwidth}
        \centering
        \includegraphics[width=1.\textwidth]{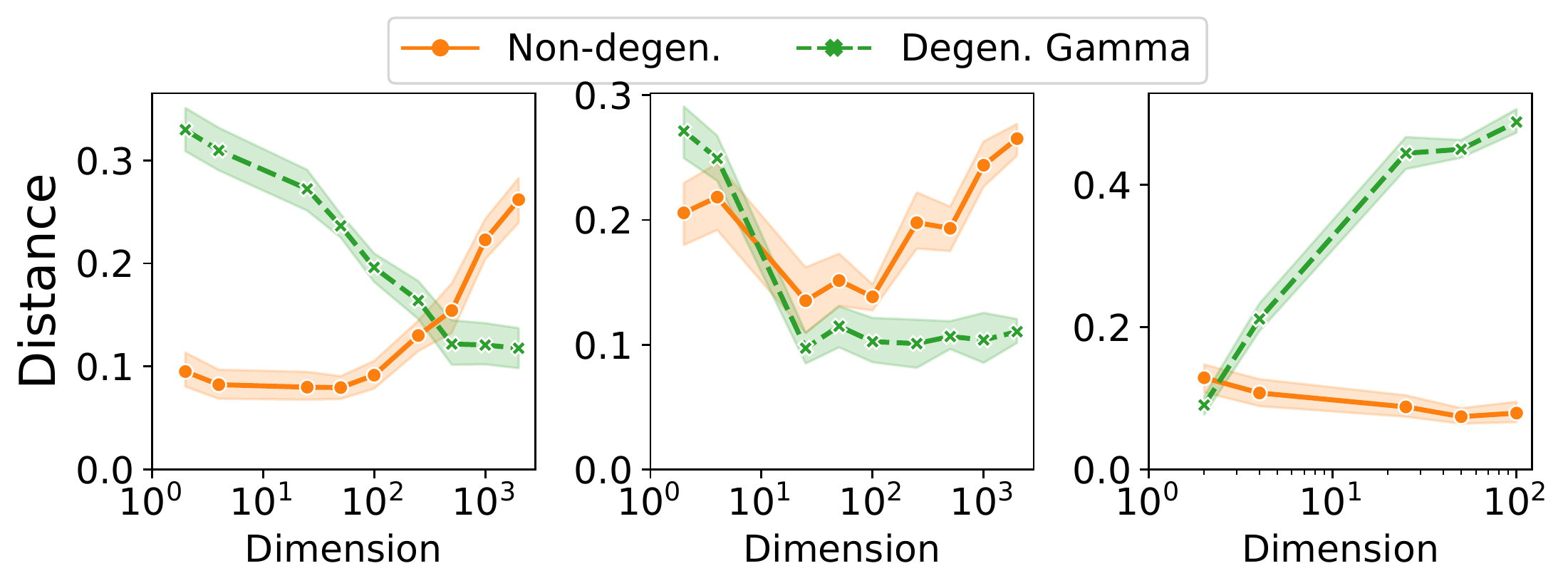}
        \vspace{-2.em}
        \caption{$L_\infty$ distance between the c.d.f.\ of $\ksd_n$ with RBF and those of the theoretical limits as $d$ varies. \emph{Left.} $n = 50$ fixed (high dimensions). \emph{Middle.} $n = \Theta(d^{1/2})$ (high dimensions). \emph{Right:} $n = \Theta(d^2)$ (low dimensions).}
        \label{fig:dist:vs:dim:ksd}
    \end{minipage} 
    \centering
    \begin{minipage}[t]{\textwidth}
        \vspace{1em}
        \centering
        \includegraphics[width=1.\textwidth]{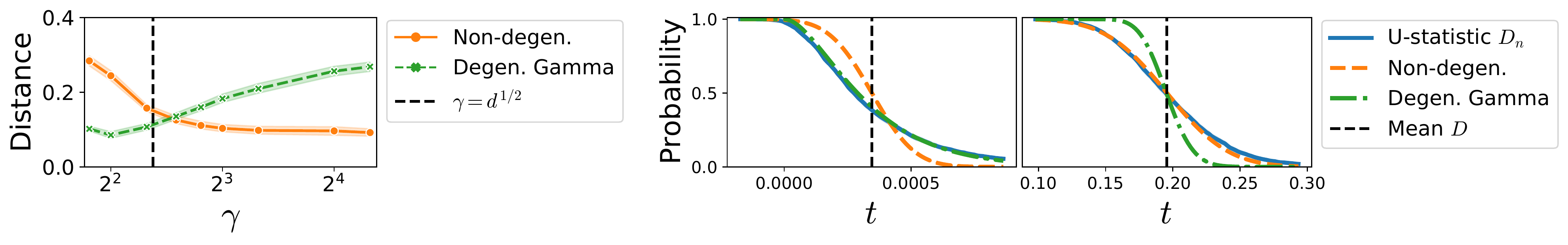}
        \vspace{-2.5em}
        \caption{Behaviour of $\P(\ksd_n > t)$ with RBF as $\gamma$ varies for $n=50$ and $d=27$. \emph{Left.} $L_\infty$ distance between the c.d.f.\ of $\ksd_n$ and the theoretical limits. \emph{Middle.} Distribution curves at $\gamma=4$. \emph{Right.} Distribution curves at $\gamma=16$.} \label{fig:gamma:vary:ksd}
    \end{minipage}
\end{figure}

\vspace{.2em}

Intriguingly, in both results, different regimes arise based on how $\gamma$ compares with the noise scale $\|\Sigma\|_2 = d^{1/2}$. In fact, a phase transition as $\gamma$ drops from $\omega(d^{1/2})$ to $o(d^{1/2})$ has been reported in \cite{ramdas2015decreasing} but with no further comments\footnote[3]{Their bandwidth $\gamma_{\rm Ramdas}$ is defined to equal our $\sqrt{2\gamma}$. Phase transition occurs at $\gamma_{\rm Ramdas}=d^{1/4}$ in their Figure 1. While their figure is for MMD with threshold chosen by a permutation test, ours is for KSD with a fixed threshold.}\footnote[4]{This was investigated in \citet[Section 10.4]{ramdas2015computational} in a special case when $\gamma = \omega(\|\bmu\|_2^2 + d)$ (case \emph{(ii)} of \Cref{prop:MMD:variance:ratio}) and $n = o(d^{5/2})$, where the author derived the test power of the RBF-kernel MMD for different SNRs.}. Our results offer one explanation: Such transitions may happen due to a change in the dependence of $\rho_d$ on $\gamma$, $\|\mu\|_2$ and $d$. \cref{fig:gamma:vary:ksd} shows a transition across different limits as $\gamma$ varies, where the transition occurs at around $\gamma \sim d^{1/2}$.

\paragraph{Linear kernel.} \cref{sect:chi:sq} discussed that the limit of $D_n$ can be non-Gaussian. This is true for MMD with a \emph{linear} kernel $\kappa(\bx, \bx') = \bx^\top \bx'$: It satisfies Lemma \ref{lem:chisq:limit:not:gaussian} with $K_*=d$ and the limit is a shifted-and-rescaled chi-square. \cref{fig:intro} verifies this for some $\Sigma \neq I_d$ by showing an asymmetric distribution curve close to the chi-square limit. We remark that a linear kernel, while not commonly used, is a valid choice here since $\mmd=0$ iff $P=Q$ under our setup.

\paragraph{Simulations.} We set $\mu = (2,0,\ldots,0)^\top \in \R^d$, $\Sigma=I_d$ and $\gamma=\gamma_{\rm med}$ for KSD with RBF and MMD with RBF. The exact setup for MMD with linear kernel is described in \cref{appendix:mmd:linear:moments}. The limits for comparison are the non-degenerate Gaussian limit in \eqref{eqn:berry:esseen} (``Non-degen.") and Gamma / shifted-and-rescaled chi-square (``Degen. Gamma" / ``Degen. Chi-square") distributions that match the degenerate limit in Proposition \ref{prop:u:chi:sq:general} by mean and variance. \cref{fig:intro} plots the distribution curves for KSD with RBF and MMD with linear kernel. \cref{fig:power:mmd:rbf} plots the same quantity for MMD with RBF. \cref{fig:dist:vs:dim:ksd} and \cref{fig:gamma:vary:ksd} examine the behaviour of KSD with RBF as $d$ or $\gamma$ varies (as a data-independent function of $d$, similar to \cite{ramdas2015decreasing}). Results involving $D_n$ are averaged over 30 random seeds, and shaded regions are $95\%$ confidence intervals\footnote[5]{The shaded regions are not visible for $\P(D_n > t)$ in \cref{fig:intro}, \ref{fig:power:mmd:rbf} and \ref{fig:gamma:vary:ksd} as the confidence intervals are very narrow.}. Code for reproducing all experiments can be found at \href{https://github.com/XingLLiu/u-stat-high-dim.git}{\texttt{github.com/XingLLiu/u-stat-high-dim.git}}.

\acks{KHH is supported by the Gatsby Charitable Foundation.
XL is supported by the President’s PhD Scholarships
of Imperial College London and the EPSRC StatML
CDT programme EP/S023151/1.
ABD is supported by Wave 1 of The UKRI Strategic Priorities Fund under the EPSRC Grant EP/T001569/1 and EPSRC Grant EP/W006022/1, particularly the “Ecosystems of Digital Twins” theme within those grants \& The Alan Turing Institute. We thank Antonin Schrab, Heishiro Kanagawa and Arthur Gretton for their helpful comments.}

\newpage
\bibliography{ref}

\appendix

\newpage
\noindent
The appendix is organised as follows. The first few appendices provide additional content:


\noindent
\textbf{\cref{appendix:gaussian:mean:shift}} states additional results for \cref{sect:examples} including moment computations and verification of assumptions.

\noindent
\textbf{\cref{appendix:tools}} presents auxiliary tools used in subsequent proofs.

\vspace{1em}

\noindent 
The remaining appendices consist of proofs:

\noindent
\textbf{\cref{appendix:main}} proves our main theorem. \cref{appendix:u:gaussian:quad:lemma} provides a list of intermediate lemmas that extends the proof overview in \cref{sect:proof:overview}.

\noindent
\textbf{\cref{appendix:proof:remaining}} proves the remaining results in \cref{sect:general:results}.

\noindent
\textbf{\cref{appendix:proof:ksd:mmd}} proves the results in \cref{sect:ksd:mmd}.

\noindent
\textbf{Appendices \ref{appendix:proof:gaussian:mean:shift}} and \textbf{\ref{appendix:proof:tools}} present proofs for the results in Appendices 
\ref{appendix:gaussian:mean:shift} and \ref{appendix:tools} respectively.

\vspace{1em}

\noindent
Throughout the appendix, we say that $C$ is an absolute constant whenever we mean that it is a number independent of all variables involved, including $u$, $R$, $d$, $n$ and $K$.

\section{Additional results for Gaussian mean-shift} \label{appendix:gaussian:mean:shift}

In this section, we consider the Gaussian mean-shift setup defined in \Cref{sect:examples}, where $Q = \cN(\bmu, \Sigma)$ and $P = \cN(0, \Sigma)$ with mean $\bmu \in \R^d$ and covariance matrix $\Sigma \in \R^{d \times d}$. We derive analytical expressions of the moments of U-statistics for (i) KSD with RBF, (ii) MMD with RBF and (iii) MMD with linear kernel. We also verify \cref{assumption:L_nu} for the three cases, which confirm that our error bounds apply. 

\paragraph{Remark on verification of \cref{assumption:moment:ratio}.} Recall that \cref{assumption:moment:ratio}, which controls the moment ratios $\Mfullnu / \sigfull$ and $\Mcondnu / \sigcond$ for some $\nu \in (2,3]$, is required for our bounds to imply a convergence theorem. As discussed in the main text, this issue is not specific to our theorem and is also relevant to e.g.~Berry-Ess\'een bounds for sample averages of $\{f(\bX_i)\}_{i=1}^n$ for $f:\R^d \rightarrow \R$ and $d$ large. A detailed verification requires a careful calculation to control the order of $\Mfullnu, \sigfull, \Mcondnu$ and $\sigcond$. For KSD and MMD with the RBF kernel, a careful control of $\sigfull$ and $\sigcond$ has already been done in the proof of Proposition \ref{prop:ksd:var:ratio} and Proposition \ref{prop:MMD:variance:ratio}, which involves examining multiple cases depending on the relative sizes of $\gamma$, $\|\mu\|_2$ and $d$ followed by an elaborate calculation. To perform this verification for all cases in full generality, in principle, one may expand on those calculations and follow a similar tedious argument. In the sections below, we perform this verification only for the setup in \cref{fig:intro}, i.e.~KSD with the RBF kernel in the case $\|\mu\|_2=\Theta(1)$ and $\gamma=\Omega(d)$ and MMD with the linear kernel in the general case. For MMD with the RBF kernel, we discuss the relevance of this verification to \cite{gao2021two}, who has done a verification of similar quantities but also in a special case. In \Cref{fig:moment:ratios}, we also include simulations verifying \Cref{assumption:moment:ratio} under the setups considered in \Cref{fig:intro}-\ref{fig:dist:vs:dim:ksd}, where we demonstrate that the moment ratios stay around 1 as the dimension varies from 1 to 2000.

\subsection{A decomposition of the RBF kernel} \label{appendix:rbf:decompose}

For both MMD and KSD, the key in verifying the assumptions for the RBF kernel is a functional decomposition. The usual Mercer representation of the RBF kernel is available only with respect to a univariate zero-mean Gaussian measure and involves some cumbersome Hermite polynomials. Since we do not actually require orthogonality of the functions in \cref{assumption:L_nu}, we opt for a simpler functional representation as given below. We also assume WLOG that the bandwidth $\gamma > 8$, since we only consider the case $\gamma=\omega(1)$ in our setup.

\begin{lemma} \label{lem:rbf:decomposition} Assume that $\gamma > 8$. Consider two independent $d$-dimensional Gaussian vectors $\bU \sim \cN(\mu_1, I_d)$ and $\bV \sim \cN(\mu_2, I_d)$ for some mean vectors $\mu_1,\mu_2 \in \R^d$. Then, for any $\nu \in (2,4]$ and $\mu_1,\mu_2 \in \R^d$, we have that
\begin{align*}
    \mean\Big[ \Big| \exp\Big( - \mfrac{1}{2\gamma} \| \bU - \bV\|_2^2 \Big) - 
    \mprod_{j=1}^d \Big(
    \msum_{k=0}^K \lambda^*_{k} \phi^*_{k}(U_j) \phi^*_{k}(V_j)
    \Big)
    \Big|^\nu \Big]
    \;\xrightarrow{K \rightarrow \infty}\;
    0\;.
\end{align*}
where $\phi^*_{k}(x) \coloneqq x^k e^{- x^2 / (2\gamma) }$ and $\lambda^*_{k} \coloneqq \frac{1}{k! \, \gamma^k} $ for each $k \in \N \cup \{0\}$.
\end{lemma}

To see that Lemma \ref{lem:rbf:decomposition} indeed gives the functional decomposition we want in \cref{assumption:L_nu}, we need to rewrite the product of sums into a sum. To this end, let $g_d$ be the $d$-tuple generalisation of the Cantor pairing function from $\N$ to $(\N \cup \{0\} )^d$ and $[g_d(k)]_l$ be the $l$-th element of $g_d(k)$. Given $\{\lambda^*_l\}_{l=0}^\infty$ and $\{\phi^*_l\}_{l=0}^\infty$ from Lemma \ref{lem:rbf:decomposition}, we define, for every $k \in \N$ and $\bx = (x_1, \ldots, x_d) \in \R^d$,
\begin{align*}
    \alpha_{k}
    \;\coloneqq&\;
    \mprod_{l=1}^d \lambda^*_{[g_d(k)]_l}
    &\text{ and }&&
    \psi_{k}(\bx)
    \;\coloneqq&\;
    \mprod_{l=1}^d \phi^*_{[g_d(k)]_l}(x_l)\;. \tagaligneq \label{eqn:rbf:decomposition:alt:defn}
\end{align*}
With this construction, for each $K \in \N$, we can now write
\begin{align*}
    \mprod_{j=1}^d \Big(
    \msum_{k=0}^K &\lambda^*_{k} \phi^*_{k}(U_j) \phi^*_{k}(V_j)
    \Big)
    \\
    \;=&\;
    \msum_{k_1, \ldots, k_d=0}^K (\lambda^*_{k_1} \ldots \lambda^*_{k_d}) (\phi^*_{k_1}(U_1) \ldots \phi^*_{k_d}(U_d) ) (\phi^*_{k_1}(V_1) \ldots \phi_{k_d}(V_d) )
    \\
    \;=&\;
    \msum_{k_1, \ldots, k_d=0}^K \; \alpha_{g_d^{-1}(k_1, \ldots, k_d )}
    \; \psi_{g_d^{-1}(k_1, \ldots, k_d )}(\bU)
    \; \psi_{g_d^{-1}(k_1, \ldots, k_d )}(\bV)\;.
\end{align*}
Since the Cantor pairing function is such that $\min_{l \leq d} [g_d(K)]_l \rightarrow \infty$ as $K \rightarrow \infty$, Lemma \ref{lem:rbf:decomposition} indeed gives a functional decomposition in terms of $\{\alpha_k\}_{k=1}^\infty$ and $\{\psi_k\}_{k=1}^\infty$ as
\begin{align*}
    \mean\Big[ \Big| \exp\Big( - \mfrac{1}{2\gamma} \| \bU - \bV\|_2^2 \Big) - 
    \msum_{k=1}^K \alpha_{k} \psi_{k}(\bU) \psi_{k}(\bV)
    \Big|^\nu \Big]
    \;\xrightarrow{K \rightarrow \infty}\;
    0\;. \tagaligneq \label{eqn:rbf:decomposition:alt}
\end{align*}
We remark that both $\alpha_k$ and $\psi_k$ are independent of the mean vectors $\mu_1$ and $\mu_2$, which makes this representation useful for a generic mean-shift setting.

\subsection{KSD U-statistic with RBF kernel}
\label{appendix:ksd:rbf:moments}
Under the Gaussian mean-shift setup with an identity covariance matrix, gradient of the log-density is given by $\nabla \log p(\bx) = - \bx$ for $\bx \in \R^d$ and the U-statistic for the RBF-kernel KSD is
\begin{align*}
    \uPksd(\bx, \bx')
    \;=&\; 
    \big( \nabla \log p(\bx) \big)^\top \big( \nabla \log p(\bx') \big) \kappa(\bx, \bx') \;+\; \big( \nabla \log p(\bx) \big)^\top  \nabla_2 \kappa(\bx, \bx') 
    \\
    & + \big( \nabla \log p(\bx') \big)^\top  \nabla_1 \kappa(\bx, \bx') \;+\; \Tr(\nabla_1 \nabla_2 \kappa(\bx, \bx'))
    \\
    \;=&\;
    \exp\Big( - \mfrac{\| \bx - \bx' \|_2^2 }{2\gamma} \Big)
    \Big(
        \bx^\top \bx' 
        + \mfrac{1}{\gamma} \bx^\top (\bx' - \bx) 
        + \mfrac{1}{\gamma} (\bx')^\top (\bx - \bx')
    + \Big( \mfrac{d}{\gamma} - \mfrac{\| \bx - \bx' \|_2^2}{\gamma^2} \Big) 
    \Big)
    \\
    \;=&\;
    \exp\Big( - \mfrac{ \| \bx - \bx' \|_2^2 }{2\gamma}  \Big)
    \Big( 
        \bx^\top \bx' 
        - \mfrac{\gamma + 1}{\gamma^2} \| \bx - \bx' \|_2^2
        + \mfrac{d}{\gamma}
    \Big)
    \;. \tagaligneq \label{eqn:defn:ksd:rbf}
\end{align*}
We first verify that \cref{assumption:L_nu} holds by adapting $\{\alpha_k\}_{k=1}^\infty$ and $\{\psi_k\}_{k=1}^\infty$ from \cref{appendix:rbf:decompose}.

\begin{lemma} \label{lem:ksd:assumption:L_nu} Assume that $\gamma > 24$. For $k' \in \N$, consider
\begin{align*} 
    &\lambda_{(k'-1)(d+3) + 1} 
    \;=\;
    - \mfrac{\gamma+1}{\gamma^2} \alpha_{k'}
    \;,
    &&
    \phi_{(k'-1)(d+3) + 1} (\bx) 
    \;=\;
    \psi_{k'}(\bx) (\|\bx\|_2^2 + 1)
    \;,
    \\
    &\lambda_{(k'-1)(d+3) + 2} 
    \;=\;
    \mfrac{\gamma+1}{\gamma^2} \alpha_{k'}
    \;,
    &&
    \phi_{(k'-1)(d+3) + 2} (\bx) 
    \;=\;
    \psi_{k'}(\bx) \|\bx\|_2^2 
    \;,
    \\
    &\lambda_{(k'-1)(d+3) + 3} 
    \;=\;
    \Big( \mfrac{d}{\gamma} + \mfrac{\gamma+1}{\gamma^2} \Big) \alpha_{k'}
    \;,
    &&
    \phi_{(k'-1)(d+3) + 3}  (\bx) 
    \;=\;
    \psi_{k'}(\bx)
    \;,
\end{align*} 
and for $l=1, \ldots, d$, define
\begin{align*}
    &\lambda_{(k'-1)(d+3) + 3 + l} 
    \;=\;
    \mfrac{\gamma^2+2\gamma+2}{\gamma^2} \alpha_{k'}
    \;,
    &&
    \phi_{(k'-1)(d+3) + 3 + l}  (\bx) 
    \;=\;
    \psi_{k'}(\bx) x_l
    \;.
\end{align*}
Then \cref{assumption:L_nu} holds with any $\nu \in (2,3]$ for $u=\uPksd$, $\{\lambda_k\}_{k=1}^\infty$ and $\{\phi_k\}_{k=1}^\infty$ defined above.
\end{lemma}

\vspace{.2em}

The following result (proved in Appendix \ref{pf:ksd:moments:analytical}) provides analytical forms or upper bounds for the moments of KSD U-statistic. 

\begin{lemma}[KSD moments]
\label{lem:ksd:moments:analytical}
    Let $\kappa$ be a RBF kernel with bandwidth $\gamma = \omega(1)$, and let $\bX, \bX'$ be independent draws from $Q$. Under the mean-shift setup with an identity covariance matrix, it follows that
    \begin{proplist}
        \item For every $\bx \in \R^d$,
        \begin{align*}
            \gksd(\bx)
            \;\coloneqq&\;
            \E[ \uPksd(\bx, \bX') ]
            \\
            \;=&\;
            \left( \mfrac{\gamma}{\gamma + 1} \right)^{d/2} \exp\left( - \mfrac{1}{2 (\gamma + 1)} \| \bx - \bmu \|_2^2 \right) \left( 
                \mfrac{2 + \gamma}{1 + \gamma} \bmu^\top \bx 
                - \mfrac{1}{1 + \gamma} \| \bmu \|_2^2
            \right) \;;
        \end{align*}
        \item The mean is given by $\ksd(Q, P) = \left( \mfrac{\gamma}{\gamma + 2} \right)^{d/2} \| \mu \|_2^2$\;;
        \item The variance of the conditional expectation $\gksd(\bX)$ is given by
        \begin{align*}
            \sigcond^2 
            \;=&\;
            \left( \mfrac{\gamma^2}{(1 + \gamma)(3 + \gamma)} \right)^{d/2}
            \bigg(
            \mfrac{(2 + \gamma)^2}{(1 + \gamma)(3 + \gamma)} \| \bmu \|_2^2
            + \left( 1 - \left( \mfrac{(1 + \gamma)(3 + \gamma)}{(2 + \gamma)^2} \right)^{d/2} \right) \| \bmu \|_2^4 \bigg)
            \;;
        \end{align*}
        \item The variance of $\uPksd(\bX, \bX')$ is given by
        \begin{align*}
            \sigfull^2
            \;=&\;
            \left( \mfrac{\gamma}{4 + \gamma} \right)^{d/2}
            \bigg(
                d 
                + \mfrac{d^2}{\gamma^2}
                + \mfrac{2d \| \bmu \|_2^2}{\gamma}
                + 2\| \bmu \|_2^2
                + \left(1 - \left( \mfrac{\gamma (4 + \gamma)}{(2 + \gamma)^2} \right)^{d/2} \right) \| \bmu \|_2^4
                \\
                &\qquad\qquad\qquad\;
                 + o\left( 
                    d 
                    + \mfrac{d^2}{\gamma^2}
                    + \mfrac{d \| \bmu \|_2^2}{\gamma}
                    +  \| \bmu \|_2^2
                \right)
            \bigg)\;;
        \end{align*}
        \item For any $\nu > 2$, there exist positive constants $C_1, C_2$ depending only on $\nu$ such that the $\nu$-th absolute moment of the conditional expectation satisfies
        \begin{align*}
            \E[ | \gksd(\bX) |^\nu ]
            \;\leq\;
            \left( \mfrac{\gamma}{1 + \gamma} \right)^{\nu d / 2 }
            \left( \mfrac{1 + \gamma}{1 + \nu + \gamma} \right)^{d/2}
            ( C_1 \| \mu \|_2^\nu + C_2 \| \mu \|_2^{2\nu} ) 
            \;.
        \end{align*}
        \item For any $\nu > 2$, there exist positive constants $C_3, C_4, C_5, C_6$ depending only on $\nu$ such that  the $\nu$-th absolute moment of $\uPksd(\bX, \bX')$ satisfies
        \begin{align*}
            \E[ | \uPksd(\bX, \bX') |^\nu ]
            \;&\leq\;
            \left( \mfrac{\gamma}{2\nu + \gamma} \right)^{d/2}
            \Big(
                C_3 d^{\nu / 2} 
                + C_4 \Big( \mfrac{d}{\gamma} \Big)^\nu
                + C_5 \| \bmu \|_2^\nu
                + C_6 \| \bmu \|_2^{2\nu}
                \\
                &\qquad\qquad\qquad\quad\;\;
                + o\Big( 
                    d^{\nu / 2}
                    + \mfrac{d^\nu}{\gamma^\nu} 
                    + \mfrac{\| \bmu \|_2^{2\nu}}{\gamma^\nu} 
                \Big)
            \Big)
            \;.
        \end{align*}
    \end{proplist}
    In particular, when $\|\mu\|_2=\Theta(1)$ and $\gamma =\Omega(d)$, \cref{assumption:moment:ratio} holds with any $\nu>2$ for $\uPksd$ under $Q$.
\end{lemma}

\begin{figure}
    \centering
    \includegraphics[width=0.8\textwidth]{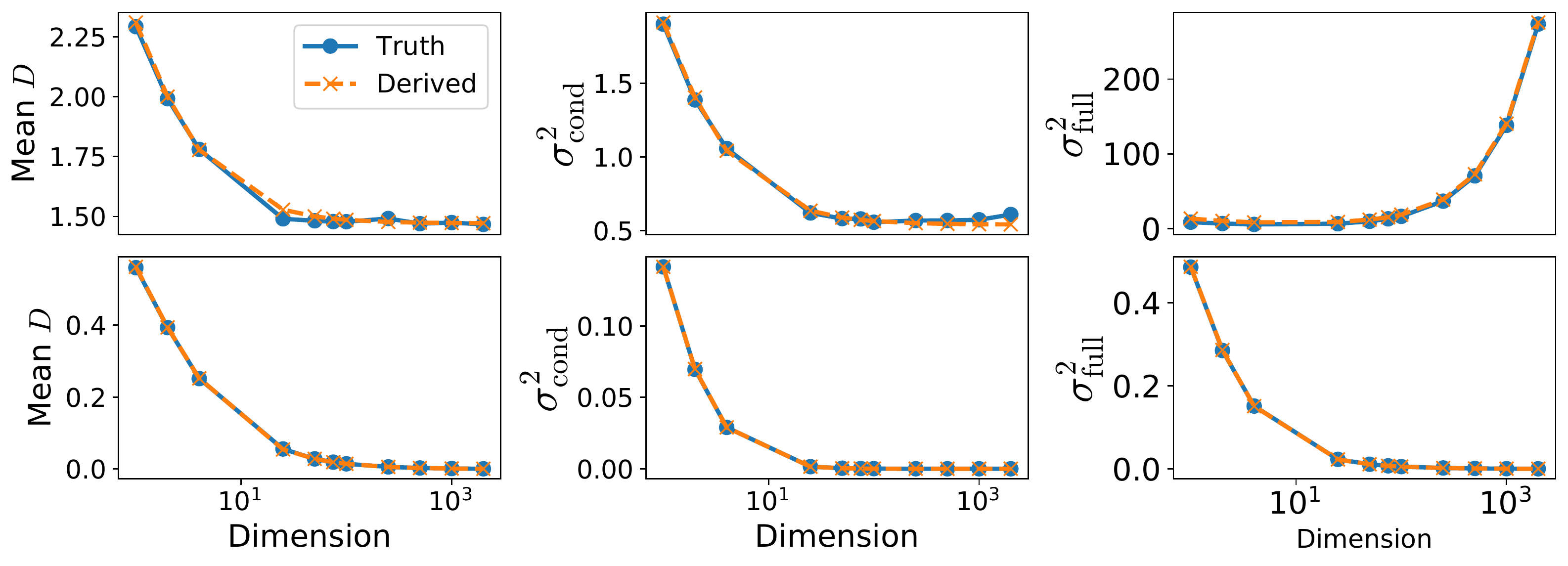}
    \caption{Verifying the analytical expressions for the first two moments of KSD and MMD. \emph{Top.} KSD moments derived in Lemma \ref{lem:ksd:moments:analytical}. \emph{Bottom.} MMD moments derived in Lemma \ref{lem:mmd:moments:analytical}. The ground truth is estimated using $n=4000$ samples for KSD and $n=10000$ samples for MMD, respectively, and the reported results are averaged over 5 random seeds.}
    \label{fig:mmd:moments}
\end{figure}


\begin{figure}
    \centering
    \begin{minipage}{0.47\textwidth}
        \centering
        \includegraphics[width=\textwidth]{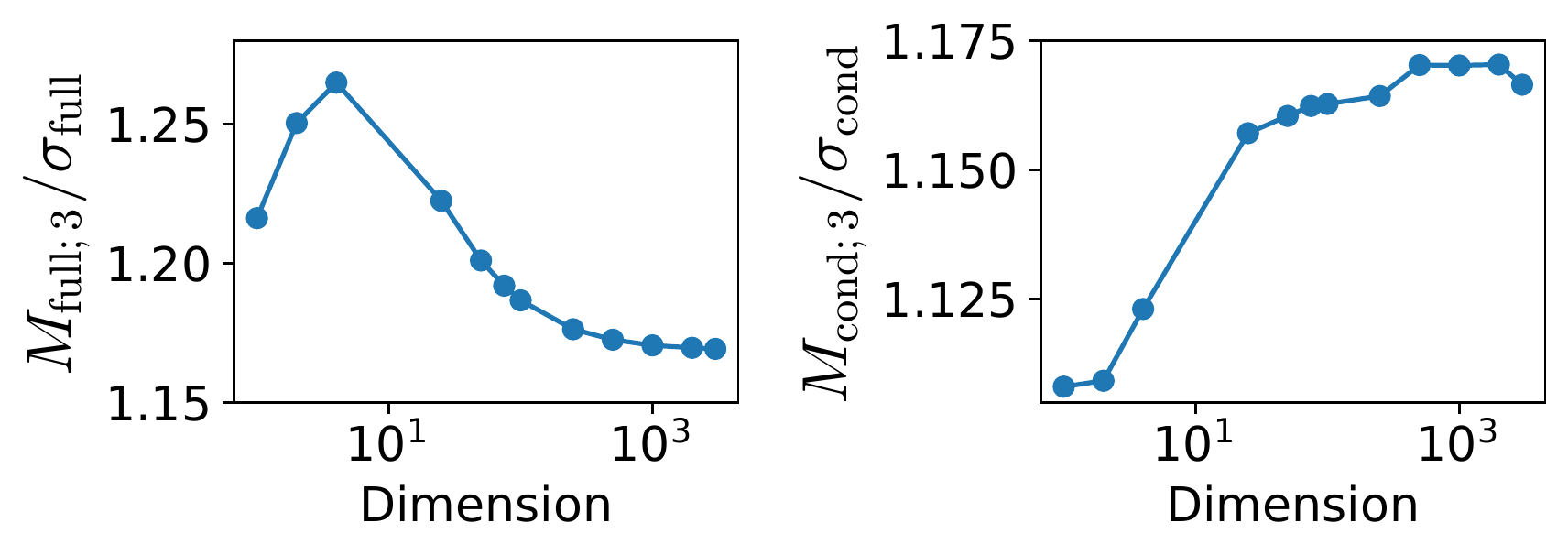}
        \label{fig:ksd:moment:ratios}
    \end{minipage}
    \hfill
    \begin{minipage}{0.47\textwidth}
         \centering
         \includegraphics[width=\textwidth]{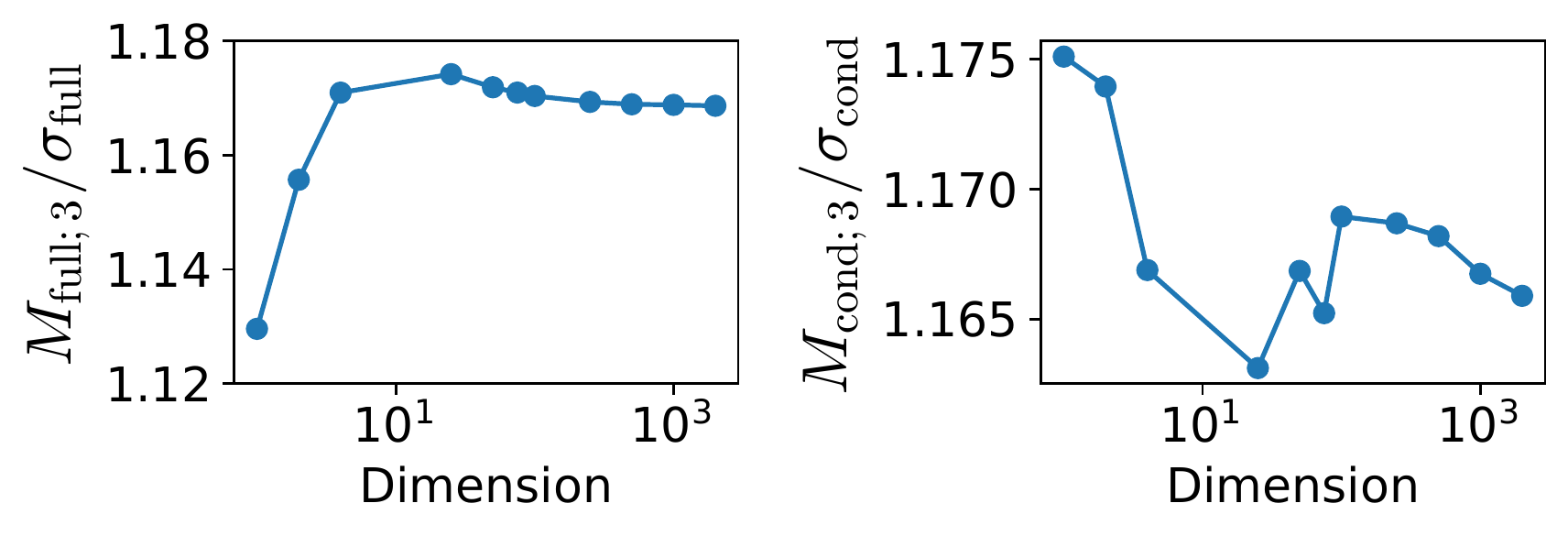}
         \label{fig:mmd:moment:ratios}
    \end{minipage}
    \vspace{-1.5em}
    \caption{Verifying \Cref{assumption:moment:ratio} for $\nu = 3$ for KSD and MMD with RBF kernels. All moment ratios appear to be bounded by a dimension-independent constant. \emph{Left and middle-left.} KSD moment ratios. \emph{Middle-right and right.} MMD moment ratios. The reported results are averaged over 5 random seeds.}
    \label{fig:moment:ratios}
\end{figure}

\subsection{MMD U-statistic with RBF kernel}
\label{appendix:mmd:rbf:moments}

Under the Gaussian mean-shift setup with an identity covariance matrix, the MMD U-statistic with a RBF kernel has the form
\begin{align*}
    &\ummd(\bz, \bz')
    \;=\;
    \kappa(\bx, \bx') + \kappa(\by, \by') - \kappa(\bx, \by') - \kappa(\bx', \by)
    \\
    &=\;
    \exp\left( - \mfrac{\| \bx - \bx'\|_2^2}{2\gamma} \right) 
    + \exp\left( - \mfrac{\|\by - \by'\|_2^2}{2\gamma} \right) 
    - \exp\left( - \mfrac{\| \bx - \by'\|_2^2}{2\gamma} \right) 
    - \exp\left( - \mfrac{\| \bx' - \by \|_2^2}{2\gamma} \right) 
    \;, \tagaligneq \label{eqn:defn:mmd:rbf}
\end{align*}
for $\bz \coloneqq (\bx, \by), \bz' \coloneqq (\bx', \by') \in \R^{2d}$. We first verify that \cref{assumption:L_nu} holds again by adapting $\{\alpha_k\}_{k=1}^\infty$ and $\{\psi_k\}_{k=1}^\infty$ from \cref{appendix:rbf:decompose}.

\begin{lemma} \label{lem:mmd:assumption:L_nu} Assume that $\gamma > 8$. Then \cref{assumption:L_nu} holds with any value $\nu \in (2,3]$ and the function $u((\bx,\by),(\bx',\by'))=\ummd((\bx,\by),(\bx',\by'))$ for $\bx,\by,\bx',\by' \in \R^d$, with the sequences of values and functions given for each $k \in \N$ as $\gamma_k = \alpha_k$ and $\phi_k(\bx, \by) = \psi_k(\bx) - \psi_k(\by)$.
\end{lemma}

\vspace{.2em}

We next compute the moments. The analytical form of the population MMD (i.e.~the expectation) has been derived in previous works under both the Gaussian mean-shift setup with a general covariance matrix $\Sigma$ (\citet[Proposition 2, Corollary 19]{wynne2022kernel}; \citep[Proposition 1]{ramdas2015decreasing}) and an expression up to the learning term was also derived under a more general mean-shift setup \citep[Lemma 1]{reddi2015high}. We only consider the Gaussian mean-shift case with $\Sigma=I_d$ but provide expressions for the second moments and a generic moment bound, while making minimal assumptions on the kernel bandwidth compared to \citet{reddi2015high}.

\begin{lemma}[RBF-MMD moments]
\label{lem:mmd:moments:analytical}
    Let $\kappa$ be a RBF kernel and let $\bX, \bX' \sim Q$ and $\bY, \bY' \sim P$ be mutually independent draws. Under the mean-shift setup with an identity covariance matrix, it follows that
    \begin{proplist}
        \item For every $\bz = (\bx, \by) \in \R^{2d}$, 
        \begin{align*}
            \gmmd&(\bz) 
            \;\coloneqq\; \E[ \ummd(\bz, \bZ') ] \\
            \;&= \left( \mfrac{\gamma}{1 + \gamma} \right)^{d/2} 
            \Big(
            e^{ -\frac{1}{2 (1 + \gamma) } \| \bx - \bmu \|_2^2}
            + e^{ -\frac{1}{2 (1 + \gamma) } \| \by\|_2^2} 
            - e^{ -\frac{1}{2 (1 + \gamma) } \| \bx \|_2^2}
            - e^{ -\frac{1}{2 (1 + \gamma) } \| \by - \bmu \|_2^2}
            \Big) \;;
        \end{align*}
        \item The mean is given by\;$
                \mmd(Q, P)
                \;=\;
                2 \left( \mfrac{\gamma}{2 + \gamma} \right)^{d/2} \Big( 1 - \exp\left( -\mfrac{1}{2 (2 + \gamma) } \| \mu \|_2^2 \right) \Big)
            \; ;$
        \item The variance of the conditional expectation is given by
        \begin{align*}
            \sigcond^2
            \;&=\;
            2 \left( \mfrac{\gamma}{1 + \gamma} \right)^{d/2} \left( \mfrac{\gamma}{3 + \gamma} \right)^{d/2} \\
            &\;\quad \times
            \bigg(
                1 + \exp\left( -\mfrac{1}{3 + \gamma} \| \mu \|_2^2 \right) 
                + 2 \left( \mfrac{3 + \gamma}{2 + \gamma} \right)^{d/2} \left( \mfrac{1 + \gamma}{2 + \gamma} \right)^{d/2} \exp\left( - \mfrac{1}{2(2 + \gamma)} \| \mu \|_2^2 \right) \\
                &\qquad \quad
                - 2 \exp\left( - \mfrac{7 + 5\gamma}{4 (1 + \gamma)(3 + \gamma)} \| \mu \|_2^2 \right) 
                - \left( \mfrac{3 + \gamma}{2 + \gamma} \right)^{d/2} \left( \mfrac{1 + \gamma}{2 + \gamma} \right)^{d/2} \\
                &\qquad \quad
                - \left( \mfrac{3 + \gamma}{2 + \gamma} \right)^{d/2} \left( \mfrac{1 + \gamma}{2 + \gamma} \right)^{d/2} \exp\left( - \mfrac{1}{2 + \gamma} \| \mu \|_2^2 \right)
            \bigg) \; ;
        \end{align*}
        \item The variance is given by
        \begin{align*}
            \sigfull^2
            \;&=\;
            2 \left( \mfrac{\gamma}{4 + \gamma} \right)^{d/2} \bigg( 1 + \exp\left(- \mfrac{1}{4 + \gamma} \| \mu \|_2^2 \right) \bigg)
            - 2 \left( \mfrac{\gamma}{2 + \gamma} \right)^d 
            \\
            &\;
            \quad- 8 \left( \mfrac{\gamma}{3 + \gamma} \right)^{d/2} \left( \mfrac{\gamma}{1 + \gamma} \right)^{d/2}
            \exp\left( - \mfrac{2 + \gamma}{2(1 + \gamma)(3 + \gamma)} \| \mu \|_2^2 \right) 
            \\
            &\;
            \quad- 2 \left( \mfrac{\gamma}{2 + \gamma} \right)^d \exp\left( - \mfrac{1}{2 + \gamma} \| \mu \|_2^2 \right)
            + 8 \left( \mfrac{\gamma}{2 + \gamma} \right)^d \exp\left( - \mfrac{1}{2(2 + \gamma)} \| \bmu \|_2^2 \right)
            \; .
        \end{align*}
    \end{proplist}
\end{lemma}

While we do not verify \cref{assumption:moment:ratio} here, we remark that \cite{gao2021two} also encounter similar moment ratios when deriving finite-sample bounds for MMD with a studentised version of U-statistic (see e.g.~their Theorem 13). They show that those ratios are controlled under an elaborate list of assumptions; in particular, those assumptions hold for the RBF kernel under a condition that amounts to choosing $\gamma=\Theta(d)$ in our Gaussian mean-shift setup. For our case, as discussed, a rigorous verification of \cref{assumption:moment:ratio} can be done by following the proofs of Propositions \ref{prop:ksd:var:ratio} and \ref{prop:MMD:variance:ratio} to control $\Mcondnu$ and $\Mfullnu$ for any $\nu > 2$. \cref{fig:moment:ratios} also verifies \cref{assumption:moment:ratio} by simulation.

\subsection{MMD U-statistic with linear kernel}
\label{appendix:mmd:linear:moments}
In this section, we consider the mean-shift setup with a general covariance matrix $\Sigma \in \R^{d \times d}$, i.e., $Q = \cN(\bmu, \Sigma)$ and $P = \cN(0, \Sigma)$. The MMD with a linear kernel $\kappa(\bx, \bx') = \bx^\top \bx'$ has the form
\begin{align*}
    \ummd( \bz, \bz' )
    \;=&\;
    \bx^\top \bx' + \by^\top \by' - \bx^\top \by' - \by^\top \bx'
    \;,
\end{align*}
where $\bz \coloneqq (\bx, \by), \bz' \coloneqq (\bx', \by') \in \R^{2d}$. In this case, \cref{assumption:L_nu} holds directly because we can represent $\ummd$ as 
\begin{align*}
    \ummd(\bz, \bz') \;=\; 
    (\bx-\by)^\top (\bx'-\by') 
    \;=\;
    \msum_{l=1}^d (x_l - y_l)(x'_l-y'_l)
    \;=\;
    \msum_{l=1}^d \gamma_l \psi_l(\bz) \psi_l(\bz')\;, \tagaligneq \label{eqn:defn:mmd:linear}
\end{align*}
where $\gamma_l=1$, $\psi_l(\bz)=x_l-y_l$ and $\psi_l(\bz')=x'_l-y'_l$.

\vspace{1em}

We next compute the moment terms and verify that \cref{assumption:moment:ratio} holds. The next result, proved in \Cref{pf:linear:mmd:moments:analytical}, gives the analytical expressions of the first two moments of the linear-kernel MMD.
\begin{lemma}[Linear-MMD moments]
\label{lem:linear:mmd:moments:analytical}
    Let $\kappa$ be a linear kernel, and let $\bX, \bX' \sim Q$ and $\bY, \bY' \sim P$ be mutually independent draws. Write $\bZ=(\bX,\bY)$ and $\bZ'=(\bX',\bY')$.  Under the mean-shift setup, it follows that
    \begin{proplist}
        \item For every $\bz = (\bx, \by) \in \R^{2d}$, we have $\gmmd(\bz) \coloneqq \E[ \ummd(\bz, \bZ') ] = \bmu^\top \by - \bmu^\top \bx$\;;
        \item The mean is given by $\mmd(Q, P) = \| \bmu \|_2^2$\;;
        \item The variance of the conditional expectation $\gmmd(\bZ)$ is given by $\sigcond^2 = 2 \bmu^\top \Sigma \bmu$\;;
        \item The variance of $\ummd(\bZ, \bZ')$ is given by $\sigfull^2 = 4 \Tr(\Sigma^2) + 4 \bmu^\top \Sigma \bmu$\;;
        \item The third central moment of $\gmmd(\bZ)$ satisfies $\Mcondthree^3 \leq C (\mu^\top \Sigma \mu)^{3/2}$ for some absolute constant $C$\;;
        \item The third central moment of $\ummd(\bZ, \bZ')$ satisfies $\Mfullthree^3  \leq C \big( \Tr(\Sigma^2) + \mu^\top \Sigma \mu \big)^{3/2}$ for some absolute constant $C$\;.
    \end{proplist} 
    In particular, \cref{assumption:moment:ratio} holds with $\nu=3$ for $\ummd$ defined in \eqref{eqn:defn:mmd:linear} under $Q$. 
\end{lemma}

In the last example in \Cref{sect:examples}, we chose $\mu = (0, 10, \ldots, 0) \in \R^d$ and a diagonal $\Sigma$ with $\Sigma_{11} = 0.5(d+1)$, $\Sigma_{ii} = 0.5$ for $i > 1$ and $\Sigma_{ij} = 0$ otherwise. Note that by the invariance of Gaussian distributions under orthogonal transformation, this is equivalent to choosing
$\Sigma$ as $0.5 \bI_d + 0.5 \bJ_d$, where $\bI_d \in \R^{d\times d}$ is the identity matrix, $\bJ_d \in \R^{d\times d}$ is the all-one matrix and $\mu$ is transformed by an appropriate orthogonal matrix of eigenvectors. Notably, this choice ensures the limit of $\ummd$ remains non-Gaussian. Indeed, when $Q$ and $P$ are Gaussian, the statistic $\mmd_n$ can be written as a sum of shifted-and-rescaled chi-squares, where the scaling factors are  $0.5(d+1), 0.5, \ldots, 0.5$, the eigenvalues of $\Sigma$. As $d$ grows, the eigenvalue $0.5(d+1)$ dominates, and the limiting distribution is then dominated by the first summand, thereby yielding a chi-square limit up to shifting and rescaling. This is numerically demonstrated in the right figure of \cref{fig:intro}. As a remark, we do not expect this exact setting to occur in practice; it should instead be treated as a toy setup to demonstrate the possibility of non-Gaussianity and convey an intuition of when this may occur.



\section{Auxiliary tools} \label{appendix:tools}

\subsection{Generic moment bounds} \label{appendix:moments}

We first present two-sided bounds on the moments of a martingale, which are useful in bounding $\nu$-th moment terms of different statistics. The original result is due to \citet{burkholder1966martingale}, and the constant $C_\nu$ is provided by  \citet{von1965inequalities} and \citet{dharmadhikari1968bounds}. 

\begin{lemma} \label{lem:martingale:bound} Fix $\nu > 1$. For a martingale difference sequence $Y_1,\ldots,Y_n$ taking values in $\R$,
\begin{align*}
    c_\nu \, n^{\min\{0, \nu/2-1 \}} \msum_{i=1}^n \mean[ |Y_i |^\nu ]
    \;\leq\;
    \mean \big[ \big| \msum_{i=1}^n Y_i \big|^\nu \big] 
    \;\leq\; 
    C_{\nu} \, n^{\max\{0,\nu/2-1\}} \msum_{i=1}^n \mean[ |Y_i |^\nu ]\;,
\end{align*}
for $C_{\nu} := \max\big\{ 2,  (8(\nu-1)\max\{1,2^{\nu-3}\})^{\nu} \big\}$ and some absolute constant $c_\nu > 0$ that depends only on $\nu$. 
\end{lemma}

The next moment computation for a quadratic form of Gaussians is used throughout the proof:

\begin{lemma}[Lemma 2.3, \cite{magnus1978moments}] \label{lem:quadratic:gaussian:moments} Given a standard Gaussian vector $\eta$ in $\R^m$ and a symmetric $m \times m$ matrix $A$, we have that $\mean[ \eta^\top A \eta] = \Tr(A)$ and
\begin{align*}
    &
    \mean[ (\eta^\top A \eta)^2] \;=\; \Tr(A)^2 + 2 \Tr(A^2)\;,
    &&
    \mean[ (\eta^\top A \eta)^3] \;=\; \Tr(A)^3 + 6 \Tr(A)\Tr(A^2) + 8 \Tr(A^3)\;.
\end{align*}
\end{lemma}

The next two results are used for the moment computation involving an RBF kernel.

\begin{lemma}
\label{lem:gauss:integral:single}
    Fix $\bm_i \in \R^d$ and $a_i > 0$ for $i = 1, 2$. Let $\bX \sim \cN(\bm_1, a_1^2 I_d)$, and let $f: \R^d \to \R$ be a deterministic function such that $\E[| f(\bX) |] < \infty$. It follows that
    \begin{align*}
        \E\Big[ f(\bX) \exp\Big( - \mfrac{\| \bX - \bm_2 \|_2^2 }{2a_2^2} \Big) \Big] 
        \;&=\;
        \Big(\mfrac{a_2^2}{a_1^2 + a_2^2} \Big)^{d/2} \exp\Big( - \mfrac{ \| \bm_1 - \bm_2 \|_2^2}{2(a_1^2 + a_2^2)} \Big) \E[f(\bW)] \;,
    \end{align*}
    where $\bW \sim \cN(\bm, a^2 I_d )$ with $\,\bm
        \;\coloneqq\;
        \frac{a_1^2 a_2^2}{a_1^2 + a_2^2} \big( \frac{1}{a_1^2} \bm_1 + \frac{1}{a_2^2} \bm_2 \big)$ 
    and
    $a^2 \coloneqq \frac{a_1^2 a_2^2}{a_1^2 + a_2^2}$\;.
\end{lemma}

\begin{lemma}
\label{lem:gauss:integral:double}
    Fix $\bm_1, \bm_2 \in \R^d$ and $a_i > 0$ for $i = 1, 2, 3$. Let $\bX \sim \cN(\bm_1, a_1^2 I_d)$ and $\bX' \sim \cN(\bm_2, a_2^2 I_d)$, and let $f: \R^d \times \R^d \to \R$ be a deterministic function with $\E[| f(\bX, \bX') |] < \infty$. Then
    \begin{align*}
        \E\Big[ f(\bX, \bX') &\exp\left( - \mfrac{ \| \bX - \bX' \|_2^2}{2a_3^2} \right) \Big] 
        \\
        \;=&\;
        \Big( \mfrac{a_3^2}{a_1^2 + a_2^2 + a_3^2} \Big)^{d/2} 
        \exp\Big( - \mfrac{ \| \bm_1 - \bm_2 \|_2^2}{2(a_1^2 + a_2^2 + a_3^2)} \Big) \E\Big[ f\Big(\bW, \bW' + \mfrac{a_2^2 }{a_2^2+a_3^2} \bW\Big) \Big]\;,
    \end{align*}
    where $\bW \sim \cN\left(\bm, a^2 I_d \right) $ and $\bW' \sim \cN\left(\bm', (a')^2 I_d \right)$ are independent with
    \begin{align*}
        \bm
        \;\coloneqq&\;
        \mfrac{a_1^2(a_2^2 + a_3^2)}{a_1^2 + a_2^2 + a_3^2} \left( \mfrac{1}{a_1^2} \bm_1 + \mfrac{1}{a_2^2 + a_3^2} \bm_2 \right) \;,
        &&
        \quad
        a^2
        \;\coloneqq\;
        \mfrac{a_1^2(a_2^2 + a_3^2)}{a_1^2 + a_2^2 + a_3^2} \;, \\
        \bm'
        \;\coloneqq&\;
        \mfrac{a_3^2}{a_2^2 + a_3^2} \bm_2 \;,
        &&
        (a')^2
        \;\coloneqq\;
        \mfrac{a_2^2 a_3^2}{a_2^2 + a_3^2} \;.
    \end{align*}
\end{lemma}

\vspace{.2em}

\subsection{Moment bounds for U-statistics} \label{appendix:u:moments}

We first present a result that bounds the moments of a U-statistic $D_n$ defined as in \eqref{eqn:general:u:defn}.

\begin{lemma} \label{lem:u:moments} Fix $n \geq 2$ and $\nu \geq 2$. Then, there exist absolute constants $c_\nu, C_{\nu} > 0$ depending only on $\nu$ such that
\begin{align*}
    \mean [ | D_n - \mean D_n |^\nu ] 
    \;&\leq\;
    C_{\nu} \, n^{\nu/2} 
    (n-1)^{-\nu} \Mcondnu^\nu 
    +  
    C_{\nu} \, (n-1)^{-\nu} \Mfullnu^\nu
    \;,
    \\
    \mean [ | D_n - \mean D_n |^\nu ] 
    \;&\geq\;
    c_{\nu}
    n(n-1)^{-\nu} \Mcondnu^\nu 
    +  
    c_{\nu}
    n^{-(\nu-1)} (n-1)^{-(\nu-1)} \Mfullnu^\nu
    \;.
\end{align*}
In other words,
\begin{align*}
    \mean [ | D_n - \mean D_n |^\nu ] 
    \;&=\; 
    O( n^{-\nu/2} \Mcondnu^\nu + n^{-\nu} \Mfullnu^\nu )\;,
    \\
    \mean [ | D_n - \mean D_n |^\nu ] 
    \;&=\; 
    \Omega( n^{-(\nu-1)} \Mcondnu^\nu + n^{-2(\nu-1)} \Mfullnu^\nu )\;.
\end{align*}
\end{lemma}

The next two results summarise how the moments of variables under the functional decomposition in \cref{assumption:L_nu} interact with the moments of the original statistic $u$ under $R$:

\begin{lemma} \label{lem:factorisable:moment} Let $\{\phi_k\}_{k=1}^\infty$,  $\{\lambda_k\}_{k=1}^\infty$ and $\varepsilon_{K;\nu}$ be defined as in \cref{assumption:L_nu}. For $\bX_1, \bX_2 \overset{i.i.d.}{\sim} R$, write $\mu_k \coloneqq \mean[\phi_k(\bX_1)]$ and let the moment terms $D, \Mcondnu, \Mfullnu$ be defined as in \cref{sect:moment:terms}. Then we have the following:
\begin{proplist}
    \item $
    \big| 
    \sum_{k=1}^K \lambda_k \mu_k^2
    -
    D
    \big| 
    \leq
    \varepsilon_{K;1}
$;
    \item for any $\nu \in [1,3]$, we have that
\begin{align*}
    \mfrac{1}{4} (\Mcondnu)^\nu - \varepsilon_{K;\nu}^\nu
    \;\leq\; 
    \mean\Big[ \Big| \msum_{k=1}^K \lambda_k (\phi_k(\bX_1) - \mu_k) \mu_k \Big|^\nu \Big] 
    \;\leq\; 
    4 ( (\Mcondnu)^\nu + \varepsilon_{K;\nu}^\nu )\;;
\end{align*}
    \item there exist some absolute constants $c, C > 0$ such that
\end{proplist}
\begin{align*}
    \mean\Big[ \Big|
    \msum_{k=1}^K
    \lambda_k (\phi_k(\bX_1) - \mu_k)(\phi_k(\bX_2) - \mu_k)
    \Big|^\nu \Big]
    \leq&\;
    4 C (\Mfullnu)^\nu    
    -
    \mfrac{1}{2} (\Mcondnu)^\nu
    + 
    (4 C + 2) \varepsilon_{K;\nu}^\nu\;,
    \\
    \mean\Big[ \Big|
    \msum_{k=1}^K
    \lambda_k (\phi_k(\bX_1) - \mu_k)(\phi_k(\bX_2) - \mu_k)
    \Big|^\nu \Big]
    \geq&\;
    \mfrac{c}{4}  (\Mfullnu)^\nu    
    - 
    8 (\Mcondnu)^\nu
    -
    (c + 8) \varepsilon_{K;\nu}^\nu\;.
\end{align*}
\end{lemma}

\vspace{1em}

The next result assumes the notations of Lemma \ref{lem:factorisable:moment}, and additionally denotes
\begin{align*}
    &
    \Lambda^K\;\coloneqq\; \diag\{ \lambda_1, \ldots, \lambda_K\}\;\in\R^{K \times K}\;,
    &&
    \phi^K(x)\;\coloneqq\; ( \phi_1(x), \ldots, \phi_K(x) )^\top\;\in\R^K\;.
\end{align*}

\vspace{1em}

\begin{lemma} \label{lem:factorisable:moment:gaussian} For $ \mu^K \coloneqq  \mean[  \phi^K(\bX_1)]$ and
$ \Sigma^K \coloneqq \Cov[ \phi^K(\bX_1)]$, we have
\begin{align*}
    \sigcond^2 - 4 \sigcond \varepsilon_{K;2} -4\varepsilon_{K;2}^2
    \;\leq\;
    &(\mu^K)^\top \Lambda^K \Sigma^K \Lambda^K (\mu^K) 
    \;\leq\; 
    (\sigcond + 2\varepsilon_{K;2})^2\;.
    \\
    (\sigfull - \varepsilon_{K;2})^2
    \;\leq\;
    &\Tr( (\Lambda^K \Sigma^K)^2 ) \;\leq\; (\sigfull + \varepsilon_{K;2})^2\;.
\end{align*}
In particular, for $\nu \in [1,3]$ and two i.i.d.\,zero-mean Gaussian vector $\bZ_1$ and $\bZ_2$ in $\R^K$ with variance $\Sigma^K$, there exists some absolute constant $C>0$ such that
\begin{align*}
    &
    \mean [| (\mu^K)^\top \Lambda^K\bZ_1  |^\nu ] \;\leq\; 7 \big(
    \sigcond^\nu +  8 \varepsilon_{K;2}^\nu \big)\;,
    \qquad
    \mean [| \bZ_1^\top \Lambda^K\bZ_2  |^\nu ] \;\leq\; 6 \big(
    \sigfull^\nu +  \varepsilon_{K;2}^\nu \big)\;,
    \\
    &
    \mean \big[ \big| (\phi^K(\bX_1) - \mu^K)^\top \Lambda^K\bZ_1  \big|^\nu \big]
    \;\leq\;
    8 C (\Mfullnu)^\nu    
    -
     (\Mcondnu)^\nu
    + 
    (8 C + 4) \varepsilon_{K;\nu}^\nu\;.
\end{align*}
\end{lemma}

\vspace{1em}

The next lemma gives an equivalent expression for $W_n^K$ defined in \eqref{eqn:defn:WnK} and also controls the moments of $W_n^K$.

\vspace{1em}

\begin{lemma} \label{lem:WnK:moments} Let $\{\eta^K_i\}_{i=1}^n$ be a sequence of i.i.d.~standard Gaussian vectors in $\R^K$. Then
\begin{proplist}
    \item the distribution of $W_n^K$ satisfies 
\begin{align*}
    W_n^K \;\overset{d}{=}\;
    \mfrac{1}{n^{3/2}(n-1)^{1/2}} 
    \Big( \msum_{i,j=1}^n (\eta^K_i)^\top (\Sigma^K)^{1/2} \Lambda^K (\Sigma^K)^{1/2} \eta^K_j
    -
    n \Tr( \Sigma^K \Lambda^K )
    \Big)    
    + D\;;
\end{align*}
    \item the mean satisfies $\mean[ W_n^K] = D$ for every $K \in \N$;
    \item the variance is controlled as
\begin{align*}
    \mfrac{2}{n(n-1)} (\sigfull - \varepsilon_{K;2})^2
    \;\leq&\; 
    \Var[ W_n^K ] 
    \;\leq\; 
    \mfrac{2}{n(n-1)} (\sigfull + \varepsilon_{K;2})^2\;;
\end{align*}
    \item the third central moment is controlled as
\begin{align*}
    \mean\big[ (W_n^K - D)^3 \big] \;\leq&\; 
    \mfrac{8 \big( \mean[u(\bX_1,\bX_2)u(\bX_2,\bX_3)u(\bX_3,\bX_1)] - \Mfullthree^3 + ( \Mfullthree + \varepsilon_{K;3})^3 \big)}{n^{3/2}(n-1)^{3/2}} \;,
    \\
    \mean\big[ (W_n^K - D)^3 \big] \;\geq&\; 
    \mfrac{8 \big( \mean[u(\bX_1,\bX_2)u(\bX_2,\bX_3)u(\bX_3,\bX_1)] + \Mfullthree^3 - ( \Mfullthree + \varepsilon_{K;3})^3 \big)}{n^{3/2}(n-1)^{3/2}} \;;
\end{align*}
    \item the fourth central moment is controlled as
\begin{align*}
    \mean\big[ (W_n^K - D)^4 \big] 
    \;\leq&\;
    \mfrac{12}{n^2(n-1)^2} 
    \Big(
    4 \,\mean[ u(\bX_1,\bX_2) u(\bX_2,\bX_3) u(\bX_3,\bX_4)  u(\bX_4,\bX_1) ]
    \\
    &\;\qquad\qquad\quad
    -
    4 \Mfullfour^4
    + 
    4 (\Mfullfour + \varepsilon_{K;4})^4
    +
    (\sigfull + \varepsilon_{K;2})^4
    \Big)
    \;,
    \\
    \mean\big[ (W_n^K - D)^4 \big] 
    \;\geq&\;
    \mfrac{12}{n^2(n-1)^2} 
    \Big(
    4 \,\mean[ u(\bX_1,\bX_2) u(\bX_2,\bX_3) u(\bX_3,\bX_4)  u(\bX_4,\bX_1) ]
    \\
    &\;\qquad\qquad\quad
    +
    4 \Mfullfour^4
    - 
    4 (\Mfullfour + \varepsilon_{K;4})^4
    +
    (\sigfull - \varepsilon_{K;2})^4
    \Big)
    \;;
\end{align*}
    \item we also have a generic moment bound: For $m \in \N$, there exists some absolute constant $C_m > 0$ depending only on $m$ such that
\begin{align*}
    \mean \big[ (W_n^K)^{2m} \big] 
     \;\leq&\;
     \mfrac{C_m}{n^{m}(n-1)^{m} }
     (\sigfull + \varepsilon_{K;2})^{2m}
     +
     C_m \, D^{2m}\;;
\end{align*}
    \item if \cref{assumption:L_nu} holds for some $\nu \geq 2$ then $\lim_{K \rightarrow \infty} \Var[ W_n^K ] = \frac{2}{n(n-1)} \sigfull^2$. If \cref{assumption:L_nu} holds for some $\nu \geq 3$, then
\begin{align*}
    \lim_{K \rightarrow \infty} 
    \mean\big[ (W_n^K - D)^3 \big] \;=&\; 
    \mfrac{8 \mean[u(\bX_1,\bX_2)u(\bX_2,\bX_3)u(\bX_3,\bX_1)]}{n^{3/2}(n-1)^{3/2}} \;,
\end{align*}
and if \cref{assumption:L_nu} holds for some $\nu \geq 4$, then
\begin{align*}
    \lim_{K \rightarrow \infty} 
    \mean\big[ (W_n^K - D)^4 \big] \;=&\; 
    \mfrac{12 (4 \mean[u(\bX_1,\bX_2)u(\bX_2,\bX_3)u(\bX_3,\bX_4)u(\bX_4,\bX_1)] + \sigfull^4)}{n^2(n-1)^2} \;.
\end{align*}
\end{proplist}
\end{lemma}

\subsection{Distribution bounds} \label{appendix:distribution}

The following is a standard approximation of an indicator function for bounding the probability of a given event; see e.g. the proof of Theorem 3.3, \citet{chen2011normal}.

\begin{lemma} \label{lem:smooth:approx:ind} Fix any $m \in \N \cup \{0\}$ and $\tau, \delta \in \R$. Then there exists an $m$-times differentiable $\R \rightarrow \R$ function $h_{m;\tau;\delta}$ such that $h_{m;\tau+\delta;\delta}(x) \;\leq\; \ind_{\{x > \tau\}} \leq h_{m;\tau;\delta}(x)$. For $0 \leq r \leq m$, the $r$-th derivative $h^{(r)}_{m;\tau;\delta}$ is continuous and bounded above by $\delta^{-r}$. Moreover, for every $\epsilon \in [0,1]$, $h^{(m)}$ satisfies that
\begin{align*}
    | h^{(m)}_{m;\tau;\delta}(x) - h^{(m)}_{m;\tau;\delta}(y) |\;\leq\; C_{m,\epsilon} \, \delta^{-(m+\epsilon)}  \, |x - y|^{\epsilon}\;,
\end{align*}
with respect to the constant $C_{m,\epsilon} = \binom{m}{\lfloor m/2 \rfloor} (m+1)^{m+\epsilon}$.
\end{lemma}

\vspace{.2em} 

The next bound is useful for approximating the distribution of a sum of two (possibly correlated) random variables $X$ and $Y$ by the distribution of $X$ alone, provided that the influence of $Y$ is small.

\begin{lemma} \label{lem:approx:XplusY:by:Y} For two real-valued random variables $X$ and $Y$, any $a, b \in \R$ and $\epsilon > 0$, we have
\begin{align*}
    \P( a \leq X+Y \leq b) 
    \;\leq\;
    \P( a - \epsilon \leq X \leq b+\epsilon ) + \P(  |Y| \geq \epsilon )\;,
    \\
    \P( a \leq X+Y \leq b) 
    \;\geq\;
    \P( a + \epsilon \leq X \leq b-\epsilon ) - \P(  |Y| \geq \epsilon )
    \;.
\end{align*}
\end{lemma}

\vspace{.2em}

Theorem 8 of \cite{carbery2001distributional} gives a general anti-concentration result for a polynomial of random variables drawn from a log-concave density. The next lemma restates the result in the case of a quadratic form of a $K$-dimensional standard Gaussian vector $\eta$.

\begin{lemma} \label{lem:cdf:bound:gaussian:quad} Let $p(\bx)$ be a degree-two polynomial of $\bx \in \R^K$ taking values in $\R$. Then there exists an absolute constant $C$ independent of $p$ and $\eta$ such that, for every $t \in \R$,
\begin{align*}
    \P \big( |p(\eta)| \leq t \big) \;\leq\; C t^{1/2} (\mean[ |p(\eta)|^2 ])^{-1/4} \;\leq\; C t^{1/2} (\Var[ p(\eta) ])^{-1/4} \;.
\end{align*}
\end{lemma}

\subsection{Weak Mercer representation} \label{appendix:weak:mercer}

In \cref{sec:kernel:general:results}, we have used the \emph{weak Mercer representation} from \cite{steinwart2012mercer}. We summarise their result below, which combines their Lemma 2.3, Lemma 2.12 and Corollary 3.2:

\begin{lemma} \label{lem:mercer} Consider a probability measure $R$ on $\R^b$, $\bV_1,\bV_2 \overset{i.i.d.}{\sim} R$ and a measurable kernel $\kappa^*$ on $\R^b$. If $\mean[ \kappa^*(\bV_1,\bV_1) ] < \infty$, there exists a sequence of functions $\{\phi_k\}_{k=1}^\infty$ in $L_2(\R^b, R)$ and a bounded sequence of non-negative values $\{\lambda_k\}_{k=1}^\infty$ with $\lim_{k\rightarrow \infty}\lambda_k=0$, such that as $K$ grows,
$ \big| \sum_{k=1}^K \lambda_k \phi_k(\bV_1) \phi_k(\bV_2) - \kappa^*(\bV_1,\bV_2) \big| \rightarrow 0$. The series converges $R \otimes R$ almost surely. 
\end{lemma}

\section{Proof of the main result} \label{appendix:main}

In this section, we prove \cref{thm:u:gaussian:quad:general}. The proof is necessarily tedious as we seek to control ``spectral" approximation errors (i.e.~the error from a truncated functional decomposition) and multiple stochastic approximation errors at the same time. The section is organised as follows:
\begin{itemize}
    \item In \cref{appendix:u:gaussian:quad:lemma}, we list notations and key lemmas that formalise the steps in the proof outline in \cref{sect:proof:overview};
    \item In \cref{appendix:u:gaussian:quad:thm:general:body}, we present the proof body of \cref{thm:u:gaussian:quad:general}, which directly combines results from the different lemmas;
    \item In \cref{appendix:u:spectral:approx}, \ref{appendix:u:lindeberg:approx}, \cref{appendix:u:gaussian:quad:approx} and \ref{appendix:u:gaussian:quad:bound}, we present the proof of the key lemmas. Each section starts with an informal sketch of proof ideas followed by the actual proof of the result. 
\end{itemize}

\subsection{Auxiliary lemmas} \label{appendix:u:gaussian:quad:lemma}

Recall that our goal is to study the distribution of
\begin{align*}
    D_n \;\coloneqq\;\mfrac{1}{n(n-1)} \msum_{1 \leq i \neq j \leq n} u(\bX_i,\bX_j)\;.
\end{align*}
The three results in this section form the key steps of the proof. We fix $\sigma > 0$ to be some normalisation constant to be chosen later. 

\vspace{1em}

\emph{1. ``Spectral" approximation.} For $K \in \N$, we define the truncated version of $D_n$ by
\begin{align*}
    D^K_n 
    \;\coloneqq&\; 
    \mfrac{1}{n(n-1)} \msum_{1 \leq i \neq j \leq n} \msum_{k=1}^K \lambda_k \phi_k(\bX_i) \phi_k(\bX_j)
    \\
    \;=&\;
    \mfrac{1}{n(n-1)} \msum_{1 \leq i \neq j \leq n} (\phi^K(\bX_i))^\top \Lambda^K \phi^K(\bX_j)
    \;. 
\end{align*}
We also denote the rescaled statistics for convenience as
\begin{align*}
    &
    \tilde D_n 
    \;\coloneqq\; \mfrac{\sqrt{n(n-1)}}{\sigma} \, D_n\;,
    &&
    \tilde D_n^K 
    \;\coloneqq\; \mfrac{\sqrt{n(n-1)}}{\sigma} \, D_n^K\;.
\end{align*}
The first lemma allows us to study the distribution of $D_n^K$ in lieu of that of $D_n$ up to some approximation error that vanishes as $K$ grows.
\begin{lemma} \label{lem:u:spectral:approx} Fix $\delta, \sigma > 0$, $K \in \N$ and $t \in \R$. Then
\begin{align*}
    &
    \P( \tilde D_n^K > t+\delta ) - \varepsilon_K'
    \leq
    \P( \tilde D_n > t )  
    \leq
    \P (  \tilde D_n^K   > t-\delta ) +  \varepsilon_K' \;,
    &&
    \varepsilon_K' \coloneqq 
    \mfrac{3 n^{1/4}(n-1)^{1/4} \varepsilon_{K;1}^{1/2}}{\sigma^{1/2} \delta^{1/2} }\;.
\end{align*}
\end{lemma}

\vspace{1em}

\emph{2. Gaussian approximation via Lindeberg's technique.} The distribution of $D_n^K$ is easier to handle, as it is a double sum of a simple quadratic form of $K$-dimensional random vectors. 
Let $\bZ_1, \ldots, \bZ_n$ be i.i.d.\,Gaussian random vectors in $\R^K$ with mean and variance matching those of $\phi^K(\bX_1)$, and denote $Z_{ik}$ as the $k$-th coordinate of $\bZ_i$.
The goal is to replace $D_n^K$ by the random variable
\begin{align*}
    D^K_Z
    \;\coloneqq\; 
    \mfrac{1}{n(n-1)} \msum_{1 \leq i \neq j \leq n} \bZ_i^\top \Lambda^K\bZ_j
    \;=\;
    \mfrac{1}{n(n-1)} \msum_{1 \leq i \neq j \leq n} \msum_{k=1}^K \lambda_k Z_{ik} Z_{jk} 
    \;. 
\end{align*}
Notice that $D^K_Z$ takes the same form as $D_n^K$ except that each $\phi^K(\bX_i)$ is replaced by $\bZ_i$. Analogous to $\tilde D_n$ and $\tilde D_n^K$, we also define a rescaled version as
\begin{align*}
    \tilde D_Z^K \;\coloneqq\; \mfrac{\sqrt{n(n-1)}\, D_Z^K}{\sigma}\;.
\end{align*}
The second lemma replaces the distribution $\tilde D_n^K$ by that of $\tilde D_Z^K$, up to some approximation error that vanishes as $n$ grows:

\begin{lemma} \label{lem:u:lindeberg:approx}Fix $\delta, \sigma > 0$, $K \in \N$, $t \in \R$ and any $\nu \in (2,3]$. Then
\begin{align*}
    &
    \P( \tilde D_n^K > t-\delta ) 
    \;\leq\;
    \P(\tilde D_Z^K > t-2\delta) + E_{\delta;K}\;,
    &&
    \P( \tilde D_n^K > t+\delta ) 
    \;\geq\; 
    \P(\tilde D_Z^K > t+2\delta) - E_{\delta;K}\;,
\end{align*}
where the approximation error is defined as, for some absolute constant $C>0$,
\begin{align*}
    E_{\delta;K} \;\coloneqq&\; 
    \mfrac{C}{\delta^\nu n^{\nu/2-1}}
    \Big(
    \mfrac{ 
    (\Mfullnu)^\nu    
    + 
    \varepsilon_{K;\nu}^\nu}
    {\sigma^{\nu} }
    +
    \mfrac{
     (\Mcondnu)^\nu + \varepsilon_{K;\nu}^{\nu} }{(n-1)^{-\nu/2}\,\sigma^{\nu}}
    \Big)\;. 
\end{align*}
\end{lemma}

\vspace{1em}

\emph{3. Replace $D^K_Z$ by $U^K_n$.} As in the statement of \cref{thm:u:gaussian:quad:general}, let $\{\eta^K_i \}_{i=1}^n$ be the i.i.d.~standard normal vectors in $\R^K$, and recall the notations $\mu^K \coloneqq \mean[  \phi^K(\bX_1)]$ and $ \Sigma^K \coloneqq \Cov[ \phi^K(\bX_1)]$. We can then express $D^K_Z$ as
\begin{align*}
    D^K_Z
    \;=&\; 
    \mfrac{1}{n(n-1)} \msum_{1 \leq i \neq j \leq n} \big((\Sigma^K)^{1/2} \eta^K_i + \mu^K\big)^\top \Lambda^K
    \big((\Sigma^K)^{1/2} \eta^K_j + \mu^K\big)
    \\
    \;=&\;
    \mfrac{1}{n(n-1)} \msum_{1 \leq i \neq j \leq n} (\eta^K_i)^\top (\Sigma^K)^{1/2} \Lambda^K (\Sigma^K)^{1/2} \eta^K_j
    +
    \mfrac{2}{n} \msum_{i=1}^n 
    (\mu^K)^\top \Lambda^K (\Sigma^K)^{1/2} \eta^K_i
    \\
    &\;+ (\mu^K)^\top \Lambda^K \mu^K
    \;.
\end{align*}
This is similar to the desired variable $U_n^K$ except for the third term:
\begin{align*}
    U_n^K
    =&\; 
    \mfrac{1}{n(n-1)} \msum_{1 \leq i \neq j \leq n} (\eta^K_i)^\top (\Sigma^K)^{1/2} \Lambda^K (\Sigma^K)^{1/2} \eta^K_j
    +
    \mfrac{2}{n} \msum_{i=1}^n 
    (\mu^K)^\top \Lambda^K (\Sigma^K)^{1/2} \eta^K_i
    + D
    \;.
\end{align*}
As before, we denote $\tilde U_n^K \coloneqq \frac{\sqrt{n(n-1)} U_n^K}{\sigma}$. The next lemma shows that the distribution of $\tilde D^K_Z$ can be approximated by that of $\tilde U_n^K$, up to some approximation error that vanishes as $K \rightarrow \infty$.
\begin{lemma}  \label{lem:u:gaussian:quad:approx} For any $a, b \in \R$ and $\epsilon > 0$, we have that
\begin{align*}
    \P( a \leq \tilde D_Z^K \leq b) 
    \;\leq&\;\;
    \P\big( a - \epsilon \leq \tilde U_n^K \leq b+\epsilon \big) 
    + 
    \mfrac{ \varepsilon_{K;1}}{ \epsilon \, n^{-1/2} (n-1)^{-1/2} \sigma}\;,
    \\
    \P( a \leq \tilde D_Z^K \leq b) 
    \;\geq&\;\;
    \P( a + \epsilon \leq \tilde U_n^K \leq b-\epsilon ) 
    - 
    \mfrac{ \varepsilon_{K;1}}{ \epsilon \, n^{-1/2} (n-1)^{-1/2} \sigma}\;.
\end{align*}
\end{lemma}

\vspace{1em}

\emph{4. Bound the distribution of $\tilde U^K_n$ over a short interval.} If we are to use Lemma \ref{lem:u:spectral:approx} and Lemma \ref{lem:u:lindeberg:approx} directly, we would end up comparing $\P(\tilde D_n > t)$ against the probabilities $\P(\tilde U^K_n > t + 2\delta)$ and $\P(\tilde U^K_n > t - 2\delta)$ for some small $\delta > 0$. It turns out these are not too different from $\P(\tilde U^K_n > t )$: As $\tilde U^K_n$ is a quadratic form of Gaussians, we can ensure it is ``well spread-out" such that the probability mass of $\tilde U^K_n$ within a small interval $(t-2\delta, t+2\delta)$ is not too large. This is ascertained by the following lemma: 

\begin{lemma} \label{lem:u:gaussian:quad:bound} For $a \leq b \in \R$, there exists some absolute constant $C$ such that
\begin{align*}
    \P( a \leq \tilde U_n^K \leq b) 
    \;\leq&\; C (b-a)^{1/2}  \Big( 
    \mfrac{1}{\sigma^2} (\sigfull - \varepsilon_{K;2})^2
    +
    \mfrac{n-1}{\sigma^2} (\sigcond^2 -2 \sigcond \varepsilon_{K;2} - 4 \varepsilon_{K;2})
    \Big)^{-1/4}\;.
\end{align*}
\end{lemma}

\vspace{.5em}

\subsection{Proof body of \cref{thm:u:gaussian:quad:general}} \label{appendix:u:gaussian:quad:thm:general:body}
Fix $\delta, \sigma > 0$, $K \in \N$ and $t \in \R$. By Lemma \ref{lem:u:spectral:approx}, we have that
\begin{align*}
    &
    \P( \tilde D_n^K > t+\delta ) - \varepsilon_K'
    \leq
    \P( \tilde D_n > t )  
    \leq
    \P (  \tilde D_n^K   > t-\delta ) +  \varepsilon_K' \;,
    &&
    \varepsilon_K' \coloneqq 
    \mfrac{3 n^{1/4}(n-1)^{1/4} \varepsilon_{K;1}^{1/2}}{\sigma^{1/2} \delta^{1/2} }\;.
\end{align*}
By Lemma \ref{lem:u:lindeberg:approx}, we have
\begin{align*}
    &
    \P( \tilde D_n^K > t-\delta ) 
    \;\leq\;
    \P(\tilde D_Z^K > t-2\delta) + E_{\delta;K}\;,
    &&
    \P( \tilde D_n^K > t+\delta ) 
    \;\geq\; 
    \P(\tilde D_Z^K > t+2\delta) - E_{\delta;K}\;,
\end{align*}
where the error term is defined as, for some absolute constant $C' >0$,
\begin{align*}
    E_{\delta;K} \;\coloneqq&\; 
    \mfrac{C' }{\delta^\nu n^{\nu/2-1}}
    \Big(
    \mfrac{ 
    (\Mfullnu)^\nu    
    + 
    \varepsilon_{K;\nu}^\nu}
    {\sigma^{\nu} }
    +
    \mfrac{
     (\Mcondnu)^\nu + \varepsilon_{K;\nu}^{\nu} }{(n-1)^{-\nu/2}\,\sigma^{\nu}}
    \Big)\;. 
\end{align*}
To combine the two bounds, we consider the following decomposition:
\begin{align*}
    \P(\tilde D_Z^K > t-2\delta) \;&=\; \P(\tilde D_Z^K > t) + \P( t-2\delta < \tilde D_Z^K \leq t)\;,
    \\
    \P(\tilde D_Z^K > t+2\delta) \;&=\; \P(\tilde D_Z^K > t) - \P( t < \tilde D_Z^K \leq t+2\delta)\;. \tagaligneq \label{eqn:thm:u:gaussian:quad:general:decompose}
\end{align*}
This allows us to combine the earlier two bounds as
\begin{align*}
    \big| \P( \tilde D_n > t )  - \P(\tilde D_Z^K > t) \big|
    \;\leq&\;
    \mmax\{ 
    \P( t-2\delta \leq \tilde D_Z^K < t)\;,
    \P( t < \tilde D_Z^K \leq t+2\delta)
    \}
    + E_{\delta;K} + \varepsilon'_K\;, 
\end{align*}
which gives the error of approximating the c.d.f.~of $\tilde D_n$ by that of $\tilde D^K_Z$. Now fix some $\epsilon > 0$. By applying
Lemma \ref{lem:u:gaussian:quad:approx} and taking appropriate limits of the endpoints to change $\leq$ to $<$, $\geq$ to $>$ and taking the right endpoint to positive infinity, we can now approximate the c.d.f.~of $\tilde D^K_Z$ by that of $\tilde U_n^K$: 
\begin{align*}
    \P( t-2\delta \leq \tilde D_Z^K < t) 
    \;\leq&\;\;
    \P\big( t-2\delta - \epsilon \leq \tilde U_n^K < t+\epsilon \big) 
    + 
    \mfrac{ \varepsilon_{K;1}}{ \epsilon \, n^{-1/2} (n-1)^{-1/2} \sigma}\;,
    \\
    \P( t \leq \tilde D_Z^K < t+2\delta) 
    \;\leq&\;\;
    \P( t - \epsilon \leq \tilde U_n^K < t+2\delta+\epsilon ) 
    +
    \mfrac{ \varepsilon_{K;1}}{ \epsilon \, n^{-1/2} (n-1)^{-1/2} \sigma}\;,
    \\
    \P( \tilde D_Z^K > t) 
    \;\leq&\;\;
    \P\big( \tilde U_n^K > t-\epsilon \big) 
    + 
    \mfrac{ \varepsilon_{K;1}}{ \epsilon \, n^{-1/2} (n-1)^{-1/2} \sigma}\;,
    \\
    \P( \tilde D_Z^K > t) 
    \;\geq&\;\;
    \P( \tilde U_n^K > t + \epsilon  ) 
    - 
    \mfrac{ \varepsilon_{K;1}}{ \epsilon \, n^{-1/2} (n-1)^{-1/2} \sigma}\;.
\end{align*}
Substituting the bounds into the earlier bound and using a similar decomposition to \eqref{eqn:thm:u:gaussian:quad:general:decompose}, we get that the error of approximating the c.d.f.~of $\tilde D_n$ by that of $\tilde U_n^K$ is
\begin{align*}
    \big| \P( \tilde D_n > t )  - \P(\tilde U_n^K > t) \big|
    \;&\leq\;
    \mmax\{ 
    \P( t-\epsilon \leq \tilde U_n^K < t)\;,
    \P( t < \tilde U_n^K \leq t+\epsilon)
    \}
    \\
    &\;
    +
    \mmax\{ 
    \P( t-2\delta-\epsilon \leq \tilde U_n^K < t+\epsilon)\;,
    \P( t-\epsilon < \tilde U_n^K \leq t+2\delta+\epsilon)
    \}
    \\
    &\;
    + E_{\delta;K} + \varepsilon'_K + \mfrac{ 4 \varepsilon_{K;1}}{ \epsilon \, n^{-1/2} (n-1)^{-1/2} \sigma}\;.
\end{align*}
To bound the maxima, we recall that by Lemma \ref{lem:u:gaussian:quad:bound}, there exists some absolute constant $C''$ such that for any $a \leq b \in \R$,
\begin{align*}
    \P( a \leq \tilde U_n^K \leq b) 
    \;\leq&\; C'' (b-a)^{1/2}  \Big( 
    \mfrac{1}{\sigma^2} (\sigfull - \varepsilon_{K;2})^2
    +
    \mfrac{n-1}{\sigma^2} (\sigcond^2 -2 \sigcond \varepsilon_{K;2} - 4 \varepsilon)
    \Big)^{-1/4}\;.
\end{align*}
Substituting this into the above bound while noting $(2\delta+2\epsilon)^{1/2} \leq 2 \delta^{1/2} + 2 \epsilon^{1/2}$, we get that
\begin{align*}
    \big| \P( \tilde D_n &> t )  - \P(\tilde U_n^K > t) \big|
    \\
    \;\leq&\;
    C''
    \big( 6  \epsilon^{1/2} 
    +
    4 \delta^{1/2} \big)
    \Big( 
    \mfrac{1}{\sigma^2} (\sigfull - \varepsilon_{K;2})^2
    +
    \mfrac{n-1}{\sigma^2} (\sigcond^2 -2 \sigcond \varepsilon_{K;2} - 4 \varepsilon_{K;2})
    \Big)^{-1/4}
    \\
    &\;
    + E_{\delta;K} + \varepsilon'_K + \mfrac{ 4 \varepsilon_{K;1}}{ \epsilon \, n^{-1/2} (n-1)^{-1/2} \sigma}\;. 
\end{align*}
We now take $K \rightarrow \infty$. By \cref{assumption:L_nu}, $\varepsilon_{K;2} \rightarrow 0$ in the first term and the two trailing error terms vanish. The second error term becomes
\begin{align*}
    E_{\delta;K} \;\rightarrow\; 
    \mfrac{C'}{\delta^\nu n^{\nu/2-1}}
    \Big(
    \mfrac{ 
    (\Mfullnu)^\nu}
    {\sigma^{\nu} }
    +
    \mfrac{
     (\Mcondnu)^\nu }{(n-1)^{-\nu/2}\,\sigma^{\nu}}
    \Big)\;. 
\end{align*}
By additionally taking $\epsilon \rightarrow 0$ in the first term and taking a supremum over $t$ on both sides, we then obtain
\begin{align*}
    \msup_{t \in \R}
    \Big| \P( \tilde D_n > t )  - \lim_{K \rightarrow \infty} \P(\tilde U_n^K > t) \Big|
    \;\leq&\;
    4 C'' \delta^{1/2}
    \Big( 
    \mfrac{\sigfull^2}{\sigma^2}
    +
    \mfrac{\sigcond^2}{(n-1)^{-1}\sigma^2} 
    \Big)^{-1/4}
    \\
    &\;
    + 
    \mfrac{C'}{\delta^\nu n^{\nu/2-1}}
    \Big(
    \mfrac{ 
    (\Mfullnu)^\nu}
    {\sigma^{\nu} }
    +
    \mfrac{
     (\Mcondnu)^\nu }{(n-1)^{-\nu/2}\,\sigma^{\nu}}
    \Big)\;.
\end{align*}
Finally, we choose
\begin{align*}
    \delta \;=\; n^{- \frac{\nu - 2}{2\nu+1}}  \Big(
    \mfrac{ 
    (\Mfullnu)^\nu}
    {\sigma^{\nu} }
    +
    \mfrac{
     (\Mcondnu)^\nu }{(n-1)^{-\nu/2}\,\sigma^{\nu}}
    \Big)^{\frac{2}{2\nu+1}}
\end{align*}
and $\sigma = \sigmax \coloneqq \max\{ \sigfull, (n-1)^{1/2} \sigcond\}$. Then $\big( \frac{\sigfull^2}{\sigma^2} + \frac{\sigcond^2}{(n-1)^{-1}\sigma^2} \big)^{-1/4} \leq 1$, and by redefining constants, we get that there exists some absolute constant $C > 0$ such that
\begin{align*}
    &\;\msup_{t \in \R}
    \Big| \P\Big( \mfrac{\sqrt{n(n-1)}}{\sigmax} D_n > t \Big)  - \lim_{K \rightarrow \infty} \P\Big( \mfrac{\sqrt{n(n-1)}}{\sigmax} U_n^K > t\Big) \Big|
    \\
    \;&\leq\;
    C \, n^{- \frac{\nu - 2}{4\nu+2}}  \Big(
    \mfrac{ 
    (\Mfullnu)^\nu}
    {\sigmax^{\nu} }
    +
    \mfrac{
     (\Mcondnu)^\nu }{(n-1)^{-\nu/2}\,\sigmax^{\nu}}
    \Big)^{\frac{1}{2\nu+1}}  \tagaligneq \label{eqn:thm:u:gaussian:quad:general:last:step}
    \\
    \;&\leq\;
    2^{\frac{1}{2\nu+1}} C \, n^{- \frac{\nu - 2}{4\nu+2}}  \Big(
    \mfrac{\Mmaxnu}{\sigmax} \Big)^{\frac{\nu}{2\nu+1}}\;,
\end{align*}
where we have recalled that $\Mmaxnu \coloneqq \max\{ \Mfullnu, (n-1)^{1/2} \Mcondnu\}$. This finishes the proof.

\subsection{Proof of Lemma \ref{lem:u:spectral:approx}} \label{appendix:u:spectral:approx}
\emph{Proof overview.} The proof idea is reminiscent of the standard technique for proving that convergence in probability implies weak convergence. We first approximate each probability by the expectation of a $\delta^{-1}$ Lipschitz function $h$ that is uniformly bounded by $1$. This introduces an approximation error of $\delta$, while replaces the difference in probability by the difference $\mean[ h(\tilde D_n) - h(\tilde D^K_n)]$. The expectation can be further split by the events $\{ | \tilde D_n - \tilde D^K_n| < \epsilon \}$ and $\{ | \tilde D_n - \tilde D^K_n| \geq \epsilon \}$. In the first case, the expectation can be bounded by a Lipschitz argument; in the second case, we can use the boundedness of $h$ to bound the expectation by $2\P(  | \tilde D_n - \tilde D^K_n| \geq \epsilon )$, which is in turn bounded by a Markov argument to give the ``spectral" approximation error. Choosing $\epsilon$ appropriately gives the above error term.

\vspace{1em}

\begin{proof}[Proof of Lemma \ref{lem:u:spectral:approx}] For any $\tau \in \R$ and $\delta > 0$, let $h_{\tau;\delta}$ be the function defined in Lemma \ref{lem:smooth:approx:ind} with $m=0$, which satisfies
$$
h_{\tau+\delta;\delta}(x) \;\leq\; \ind_{\{x > \tau\}} \;\leq\; h_{\tau;\delta}(x)\;.
$$
By applying the above bounds with $\tau$ set to $t$ and $t-\delta$, we get that
\begin{align*}
    \P( \tilde D_n > t ) - \P \big( \tilde D_n^K > t-\delta \big)
    \;=&\; 
    \mean[ \ind_{\{\tilde D_n > t\}} - \ind_{\{\tilde D_n^K > t-\delta\}} ]
    \;\leq\; \mean[ h_{t;\delta}(\tilde D_n) - h_{t;\delta}(\tilde D_n^K) ] \;,
\end{align*}
and similarly
\begin{align*}
    \P\big(\tilde D_n^K > t+\delta \big) - \P( \tilde D_n > t ) 
    \;\leq\; \mean[ h_{t+\delta;\delta}(\tilde D_n^K) - h_{t+\delta;\delta}(\tilde D_n)  ] \;.
\end{align*}
Therefore, defining $\xi_\tau := | \mean[ h_{\tau;\delta}(\tilde D_n) - h_{\tau;\delta}(\tilde D_n^K) ] |$, we get that
\begin{align*}
    \P\big(\tilde D_n^K > t+\delta \big) - \xi_{t+\delta}
    \;\leq\;
    \P( \tilde D_n > t )  
    \;\leq\; 
    \P \big( \tilde D_n^K > t-\delta \big) + \xi_t \;.
\end{align*}
To bound quantities of the form $\xi_\tau$, fix any $\epsilon > 0$ and write $\xi_\tau = \xi_{\tau, 1} + \xi_{\tau, 2} $ where
\begin{align*}
    \xi_{\tau, 1}
    \;\coloneqq&\;
    \Big| \mean \Big[ 
    \big( h_{\tau;\delta}(\tilde D_n) - h_{\tau;\delta}(\tilde D_n^K) \big) 
    \ind_{\{ |\tilde D_n - \tilde D_n^K| \leq \epsilon \}} \Big] \Big|
    \;,
    \\
    \xi_{\tau, 2}
    \;\coloneqq&\;
    \Big| \mean \Big[ 
    \big( h_{\tau;\delta}(\tilde D_n) - h_{\tau;\delta}(\tilde D_n^K) \big) 
    \ind_{\{ |\tilde D_n - \tilde D_n^K| > \epsilon \}} \Big] \Big|\;.
\end{align*}
The first term can be bounded by recalling from Lemma \ref{lem:smooth:approx:ind} that $h_{\tau;\delta}$ is $\delta^{-1}$-Lipschitz:
\begin{align*}
     \xi_{\tau, 1} 
     \;\leq&\;
     \delta^{-1}
     \mean \big[ 
     \big| \tilde D_n - \tilde D_n^K \big|
        \ind_{\{ |\tilde D_n - \tilde D_n^K| \leq \epsilon \}} 
    \big] 
    \;\leq\;
     \delta^{-1} \epsilon \,
     \P \big( |\tilde D_n - \tilde D_n^K| \leq \epsilon \big)
     \;\leq\;
    \delta^{-1} \epsilon\;.
\end{align*}
The second term can be bounded by noting that $h_{\tau;\delta}$ is uniformly bounded above by $1$ and applying Markov's inequality:
\begin{align*}
    \xi_{\tau, 2}
    \;\leq\;
    2 \mean[ \ind_{\{ |\tilde D_n - \tilde D_n^K| > \epsilon \}} ]
    \;=\;
    2 \P(  |\tilde D_n - \tilde D_n^K| > \epsilon )
    \;\leq\;
    2 \epsilon^{-1} \mean\big[ | \tilde D_n - \tilde D_n^K | \big]\;.
\end{align*}
By the definition of $\tilde D_n$ and $\tilde D_n^K$, a triangle inequality and noting that $\bX_1, \ldots, \bX_n$ are i.i.d.\,, the absolute moment term can be bounded as
\begin{align*}
    \mean\big[ |\tilde D_n - \tilde D_n^K| \big]
    \;=&\;
    \mfrac{\sqrt{n(n-1)}}{\sigma} \,
    \mean\big[ | D_n - D_n^K | \big]
    \\
    \;=&\;
    \mfrac{1}{\sigma \sqrt{n(n-1)}}
    \Big\|
    \msum_{1\leq i \neq j \leq n}
    \big( u(\bX_i, \bX_j)   
    -
    \msum_{k=1}^K \lambda_k \phi_k(\bX_i) \phi_k(\bX_j) \big)
    \Big\|_{L_1}
    \\
    \;\leq&\;
    \mfrac{\sqrt{n(n-1)}}{\sigma} 
    \Big\|
    u(\bX_1, \bX_2)   
    -
    \msum_{k=1}^K \lambda_k \phi_k(\bX_1) \phi_k(\bX_2) 
    \Big\|_{L_1}
    \\
    \;=&\;  \sigma^{-1} \sqrt{n(n-1)} \, \varepsilon_{K;1} \;.
\end{align*}
Combining the bounds on $\xi_{\tau, 1}, \xi_{\tau, 2}$ and $\mean[ | \tilde D_n-\tilde D_n^K| ]$ and choosing $\epsilon = \big( \sqrt{n (n-1)} \sigma^{-1} \delta \varepsilon_{K;1}  )^{1/2} $, we get that
\begin{align*}
    \xi_\tau 
    \;\leq&\; 
    \delta^{-1} \epsilon + 2 \sqrt{ n (n-1)}\,\epsilon^{-1} \sigma^{-1} \varepsilon_{K;1} 
    \;=\;
    \mfrac{3 n^{1/4}(n-1)^{1/4} \varepsilon_{K;1}^{1/2}}{\sigma^{1/2} \delta^{1/2} }
    \;\eqqcolon\; \varepsilon_K'
    \;,
\end{align*}
which yields the desired bound
\begin{align*}
    \P( \tilde D_n^K > t+\delta ) - \varepsilon_K'
    \;\leq\;
    \P(  \tilde D_n > t )  
    \;\leq\; 
    \P (  \tilde D_n^K > t-\delta ) + \varepsilon_K' \;.
\end{align*}

\end{proof}
 
\subsection{Proof of Lemma \ref{lem:u:lindeberg:approx}} \label{appendix:u:lindeberg:approx}

For convenience, we denote $\bV_i \coloneqq \phi^K(\bX_i)$ throughout this section.

\vspace{1em}

\emph{Proof overview.} The key idea in the proof rests on Lindeberg's telescoping sum argument for central limit theorem. We follow \cite{chatterjee2006generalization}'s adaptataion of the Lindeberg idea for statistics that are not asymptotically normal. As before, the difference in probability is first approximated by a difference in expectation $\mean[h(\tilde D_n^K) - h(\tilde D^K_Z)]$ with respect to some function $h$, which introduces a further approximation error $\delta$. The next step is to note that both $\tilde D_n^K$ and $\tilde D^K_Z$ can be expressed in terms of some common function $\tilde f$, such that
\begin{align*}
    &
    \tilde D_n^K \;=\; \tilde f( \bV_1, \ldots, \bV_n )\;,
    &&
    \tilde D^K_Z \;=\; \tilde f( \bZ_1, \ldots, \bZ_n )\;.
\end{align*}
Denoting $g = h \circ \tilde f$, we can then write the difference in expectation in terms of Lindeberg's telescoping sum as
\begin{align*}
    \mean[h(\tilde D_n^K) - h(\tilde D^K_Z)] 
    \;=&\; 
    \mean[ g( \bV_1, \ldots, \bV_1 ) - g( \bZ_1, \ldots, \bZ_n ) ]
    \\
    \;=&\;
    \msum_{i=1}^n
    \big( \mean[ g( \bV_1, \ldots,\bV_{i-1},  \bV_i, \bZ_{i+1}, \ldots, \bZ_n ) 
    \\
    &\;\qquad \qquad
    - g( \bV_1, \ldots, \bV_{i-1}, \;\bZ_i, \;\bZ_{i+1},\ldots, \bZ_n )  ] \big)\;.
\end{align*}
Since each summand differs only in the $i$-th argument, we can perform a second-order Taylor expansion about the $i$-th argument provided that the function $h$ such that $h$ is twice-differentiable. The second-order remainder term is further ``Taylor-expanded" to an additional $\epsilon$-order for any $\epsilon \in [0,1]$ by choosing $h''$ to be $\epsilon$-H\"older. Write $D_i$ as the differential operator with respect to the $i$-th argument and denote $\tilde f_i(\bx) \coloneqq \tilde f( \bV_1, \ldots,\bV_{i-1}, \bx, \bZ_{i+1}, \ldots, \bZ_n )$. Then informally speaking, the Taylor expansion argument amounts to bounding each summand as
\begin{align*}
    \big| \text{(summand)}_i \big| \;\leq&\; \mean[ D_i (h \circ \tilde f_i)(0) (\bV_i - \bZ_i)] + \mfrac{1}{2} \mean[ D_i^2 (h \circ \tilde f_i)(0) (\bV_i^2 - \bZ_i^2)]
    \\
    &\; + \mfrac{1}{6} \big(\text{H\"older constant of } h''\big)  \times \mean\big[ \big| D_i \tilde f_i(0) \bV_i  \big|^{2+\epsilon} +  \big| D_i \tilde f_i(0) \bZ_i   \big|^{2+\epsilon}  \big]\;,
\end{align*}
where we have used the fact that $\tilde f_i$ is a linear function in expressing the last quantity. The first two terms vanish because $h \circ \tilde f_i$ is independent of $\big(\bV_i, \bZ_i\big)$ and the first two moments of $\bV_i$ and $\bZ_i$ match. The third term is bounded carefully by noting the moment structure of $\bV_i$ and $\bZ_i$ to give the error term $\frac{1}{n} E_{\delta;K}$. Summing the errors over $1 \leq i \leq n$ then gives the Gaussian approximation error bound in Lemma \ref{lem:u:lindeberg:approx}.

\vspace{1em}

\begin{proof}[Proof of Lemma \ref{lem:u:lindeberg:approx}] For any $\tau \in \R$ and $\delta > 0$, let $h_{\tau;\delta}$ be the twice-differentiable function defined in Lemma \ref{lem:smooth:approx:ind} (i.e.\,$m=2$), which satisfies
$$
h_{\tau+\delta;\delta}(x) \;\leq\; \ind_{\{x > \tau\}} \;\leq\; h_{\tau;\delta}(x)\;.
$$
By applying the above bounds with $\tau$ set to $t-\delta$ and $t-2\delta$, we get that
\begin{align*}
    \P( \tilde D_n^K > t-\delta ) - \P( \tilde D_Z^K > t-2\delta)
    \;=&\; 
    \mean[ \ind_{\{\tilde D_n^K > t-\delta\}} - \ind_{\{\tilde D_Z^K > t-2\delta\}} ]
    \\
    \;\leq&\; \mean[ h_{t-\delta;\delta}( \tilde D_n^K) - h_{t-\delta;\delta}(\tilde D_Z^K) ] \;, 
\end{align*}
and similarly
\begin{align*}
    \P(\tilde D_Z^K > t+2\delta) - \P( \tilde D_n^K > t+\delta ) 
    \;=&\; 
    \mean[ \ind_{\{\tilde D_Z^K > t+2\delta\}} - \ind_{\{\tilde D_n^K  > t+\delta\}}]
    \\
    \;\leq&\; \mean[ h_{t+2\delta;\delta}(\tilde D_Z^K) - h_{t+2\delta;\delta}(\tilde D_n^K)  ] \;.
\end{align*}
Therefore, we obtain that
\begin{align*}
    &
    \P( \tilde D_n^K > t-\delta ) 
    \;\leq\;
    \P(\tilde D_Z^K > t-2\delta) + E'_{\delta;K}\;,
    &&
    \P( \tilde D_n^K > t+\delta ) 
    \;\geq\; 
    \P(\tilde D_Z^K > t+2\delta) - E'_{\delta;K}\;, \tagaligneq \label{eqn:lindeberg:intermediate}
\end{align*}
where $E'_{\delta;K} \coloneqq \msup_{\tau \in \R} | \mean[ h_{\tau;\delta}(\tilde D_n^K) - h_{\tau;\delta}(\tilde D_Z^K) ] |$. The next step is to bound $E'_{\delta;K}$, to which we apply Lindeberg's technique for proving central limit theorem. We denote the scaled mean as
\begin{align*} 
    \tilde \mu 
    \;\coloneqq\; 
    \mfrac{\mean[\bV_1]}{\sigma^{1/2} (n(n-1))^{1/4}} 
    \;=\; 
    \mfrac{\mean[\bZ_1]}{\sigma^{1/2} (n(n-1))^{1/4}}\;,
\end{align*}
and define the centred and scaled versions of $\bV_i$ and $\bZ_i$ respectively as 
\begin{align*}
    &
    \tilde \bV_i \;\coloneqq\; \mfrac{\bV_i}{\sigma^{1/2} (n(n-1))^{1/4}} - \tilde \mu\;,
    &&
    \tilde \bZ_i \;\coloneqq\; \mfrac{\bZ_i}{\sigma^{1/2} (n(n-1))^{1/4}} - \tilde \mu\;.
\end{align*}
We also define the function $f: (\R^{K})^n \rightarrow \R$ by
\begin{align*}
    f(\bv_1, \ldots, \bv_n) 
    \coloneqq
    \msum_{1 \leq i \neq j \leq n} (\bv_i+ \tilde \mu)^\top \Lambda^K(\bv_j+ \tilde \mu)
    \;,
    \;\;
    \text{ where we recall }
    \Lambda^K\coloneqq \diag\{\lambda_1, \ldots, \lambda_K\}\;.
\end{align*}
This allows us to express the random quantities in \eqref{eqn:lindeberg:intermediate} as
\begin{align*}
    &
    \tilde D_n^K \;=\; f( \tilde \bV_1, \ldots, \tilde \bV_n)\;,
    &&
    \tilde D_Z^K \;=\; f( \tilde \bZ_1, \ldots, \tilde \bZ_n)\;.
\end{align*}
By defining the random function
\begin{align*}
    &
    F_i(\bv) \;:=\;  f(\tilde \bV_1, \ldots, \tilde \bV_{i-1}, \bv, \tilde \bZ_{i+1}, \ldots, \tilde \bZ_n) 
    &&
    \text{ for } \bv \in \R^K \text{ and } 1 \leq i \leq n\;,
\end{align*}
we can write $E'_{\delta;K}$ into Lindeberg's telescoping sum as
\begin{align*}
    E'_{\delta;K}
    \;=&\;
    \msup_{\tau \in \R}
    | \mean[ h_{\tau;\delta} \circ f(\tilde \bV_1, \ldots, \tilde \bV_n) - h_{\tau;\delta}  \circ f(\tilde \bZ_1, \ldots, \tilde \bZ_n) ] |
    \\
    \;=&\; 
    \msup_{\tau \in \R}
    \Big| \msum_{i=1}^n \mean[ h_{\tau;\delta}(F_i( \tilde \bV_i) - h_{\tau;\delta}(F_i(\tilde \bZ_i)) ] \Big|
    \\
    \;\leq&\;
    \msup_{\tau \in \R}
    \msum_{i=1}^n | \mean[ h_{\tau;\delta} \circ F_i(\tilde \bV_i) - h_{\tau;\delta} \circ F_i(\tilde \bZ_i) ] |\;. 
\end{align*}
Since $h_{\tau;\delta} \circ f$ is twice-differentiable, by a second-order Taylor expansion around $\bzero \in \R^K$, there exists random values $\theta_V, \theta_Z \in (0,1)$ almost surely such that
\begin{align*}
    h_{\tau;\delta} \circ F_i(\tilde \bV_i) 
    \;&=\; 
    \mfrac{\partial h_{\tau;\delta} \circ F_i(\bx)}{\partial \bx} \Big|_{\bx=\bzero} \, \tilde \bV_i
    +
    \mfrac{1}{2}
    \mfrac{\partial^2 h_{\tau;\delta} \circ F_i(\bx)}{\partial \bx^2} \Big|_{\bx=\theta_V \tilde \bV_i}
    \, \tilde \bV_i^{\otimes 2}
    \;,
    \\
    h_{\tau;\delta} \circ F_i(\tilde \bZ_i) 
    \;&=\; 
    \mfrac{\partial h_{\tau;\delta} \circ F_i(\bx)}{\partial \bx} \Big|_{\bx=\bzero} \, \tilde \bZ_i
    +
    \mfrac{1}{2}
    \mfrac{\partial^2 h_{\tau;\delta} \circ F_i(\bx)}{\partial \bx^2} \Big|_{\bx=\theta_Z \tilde \bZ_i}
    \, \tilde \bZ_i^{\otimes 2}
    \;.
\end{align*}
Substituting this into the sum above gives
\begin{align*}
    E'_{\delta;K} \;\leq\; 
    \msup_{\tau \in \R}
    \Big(&
    \msum_{i=1}^n 
    \Big| \mean\Big[ 
    \mfrac{\partial h_{\tau;\delta} \circ F_i(\bx)}{\partial \bx} \Big|_{\bx=\bzero} \, 
    \big( \tilde \bV_i - \tilde \bZ_i)
    \Big] \Big|
    \\
    &\;
    + 
    \mfrac{1}{2}
    \msum_{i=1}^n 
    \Big| \mean\Big[  
    \mfrac{\partial^2 h_{\tau;\delta} \circ F_i(\bx)}{\partial \bx^2} \Big|_{\bx=\theta_V \tilde \bV_i}
    \, \tilde \bV_i^{\otimes 2}
    -
    \mfrac{\partial^2 h_{\tau;\delta} \circ F_i(\bx)}{\partial \bx^2} \Big|_{\bx=\theta_Z \tilde \bZ_i}
    \, \tilde \bZ_i^{\otimes 2}
    \Big] \Big| \, \Big)\;.
\end{align*}
The first sum vanishes because the only randomness of the derivative comes from $F_i$, who is independent of $(\tilde \bV_i, \tilde \bZ_i)$, and the mean of $\tilde \bV_i$ and $\tilde \bZ_i$ match. To handle the second sum, we make use of independence again and the fact that the second moment of $\tilde \bV_i$ and $\tilde \bZ_i$ also match: By subtracting and adding the term 
\begin{align*}
    \mean\Big[ \mfrac{\partial^2 h_{\tau;\delta} \circ F_i(\bx)}{\partial \bx^2} \Big|_{\bx=\bzero} ( \tilde \bV_i)^{\otimes 2} \Big]
    \;=\;
    \mean\Big[ \mfrac{\partial^2 h_{\tau;\delta} \circ F_i(\bx)}{\partial \bx^2} \Big|_{\bx=\bzero} ( \tilde \bZ_i)^{\otimes 2} \Big]\;,
\end{align*}
we can apply a triangle inequality to get that
\begin{align*}
    E'_{\delta;K} \;\leq\; 
    \mfrac{1}{2}
    \msup_{\tau \in \R}
    \Big(
    &
    \msum_{i=1}^n 
    \Big|
    \mean\Big[  
    \Big(
    \mfrac{\partial^2 h_{\tau;\delta} \circ F_i(\bx)}{\partial \bx^2} \Big|_{\bx=\theta_V \tilde \bV_i}
    -
    \mfrac{\partial^2 h_{\tau;\delta} \circ F_i(\bx)}{\partial \bx^2} \Big|_{\bx=\bzero}
    \Big)
    \, \tilde \bV_i^{\otimes 2}
    \Big] \Big|
    \\
    &\;
    +
    \msum_{i=1}^n 
    \Big| \mean\Big[  
    \Big(
    \mfrac{\partial^2 h_{\tau;\delta} \circ F_i(\bx)}{\partial \bx^2} \Big|_{\bx=\theta_Z \tilde \bZ_i}
    -
    \mfrac{\partial^2 h_{\tau;\delta} \circ F_i(\bx)}{\partial \bx^2} \Big|_{\bx=\bzero}
    \Big)
    \, \tilde \bZ_i^{\otimes 2}
    \Big] \Big|
    \, \Big)
    \;. \tagaligneq \label{eqn:lindeberg:eta:tau:intermediate}
\end{align*}
The final step is to bound the two sums by exploiting the derivative structure of $h_{\tau;\delta}$ and $F_i$. Note that $F_i$ is a linear function: its first derivative is given by
\begin{align*}
    \partial F_i(\bx) 
    \;=\; 
    2 \msum_{1 \leq j < i} \Lambda^K\tilde \bV_j
    +
    2 \msum_{i < j \leq n} \Lambda^K\tilde \bZ_j
    + 
    2 (n-1) \Lambda^K\tilde \mu
    \;\in\;\R^K
    \;,
\end{align*}
which is independent of $\bx$, while its higher derivatives vanish. By a second-order chain rule, this implies that almost surely
\begin{align*}
    &\;
    \Big|
    \Big(
    \mfrac{\partial^2 h_{\tau;\delta} \circ F_i(\bx)}{\partial \bx^2} \Big|_{\bx=\theta_V \tilde \bV_i}
    -
    \mfrac{\partial^2 h_{\tau;\delta} \circ F_i(\bx)}{\partial \bx^2} \Big|_{\bx=\bzero} \Big) \, \tilde \bV_i^{\otimes 2}
    \Big|
    \\
    \;=&\;
    \Big|
    \big(\partial^2 h_{\tau;\delta} \big(  F_i(\theta_V \tilde \bV_i) \big)
    -
    \partial^2 h_{\tau;\delta} \big(  F_i(\bzero) \big)
    \big)
    \,
    ( \partial F_i(\bzero)^\top \tilde \bV_i )^2 \Big|
    \\
    \;\leq&\;
    \big| \partial^2 h_{\tau;\delta} \big(  F_i(\theta_V \tilde \bV_i) \big)
    -
    \partial^2 h_{\tau;\delta} \big(  F_i(\bzero) \big)
    \big| \, \big|  \partial F_i(\bzero)^\top \tilde \bV_i  \big|^2\;.
\end{align*}
For $\nu \in (2,3]$, by the H\"older property of $\partial^2 h_{\tau;\delta}$ from Lemma \ref{lem:smooth:approx:ind}, we get that almost surely,
\begin{align*}
    \big|  \partial^2 h_{\tau;\delta} (F_i(\theta_V \tilde \bV_i)) - \partial^2 h_{\tau;\delta} (F_i(\bzero)) \big|
    \;\leq&\;
    18 \times 3^{\nu - 2} \delta^{-\nu} | F_i(\theta_V\tilde \bV_i) - F_i(\bzero)|^{\nu - 2}
    \\
    \;=&\;
    18 \times 3^{\nu - 2} \delta^{-\nu} | \partial F_i(\bzero)^\top (\theta_V \tilde \bV_i) |^{\nu - 2}
    \\
    \;\leq&\;
    54 \delta^{-\nu} | \partial F_i(\bzero)^\top \tilde \bV_i |^{\nu - 2}
    \;.
\end{align*}
In the last inequality, we have used that $\theta_V$ takes value in $[0,1]$. Combining the results, we get that each summand in the first sum in \eqref{eqn:lindeberg:eta:tau:intermediate} can be bounded as
\begin{align*}
    \Big|
    \mean\Big[  
    \Big(
    \mfrac{\partial^2 h_{\tau;\delta} \circ F_i(\bx)}{\partial \bx^2} \Big|_{\bx=\theta_V \tilde \bV_i}
    -
    \mfrac{\partial^2 h_{\tau;\delta} \circ F_i(\bx)}{\partial \bx^2} \Big|_{\bx=\bzero}
    \Big)
    \, \tilde \bV_i^{\otimes 2}
    \Big] \Big|
    \;\leq&\;
    54 \delta^{-\nu} \mean \big[ | \partial F_i(\bzero)^\top \tilde \bV_i |^{\nu} \big]\;.
\end{align*}
The exact same argument applies to the summands of the second sum to give
\begin{align*}
    \Big|
    \mean\Big[  
    \Big(
    \mfrac{\partial^2 h_{\tau;\delta} \circ F_i(\bx)}{\partial \bx^2} \Big|_{\bx=\theta_Z \tilde \bZ_i}
    -
    \mfrac{\partial^2 h_{\tau;\delta} \circ F_i(\bx)}{\partial \bx^2} \Big|_{\bx=\bzero}
    \Big)
    \, \tilde \bZ_i^{\otimes 2}
    \Big] \Big|
    \;\leq&\;
    54 \delta^{-\nu} \mean\big[  | \partial F_i(\bzero)^\top \tilde \bZ_i |^{\nu} \big]\;,
\end{align*}
so a substitution back into \eqref{eqn:lindeberg:eta:tau:intermediate} gives
\begin{align*}
    E'_{\delta;K} \;\leq&\; 
    27 \delta^{-\nu}
    \msum_{i=1}^n 
    \Big( 
    \mean\big[ | \partial F_i(\bzero)^\top \tilde \bV_i |^{\nu} \big]
    + 
    \mean \big[ | \partial F_i(\bzero)^\top \tilde \bV_i |^{\nu} \big] 
    \Big)
    \;.
\end{align*}
We defer to Lemma \ref{lem:lindeberg:nu:moment:bound} to show that there exists an absolute constant $C'>0$ such that the moment terms can be bounded as
\begin{align*}
     \mean\big[ | \partial F_i(\bzero)^\top \tilde \bV_i |^\nu \big]
     +
      \mean\big[ | \partial F_i(\bzero)^\top \tilde \bZ_i |^\nu \big]
    \;\leq&\;
      \mfrac{C'}{n^{\nu/2}} 
    \Big(
    \mfrac{ 
    (\Mfullnu)^\nu    
    + 
    \varepsilon_{K;\nu}^\nu}
    {\sigma^{\nu} }
    +
    \mfrac{
     (\Mcondnu)^\nu + \varepsilon_{K;\nu}^{\nu} }{(n-1)^{-\nu/2}\,\sigma^{\nu}}
    \Big)\;. \tagaligneq \label{eqn:lindeberg:nu:moment:bound}
\end{align*}
Combining with \eqref{eqn:lindeberg:intermediate} and defining $E_{\delta;K}$ to be the upper bound for $E'_{\delta;K}$, we get that
\begin{align*}
    &
    \P( \tilde D_n^K > t-\delta ) 
    \;\leq\;
    \P(\tilde D_Z^K > t-2\delta) + E_{\delta;K}\;,
    &&
    \P( \tilde D_n^K > t+\delta ) 
    \;\geq\; 
    \P(\tilde D_Z^K > t+2\delta) - E_{\delta;K}\;, 
\end{align*}
where we have made the $K$-dependence explicit and define, for $C \coloneqq 27C'$,
\begin{align*}
    E_{\delta;K} \;\coloneqq&\; 
    \mfrac{C}{\delta^\nu n^{\nu/2-1}}
    \Big(
    \mfrac{ 
    (\Mfullnu)^\nu    
    + 
    \varepsilon_{K;\nu}^\nu}
    {\sigma^{\nu} }
    +
    \mfrac{
     (\Mcondnu)^\nu + \varepsilon_{K;\nu}^{\nu} }{(n-1)^{-\nu/2}\,\sigma^{\nu}}
    \Big)\;. 
\end{align*}
\end{proof}

\begin{lemma} \label{lem:lindeberg:nu:moment:bound} \eqref{eqn:lindeberg:nu:moment:bound} holds.
\end{lemma}

\begin{proof}[Proof of Lemma \ref{lem:lindeberg:nu:moment:bound}] We seek to bound $\mean[ | \partial F_i(\bzero)^\top \tilde \bV_i |^{\nu} ] + \mean[ | \partial F_i(\bzero)^\top \tilde \bZ_i |^{\nu} ]$ for $\nu \in (2,3]$ and
\begin{align*}
    \partial F_i(\bzero) 
    \;=\; 
    2 \msum_{1 \leq j < i} \Lambda^K\tilde\bV_j
    +
    2 \msum_{i < j \leq n} \Lambda^K\tilde \bZ_j
    + 
    2 (n-1) \Lambda^K\tilde \mu
    \;\in\;\R^K
    \;.
\end{align*}
We first focus on bounding the first expectation. By convexity of the function $x \mapsto | x |^\nu$, we can apply Jensen's inequality to bound
\begin{align*}
    \mean\big[ | \partial F_i(\bzero)^\top& \tilde \bV_i |^\nu \big]
    \;=\;
    \mean\Big[ \Big| 2 \msum_{j < i}  \tilde\bV_j^\top \Lambda^K\tilde\bV_i
    + 
    2 \msum_{j > i} \tilde \bZ_j^\top \Lambda^K\tilde\bV_i 
    +
    2 (n-1) \tilde \mu^\top \Lambda^K\tilde \bV_i \Big|^\nu \Big]
    \\
    \;\leq&\;
    \mfrac{1}{3}
    \mean\big[ \big| 6 \msum_{j < i} \tilde \bV_j^\top \Lambda^K\tilde\bV_i \big|^\nu \big]
    + 
    \mfrac{1}{3}
    \mean\big[ \big| 6 \msum_{j > i} \tilde \bZ_j^\top \Lambda^K\tilde \bV_i \big|^\nu \big]
    + 
    \mfrac{1}{3}
    \mean\big[ \big| 6(n-1) \tilde \mu^\top \Lambda^K\tilde\bV_1 \big|^\nu \big]
    \\
    \;\leq&\;
    72 \big(
    \mean\big[ \big| \msum_{j < i} \tilde\bV_j^\top\Lambda^K\tilde\bV_i \big|^\nu \big]
    + 
    \mean\big[ \big| \msum_{j > i} \tilde \bZ_j^\top\Lambda^K\tilde \bV_i \big|^\nu \big]
    + 
    \mean\big[ \big| (n-1) \tilde \mu^\top\Lambda^K\tilde\bV_1 \big|^\nu \big]
    \big)
    \;,
\end{align*}
where we have noted that $\nu \leq 3$. Since $\tilde\bV_i$'s are i.i.d., $\tilde\bZ_i$'s are i.i.d.~and all variables involved are zero-mean, $(\tilde \bV_j^\top \Lambda^K\tilde \bV_i )_{j=1}^{i-1}$ forms a martingale difference sequence with respect to the filtration $\sigma(\tilde \bV_i, \tilde\bV_1), \ldots, \sigma(\tilde \bV_i, \tilde \bV_1, \ldots, \tilde \bV_{i-1})$, and so is $( \tilde \bZ_j^\top \Lambda^K\tilde \bV_i )_{j=i+1}^n$ with respect to the filtration $\sigma(\tilde \bV_i, \tilde \bZ_{i+1}), \ldots, \sigma(\tilde \bV_i, \tilde \bZ_{i+1}, \ldots, \tilde \bZ_n)$. This allows the above two moments of sums to be bounded via the martigale moment inequality from Lemma \ref{lem:martingale:bound}: There exists an absolute constant $C_0 > 0$ such that
\begin{align*}
    \mean\big[ | \partial F_i(\bzero)^\top &\tilde \bV_i |^\nu \big]
    \\
    \;\leq&\; 
    C_0
    \Big( 
    (i-1)^{\nu/2-1} 
    \msum_{j=1}^{i-1} \mean[ |\tilde \bV_j^\top \Lambda^K\tilde\bV_i |^\nu ]
    +
    (n-i)^{\nu/2-1} 
    \msum_{j=i+1}^{n} \mean[ |\tilde \bZ_j^\top \Lambda^K\tilde\bV_i |^\nu ]
    \\
    &\;\;\;\;\;\;\;
    + 
    (n-1)^\nu \, \mean[ | \tilde \mu^\top \Lambda^K\tilde \bV_1 |^\nu ] 
    \Big)
    \\
    \;\leq&\;
    C_0 (n-1)^{\nu/2}
    \Big(  
    \mean[ |\tilde \bV_1^\top \Lambda^K\tilde \bV_2 |^\nu ]
    +
    \mean[ |\tilde \bZ_1^\top  \Lambda^K \tilde \bV_1 |^\nu ]
    + 
    (n-1)^{\nu/2}
    \mean[ | \tilde \mu^\top \Lambda^K\tilde \bV_1 |^\nu ]
    \Big)\;.
\end{align*}
By the exact same argument, the other expectation we want to bound can also be controlled as
\begin{align*}
    \mean\big[ | \partial F_i(\bzero)^\top &\tilde \bZ_i |^\nu \big]
    \\
    &\leq\;
    C_0 (n-1)^{\nu/2}
    \Big( 
    \mean[ |\tilde \bZ_1^\top \Lambda^K \tilde\bZ_2 |^\nu ]
    +
    \mean[ |\tilde \bZ_1^\top \Lambda^K \tilde \bV_1 |^\nu ]
    + 
    (n-1)^{\nu/2}
    \mean[ | \tilde \mu^\top \Lambda^K\tilde \bZ_1 |^\nu ]
    \Big)\;.
\end{align*}
Finally, we  relate these moments terms to moments of $u(\bX_1,\bX_2)$, up to error terms that vanish as $K \rightarrow \infty$: Denoting $\mu_k \coloneqq \mean[\phi_k(\bX_1)]$, we have that by Lemma \ref{lem:factorisable:moment},
\begin{align*}
    \mean[ | \tilde \mu^\top \Lambda^K \tilde \bV_1 |^\nu ]
    \;=\;
    \mfrac{1}{\sigma^{\nu} n^{\nu/2}(n-1)^{\nu/2}}
    \mean\Big[ \Big| \msum_{k=1}^K \lambda_k (\phi_k(\bX_1) - \mu_k) \mu_k \Big|^\nu \Big] 
    \;\leq\;
    \mfrac{
    4 ( (\Mcondnu)^\nu + \varepsilon_{K;\nu}^{\nu} )}{\sigma^{\nu}\, n^{\nu/2}(n-1)^{\nu/2}}\;,
\end{align*}
and for some absolute constant $C_1 > 0$,
\begin{align*}
    \mean[ |\tilde \bV_1^\top \Lambda^K \tilde \bV_2 |^\nu ]
    \;=&\;
    \mfrac{1}{\sigma^{\nu} n^{\nu/2}(n-1)^{\nu/2}} \mean\Big[ \Big|
    \msum_{k=1}^K \lambda_k (\phi_k(\bX_1) - \mu_k) (\phi_k(\bX_2)-\mu_k)
    \Big|^\nu \Big]
    \\
    \;\leq&\;
    \mfrac{4 C_1 (\Mfullnu)^\nu    
    -
    \frac{1}{2} (\Mcondnu)^\nu
    + 
    (4 C_1 + 2) \varepsilon_{K;\nu}^\nu}
    {\sigma^{\nu} n^{\nu/2}(n-1)^{\nu/2}}
    \;\leq\;
    \mfrac{4 C_1 (\Mfullnu)^\nu    
    + 
    (4 C_1 + 2) \varepsilon_{K;\nu}^\nu}
    {\sigma^{\nu} n^{\nu/2}(n-1)^{\nu/2}}
    \;.
\end{align*}
For the moment terms involving the Gaussians $\tilde \bZ_1$ and $\tilde \bZ_2$, we apply Lemma \ref{lem:factorisable:moment:gaussian} to show that 
\begin{align*}
    \mean[ | \tilde \mu^\top  \Lambda^K \tilde \bZ_1 |^\nu ]
    \;=&\;
    \mfrac{\mean[ | (\mean[\bV_1])^\top  \Lambda^K \bZ_1 |^\nu ]}{\sigma^{\nu} n^{\nu/2}(n-1)^{\nu/2}}
    \;\leq\; 
    \mfrac{
    7 (\sigcond^\nu + 8 \varepsilon_{K;2}^\nu )
    }{\sigma^{\nu} n^{\nu/2}(n-1)^{\nu/2}}
    \;\leq\;
    \mfrac{
    7 ((\Mcondnu)^\nu + 8 \varepsilon_{K;\nu}^\nu )
    }{\sigma^{\nu} n^{\nu/2}(n-1)^{\nu/2}}
    \;,
    \\
    \mean[ |\tilde \bZ_1^\top \Lambda^K \tilde \bZ_2 |^\nu ]
    \;=&\;
    \mfrac{\mean[ | \bZ_1^\top \Lambda^K \bZ_2 |^\nu ]}{\sigma^{\nu} n^{\nu/2}(n-1)^{\nu/2}}
     \;\leq\;
     \mfrac{
     6 ( \sigfull^\nu + \varepsilon_{K;2}^\nu)
     }{\sigma^{\nu} n^{\nu/2}(n-1)^{\nu/2}}
     \;\leq\;
     \mfrac{
     6 ( (\Mfullnu)^\nu + \varepsilon_{K;\nu}^\nu)
     }{\sigma^{\nu} n^{\nu/2}(n-1)^{\nu/2}}
     \;.
\end{align*}
In the last inequalities for both bounds, we have noted that $L_2$ norm is dominated by $L_\nu$ norm since $\nu > 2$. Meanwhile by Lemma \ref{lem:factorisable:moment:gaussian} again, there exists some absolute constant $C_2 > 0$ such that
\begin{align*}
    \mean[ |\tilde \bZ_1^\top \Lambda^K \tilde \bV_1 |^\nu ]
    \;=&\;
    \mfrac{\mean[ | (\bV_1 - \mean[\bV_1] )^\top \Lambda^K \bZ_1 |^\nu ]}{\sigma^{\nu} n^{\nu/2}(n-1)^{\nu/2}}
    \;\leq\;
    \mfrac{8 C_2 (\Mfullnu)^\nu + (8C_2+4) \varepsilon_{K;\nu}^\nu}{\sigma^{\nu} n^{\nu/2}(n-1)^{\nu/2}}
    \;.
\end{align*}
Substituting the five moment bounds into the earlier bounds on $\mean[ | \partial F_i(\bzero)^\top \tilde \bV_i |^\nu ]$ and $\mean[ | \partial F_i(\bzero)^\top \tilde \bZ_i |^\nu ]$ and combining the constant terms, we get that there exists an absolute constant $C>0$ such that
\begin{align*}
     \mean\big[ | \partial F_i(\bzero)^\top \tilde \bV_i |^\nu \big]
     +
      \mean\big[ | \partial F_i(\bzero)^\top \tilde \bZ_i |^\nu \big]
    \;\leq&\;
      \mfrac{C}{n^{\nu/2}} 
    \Big(
    \mfrac{ 
    (\Mfullnu)^\nu    
    + 
    \varepsilon_{K;\nu}^\nu}
    {\sigma^{\nu} }
    +
    \mfrac{
     (\Mcondnu)^\nu + \varepsilon_{K;\nu}^{\nu} }{(n-1)^{-\nu/2}\,\sigma^{\nu}}
    \Big)\;.
\end{align*}
\end{proof}

\subsection{Proof of Lemma \ref{lem:u:gaussian:quad:approx}}  \label{appendix:u:gaussian:quad:approx}

\emph{Proof overview.} For convenience, we write
\begin{align*}
    U_0 \;\coloneqq\; 
    \mfrac{1}{n(n-1)} \msum_{1 \leq i \neq j \leq n} (\eta^K_i)^\top (\Sigma^K)^{1/2} \Lambda^K (\Sigma^K)^{1/2} \eta^K_j
    +
    \mfrac{2}{n} \msum_{i=1}^n 
    (\mu^K)^\top \Lambda^K (\Sigma^K)^{1/2} \eta^K_i\;,
\end{align*}
so that
\begin{align*}
    &
    \tilde D^K_Z \;=\; \mfrac{\sqrt{n(n-1)}}{\sigma} U_0 +  \mfrac{\sqrt{n(n-1)}}{\sigma}(\mu^K)^\top \Lambda^K \mu^K\;,
    &&
    \tilde U_n^K \;=\; \mfrac{\sqrt{n(n-1)}}{\sigma} U_0 +  \mfrac{\sqrt{n(n-1)}}{\sigma} D\;.
\end{align*}
To approximate the distribution of $\tilde D^K_Z$ by that of $\tilde U_n^K$, the proof boils down to replacing $(\mu^K)^\top \Lambda^K \mu^K$ by $D$. We use a Markov-type argument so that we obtain an error term that is separate from the distribution terms.

\vspace{1em}

\begin{proof}[Proof of Lemma \ref{lem:u:gaussian:quad:approx}] Recall that Lemma \ref{lem:approx:XplusY:by:Y} allows us to approximate the distribution of a sum of two random variables by a single one provided that the other is negligible. Writing
\begin{align*}
    \tilde D_Z^K \;=\; \tilde U_n^K + (\tilde D_Z^K - \tilde U_n^K)
    \;=\; \tilde U_n^K + \mfrac{\sqrt{n(n-1)}}{\sigma} \big((\mu^K)^\top \Lambda^K \mu^K - D\big)\;,
\end{align*}
we can apply Lemma \ref{lem:approx:XplusY:by:Y} to obtain that for any $a, b \in \R$ and $\epsilon > 0$,
\begin{align*}
    \P( a \leq \tilde D_Z^K \leq b) 
    \;\leq&\;
    \P\big( a - \epsilon \leq \tilde U_n^K \leq b+\epsilon \big) 
    + 
    \P\Big( \mfrac{\sqrt{n(n-1)}}{\sigma} \big|(\mu^K)^\top \Lambda^K \mu^K - D\big|  \geq \epsilon \Big)\;,
    \\
    \P( a \leq \tilde D_Z^K \leq b) 
    \;\geq&\;
    \P( a + \epsilon \leq \tilde U_n^K \leq b-\epsilon ) 
    - 
    \P\Big( \mfrac{\sqrt{n(n-1)}}{\sigma} \big|(\mu^K)^\top \Lambda^K \mu^K - D\big|  \geq \epsilon \Big)\;.
\end{align*}
Note that $|(\mu^K)^\top \Lambda^K \mu^K - D|$ is deterministic. By a Markov inequality and the bound from Lemma \ref{lem:factorisable:moment}, we get that
\begin{align*}
    \P\Big( \mfrac{\sqrt{n(n-1)}}{\sigma} \big|(\mu^K)^\top \Lambda^K \mu^K - D\big|  \geq \epsilon \Big)
    \;\leq&\;
    \mfrac{\sqrt{n(n-1)}}{\epsilon \sigma} \mean\Big[ \big|(\mu^K)^\top \Lambda^K \mu^K - D\big| \Big]
    \\
    \;=&\;
    \mfrac{ \big| \msum_{k=1}^K \lambda_K \mu_k^2 - D \big|}{ \epsilon \, n^{-1/2} (n-1)^{-1/2} \sigma}
    \;\leq\;
    \mfrac{ \varepsilon_{K;1}}{ \epsilon \, n^{-1/2} (n-1)^{-1/2} \sigma}
    \;.
\end{align*}
Combining the two results gives the desired bounds.
\end{proof}

\subsection{Proof of Lemma \ref{lem:u:gaussian:quad:bound}}
\label{appendix:u:gaussian:quad:bound}

\emph{Proof overview.} The key ingredient of the proof is Theorem 8 of \cite{carbery2001distributional}, which gives an anti-concentration bound for the distribution of a polynomial of Gaussians in terms of its variance. In Lemma \ref{lem:cdf:bound:gaussian:quad}, we have rewritten the result in the special case of a degree-two polynomial, which allows us to control the distribution of $\tilde U^K_n$ in terms of its variance. 

\vspace{1em}

We introduce some matrix shorthands: For any $m \in \N$, denote $O_m$ as the zero matrix in $\R^{m \times m}$, $J_m$ as the all-one matrix in $\R^{m \times m}$ and $I_m$ as the identity matrix in $\R^{m \times m}$. Define the $nK \times nK$ matrix $M$ as
\begin{align*}
    M
    \;\coloneqq\;
    \begin{pmatrix}
    O_K & \Lambda^K & \ldots & \Lambda^K \\
    \Lambda^K & O_K & \ddots &  \vdots \\
    \vdots  & \ddots & \ddots & \Lambda^K \\
    \Lambda^K & \hdots & \Lambda^K & O_K \\
    \end{pmatrix}
    \;=\;
    \Lambda^K \otimes (J_n - I_n)
    \;,
\end{align*}
as well as
\begin{align*}
    \mu \coloneqq \big( (\mu^K)^\top, \ldots, (\mu^K)^\top \big)^\top \in \R^{nK}
    \;,
    \;\;
    \Sigma \coloneqq \Sigma^K  \otimes I_n  \in \R^{nK \times nK}\;,
    \;\;
    \Lambda \coloneqq \Lambda^K  \otimes I_n  \in \R^{nK \times nK}
    \;.
\end{align*} 
We also consider the concatenated $nK$-dimensional standard Gaussian vector
\begin{align*}
    \eta \;\coloneqq\;  
    \big( ( \eta_1^K)^\top, \ldots, ( \eta_n^K)^\top \big)^\top\;.
\end{align*}
  
\vspace{1em}

\begin{proof}[Proof of Lemma \ref{lem:u:gaussian:quad:bound}]  The goal is to bound the distribution function between $a \leq b \in \R$ of
\begin{align*}
     \tilde U_n^K  \;=\; \mfrac{\sqrt{n(n-1)}}{\sigma} \, U_n^K 
     \;=&\;
    \mfrac{1}{\sigma \sqrt{n(n-1)}} \msum_{1 \leq i \neq j \leq n} (\eta^K_i)^\top (\Sigma^K)^{1/2} \Lambda^K (\Sigma^K)^{1/2} \eta^K_j
    \\
    &\;
    +
    \mfrac{2 \sqrt{n-1}}{\sigma \sqrt{n} } \msum_{i=1}^n 
    (\mu^K)^\top \Lambda^K (\Sigma^K)^{1/2} \eta^K_i
    + 
    \mfrac{\sqrt{n(n-1)}}{\sigma}\,
    D
    \\
    \;=&\;
    \mfrac{1}{\sigma \sqrt{n(n-1)}} \eta^\top \Sigma^{1/2} M \Sigma^{1/2} \eta
    +
    \mfrac{2 \sqrt{n-1}}{\sigma \sqrt{n} } \mu^\top \Lambda \Sigma^{1/2} \eta
    +
    \mfrac{\sqrt{n(n-1)}}{\sigma}\,
    D
    \;.
\end{align*}
For convenience, define
\begin{align*}
    Q_1 \;\coloneqq\; \eta^\top \Sigma^{1/2} M \Sigma^{1/2} \eta\;,
    \qquad 
    Q_2 \;\coloneqq\; \mu^\top \Lambda \Sigma^{1/2} \eta\;,
    \qquad 
    \tilde U_0 \;\coloneqq\;
    \mfrac{1}{\sigma \sqrt{n(n-1)}} Q_1
    +
    \mfrac{2 \sqrt{n-1}}{\sigma \sqrt{n} } Q_2 \;.
\end{align*}
Denote $\alpha \coloneqq \frac{b-a}{2}$ and $\beta \coloneqq \frac{a+b}{2}$. Rewriting the probability in terms of $\tilde U_0$, $\alpha$ and $\beta$, we get that
\begin{align*}
    \P( a \leq \tilde U_n^K \leq b)
    \;=&\;
    \P\Big( (\beta-\alpha) \, 
    \leq 
    \tilde U_0 + \mfrac{\sqrt{n(n-1)}}{\sigma}\,
    D
    \leq \,  (\beta+\alpha) \Big)
    \\
    \;=&\;
    \P\Big(  \Big| \tilde U_0 
    + \mfrac{\sqrt{n(n-1)}}{\sigma}\,
    D
    - \beta \Big| 
    \;\leq\; 
    \alpha \Big) \;.
\end{align*}
Since $ \tilde U_0 + \frac{\sqrt{n(n-1)}}{\sigma}\, D- \beta
$ is a degree-two polynomial of $\eta$, we can apply Lemma \ref{lem:cdf:bound:gaussian:quad} to bound the above probability: For an absolute constant $C'$, we have
\begin{align*}
    \P( a \leq \tilde U_n^K \leq b) 
    \;\leq&\; C' \alpha^{1/2} \big(\Var[ \tilde U_0 ]\big)^{-1/4}\;,
    \tagaligneq \label{eqn:gaussian:quad:bound:intermediate}
\end{align*}
where the variance term can be expanded as
\begin{align*}
    \Var\big[\tilde U_0 \big] 
    \;=&\;
    \mfrac{1}{n(n-1)\sigma^2} \Var[Q_1]
    +
    \mfrac{4 (n-1)}{n \sigma^2}
    \Var[Q_2]
    + \mfrac{4}{n \, \sigma^2} \, \Cov[ Q_1, Q_2
    ] \;.
\end{align*}
We now provide bound the individual terms in the variance. By noting that each summand in $Q_1$ is zero-mean when $i \neq j$ and that each summand in $Q_2$ is zero-mean, the covariance term can be computed as
\begin{align*}
    \Cov[ Q_1, Q_2
    ]
    \;=&\; 
    \msum_{1 \leq i \neq j \leq n}\msum_{l=1}^n 
    \mean \Big[ 
    (\eta^K_i)^\top (\Sigma^K)^{1/2} \Lambda^K (\Sigma^K)^{1/2} \eta^K_j
    \times 
    (\mu^K)^\top \Lambda^K (\Sigma^K)^{1/2} \eta^K_l
    \Big]
    \\
    \;=&\;
    \mfrac{1}{2}
    \mean\Big[  (\eta^K_1)^\top (\Sigma^K)^{1/2} \Lambda^K (\Sigma^K)^{1/2} \eta^K_1
    \times 
    (\mu^K)^\top \Lambda^K (\Sigma^K)^{1/2} \eta^K_1
    \Big]\;.
\end{align*}
Denote $\xi_k$ as the $k$-th coordinate of $\eta_1^K$. Then the above expectation is taken over a linear combination of terms of the form $\xi_{k_1} \xi_{k_2} \xi_{k_3}$. If any of $k_1, k_2, k_3$ is distinct from the other two indices, the expectation is zero; if $k_1=k_2=k_3$, the expectation is again zero by property of a standard Gauassian. Therefore, we have
\begin{align*}
    \Cov[ Q_1, Q_2
    ]\;=\;0\;.
\end{align*}
On the other hand, the first variance can be computed by using the moment formula for a quadratic form of Gaussian from Lemma \ref{lem:quadratic:gaussian:moments} and the cyclic property of trace:
\begin{align*}
    \Var[Q_1] \;=&\; 2 \Tr\big( (\Sigma^{1/2} M \Sigma^{1/2})^2 \big)
    \;=\;
    2 \Tr\big( (\Sigma M)^2 \big)
    \\
    \;=&\;
    2 \Tr\big( (\Sigma^K \Lambda^K)^2 \otimes (J_n - I_n)^2 \big)
    \\
    \;=&\;
    2 \Tr\big( (\Sigma^K \Lambda^K)^2 \otimes J_n^2 \big)
    -
    4 \Tr\big( (\Sigma^K \Lambda^K)^2 \otimes J_n \big)
    +
    2 \Tr\big( (\Sigma^K \Lambda^K)^2 \otimes I_n \big)
    \\
    \;=&\; \big( 2n^2 - 4n + 2n \big)
    \Tr\big( (\Sigma^K \Lambda^K)^2 \big)
    \\
    \;=&\;
    2 n (n-1)\Tr\big( (\Lambda^K\Sigma^K)^2 \big)
    \\
    \;\geq&\; 2 n(n-1) (\sigfull-\varepsilon_{K;2})^2\;.
\end{align*}
In the last inequality, we have used the bound from Lemma \ref{lem:factorisable:moment:gaussian} on $\Tr\big( (\Lambda^K\Sigma^K)^2 \big)$. The second variance is on a Gaussian random variable and can be bounded by Lemma \ref{lem:factorisable:moment:gaussian} again as
\begin{align*}
    \Var[Q_2] \;=&\; \mu^\top \Lambda \Sigma \Lambda \mu
    \;=\;
    n (\mu^K)^\top \Lambda^K \Sigma^K \Lambda^K \mu^K
    \;\geq\;
    n (\sigcond^2 -2 \sigcond \varepsilon_{K;2} - 4 \varepsilon_{K;2})
    \;.
\end{align*}
This implies that 
\begin{align*}
    \Var\big[\tilde U_0 \big] 
    \;\geq&\;
    \mfrac{2}{\sigma^2} (\sigfull - \varepsilon_{K;2})^2
    +
    \mfrac{4 (n-1)}{\sigma^2} (\sigcond^2 -2 \sigcond \varepsilon_{K;2} - 4 \varepsilon_{K;2})
    \;.
\end{align*}
Substituting this into \eqref{eqn:gaussian:quad:bound:intermediate} and redefining the constants, we get that there exists an absolute constant $C$ such that
\begin{align*}
    \P( a \leq \tilde U_n^K \leq b) 
    \;\leq&\; C (b-a)^{1/2}  \Big( 
    \mfrac{1}{\sigma^2} (\sigfull - \varepsilon_{K;2})^2
    +
    \mfrac{n-1}{\sigma^2} (\sigcond^2 -2 \sigcond \varepsilon_{K;2} - 4 \varepsilon_{K;2})
    \Big)^{-1/4}\;.
\end{align*}
\end{proof}

\section{Proofs for the remaining results in \cref{sect:general:results}} \label{appendix:proof:remaining}

\subsection{Proofs for variants and corollaries of the main result}

The upper bound in Proposition \ref{prop:u:gaussian:quad:distribution bounds} is a concentration inequality and is obtained by a standard argument via Chebyshev's inequality. The lower bound is a combination of the anti-concentration bound for a Gaussian quadratic form from Lemma \ref{lem:u:gaussian:quad:bound} and \cref{thm:u:gaussian:quad:general}.

\vspace{1em}

\begin{proof}[Proof of Proposition \ref{prop:u:gaussian:quad:distribution bounds}] Denote
$\tilde U_n^K \coloneqq \frac{\sqrt{n(n-1)} U_n^K}{\sigmax}$. In Lemma \ref{lem:u:gaussian:quad:bound}, we have shown that for any $a, b \in \R$ with $a \leq b$, there exists some absolute constant $C'$ such that
\begin{align*}
    \P( a \leq \tilde U_n^K \leq b) 
    \leq&\; C' (b-a)^{1/2}  \Big( 
    \mfrac{1}{\sigmax^2} (\sigfull - \varepsilon_{K;2})^2
    +
    \mfrac{n-1}{\sigmax^2} (\sigcond^2 -2 \sigcond \varepsilon_{K;2} - 4 \varepsilon_{K;2})
    \Big)^{-1/4}\;.
\end{align*}
Take $K \rightarrow \infty$ and using \cref{assumption:L_nu} for $\nu \geq 2$, we get that $\varepsilon_{K;2} \rightarrow 0$. For a fixed $\epsilon > 0$, set $a=\frac{\sqrt{n(n-1)}}{\sigmax} D-\epsilon$ and $b=\frac{\sqrt{n(n-1)}}{\sigmax}D+\epsilon$, we get that
\begin{align*}
    \lim_{K \rightarrow \infty} 
    \P\Big( \mfrac{\sqrt{n(n-1)}}{\sigmax} | U_n^K - D | \leq \epsilon \Big) 
    \;\leq&\; 
    \sqrt{2}\,C'\, \epsilon^{1/2} \Big( \mfrac{\sigfull^2}{\sigmax^2} 
    +
    \mfrac{(n-1)\sigcond^2}{\sigmax^2} \Big)^{-1/4}
    \;\leq\;
    \sqrt{2}\,C'\, \epsilon^{1/2} 
    \;.
\end{align*}
Now by \cref{thm:u:gaussian:quad:general}, there exists an absolute constant $C''$ such that
\begin{align*}
    \msup_{t \in \R}
    \Big| \P\Big( \mfrac{\sqrt{n(n-1)}}{\sigmax} D_n > t \Big)  - \lim_{K \rightarrow \infty} \P\Big( \mfrac{\sqrt{n(n-1)}}{\sigmax} U_n^K > t\Big) \Big|
    \;\leq&\;
    C'' \, n^{- \frac{\nu - 2}{4\nu+2}}  \Big(
    \mfrac{\Mmaxnu}{\sigmax} \Big)^{\frac{\nu}{2\nu+1}}\;.
\end{align*}
By a triangle inequality, we get that 
\begin{align*}
     \P\Big( \mfrac{\sqrt{n(n-1)}}{\sigmax} | D_n - D| > \epsilon \Big)
     \;\geq&\;
     \P\Big( \mfrac{\sqrt{n(n-1)}}{\sigmax} | U_n^K - D| > \epsilon \Big)
     - 2  C'' \, n^{- \frac{\nu - 2}{4\nu+2}}  \Big(
    \mfrac{\Mmaxnu}{\sigmax} \Big)^{\frac{\nu}{2\nu+1}}
    \\
    \;\geq&\;
    1
    -
    \sqrt{2} \, C' \epsilon^{1/2}
    -
    2  C'' \, n^{- \frac{\nu - 2}{4\nu+2}}  \Big(
    \mfrac{\Mmaxnu}{\sigmax} \Big)^{\frac{\nu}{2\nu+1}}\;.
\end{align*}
By replacing $\epsilon$ with $\frac{\sqrt{n(n-1)}}{\sigmax} \epsilon$ and redefining constants, we get the desired lower bound that there exists absolute constants $C_1, C_2 > 0$ such that
\begin{align*}
    \P( | D_n - D| > \epsilon )
    \;\geq&\;
    1
    -
    C_1 \Big( \mfrac{\sqrt{n(n-1)}}{\sigmax} \Big)^{1/2} \epsilon^{1/2}
    -
    C_2 \, n^{- \frac{\nu - 2}{4\nu+2}}  \Big(
    \mfrac{\Mmaxnu}{\sigmax} \Big)^{\frac{\nu}{2\nu+1}}\;.
\end{align*}
For the upper bound, we apply a Chebyshev inequality directly to $D_n$ and bound the variance by Lemma \ref{lem:u:moments}: There exists some absolute constant $C'_3 > 0$ such that
\begin{align*}
    \P( | D_n - D| > \epsilon )
    \;\leq\;
    \epsilon^{-2} \Var[ D_n]
    \;\leq&\;
    C'_3 \epsilon^{-2}
    \Big(\mfrac{\sigcond^2}{n^{-1}(n-1)^2}
    +  
    \mfrac{\sigfull^2}{(n-1)^2} \Big)
    \\
    \;\leq&\;
    C'_3 \epsilon^{-2}
    \Big(\mfrac{\sigmax}{n-1} \Big)^2
    \;\leq\;
    C_3 \epsilon^{-2}
    \Big(\mfrac{\sigmax}{\sqrt{n(n-1)}} \Big)^2\;.
\end{align*}
In the last inequality, we have noted that $\frac{1}{n-1} \leq \frac{2}{n}$ for $n \geq 2$ and defined $C_3 = 2 C'_3$. This finishes the proof.
\end{proof}

\vspace{1em}

\cref{thm:u:gaussian:quad:general} provides an approximation of the distribution of $D_n$ by that of a Gaussian quadratic form. Proposition \ref{prop:u:chi:sq:general} combines \cref{thm:u:gaussian:quad:general} with a Markov argument, which makes a further approximation of the Gaussian quadratic form by a weighted sum of chi-squares $U_n^K$. The approximation error introduced vanishes as $n,d$ grow provided that $\rho_d = \omega(n^{1/2})$, i.e.~$n^{-1/2} \sigfull = \omega ( \sigcond )$. 

\vspace{1em}

\begin{proof}[Proof of Proposition \ref{prop:u:chi:sq:general}] We first seek to compare $W_n^K$ to the distribution of
\begin{align*}
    U_n^K
    \;=\; 
    \mfrac{1}{n(n-1)} \msum_{1 \leq i \neq j \leq n} (\eta^K_i)^\top (\Sigma^K)^{1/2} \Lambda^K (\Sigma^K)^{1/2} \eta^K_j
    +
    \mfrac{2}{n} \msum_{i=1}^n 
    (\mu^K)^\top \Lambda^K (\Sigma^K)^{1/2} \eta^K_i
    +
    D\;,
\end{align*}
where $\{\eta^K_i\}_{i=1}^n$ are i.i.d.\,standard Gaussian vectors in $\R^K$. The first step is to write
\begin{align*}
    U_n^K \;=\; \mfrac{\sqrt{n-1}}{\sqrt{n}} W_0 + D +  \Big(1-\mfrac{\sqrt{n-1}}{\sqrt{n}}\Big) W_0 + W_1 + W_2\;,
\end{align*}
where we have defined the zero-mean random variables
\begin{align*}
    W_0 \;\coloneqq&\;
    \mfrac{1}{n(n-1)} 
    \Big( \msum_{i,j=1}^n (\eta^K_i)^\top (\Sigma^K)^{1/2} \Lambda^K (\Sigma^K)^{1/2} \eta^K_j
    -
    n \Tr( \Sigma^K \Lambda^K )
    \Big)\;,
    \\
    W_1
    \;\coloneqq&\;
    \mfrac{1}{n(n-1)} 
    \Big(
    \msum_{i=1}^n (\eta^K_i)^\top (\Sigma^K)^{1/2} \Lambda^K (\Sigma^K)^{1/2} \eta^K_i
    -
    n \Tr( \Sigma^K \Lambda^K )
    \Big)
    \;,
    \\
    W_2
    \;\coloneqq&\;
    \mfrac{2}{n} \msum_{i=1}^n 
    (\mu^K)^\top \Lambda^K (\Sigma^K)^{1/2} \eta^K_i
    \;.
\end{align*}
Fix $\epsilon_0, \epsilon_1, \epsilon_2 > 0$. We first use the bound from Lemma \ref{lem:approx:XplusY:by:Y}: For any $a, b \in \R$, we have
\begin{align*}
    &\P\Big( a \leq \mfrac{\sqrt{n(n-1)}}{\sigfull} \Big( \mfrac{\sqrt{n-1}}{\sqrt{n}} W_0+D \Big) \leq b \Big) 
    \\
    \;\leq&\;
    \P\Big( a - \epsilon_0 - \epsilon_1 - \epsilon_2 \leq \mfrac{\sqrt{n(n-1)}}{\sigfull} U_n^K \leq b + \epsilon_0 + \epsilon_1 + \epsilon_2  \Big)     \\
    &\;
    + \P\Big( \mfrac{\sqrt{n(n-1)}}{\sigfull} \Big( 1 - \mfrac{\sqrt{n-1}}{\sqrt{n}}\Big)  | W_0| \geq \epsilon_0 \Big)
    + \P\Big( \mfrac{\sqrt{n(n-1)}}{\sigfull}  |W_1| \geq \epsilon_1 \Big)
    \\
    &\;
    + \P\Big( \mfrac{\sqrt{n(n-1)}}{\sigfull}  |W_2| \geq \epsilon_2 \Big)
    \;
\end{align*}
and
\begin{align*}
    &\P\Big( a \leq \mfrac{\sqrt{n(n-1)}}{\sigfull} \Big( \mfrac{\sqrt{n-1}}{\sqrt{n}} W_0+D \Big) \leq b \Big) 
    \\
    \;\geq&\;
    \P\Big( a + \epsilon_0  + \epsilon_1 + \epsilon_2 \leq \mfrac{\sqrt{n(n-1)}}{\sigfull} U_n^K  \leq b - \epsilon_0 - \epsilon_1 - \epsilon_2  \Big) 
    \\
    &\;
    - \P\Big( \mfrac{\sqrt{n(n-1)}}{\sigfull} \Big( 1 - \mfrac{\sqrt{n-1}}{\sqrt{n}}\Big)  | W_0| \geq \epsilon_0 \Big)
    - \P\Big( \mfrac{\sqrt{n(n-1)}}{\sigfull}  |W_1| \geq \epsilon_1 \Big)
    \\
    &\;- \P\Big( \mfrac{\sqrt{n(n-1)}}{\sigfull}  |W_2| \geq \epsilon_2 \Big)
    \;.
\end{align*}
We now bound the error terms. By the Chebyshev's inequality, the variance formula of a quadratic form of Gaussians from Lemma \ref{lem:quadratic:gaussian:moments} and the bound from Lemma \ref{lem:factorisable:moment:gaussian}, we get that
\begin{align*}
    \P\Big( \mfrac{\sqrt{n(n-1)}}{\sigfull} |W_1| \geq \epsilon_1 \Big)
    \;\leq\;
    \epsilon_1^{-2} \Var\Big[ \mfrac{\sqrt{n(n-1)}}{\sigfull} W_1 \Big]
    \;=&\;
     \mfrac{2}{\epsilon_1^2 (n-1) \sigfull^2} \Tr\big( ( \Lambda^K \Sigma^K)^2 \big)
     \\
     \;\leq&\;
     \mfrac{2 (\sigfull+\varepsilon_{K;2})^2}{\epsilon_1^2 (n-1) \sigfull^2} \;.
\end{align*}
Similarly, by the Chebyshev's inequality, the variance formula of a Gaussian and the bound from Lemma \ref{lem:factorisable:moment:gaussian}, we get that
\begin{align*}
    \P\Big( \mfrac{\sqrt{n(n-1)}}{\sigfull} |W_2| \geq \epsilon \Big)
    \;\leq\;
    \epsilon_2^{-2} \Var\Big[ \mfrac{\sqrt{n(n-1)}}{\sigfull} W_2  \Big]
    \;=&\;
     \mfrac{4(n-1)}{\epsilon_2^2 \sigfull^2} \mean\big[ (\mu^K)^\top \Lambda^K \Sigma^K \Lambda^K \mu^K \big]
    \\
    \;\leq&\;
    \mfrac{4(n-1)(\sigcond+2\varepsilon_{K;2})^2}{\epsilon_2^2 \sigfull^2}\;.
\end{align*}
By Lemma \ref{lem:WnK:moments}, we can replace $W_0$ by using the following equality in distribution:
\begin{align*}
    \mfrac{\sqrt{n-1}}{\sqrt{n}} W_0
    \;=&\;
    \mfrac{1}{n^{3/2}(n-1)^{1/2}} 
    \Big( \msum_{i,j=1}^n (\eta^K_i)^\top (\Sigma^K)^{1/2} \Lambda^K (\Sigma^K)^{1/2} \eta^K_j
    -
    n \Tr( \Sigma^K \Lambda^K )
    \Big)
    \\
    \;\overset{d}{=}&\; W_n^K - D\;.
\end{align*}
Finally, using a Chebyshev's inequality together with the moment bound in Lemma \ref{lem:WnK:moments}, we get that
\begin{align*}
    \P\Big( \mfrac{\sqrt{n(n-1)}}{\sigfull} \Big( 1 - \mfrac{\sqrt{n-1}}{\sqrt{n}}\Big)  | W_0| \geq \epsilon_0 \Big)
    \;\leq&\;
    \mfrac{n(n-1)}{ \epsilon_0^2 \sigfull^2} \Big( 1 - \mfrac{\sqrt{n-1}}{\sqrt{n}} \Big)^2 
    \Var\big[ W_0 \big]
    \\
    \;=&\;
    \mfrac{n^2}{ \epsilon_0^2 \sigfull^2} \Big( 1 - \mfrac{\sqrt{n-1}}{\sqrt{n}} \Big)^2 
    \Var\big[ W_n^K \big]
    \\
    \;\leq&\;
    \mfrac{2n (\sigfull+\varepsilon_{K;2})^2}{\epsilon_0^2  (n-1) \sigfull^2} \Big( 1 - \mfrac{\sqrt{n-1}}{\sqrt{n}} \Big)^2 
    \\
    \;\leq&\;
    \mfrac{2 (\sigfull+\varepsilon_{K;2})^2}{\epsilon_0^2 (n-1) \sigfull^2}\;.
\end{align*}
In the last inequality, we have noted that $\sqrt{n} - \sqrt{n-1} \leq 1$. Combining the above bounds, we get that
\begin{align*}
    \P\Big( a \leq \mfrac{\sqrt{n(n-1)}}{\sigfull} W_n^K \leq b \Big) 
    \;\leq&\;
    \P\Big( a - \epsilon_0 - \epsilon_1 - \epsilon_2 \leq \mfrac{\sqrt{n(n-1)}}{\sigfull} U_n^K \leq b + \epsilon_0 + \epsilon_1 + \epsilon_2  \Big) 
    \\
    &\;
    +
    \mfrac{2 (\sigfull+\varepsilon_{K;2})^2}{(n-1) \sigfull^2} \big( \epsilon_0^{-2} + \epsilon_1^{-2} \big)
    +
    \mfrac{4(n-1)(\sigcond+2\varepsilon_{K;2})^2}{\epsilon_2^2 \sigfull^2}
    \;,
    \\
    \P\Big( a \leq \mfrac{\sqrt{n(n-1)}}{\sigfull}  W_n^K \leq b \Big) 
    \;\geq&\;
    \P\Big( a + \epsilon_0 + \epsilon_1 + \epsilon_2 \leq \mfrac{\sqrt{n(n-1)}}{\sigfull} U_n^K  \leq b - \epsilon_0 - \epsilon_1 - \epsilon_2  ) 
    \\
    &\;
    -
    \mfrac{2 (\sigfull+\varepsilon_{K;2})^2}{(n-1) \sigfull^2} \big( \epsilon_0^{-2} + \epsilon_1^{-2} \big)
    -
    \mfrac{4(n-1)(\sigcond+2\varepsilon_{K;2})^2}{\epsilon_2^2 \sigfull^2}
    \;.
\end{align*}
Taking $b \rightarrow \infty$ and $a \rightarrow t$ from the right, we get that
\begin{align*}
    \Big| 
    \P\Big( &\mfrac{\sqrt{n(n-1)}}{\sigfull} W_n^K > t \Big)  
    -
    \P\Big(  \mfrac{\sqrt{n(n-1)}}{\sigfull} U_n^K > t  \Big) 
    \Big|
    \\
    \;\leq&\;
    \max\Big\{ 
    \P\Big( t-\epsilon_0-\epsilon_1-\epsilon_2 \leq \mfrac{\sqrt{n(n-1)}}{\sigfull} U_n^K \leq t  \Big)\;,\;
     \P\Big( t \leq \mfrac{\sqrt{n(n-1)}}{\sigfull} U_n^K \leq \epsilon_0+\epsilon_1+\epsilon_2  \Big)
     \Big\}  
     \\
     &\;
     + 
    \mfrac{2 (\sigfull+\varepsilon_{K;2})^2}{(n-1) \sigfull^2} \big( \epsilon_0^{-2} + \epsilon_1^{-2} \big)
    +
    \mfrac{4(n-1)(\sigcond+2\varepsilon_{K;2})^2}{\epsilon_2^2 \sigfull^2}\;.
\end{align*}
This allows us to follow a similar argument to the proof of \cref{thm:u:gaussian:quad:general} to approximate $W_n^K$ by $U_n^K$. To bound the maxima, we apply Lemma \ref{lem:u:gaussian:quad:bound} with $\sigma=\sigfull$: There exists some absolute constant $C'$ such that for any $a \leq b \in \R$,
\begin{align*}
    \P\Big( a \leq  \mfrac{\sqrt{n(n-1)}}{\sigfull} &U_n^K \leq b \Big) 
    \\
    \;\leq&\; C' (b-a)^{1/2}  \Big( 
    \mfrac{1}{\sigfull^2} (\sigfull - \varepsilon_{K;2})^2
    +
    \mfrac{n-1}{\sigfull^2} (\sigcond^2 -2 \sigcond \varepsilon_{K;2} - 4 \varepsilon_{K;2})
    \Big)^{-1/4}\;.
\end{align*}
By additionally noting that $(\epsilon_0 + \epsilon_1+\epsilon_2)^{1/2} \leq \sqrt{\epsilon_0} + \sqrt{\epsilon_1} + \sqrt{\epsilon_2}$, we get that
\begin{align*}
    \Big| 
    \P\Big( &\mfrac{\sqrt{n(n-1)}}{\sigfull} W_n^K > t \Big)  
    -
    \P\Big(  \mfrac{\sqrt{n(n-1)}}{\sigfull} U_n^K > t  \Big) 
    \Big|
    \\
    \;\leq&\;
    C' (\sqrt{\epsilon_0} + \sqrt{\epsilon_1} + \sqrt{\epsilon_2})  
    \Big( 
    \mfrac{1}{\sigfull^2} (\sigfull - \varepsilon_{K;2})^2
    +
    \mfrac{n-1}{\sigfull^2} (\sigcond^2 -2 \sigcond \varepsilon_{K;2} - 4 \varepsilon_{K;2})
    \Big)^{-1/4}
     \\
     &\;
     + 
    \mfrac{2 (\sigfull+\varepsilon_{K;2})^2}{(n-1) \sigfull^2} \big( \epsilon_0^{-2} + \epsilon_1^{-2} \big)
    +
    \mfrac{4(n-1)(\sigcond+2\varepsilon_{K;2})^2}{\epsilon_2^2 \sigfull^2}\;.
\end{align*}
Taking $K \rightarrow \infty$ on both sides, the inequality becomes
\begin{align*}
    \Big| 
    \lim_{K \rightarrow \infty} \P\Big( &\mfrac{\sqrt{n(n-1)}}{\sigfull} W_n^K > t \Big)  
    -
    \lim_{K \rightarrow \infty} \P\Big(  \mfrac{\sqrt{n(n-1)}}{\sigfull} U_n^K > t  \Big) 
    \Big|
    \\
    \;\leq&\;
    C' (\sqrt{\epsilon_0} + \sqrt{\epsilon_1} + \sqrt{\epsilon_2})  
    \Big( 1
    +
    \mfrac{(n-1)\sigcond^2}{\sigfull^2} 
    \Big)^{-1/4}
    + 
    \mfrac{2}{n-1} \big(  \epsilon_0^{-2} + \epsilon_1^{-2} \big)
    +
    \mfrac{4(n-1) \sigcond^2}{\epsilon_2^2 \sigfull^2}
    \\
    \;\leq&\;
    C' (\sqrt{\epsilon_0} + \sqrt{\epsilon_1} + \sqrt{\epsilon_2})  
    + 
    \mfrac{2}{n-1} \big(  \epsilon_0^{-2} + \epsilon_1^{-2} \big)
    +
    \mfrac{4(n-1) \sigcond^2}{\epsilon_2^2 \sigfull^2}
    \;.
\end{align*}
Choosing $\epsilon_0 = \epsilon_1 = (n-1)^{-2/5}$ and $\epsilon_2 = \big( (n-1)\sigcond^2/\sigfull^2\big)^{2/5}$, redefining constants and taking a supremum over $t \in \R$, we get that there exists some absolute constant $C'' > 0$ such that
\begin{align*}
    \msup_{t \in \R} \Big| 
    \lim_{K \rightarrow \infty} \P\Big( \mfrac{\sqrt{n(n-1)}}{\sigfull} W_n^K > t \Big)  
    -
    \lim_{K \rightarrow \infty} \P\Big(  &\mfrac{\sqrt{n(n-1)}}{\sigfull} U_n^K > t  \Big) 
    \Big|
    \\
    \;\leq&\;
    C'' \Big( \mfrac{1}{(n-1)^{1/5}} + \Big( \mfrac{\sqrt{n-1}\, \sigcond}{\sigfull} \Big)^{2/5} \Big)  
    \;.
\end{align*}
The final step is to relate this bound to $D_n$. Consider the last step \eqref{eqn:thm:u:gaussian:quad:general:last:step} of the proof of \cref{thm:u:gaussian:quad:general} in \cref{appendix:u:gaussian:quad:thm:general:body}. If we set $\sigma=\sigfull$ instead of $\sigmax$, we get that there exists some absolute constant $C''' > 0$ such that
\begin{align*}
     \msup_{t \in \R} \Big| \P\Big( \mfrac{\sqrt{n(n-1)}}{\sigfull} D_n > t \Big)  - \lim_{K \rightarrow \infty} \P\Big( &\mfrac{\sqrt{n(n-1)}}{\sigfull} U_n^K > t\Big) \Big|
    \\
    \;&\leq\;
    C''' \, n^{- \frac{\nu - 2}{4\nu+2}}  \Big(
    \mfrac{ 
    (\Mfullnu)^\nu}
    {\sigfull^{\nu} }
    +
    \mfrac{
     (\Mcondnu)^\nu }{(n-1)^{-\nu/2}\,\sigfull^{\nu}}
    \Big)^{\frac{1}{2\nu+1}} \;.
\end{align*}
Setting $C= \max\{C'', C'''\}$ and using a triangle inequality, we get the desired bound that
\begin{align*}
    \msup_{t \in \R} &\Big| \P\Big( \mfrac{\sqrt{n(n-1)}}{\sigfull} D_n > t \Big) - \lim_{K \rightarrow \infty} \P\Big( \mfrac{\sqrt{n(n-1)}}{\sigfull} W_n^K > t\Big) \Big|
    \\
    \;&\leq\;
    C
    \Big( 
    \mfrac{1}{(n-1)^{1/5}} 
    + 
    \Big( \mfrac{\sqrt{n-1}\, \sigcond}{\sigfull} \Big)^{2/5} 
    +
    \, n^{- \frac{\nu - 2}{4\nu+2}}  \Big(
    \mfrac{ 
    (\Mfullnu)^\nu}
    {\sigfull^{\nu} }
    +
    \mfrac{
     (\Mcondnu)^\nu }{(n-1)^{-\nu/2}\,\sigfull^{\nu}}
    \Big)^{\frac{1}{2\nu+1}}  \Big)\;.
\end{align*}
\end{proof}

\subsection{Proofs for results on $W_n$}

\begin{proof}[Proof of Proposition \ref{prop:W_n:existence:all:moments}] To prove the existence of distribution, we seek to apply L\'evy's continuity theorem. We first verify that there exists a sufficiently large $K^*$ such that the sequence $(W_n^K)_{K \geq K^*}$ is tight. Since  \cref{assumption:L_nu} holds for some $\nu \geq 2$, we get that as $K \rightarrow \infty$,
\begin{align*}
    \varepsilon_{K;2} \;\coloneqq\; \mean\big[ \big| \msum_{k=1}^K \lambda_k \phi_k(\bX_1)\phi_k(\bX_2) - u(\bX_1,\bX_2) \big|^2 \big]^{1/2}  \rightarrow 0\;.
\end{align*}
In particular, there exists some sufficiently large $K^*$ such that $\varepsilon_{K;2} \leq 1$ for all $K \geq K^*$. By Lemma \ref{lem:WnK:moments}, we have that for all $K \geq K^*$,
\begin{align*}
    \Var[ W_n^{K} ] 
    \;\leq\; 
    \mfrac{2}{n(n-1)} (\sigfull + \varepsilon_{K;2})^2
    \;\leq\;
    \mfrac{2}{n(n-1)} (\sigfull + 1)^2
    \;.
\end{align*}
Note that by assumption, we have $|D| , \sigfull < \infty$. This implies that the sequence $(W_n^K)_{K \geq K^*}$ is tight by a Markov inequality:
\begin{align*}
    \lim_{x \rightarrow \infty} \Big( \msup_{K \geq K^*} \P\big( \big| W_n^K \big| > x \big) \Big)
    \;\leq&\;
    \lim_{x \rightarrow \infty} \Big( x^{-2} \msup_{K \geq K^*} \mean[ (W_n^K)^2 ] \Big)
    \\
    \;\leq&\;
    \lim_{x \rightarrow \infty} \mfrac{2 n^{-1}(n-1)^{-1} (\sigfull + 1)^2 + D^2}{x^2}  
    \;=\;0\;.
\end{align*}
We defer to Lemma \ref{lem:W_n:exist:characteristic} to show that the characteristic function of $(W_n^K - D)$ converges pointwise as $K \rightarrow \infty$. This allows us to apply L\'evy's continuity theorem and obtain that $W_n$ exists.
\end{proof}

\vspace{1em}

\begin{proof}[Proof of Lemma \ref{lem:chisq:limit:not:gaussian}] The result holds by noting that for all $k > K^*$, $W_n^K = W_n^{K^*}$ almost surely, and the latter random variable does not depend on $K$.
\end{proof}

\vspace{1em}

\begin{lemma} \label{lem:W_n:exist:characteristic} The characteristic function of $(W_n^K - D)$ converges pointwise as $K \rightarrow \infty$.
\end{lemma}
\begin{proof}[Proof of Lemma \ref{lem:W_n:exist:characteristic}]
Define $a_k \coloneqq \frac{1}{\sqrt{n(n-1)}}\tau_{k;d}$ and  $T_k \coloneqq a_k (\xi_k^2 - 1)$, which allows us to write
\begin{align*}
    W_n^K  
    \;=\;
    \mfrac{1}{\sqrt{n(n-1)}}
    \msum_{k=1}^K \tau_{k;d} (\xi_k^2 - 1)
    +
    D
    \;=\;
    \msum_{k=1}^K T_k
    +
    D
    \;.
\end{align*}
Denote $i=\sqrt{-1}$ as the imaginary unit and $Y$ as a chi-squared random variable with degree $1$. Since each $T_k$ is a scaled and shifted chi-squared random variable with degree $1$, it has the characteristic function
\begin{align*}
    \psi_{T_k}(t) \;=\; \mean[ \exp(i t \, T_k) ] \;=\; 
    \mean[ \exp( i a_k Y t ) ] \exp ( - i  a_k t )
    \;=&\;
    ( 1 - 2 i a_k t )^{-1/2} \exp ( - i  a_k t )
    \;.
\end{align*}
Since $T_k$'s are independent, by the convolution theorem, the characteristic function of $W_n^K - D$ is given by
\begin{align*}
    \psi_{W_n^K - D}(t) 
    \;=\; 
    \exp \Big( - i  \msum_{k=1}^K a_k t \Big)
    \mprod_{k=1}^K ( 1 - 2 i a_k t )^{-1/2}
    \;.
\end{align*}
We want to prove that for every $t \in \R$, $\psi_{W_n^K - D}(t)$ converges to some function as limit $K \rightarrow \infty$. By taking the principal-valued complex logarithm (i.e.~discontinuity along negative real axis), we get that
\begin{align*}
    \log \psi_{W_n^K - D}(t) 
    \;=\; 
    \msum_{k=1}^K \Big( - i  a_k t  - \mfrac{1}{2} \log(1 - 2 i a_k t) \Big)  
    + 2 i m_K \pi 
    \;\eqqcolon\; S_K +  2 i m_K \pi \;, \tagaligneq \label{eqn:W_n:exist:mK}
\end{align*}
for some $m_K \in \N$ for each $K$ that adjusts for values at discontinuity. Now consider the real part of the logarithm:
\begin{align*}
    {\rm{Re}}\big(\log \psi_{W_n^K - D}(t) \big)
    \;=\;
    {\rm{Re}}(S_K)
    \;=&\;
    - \mfrac{1}{2} \msum_{k=1}^K \log | 1 - 2 i a_k t | 
    \\
    \;=&\;
    - \mfrac{1}{2} \msum_{k=1}^K \log\sqrt{1 + 4 a_k^2 t^2 }
    \;=\;
    - \mfrac{1}{4} \msum_{k=1}^K \log(1 + 4 a_k^2 t^2)\;.
\end{align*}
Recall by Lemma \ref{lem:factorisable:moment:gaussian} that 
\begin{align*}
     \msum_{k=1}^K a_k^2
     \;=\; 
     \mfrac{1}{n(n-1)}  \msum_{k=1}^K \tau_{k;d}^2  
     \;=\;
     \Tr( (\Sigma^K \Lambda^K)^2 )
     \xrightarrow{K \rightarrow \infty}  \sigfull^2\;. \tagaligneq \label{eqn:W_n:exist:aksquared:conv}
\end{align*}
Fix $\epsilon > 0$. The above implies that there exists a sufficiently large $K^*$ such that for all $K_1, K_2 \geq K^*$, $\sum_{k=K_1}^{K_2} a_k^2 < \epsilon$. Then for all $K_1, K_2 \geq K^*$, we have
\begin{align*}
    0 \;\leq\; \msum_{k=K_1}^{K_2} \log(1 + 4 a_k^2 t^2) \;\leq\; 4t^2 \msum_{k=K_1}^{K_2} a_k^2 \;\leq\; 4t^2 \epsilon\;.
\end{align*}
This implies that $({\rm{Re}}(S_K))_{K \in \N}$ is a Cauchy sequence and therefore converges. Now we handle the imaginary part. First let $m'_K \in \Z$ be such that 
\begin{align*}
    {{\rm{Im}}\Big(\msum_{k=1}^K \log(1 - 2 i a_k t)\Big)} \;=\; \msum_{k=1}^K \arctan(- 2 a_k t) + 2 m'_K \pi\;.
\end{align*}
Then we have
\begin{align*}
    {\rm{Im}}(S_K )
    \;=&\;
    \msum_{k=1}^K \Big( - a_k t  + \mfrac{1}{2} \arctan( 2 a_k t) \Big) - m'_K \pi
    \;\eqqcolon\;
    I_K -  m'_K \pi \tagaligneq \label{eqn:W_n:existence:defn:imaginary}
    \;. 
\end{align*}
To show that $I_K$ converges, we first note that by a third-order Taylor expansion, we have that $\arctan(x) = x + \frac{6(x_*)^2 - 2}{6(x_*^2+1)^3} x^3$ for some $x_* \in [0, x]$ (we use this to denote $[0,x]$ for $x \geq 0$ as well as $[x,0]$ for $x < 0$, with an abuse of notation). This implies that for all $K_1, K_2 \geq K^*$, where $K^*$ is defined as before,
\begin{align*}
    \Big|
    \msum_{k=K_1}^{K_2} \Big( - a_k t  + &\mfrac{1}{2} \arctan( 2 a_k t) \Big) \Big|
    \;=\;
    \Big|
    \msum_{k=K_1}^{K_2} \Big( - a_k t  + \mfrac{1}{2} \arctan( 2 a_k t) \Big) 
    \Big|
    \\
    \;\leq&\;
    \msum_{k=K_1}^{K_2} \msup_{b_k \in [0, a_k]} \Big| \mfrac{1}{2} \mfrac{24 b_k^2t^2-2}{6(4b_k^2 t^2 + 1)^3} 8 a_k^3 t^3    \Big|
    \\
    \;=&\;
    4 t^3 \msum_{k=K_1}^{K_2} |a_k|^3 \,
    \Big(\msup_{b_k \in [0, a_k]} \Big| \mfrac{24b_k^2t^2+6-8}{6(4b_k^2 t^2 + 1)^3} \Big| \Big)
    \\
    \;=&\;
    4 t^3 \msum_{k=K_1}^{K_2} |a_k|^3 \,
    \Big(\msup_{b_k \in [0, a_k]} \Big| \mfrac{1}{(4b_k^2 t^2 + 1)^2} - \mfrac{4}{3(4b_k^2 t^2 + 1)^2} \Big| \Big)
    \\
    \;\leq&\;
    20 t^3 \msum_{k=K_1}^{K_2} |a_k|^3
    \;\leq\; 20 t^3 \Big( \msum_{k=K_1}^{K_2} (a_k)^2 \Big)^{3/2} \;\leq\; 20t^3 \epsilon^{3/2}\;,
\end{align*}
where, in the last line, we have used the relative sizes of $l_p$ norms. This implies that $I_K$ converges. To show that \cref{eqn:W_n:existence:defn:imaginary} converges, we need to show that $m_K$ in \cref{eqn:W_n:existence:defn:imaginary} is eventually constant. By using \cref{eqn:W_n:existence:defn:imaginary} and a triangle inequality, we have that
\begin{align*}
    \pi | m'_{K+1} -  m'_{K} | \;\leq&\;
    | I_{K+1} - I_K | 
    + 
    \Big| {\rm{Im}}(S_{K+1}) - {\rm{Im}}(S_K) \Big|
    \\
    \;=&\;
    | I_{K+1} - I_K | + \big| a_{K+1} t + \mfrac{1}{2} \log(1-2 i a_{K+1} t) \big|\;.
\end{align*}
The first term converges to zero, since we have shown that $I_K$ converges. Since $a_K \rightarrow 0$ by \cref{eqn:W_n:exist:aksquared:conv} and the complex logarithm we use is continuous outside $\{ z: {\rm{Re}}(z) > 0\}$, the second term above also converges to zero. Therefore $| m'_{K+1} -  m'_{K} | \rightarrow 0$, and since $(m'_K)_{K \in \N}$ is an integer sequence, $(m'_K)_{K \in \N}$ converges. By \cref{eqn:W_n:existence:defn:imaginary}, this implies that ${\rm{Im}}(S_K)$ converges, and since we have shown ${\rm{Re}}(S_K)$ converges, we get that $S_K$ converges. Finally, to show that $\psi_{W_n^K - D}(t)$ converges, since ${\rm{Re}}(S_K) = {\rm{Re}}\big(\psi_{W_n^K - D}(t)\big)$, we only need to show that ${\rm{Im}}\big(\psi_{W_n^K - D}(t)\big)$ converges. By \cref{eqn:W_n:exist:mK}, this again reduces to showing that $m_K$ is eventually constant. As before, by a triangle inequality,
\begin{align*}
    2 \pi | m_{K+1} - m_K | 
    \;\leq&\; 
    | {\rm{Im}}(S_{K+1}) - {\rm{Im}}(S_K) | 
    + 
    \big| {\rm{Im}}(\log \psi_{W_n^{K+1} - D}(t))  - {\rm{Im}}(\log \psi_{W_n^{K} - D}(t)) \big|
    \\
    \;=&\;
    | {\rm{Im}}(S_{K+1}) - {\rm{Im}}(S_K) | 
    +
    \big| a_{K+1} t + \mfrac{1}{2} \log(1- 2i a_{K+1} t) \big|
    \;\xrightarrow{K \rightarrow \infty} 0\;,
\end{align*}
where the convergence of both terms has been shown earlier. This proves that the characteristic function $\psi_{W_n^K - D}(t)$ converges for every $t \in \R$.
\end{proof}

\section{Proofs for \cref{sect:ksd:mmd}}  \label{appendix:proof:ksd:mmd}

\subsection{Proofs for the general results}

\begin{proof}[Proof of Lemma \ref{lem:mmd:general}] To prove the first result, note that since $\kappa$ is a kernel, there exists a RKHS $\cH$ and a map $\Phi: \R^d \rightarrow \cH$ such that we can write
\begin{align*}
    \ummd\big((\bx,\by), (\bx',\by')\big)
    \;=&\;
    \langle \Phi(\bx), \Phi(\by') \rangle_{\cH}
    +
    \langle \Phi(\by), \Phi(\by') \rangle_{\cH}
    -
    \langle \Phi(\bx), \Phi(\by') \rangle_{\cH}
    -
    \langle \Phi(\bx'), \Phi(\by) \rangle_{\cH}
    \\
    \;=&\;
    \langle \Phi(\bx) - \Phi(\by), \Phi(\bx')-\Phi(\by') \rangle_{\cH}\;.
\end{align*}
Defining $\Phi_*\big((\bx,\by)\big) \coloneqq \Phi(\bx) - \Phi(\by)$ proves that $\ummd$ is a kernel. To prove the second result, note that by the definition of a weak Mercer representation, we have that almost surely
\begin{align*}
    \big| \msum_{k=1}^K \lambda_k \phi_k(\bZ_1) \phi_k(\bZ_2) - \ummd(\bZ_1,\bZ_2) \big| \;\xrightarrow{K \rightarrow \infty}\; 0\;,
\end{align*}
which in particular implies convergence in probability. The argument uses the Vitali convergence theorem. By \cref{assumption:moment:bounded}, there exists some  $\nu^* > \nu$ such that $\sup_{K \geq 1} \mean[| \sum_{k=1}^K \lambda_k \phi_k(\bZ_1) \phi_k(\bZ_2) |^{\nu^*}] < \infty$ and $\mean[ | \ummd(\bZ_1,\bZ_2)|^{\nu^*}] < \infty$. By a triangle inequality and a Jensen's inequality, we have
\begin{align*}
    &
    \;\msup_{K \geq 1} \mean\bigg[\Big| \msum_{k=1}^K \lambda_k \phi_k(\bZ_1) \phi_k(\bZ_2) - \ummd(\bZ_1,\bZ_2) \Big|^{\nu^*}\bigg] 
    \\
    \;\leq&\;
    \msup_{K \geq 1} \mean\bigg[\Big| \, \big| \msum_{k=1}^K \lambda_k \phi_k(\bZ_1) \phi_k(\bZ_2) \big| +
    \big| \ummd(\bZ_1,\bZ_2) \big| \, \Big|^{\nu^*}\bigg] 
    \\
    \;\leq&\;
    2^{\nu^*-1}
    \msup_{K \geq 1} \mean\Big[\big| \msum_{k=1}^K \lambda_k \phi_k(\bZ_1) \phi_k(\bZ_2) 
 \big|^{\nu^*}\Big] 
    +
    2^{\nu^*-1}
    \mean\Big[\big|\ummd(\bZ_1,\bZ_2) \big|^{\nu^*}\Big] 
    \;<\; \infty\;.
\end{align*}
This implies for any $\nu \in (2,\nu^*)$, the sequence $\big( \big( \sum_{k=1}^K \lambda_k \phi_k(\bZ_1) \phi_k(\bZ_2) - \ummd(\bZ_1,\bZ_2) \big)^\nu \big)_{K \in \N}$ is uniformly integrable, and therefore converges to zero in $L_1(\R^{2d}, P \otimes Q)$ by the Vitali convergence theorem.  Since convergence in $L_\nu$ implies convergence in $L_{\min\{\nu,3\}}$, we get that \cref{assumption:L_nu} holds for $\min \{\nu, 3\}$.
\end{proof}

\vspace{1em}

Before we prove the next result, recall that $\{\lambda_k\}_{k=1}^\infty$ and $\{\phi_k\}_{k=1}^\infty$ are defined as the weak Mercer representation for the kernel $\kappa$ under $Q$, and we have assumed that $\phi_k$'s are differentiable. We have also defined the sequence of values $\{\alpha_k\}_{k=1}^\infty$ and the sequence of functions $\{\psi_k\}_{k=1}^\infty$ in \eqref{eqn:ksd:spectral} as
\begin{align*}
    \alpha_{(k'-1) d + l} \;\coloneqq&\; \lambda_{k'}\;
    &\text{ and }&&
    \psi_{(k'-1) d + l}(\bx)
    \;\coloneqq&\; (\partial_{x_l} \log p(\bx)) \phi_{k'}(\bx) + \partial_{x_l} \phi_{k'}(\bx)
    \;,
\end{align*}
for $1 \leq l \leq d$ and $k' \in \N$. For convenience, we denote $\psi_{k';l} \coloneqq \psi_{(k'-1) d + l}$ in the proof below.

\vspace{1em}

\begin{proof}[Proof of Lemma \ref{lem:ksd:general}] Recall that $\psi_{k';l}(\bx)\coloneqq (\partial_{x_l} \log p(\bx)) \phi_{k'}(\bx) + \partial_{x_l} \phi_{k'}(\bx)$. Write $\tilde \psi_{k'}(\bx) \coloneqq ( \psi_{k';1}(\bx), \ldots, \psi_{k';n}(\bx) )^\top$. We first consider the error term with $dK'$ summands for some $K' \in \N$:
\begin{align*}
    &\;
    \mean\Big[ \Big| \msum_{k=1}^{dK'} \alpha_k \psi_k(\bX_1) \psi_k(\bX_2) - \uPksd(\bX_1,\bX_2) \Big|^\nu \Big]
    \\
    \;=&\;
    \mean\Big[ \Big| \msum_{l=1}^d \msum_{k'=1}^{K'} \lambda_{k'} \psi_{k';l}(\bX_1) \psi_{k';l}(\bX_2) - \uPksd(\bX_1,\bX_2) \Big|^\nu \Big]
    \\
    \;=&\;
    \mean\Big[ \Big| \msum_{k'=1}^{K'} \lambda_{k'} 
    \big( \tilde \psi_{k'}(\bX_1) \big)^\top \big( \tilde \psi_{k'}(\bX_2) \big) - \uPksd(\bX_1,\bX_2) \Big|^\nu \Big]
    \\
    \;=&\;
    \mean\big[ \big| T_1 + T_2 + T_3 + T_4 - \uPksd(\bX_1,\bX_2) \big|^\nu \big]
    \;,
\end{align*}
where the random quantities are defined in terms of $\bX_1, \bX_2 \overset{i.i.d.}{\sim} Q$:
\begin{align*}
    T_1
    \;\coloneqq&\;
    \big( \nabla \log p(\bX_1) \big)^\top \big( \nabla \log p(\bX_2) \big)
    \msum_{k'=1}^{K'}
    \lambda_{k'}
    \phi_{k'}(\bX_1)\phi_{k'}(\bX_2) \;,
    \\
    T_2
    \;\coloneqq&\; 
    \big( \nabla \log p(\bX_1) \big)^\top 
    \Big(
    \msum_{k'=1}^{K'}
    \lambda_{k'}
    \big(\nabla \phi_{k'}(\bX_2) \big) 
    \phi_{k'}(\bX_1)
    \Big)
    \;,
    \\
    T_3
    \;\coloneqq&\; 
    \big( \nabla \log p(\bX_2) \big)^\top
    \Big(
    \msum_{k'=1}^{K'}
    \lambda_{k'}   \big(\nabla \phi_{k'}(\bX_1) \big) \phi_{k'}(\bX_2)
    \Big)\;,
    \\
    T_4
    \;\coloneqq&\; 
    \msum_{k'=1}^{K'}
    \lambda_{k'} 
    \big( \nabla \phi_{k'}(\bX_1) \big)^\top  \big( \nabla \phi_{k'}(\bX_2) \big) \;.
\end{align*}
Recall that by \cref{assumption:moment:bounded}, there exists some $\nu^* > \nu$ such that we have $\| \kappa^*(\bZ_1,\bZ_2) \|_{L_{\nu^*}} < \infty$ and $\sup_{K \geq 1} \| \sum_{k=1}^K \lambda_k \phi_k(\bZ_1) \phi_k(\bZ_2) \|_{L_{\nu^*}} < \infty$. By using the proof of the second part of Lemma \ref{lem:mmd:general} above, for $\nu^{\Delta} \coloneqq \frac{\nu+\nu^*}{2} \in (\nu, \nu^*)$, we have
\begin{align*}
    \mean\Big[ \big| \msum_{k'=1}^{K'}
    \lambda_{k'}
    \phi_{k'}(\bX_1)\phi_{k'}(\bX_2) - u(\bX_1,\bX_2) \big|^{\nu^{\Delta}} \Big] \;\xrightarrow{ K' \rightarrow \infty}\; 0\;.
\end{align*}
Meanwhile by \cref{assumption:ksd},
$\big\| \| \nabla \log p(\bX_1) \|_2 \big\|_{L_{2{\nu^{**}}}} < \infty$, where 
\begin{align*}
    \nu^{**} 
    \;=\;
    \mfrac{\nu(\nu+\nu^*)}{\nu^* - \nu}
    \;=\; 
    \Big( \mfrac{1}{\nu} - \mfrac{2}{\nu + \nu^*} \Big)^{-1} 
    \;=\;
    \Big( \mfrac{1}{\nu} - \mfrac{1}{\nu^{\Delta}} \Big)^{-1} 
    \;>\; \nu\;.
\end{align*}
By a Cauchy-Schwarz inequality and a H\"older's inequality, we have that 
\begin{align*}
    \big\|
     \big( \nabla \log p(\bX_1) \big)^\top \big( \nabla \log p(\bX_2) \big)
    \big\|_{L_{\nu^{**}}}
    \;\leq&\;
    \big\| \, \|\nabla \log p(\bX_1)\|_2 \,
    \big\|_{L_{2\nu^{**}}}
     < \infty\;.
\end{align*}
Now by a H\"older's inequality and noting that $(\nu^{**})^{-1} + (\nu^{\Delta})^{-1} = \nu^{-1}$, we can now bound the error of using $T_1$ to approximate the first term of $\uPksd$ as
\begin{align*}
    &\;\mean[ |E_1|^\nu ] \;\coloneqq\; \mean\big[ \big| T_1 -  \big( \nabla \log p(\bX_1) \big)^\top \big( \nabla \log p(\bX_2) \big) \, u(\bX_1, \bX_2) \big|^\nu \big]
    \\
    \;&=\;
    \big\| T_1 -  \big( \nabla \log p(\bX_1) \big)^\top \big( \nabla \log p(\bX_2) \big) \, u(\bX_1, \bX_2) \big\|_{L_\nu}^\nu
    \\
    \;&\leq\; 
    \big\|
     \big( \nabla \log p(\bX_1) \big)^\top \big( \nabla \log p(\bX_2) \big)
    \big\|_{L_{\nu^{**}}}^\nu
    \, 
     \big\| \msum_{k'=1}^{K'}
    \lambda_{k'}
    \phi_{k'}(\bX_1)\phi_{k'}(\bX_2) - u(\bX_1,\bX_2) \big\|_{L_{\nu^{\Delta}}}^\nu
    \\
    \;&\xrightarrow{ K' \rightarrow \infty}\; 0\;.
\end{align*}
For $T_2$, we consider a similar approximation error quantity and apply a Cauchy-Schwarz inequality:
\begin{align*}
    \mean[ |E_2|^\nu ] \;\coloneqq&\; 
    \mean\big[ \big| T_2 -  \big( \nabla \log p(\bX_1) \big)^\top \nabla_2 \kappa(\bX_1,\bX_2)\big|^\nu \big]
    \\
    \;&=\;
    \mean\Big[ \Big| \big( \nabla \log p(\bX_1) \big)^\top \Big( 
    \msum_{k'=1}^{K'}
    \lambda_{k'}
    \big(\nabla \phi_{k'}(\bX_2)\big)
    \phi_{k'}(\bX_1)
    -
    \nabla_2 \kappa(\bX_1,\bX_2) \Big)\Big|^\nu \Big]
    \\
    &\leq\; 
    \| \, \|  \nabla  \log p(\bX_1) \|_2 \, \|_{L_{2\nu}}^\nu \Big\|  \Big\| 
    \msum_{k'=1}^{K'}
    \lambda_{k'}
    \big(\nabla \phi_{k'}(\bX_2)\big)
    \phi_{k'}(\bX_1)
    -
    \nabla_2 \kappa(\bX_1,\bX_2) \Big\|_2 \Big\|_{L_{2\nu}}^\nu
    \\
    \;&\xrightarrow{K' \rightarrow \infty}\; 0\;,
\end{align*}
where we have noted that the first term is bounded since $2\nu < 2\nu^{**}$ and used \cref{assumption:ksd}(iv). By symmetry of $\kappa$ and the fact that $\bX_1$ and $\bX_2$ are exchangeable, we have the same result for $T_3$:
\begin{align*}
    \mean[ |E_3|^\nu ] \;\coloneqq&\; 
    \mean\big[ \big| T_3 -  \big( \nabla \log p(\bX_2) \big)^\top \nabla_1 \kappa(\bX_1,\bX_2)\big|^\nu \big]
     \;\xrightarrow{K' \rightarrow \infty}\; 0\;.
\end{align*}
Meanwhile, the second condition of \cref{assumption:ksd}(iv) directly says that
\begin{align*}
    \mean[ |E_4|^\nu ] \;\coloneqq&\; 
    \mean\big[ \big| T_4 - \Tr\big( \nabla_1 \nabla_2 \kappa(\bX_1,\bX_2) \big) \big|^\nu \big] 
     \;\xrightarrow{K' \rightarrow \infty}\; 0\;.
\end{align*}
Combining the results and applying a Jensen's inequality to the convex function $x \mapsto |x|^\nu$, we have
\begin{align*}
    \mean\Big[ \Big| \msum_{k=1}^{dK'} \alpha_k \psi_k(\bX_1) \psi_k(\bX_2) - &\uPksd(\bX_1,\bX_2) \Big|^\nu \Big]
    \;=\;
    \mean\big[ \big| E_1 + E_2 + E_3 + E_4 \big|^\nu \big]
    \\
    \;\leq&\;
    \mean\big[ \big| \mfrac{1}{4} (4E_1) +  \mfrac{1}{4} (4E_2) +  \mfrac{1}{4} (4E_3) +  \mfrac{1}{4} (4E_4)   \big|^\nu \big]
    \\
    \;\leq&\;
    4^{\nu - 1} \big( \mean[ |E_1|^\nu ] + \mean[ |E_2|^\nu ] + \mean[ |E_3|^\nu ] + \mean[ |E_4|^\nu ] \big)
     \;\xrightarrow{K' \rightarrow \infty}\; 0\;.
\end{align*}
Now consider $K \in \N$ that is not necessarily divisible by $d$, and let $K'$ be the greatest integer such that $K \geq dK'$. Then by a triangle inequality and a similar Jensen's inequality as above, we get
\begin{align*}
    &\;
    \mean\Big[ \Big| \msum_{k=1}^{K} \alpha_k \psi_k(\bX_1) \psi_k(\bX_2) - \uPksd(\bX_1,\bX_2) \Big|^\nu \Big]
    \\
    \;\leq&\;
    2^{\nu-1}
    \mean\Big[ \Big| \msum_{k=1}^{dK'} \alpha_k \psi_k(\bX_1) \psi_k(\bX_2) - \uPksd(\bX_1,\bX_2) \Big|^\nu \Big]
    \\
    &\;
    +
    2^{\nu-1}
    \mean\Big[ \Big| \msum_{k=dK'+1}^{K} \alpha_k \psi_k(\bX_1) \psi_k(\bX_2) \Big|^\nu \Big]\;. \tagaligneq \label{eqn:ksd:general:intermediate}
\end{align*}
The first term is $o(1)$ as $K \rightarrow \infty$ by the previous argument, so we only need to focus on the second term. The expectation can be bounded by noting that $\alpha_k = \lambda_{K'+1} \geq 0$ for all $dK'+1 \leq k \leq K$ and using a triangle inequality followed by a Jensen's inequality:
\begin{align*}
    &\;
    \mean\Big[ \Big| \msum_{k=dK'+1}^{K} \alpha_k \psi_k(\bX_1) \psi_k(\bX_2) \Big|^\nu \Big]
    \\
    \;\leq&\;
    (\lambda_{K'+1})^\nu  \mean\Big[ 
    \Big( \mfrac{1}{K-dK'} \msum_{k=dK'+1}^{K}  (K-dK') | \psi_k(\bX_1) \psi_k(\bX_2) | \Big)^{\nu} \Big]
    \\
    \;\leq&\;
    (\lambda_{K'+1})^\nu (K-dK')^{\nu-1} \msum_{k=dK'+1}^{K}  \mean[ | \psi_k(\bX_1) \psi_k(\bX_2) |^{\nu} ]
    \\
    \;\leq&\;
    (\lambda_{K'+1})^\nu d^\nu \msup_{k \in \{dK'+1, \ldots, dK'+d\}} \mean[ | \psi_k(\bX_1) \psi_k(\bX_2) |^{\nu} ]
    \\
    \;=&\;
    (\lambda_{K'+1})^\nu d^\nu \msup_{1 \leq l \leq d } \mean[ | \psi_{dK'+l}(\bX_1) |^{\nu} ]^2 \;.
\end{align*}
In the last equality, we have noted that $\bX_1$ and $\bX_2$ are identically distributed. Now by the definition of $\psi_k$, another Jensen's inequality on $x \mapsto |x|^\nu$ and a Cauchy-Schwarz inequality, we have
\begin{align*}
    \mean[ | \psi_{dK'+l}(\bX_1) &|^{\nu} ] 
    \;=\;
    \mean[ | (\partial_{x_l} \log p(\bX_1)) \phi_{K'+1}(\bX_1) + \partial_{x_l} \phi_{K'+1}(\bX_1)  |^{\nu} ] 
    \\
    \;\leq&\;
    2^{\nu -1} \mean[ | (\partial_{x_l} \log p(\bX_1)) \phi_{K'+1}(\bX_1) |^\nu ]
    +
    2^{\nu -1} \mean[ | \partial_{x_l} \phi_{K'+1}(\bX_1)  |^\nu ]
    \\
    \;\leq&\;
    2^{\nu -1} \mean[ | \partial_{x_l} \log p(\bX_1)|^{2\nu}]^{1/2}
    \, \mean[ |\phi_{K'+1}(\bX_1) |^{2\nu} ]^{1/2}
    +
    2^{\nu -1} \mean[ | \partial_{x_l} \phi_{K'+1}(\bX_1)  |^\nu ]
    \\
    \;\leq&\;
    2^{\nu -1} \mean[ \| \nabla \log p(\bX_1) \|^{2\nu}_2 ]^{1/2}
    \, \mean[ |\phi_{K'+1}(\bX_1) |^{2\nu} ]^{1/2}
    +
    2^{\nu -1} \mean[ \| \nabla \phi_{K'+1}(\bX_1) \|_2^\nu ]
    \\
    \;=&\;
    2^{\nu -1} \| \| \nabla \log p(\bX_1) \|_2 \|_{L_{2\nu}}^{\nu}
    \, \| \phi_{K'+1}(\bX_1) \|_{L_{2\nu}}^{\nu}
    +
    2^{\nu -1} \| \| \nabla \phi_{K'+1}(\bX_1) \|_2 \|_{L_\nu}^\nu\;.
\end{align*}
By \cref{assumption:ksd}(i), (ii) and (iii), all three norms are bounded, so $ \mean[ | \psi_{dK'+l}(\bX_1) |^{\nu} ] < \infty$. By the definition of $\lambda_k$ from the weak Mercer representation, as $K \rightarrow \infty$ and therefore $K' \rightarrow \infty$, $\lambda_{K'+1} \rightarrow 0$, which implies  
\begin{align*}
    \mean\Big[ \Big| \msum_{k=dK'+1}^{K} \alpha_k \psi_k(\bX_1) \psi_k(\bX_2) \Big|^\nu \Big]
\;=\; o(1)\;.
\end{align*}
This means that both terms in \eqref{eqn:ksd:general:intermediate} converge to $0$ as $K \rightarrow \infty$. In other words,
\begin{align*}
    \mean\Big[ \Big| \msum_{k=1}^{K} &\alpha_k \psi_k(\bX_1) \psi_k(\bX_2) - \uPksd(\bX_1,\bX_2) \Big|^\nu \Big]
    \;\xrightarrow{K \rightarrow \infty}\; 0\;.
\end{align*}
Since $L_\nu$-convergence implies $L_{\min\{\nu,3\}}$-convergence and we have assumed that $\nu > 2$, we get that \cref{assumption:L_nu} holds for $\min\{\nu,3\}$ with respect to the $\uPksd$, $\alpha_k$ and $\psi_k$.
\end{proof}

\subsection{Proof of Proposition \ref{prop:ksd:var:ratio}}
From Lemma \ref{lem:ksd:moments:analytical}, we can write the variance ratio as 
\begin{align*}
    \mfrac{\sigfull^2}{\sigcond^2}
    \;=&\;
    \left( \mfrac{\gamma}{4 + \gamma} \right)^{d/2}
    \left( \mfrac{(1 + \gamma) (3 + \gamma)}{ \gamma^2 } \right)^{d/2}
    \mfrac{B}{A}
    \;=\;
    C \times
    \mfrac{B}{A}
    \;,
\end{align*}
where 
\begin{align*}
    A
    \;\coloneqq&\;
    \mfrac{(2 + \gamma)^2}{(1 + \gamma)(3 + \gamma)} \| \bmu \|_2^2
    + \Big( 1 - \left( \mfrac{(1 + \gamma)(3 + \gamma)}{(2 + \gamma)^2} \right)^{d/2} \Big) \| \bmu \|_2^4 
    \\
    \;=&\;
    \left( 1 + o(1) \right) \| \bmu \|_2^2
    + \Big( 1 - \left( 1 - \alpha \right)^{d/2} \Big) \| \bmu \|_2^4 
    \\
    B
    \;\coloneqq&\;
    d 
    + \mfrac{d^2}{\gamma^2}
    + \mfrac{2d \| \bmu \|_2^2}{\gamma}
    + 2\| \bmu \|_2^2
    + \Big(1 - \left( \mfrac{\gamma (4 + \gamma)}{(2 + \gamma)^2} \right)^{d/2} \Big) \| \bmu \|_2^4
    + o\Big( 
        d 
        + \mfrac{d^2}{\gamma^2}
        + \mfrac{d \| \bmu \|_2^2}{\gamma}
        +  \| \bmu \|_2^2
    \Big)
    \\
    \;=&\;
    d 
    + \mfrac{d^2}{\gamma^2}
    + \mfrac{2d \| \bmu \|_2^2}{\gamma}
    + 2\| \bmu \|_2^2
    + \Big(1 - \left( 1 - \delta \right)^{d/2}  \Big) \| \bmu \|_2^4
    + o\Big( 
        d 
        + \mfrac{d^2}{\gamma^2}
        + \mfrac{d \| \bmu \|_2^2}{\gamma}
        +  \| \bmu \|_2^2
    \Big)
    \\
    C
    \;\coloneqq&\;
    \Big( \mfrac{\gamma}{4 + \gamma} \Big)^{d/2}
    \Big( \mfrac{(1 + \gamma) (3 + \gamma)}{ \gamma^2 } \Big)^{d/2}
    \;=\;
    \Big( \mfrac{(1 + \gamma) (3 + \gamma)}{ \gamma (4 + \gamma) } \Big)^{d/2}
    \;,
\end{align*}
and we have written $\frac{(1 + \gamma)(3 + \gamma)}{(2 + \gamma)^2} = 1 - \alpha$ with $\alpha \coloneqq \frac{1}{(2 + \gamma)^2}$ and $\frac{\gamma(4 + \gamma)}{(2 + \gamma)^2} = 1 - \delta$ with $\delta \coloneqq \frac{4}{(2 + \gamma)^2}$. To simplify $A$ and $B$, we first rewrite
\begin{align*}
    1-\left( 1 - \alpha \right)^{d/2}
    \;=\;
    1-\exp\left( - \mfrac{d}{2} \log\left( 1 - \alpha \right) \right)
    \;\overset{(a)}{=}&\;
    1 - \exp\left( - \mfrac{d}{2} \left( \mfrac{1}{(2 + \gamma)^2} + O\left( \mfrac{1}{\gamma^{4}} \right)\right) \right)
     \\
    \;=&\;
    1-\exp\left( - \mfrac{d}{2(2 + \gamma)^2} + O\left(\mfrac{d}{ \gamma^{4}} \right) \right)
    \tagaligneq \label{eq:alpha:expansion}
    \;.
\end{align*}
In $(a)$, we have used a Taylor expansion by noting that $\gamma$ is small by the stated assumption $\gamma = \omega(1)$. Similarly we can obtain
\begin{align*}
    &1-(1 - \delta)^{d/2}
    \;=\;
    1-\exp\left( \mfrac{d}{2} \log(1 - \delta) \right)
    \;=\;
    1-\exp\left( \mfrac{d}{2} \left(- \mfrac{4}{(2 + \gamma)^2} + O\left( \mfrac{1}{\gamma^{4}} \right) \right) \right)
    \\
    &\;\qquad\qquad\qquad\qquad\qquad\qquad\qquad\qquad\;\;\;\;\;\;=\;
    1-\exp\left( - \mfrac{2d}{(2 + \gamma)^2} + O\left(\mfrac{d}{ \gamma^{4}} \right) \right)
    \tagaligneq \label{eq:gamma:expansion}
    \;,
    \\
    &\;
    C
    \;=\;
    \exp\left( \mfrac{d}{2} \log\left( 1 + \mfrac{3}{\gamma (4 + \gamma)} \right) \right)
    \;=\;
    \exp\left( 
        \mfrac{d}{2} \left( \mfrac{3}{\gamma(4 + \gamma)} + O\left( \mfrac{1}{\gamma^{4}} \right) \right)
    \right)
    \\
    &\qquad\qquad\qquad\qquad\qquad\qquad\qquad\;\;=\;
    \exp\left( 
        \mfrac{3d}{2\gamma (4 + \gamma)} + O\left(\mfrac{d}{\gamma^{4}}\right)
    \right)
    \tagaligneq \label{eq:factor:expansion}
    \;.
\end{align*}
Therefore, the terms $(1 - \alpha)^{d/2}$, $(1 - \gamma)^{d/2}$ and $C$ can be small, large or close to a constant, depending on whether $\gamma^2$ grows faster than, lower than, or at the same rate as $d$. We now consider the three cases individually.

\paragraph{Case 1: $\gamma = o(d^{1/2})$.}
In this case, $\frac{d}{(2 + \gamma)^2} = \omega(1)$, so we have $(1 - \alpha)^{d/2} = o(1)$ and $(1 - \delta)^{d/2} = o(1)$. Therefore
\begin{align*}
    A 
    \;=&\; (1 + o(1)) \| \bmu \|_2^2 + (1 + o(1)) \| \bmu \|_2^4 
    \;=\; \Theta(\| \bmu \|_2^2 + \| \bmu \|_2^4)
    \;=\;
    \Theta(\| \bmu \|_2^4)
    \;,
\end{align*}
and 
\begin{align*}
    B
    \;=&\;
    d 
    + \mfrac{d^2}{\gamma^2}
    + \mfrac{2d \| \bmu \|_2^2}{\gamma}
    + 2\| \bmu \|_2^2
    + \| \bmu \|_2^4
    + o\left( 
        d 
        + \mfrac{d^2}{\gamma^2}
        + \mfrac{d \| \bmu \|_2^2}{\gamma}
        +  \| \bmu \|_2^2
    \right)
    \\
    \;=&\;
    \Theta\left( 
        d 
        + \mfrac{d^2}{\gamma^2}
        + \mfrac{d \| \bmu \|_2^2}{\gamma}
        + \| \bmu \|_2^4
    \right)
    \;.
\end{align*}
Combining with the previous expressions for $A$, $B$ and $C$ yields
\begin{align*}
    \rho_d
    \;=\;
    \mfrac{\sigfull}{\sigcond}
    \;=\;
    \sqrt{C} \times \mfrac{\sqrt{B}}{\sqrt{A}}
    \;=&\;
    \exp\left( 
        \mfrac{3d}{4\gamma (4 + \gamma)} + O\left(\mfrac{d}{\gamma^{4}} \right)
    \right)
    \Theta\bigg( 
        \sqrt{\mfrac{d}{ \| \bmu \|_2^{4}}
        + \mfrac{d^2}{\gamma^2 \| \bmu \|_2^4}
        + \mfrac{d}{\gamma \| \bmu \|_2^2}
        + 1} \;
    \bigg)
    \\
    \;\overset{(a)}{=}&\;
    \exp\left( 
        \mfrac{3d}{4\gamma^2} + o\left(\mfrac{d}{\gamma^2} \right)
    \right)
    \Theta\bigg( 
        \sqrt{\mfrac{d^2}{\gamma^2 \| \bmu \|_2^{4}}
        + \mfrac{d}{\gamma \| \bmu \|_2^2}
        + 1} \;
    \bigg)
    \\
    \;\overset{(b)}{=}&\;
    \exp\left( 
        \mfrac{3d}{4\gamma^2} + o\left(\mfrac{d}{\gamma^2} \right)
    \right)
    \Theta\bigg( 
        \mfrac{d}{\gamma \| \bmu \|_2^2}
        + \mfrac{d^{1/2}}{\gamma^{1/2} \| \bmu \|_2}
        + 1
    \bigg)
    \;,
\end{align*}
where in $(a)$ we have used the fact that $\gamma = o(d^{1/2})$, and in $(b)$ we have noted that for $a, b, c > 0$, $\sqrt{a+b+c} \leq \sqrt{a} + \sqrt{b} + \sqrt{c}$ and by a Jensen's inequality, $\sqrt{a+b+c}  \geq \frac{1}{\sqrt{3}} (\sqrt{a} + \sqrt{b} + \sqrt{c})$. 

\paragraph{Case 2: $\gamma = \omega(d^{1/2})$.}
Since in this case $\frac{d}{\gamma^2}$ is small, we can use Taylor expansion to approximate the exponential term in \eqref{eq:alpha:expansion} to get
\begin{align*}
    1 - \left( 1 - \alpha \right)^{d/2}
    \;=\;
    1 - \exp\left( - \mfrac{d}{2(2 + \gamma)^2} + O\left(\mfrac{d}{ \gamma^{4}} \right) \right)
    \;=&\;
    1 - \left( 1 - \mfrac{d}{2 (2 + \gamma)^2} + O\Big(\mfrac{d^2}{\gamma^{4}} \Big) \right)
    \\
    \;=&\;
    \mfrac{d}{2 (2 + \gamma)^2} 
     + o\Big( \mfrac{d}{\gamma^2} \Big)
     \;.
\end{align*}
Using a similar argument applied to \eqref{eq:gamma:expansion}, we have
\begin{align*}
    1 - \left( 1 - \delta \right)^{d/2}
    \;=\;
    1 - \exp\left( - \mfrac{2d}{(2 + \gamma)^2} + O\left(\mfrac{d}{ \gamma^{4}} \right) \right)
    \;=&\;
    \mfrac{2d}{(2 + \gamma)^2} + o\Big( \mfrac{d}{\gamma^2} \Big)
    \;,
\end{align*}
and \eqref{eq:factor:expansion} yields
\begin{align*}
    C
    \;=&\;
    \exp\left( 
        \mfrac{3d}{2\gamma (4 + \gamma)} + O\left( \mfrac{d}{\gamma^{4}} \right)
    \right)
    \;=\;
    1 + \mfrac{3d}{2\gamma (4 + \gamma)} + O\Big( \mfrac{d^2}{\gamma^{4}} \Big)
    \;=\;
    1 + o(1)
    \;.
\end{align*}
We therefore conclude that
\begin{align*}
    A
    \;=&\;
    (1 + o(1)) \| \bmu \|_2^2 + \Big( \mfrac{d}{2(2 + \gamma)^2} + o \Big( \mfrac{d}{\gamma^2} \Big) \, \Big) \| \bmu \|_2^4 
    \;=\;
    \Theta\left( 
        \| \bmu \|_2^2 +  \mfrac{d}{\gamma^2} \| \bmu \|_2^4 
    \right)
    \;,
\end{align*}
A similar argument shows that
\begin{align*}
    B
    \;=&\;
    d 
    + \mfrac{d^2}{\gamma^2}
    + \mfrac{2d \| \bmu \|_2^2}{\gamma}
    + 2\| \bmu \|_2^2
    + \| \bmu \|_2^4 \left( \mfrac{2d}{ (2 + \gamma)^2 } + o\left( \mfrac{d}{\gamma^2} \right) \right)
    + o\left( 
        d 
        + \mfrac{d^2}{\gamma^2}
        + \mfrac{d \| \bmu \|_2^2}{\gamma}
        +  \| \bmu \|_2^2
    \right)
    \\
    \;=&\;
    \Theta\left(
    d 
    + \mfrac{d \| \bmu \|_2^2}{\gamma}
    + \| \bmu \|_2^2
    + \mfrac{d \| \bmu \|_2^4}{ \gamma^2 }
    \right)
    \;,
\end{align*}
where in the last line we noted that $\gamma = \omega(d^{1/2})$ implies $ \frac{d^2}{\gamma^2} = o(d)$. Combining the results gives
\begin{align*}
    \rho_d
    \;=\;
    \mfrac{\sigfull}{\sigcond}
    \;=&\;
    \sqrt{C} \times \mfrac{\sqrt{B}}{\sqrt{A}}
    \\
    \;=&\;
    \sqrt{1 + o(1)} \;\;
    \Theta\left( 
    \mfrac{d^{1/2} + \gamma^{-1/2}d^{1/2} \|\mu\|_2 + \|\mu\|_2 + \gamma^{-1} d^{1/2} \|\mu\|_2^2 }
    { \|\mu\|_2 + \gamma^{-1} d^{1/2} \|\mu\|_2^2}
    \right)
    \\
    \;=&\;
    \Theta\Big( \mfrac{d^{1/2} + \gamma^{-1/2}d^{1/2} \|\mu\|_2}{\|\mu\|_2 + \gamma^{-1} d^{1/2} \|\mu\|_2^2} + 1 \Big)
    \;=\;
    \Theta\Big( \mfrac{d^{1/2} (1 + \gamma^{-1/2} \|\mu\|_2)}{\|\mu\|_2 \, ( 1 + \gamma^{-1} d^{1/2} \|\mu\|_2)} + 1 \Big)\;.
\end{align*}

\paragraph{Case 3: $\gamma = \Theta(d^{1/2})$.}
Since in this case $\frac{d}{\gamma^{4}}$ is small, we have that $\exp\big( O\big(\frac{d}{\gamma^4} \big) \big) = 1 + O\big( \frac{d}{\gamma^4} \big)$ by a Taylor expansion. Substituting this into \eqref{eq:alpha:expansion}, we have
\vspace{-.2cm}
\begin{align*}
    0\;\leq\;
    1 - \left( 1 - \alpha \right)^{d/2}
    \;=&\;
    1 - \exp\left( - \mfrac{d}{2(2 + \gamma)^2} \right) \left(1 + O\left(\mfrac{d}{\gamma^{4}} \right) \right)
    \\
    \;=&\;
    1 - \exp\left( - \mfrac{d}{2(2 + \gamma)^2} \right)  
    + O\left(\mfrac{d}{\gamma^{4}} \right) 
    \;=\;
    \Theta(1)
    \;,
\end{align*}
where the last line holds as $1 - \exp\big( - \frac{d}{2(2 + \gamma)^2} \big) = \Theta(1)$. A similar argument applied to \eqref{eq:gamma:expansion} and \eqref{eq:factor:expansion} gives
\begin{align*}
    0\;\leq\;
    1 - \left( 1 - \delta \right)^{d/2}
    \;=&\;
    1 - \exp\left( - \mfrac{2d}{(2 + \gamma)^2} + O\left(\mfrac{d}{ \gamma^{4}} \right) \right)
    \;=\;
    \Theta(1)
    \;,
    \\
    0\;\leq\; C   
    \;=&\;
    \exp\left( 
        \mfrac{3d}{2\gamma (4 + \gamma)} + O\left(\mfrac{d}{\gamma^{4}} \right)
    \right)
    \;=\;
    \Theta(1)
    \;.
\end{align*}
Combining the above derivations yields
\begin{align*}
    A
    \;=&\;
    \Theta( \| \bmu \|_2^2 + \| \bmu \|_2^4)
    \;=\;
    \Theta( \| \bmu \|_2^4)
    \;,
    \\
    B
    \;=&\;
    \Theta\Big( d + \mfrac{d^2}{\gamma^2} + \mfrac{2d \| \bmu \|_2^2}{\gamma} + 2\| \bmu \|_2^2 + \| \bmu \|_2^4 \Big)
    \;=\;
    \Theta\Big( d + d^{1/2} \| \bmu \|_2^2 + \| \bmu \|_2^4 \Big)
    \;,
\end{align*}
where in the equality for $B$ we have used the fact that $\| \bmu \|_2^2 = \Omega(1)$ implies $\| \bmu \|_2^2 = O(\| \bmu \|_2^4)$ and that $\gamma = \Theta(d^{1/2})$ implies $\frac{d}{\gamma} = \Theta(d^{1/2})$. Therefore,
\begin{align*}
    \rho_d
    \;=\;
    \mfrac{\sigfull}{\sigcond}
    \;=\;
    \sqrt{C} \times \mfrac{\sqrt{B}}{\sqrt{A}}
    \;=&\;
    \Theta\left( 
        \sqrt{\mfrac{d}{\| \bmu \|_2^4}
        + \mfrac{d^{1/2}}{ \| \bmu \|_2^2}
        + 1}\;
    \right) 
    \;=\;
    \Theta\left( 
        \mfrac{d^{1/2}}{\| \bmu \|_2^2}
        + \mfrac{d^{1/4}}{ \| \bmu \|_2}
        + 1
    \right) 
    \;.
\end{align*}
This completes the proof.

\subsection{Proof of Proposition \ref{prop:MMD:variance:ratio}}
Recall the expressions of $\sigcond^2$ and $\sigfull^2$ for MMD-RBF from Lemma \ref{lem:mmd:moments:analytical}, which allow us to rewrite $\sigcond^2 = C A$ and $\sigfull^2 = C B$, where 
\begin{align*}
    A
    \;\coloneqq&\;
    1 + \exp\left( -\mfrac{1}{3 + \gamma} \| \mu \|_2^2 \right) 
    + 2 \left( \mfrac{3 + \gamma}{2 + \gamma} \right)^{d/2} \left( \mfrac{1 + \gamma}{2 + \gamma} \right)^{d/2} \exp\left( - \mfrac{1}{2(2 + \gamma)} \| \mu \|_2^2 \right) \\
    &\; 
    - 2 \exp\left( - \mfrac{2 + \gamma}{2(1 + \gamma)(3 + \gamma)} \| \mu \|_2^2 \right) 
    - \left( \mfrac{3 + \gamma}{2 + \gamma} \right)^{d/2} \left( \mfrac{1 + \gamma}{2 + \gamma} \right)^{d/2} \\
    &\; 
    - \left( \mfrac{3 + \gamma}{2 + \gamma} \right)^{d/2} \left( \mfrac{1 + \gamma}{2 + \gamma} \right)^{d/2} \exp\left( - \mfrac{1}{2 + \gamma} \| \mu \|_2^2 \right)
    \\
    B
    \;\coloneqq&\;
    \Big( \mfrac{3+\gamma}{4+\gamma} \Big)^{d/2}\Big( \mfrac{1+\gamma}{\gamma} \Big)^{d/2}
    \Big( 1 + \exp\Big( - \mfrac{1}{4+\gamma} \|\mu\|_2^2 \Big) \Big)
    -
    \Big( \mfrac{3+\gamma}{2+\gamma} \Big)^{d/2} \Big( \mfrac{1+\gamma}{2+\gamma} \Big)^{d/2}
    \\
    &\;
    -
    4
    \exp \Big( - \mfrac{2+\gamma}{2(1+\gamma)(3+\gamma)}  \| \mu \|_2^2 \Big)
    -
    \Big( \mfrac{3+\gamma}{2+\gamma} \Big)^{d/2} \Big( \mfrac{1+\gamma}{2+\gamma} \Big)^{d/2}
    \exp\Big( - \mfrac{1}{2+\gamma} \|\mu\|_2^2 \Big) 
    \\
    &\;
    +
    4 \Big( \mfrac{3+\gamma}{2+\gamma} \Big)^{d/2} \Big( \mfrac{1+\gamma}{2+\gamma} \Big)^{d/2}
    \exp\Big( - \mfrac{1}{2(2+\gamma)} \|\mu\|_2^2 \Big) 
    \\
    C 
    \;\coloneqq&\; 
    2\left( \mfrac{\gamma}{1 + \gamma} \right)^{d/2} \left( \mfrac{\gamma}{3 + \gamma} \right)^{d/2}
    \;.
\end{align*}
This implies that $\sigfull^2 / \sigcond^2 = B / A$, so it suffices to calculate the leading terms in $A$ and $B$, respectively. We first write $\frac{(3+\gamma)(1+\gamma)}{(2+\gamma)^2} = 1 - \frac{1}{(2+\gamma)^2} \eqqcolon 1 - \alpha$ and $\frac{(3 + \gamma)(1 + \gamma)}{(4 + \gamma)\gamma} = 1 + \frac{3}{\gamma(4 + \gamma)} =: 1 + \beta$, where $\alpha$ and $\beta$ are small as $\gamma = \omega(1)$ by assumption. Rearranging $A$ gives
\begin{align*}
    A
    \;=&\;
    1 + \exp\left( - \mfrac{1}{3 + \gamma} \| \bmu \|_2^2 \right)
    - 2 \exp\left( - \mfrac{2 + \gamma}{2(1 + \gamma)(3 + \gamma)} \| \mu \|_2^2 \right)
    \\
    &\;
    + (1 - \alpha)^{d/2} \Big(
        2\exp\Big( - \mfrac{1}{2(2 + \gamma)} \| \bmu \|_2^2 \Big)
        - 1
        - \exp\Big( - \mfrac{1}{2 + \gamma} \| \bmu \|_2^2 \Big)
    \Big) 
    \\
    \;\eqqcolon&\; A_1 + (1 - \alpha)^{d/2}  A_2\;,
\end{align*}
and similarly,
\begin{align*}
    B
    \;=&\;
    ( 1+\beta )^{d/2}
    \Big( 1 + \exp\Big( - \mfrac{1}{4+\gamma} \|\mu\|_2^2 \Big) \Big)
    - 4\exp \Big( - \mfrac{2+\gamma}{2(1+\gamma)(3+\gamma)}  \| \mu \|_2^2 \Big)
    \\
    &\;
    - (1-\alpha)^{d/2}
    \Big( 
        1 
        + \exp\Big( - \mfrac{1}{2+\gamma} \|\mu\|_2^2 \Big) 
        - 4 \exp\Big( - \mfrac{1}{2(2+\gamma)} \|\mu\|_2^2 \Big) 
    \Big)
    \\
    \;\eqqcolon&\;
    ( 1+\beta )^{d/2} B_1 - B_2 - (1-\alpha)^{d/2} B_3
    \;.
\end{align*}
These expressions can be simplified further depending on the relative growth rates of $d, \gamma$ and $\| \bmu \|_2^2$; we consider these cases individually.

\paragraph{Case 1: $\gamma = o(d^{1/2})$ and $\gamma = o(\| \bmu \|_2^2)$.}
Since $\gamma = o(\| \bmu \|_2^2)$, all exponential terms of the form $\exp\left( - \frac{1}{a(b + \gamma)}  \| \bmu \|_2^2 \right)$, for any positive constants $a, b$, are $o(1)$. Moreover, since we have assumed that $\gamma = \omega(1)$, we can apply a Taylor expansion to yield 
\begin{align}
    (1 - \alpha)^{d/2}
    \;=&\;
    \exp\left( \mfrac{d}{2}\log\left( 1 - \mfrac{1}{(2 + \gamma)^2} \right) \right)
    \nonumber
    \\
    \;=&\;
    \exp\left(
        \mfrac{d}{2} \left( - \mfrac{1}{(2 + \gamma)^2} + O\left( \mfrac{1}{\gamma^{4}} \right) \right)
    \right)
    \;=\;
    \exp\left(
        -\mfrac{d}{2 \gamma^2} + o\left(\mfrac{d}{ \gamma^{2}} \right)
    \right)
    \label{eq:mmd:alpha:expansion}
    \;.
\end{align}
Therefore, when $\gamma = o(d^{1/2})$, we have $(1 - \alpha)^{d/2} = o(1)$. Thus the dominating term in $A$ is the leading constant $1$ and
\begin{align*}
    A
    \;=\;
    1 + o(1)
    \;.
\end{align*}
To control $B$, we first consider a similar Taylor expansion by noting that $\gamma = \omega(1)$:
\begin{align*}
    (1 + \beta)^{d/2}
    \;=\;
    \exp\Big( 
        \mfrac{d}{2} \log \Big(1 + &\mfrac{3}{\gamma (4 + \gamma)} \Big)
    \Big)
    \;=\;
    \exp\left( 
        \mfrac{d}{2}\left(
            \mfrac{3}{\gamma( 4 + \gamma)} + O\left( \mfrac{1}{\gamma^{4}} \right)
        \right)
    \right)
    \\
    \;=&\;
    \exp\left( 
        \mfrac{3d}{2\gamma( 4 + \gamma)} + O\left(\mfrac{d}{ \gamma^{4}} \right)
    \right)
    \;=\;
    \exp\Big( 
        \mfrac{3d}{2\gamma^2} + o\left(\mfrac{d}{ \gamma^{2}} \right)
    \Big)
    \tagaligneq
    \label{eq:mmd:beta:expansion}
    \;.
\end{align*}
Since $\gamma = o(d^{1/2})$, we have that $(1 + \beta)^{d/2} = \omega(1)$. All exponential terms and $(1-\alpha)^{d/2}$ are $o(1)$ by the calculations above, so
\begin{align*}
    B 
    \;=\; (1 + \beta)^{d/2} + o((1 + \beta)^{d/2})
    \;=\;
    \Theta\Big( \exp\left(
        \mfrac{3d}{2\gamma^2} + o\left(\mfrac{d}{ \gamma^{2}}\right)
    \right) \Big)
    \;. 
\end{align*}
Combining the results for $A$ and $B$  gives
\begin{align*}
    \rho_d
    \;=\;
    \mfrac{\sigfull}{\sigcond} 
    \;=\; 
    \mfrac{\sqrt{B}}{\sqrt{A}}
    \;=\;
    \Theta\Big( \exp\left(
        \mfrac{3d}{4\gamma^2} + o\left(\mfrac{d}{ \gamma^{2}}\right)
    \right) \Big)
    \;.
\end{align*}

\paragraph{Case 2: $\gamma = o(d^{1/2})$ and $\gamma = \omega(\| \bmu \|_2^2)$.}
Since $\gamma = \omega(\| \bmu \|_2^2)$, we can bound $A_1$ by first extracting an exponential factor and then applying two second-order Taylor expansions:
\begin{align*}
    A_1\;=\;&\;
    1 
    + \exp\left( - \mfrac{1}{3 + \gamma} \| \bmu \|_2^2 \right)
    - 2 \exp\left( - \mfrac{2 + \gamma}{2(1 + \gamma)(3 + \gamma)} \| \mu \|_2^2 \right)
    \\
    \;=&\;
    1
    + \exp\left( - \mfrac{1}{3 + \gamma} \| \bmu \|_2^2 \right)
    \Big( 
        1 - 2\exp\left( \mfrac{\gamma}{2(1 + \gamma)(3 + \gamma)} \| \bmu \|_2^2 \right)
    \Big)
    \\
    \;=&\;
    1 
    +\Big( 
        1 
        - 
        \mfrac{ \| \bmu \|_2^2 }{3 + \gamma}
        + 
        \mfrac{\| \bmu \|_2^4}{2(3 + \gamma)^2}
        +
        O\Big( \mfrac{\| \bmu \|_2^6}{\gamma^3} \Big)
    \Big)
    \\
    &\;\qquad\qquad \times 
    \Big(
        - 1
        -
        \mfrac{\gamma  \| \bmu \|_2^2}{(1 + \gamma)(3 + \gamma)}
        -
        \mfrac{\gamma^2  \| \bmu \|_2^4}{4(1 + \gamma)^2(3 + \gamma)^2}
        + 
        O\Big( \mfrac{\| \bmu \|_2^6}{ \gamma^3} \Big)
    \Big)
    \\
    \;=&\;
    1 - 1 
    +
    \Big(\mfrac{1}{3 + \gamma} - \mfrac{\gamma}{(1 + \gamma)(3 + \gamma)} \Big) \| \bmu \|_2^2
    \\&\;
    +
    \Big( - \mfrac{1}{2(3+\gamma)^2} - \mfrac{\gamma^2}{4(1+\gamma)^2(3+\gamma)^2} + \mfrac{\gamma}{(1+\gamma)(3+\gamma)^2} \Big) \| \mu \|_2^4
    + 
    O\Big(\mfrac{\| \bmu \|_2^6}{\gamma^3}\Big)  
    \\
    \;=&\;
    \mfrac{1}{(1 + \gamma) (3 + \gamma)} \| \bmu \|_2^2
    +
    \mfrac{-2+\gamma^2}{4(3+4\gamma+\gamma^2)^2} \|\bmu\|_2^4
    + O\Big(\mfrac{\| \bmu \|_2^6}{\gamma^3}\Big)  
    \;.
\end{align*}
Note that the first term is on the order $\gamma^{-2} \|\mu\|_2^2$, the second term is on the order $\gamma^{-2} \|\mu\|_2^4$ and the third term is on the order $\gamma^{-3} \|\mu\|_2^6$. Since $\gamma^{-1}\|\mu\|_2^2 = o(1)$ and $\|\mu\|_2^2 = \Omega(1)$, the second term dominates and we get that
\begin{align*}
    A_1
    \;=&\;
    \mfrac{\|\bmu\|_2^4}{4\gamma^2} 
    + o\Big( \mfrac{\|\bmu\|_2^4 }{\gamma^2} \Big)\;.  \tagaligneq \label{eqn:mmd:A1:taylor}
\end{align*}
To control $A_2$, we use a similar Taylor expansion to get that
\begin{align*}
    A_2 
    \;=&\; 
    2\exp\left( - \mfrac{1}{2(2 + \gamma)} \| \bmu \|_2^2 \right)
    - 1
    - \exp\left( - \mfrac{1}{2 + \gamma} \| \bmu \|_2^2 \right)
    \\
    \;=&\;
    - 1 + \exp\Big( - \mfrac{1}{2 + \gamma} \| \bmu \|_2^2  \Big)
    \;
    \Big(
    2\exp\Big( \mfrac{1}{2(2 + \gamma)} \| \bmu \|_2^2 \Big) - 1
    \Big)
    \\
    \;=&\;
    -1 
    + 
    \Big( 1 -  \mfrac{\| \bmu \|_2^2 }{2 + \gamma} + \mfrac{\| \bmu \|_2^4 }{2(2 + \gamma)^2} + O\Big( \mfrac{\| \bmu \|_2^6 }{\gamma^3} \Big)\,  \Big)
    \;
    \Big( 1 +  \mfrac{\| \bmu \|_2^2 }{2 + \gamma} +  \mfrac{\| \bmu \|_2^4 }{4(2 + \gamma)^2} + O\Big( \mfrac{\| \bmu \|_2^6 }{\gamma^3} \Big)\,  \Big)
    \\
    \;=&\;
    \Big( \mfrac{1}{4(2+\gamma)^2} + \mfrac{1}{2(2+\gamma)^2} - \mfrac{1}{(2+\gamma)^2} \Big) \|\mu\|_2^4 
    + 
    O\Big( \mfrac{\| \bmu \|_2^6 }{\gamma^3} \Big)
    \\
    \;=&\; - \mfrac{\|\bmu\|_2^4}{4(2+\gamma)^2}  + 
    O\Big( \mfrac{\| \bmu \|_2^6 }{\gamma^3} \Big)\;.  \tagaligneq \label{eqn:mmd:A2:taylor}
\end{align*}
In particular, we have $A_2 = O( \gamma^{-2} \|\bmu\|_2^4 ) = O(A_1)$. Since $\gamma=o(d^{1/2})$, we have $(1 - \alpha)^{d/2} =  o(1) $ as before, which implies
\begin{align*}
    A 
    \;=\;
    A_1
    +
    (1-\alpha)^{d/2} A_2
    \;=\; 
    \mfrac{\|\bmu\|_2^4}{4\gamma^2} 
    + o\Big( \mfrac{\|\bmu\|_2^4 }{\gamma^2} \Big)
    \;.
\end{align*}
To control $B$, recall we have shown in Case 1 that $(1+\beta)^{d/2} = \omega(1)$ and $(1-\alpha)^{d/2}=o(1)$ for $\gamma=o(d^{1/2})$. All exponential terms are $O(1)$ and $B_1 = 2 + O(\gamma^{-1} \|\mu\|_2^2)$ by a Taylor expansion. By \eqref{eq:mmd:beta:expansion}, we obtain that
\begin{align*}
    B \;=\; (1+\beta)^{d/2} B_1 - B_2 - (1-\alpha)^{d/2} B_3 \;=&\; 2(1+\beta)^{d/2} + o((1+\beta)^{d/2} )
    \\
    \;=&\; \Theta\Big( \exp\left(
        \mfrac{3d}{2\gamma^2} + o\left(\mfrac{d}{ \gamma^{2}}\right)
    \right) \Big)\;.
\end{align*}
We hence conclude that
\begin{align*}
    \rho_d
    \;=\;
    \mfrac{\sigfull}{\sigcond}
    \;=\;
    \mfrac{\sqrt{B}}{\sqrt{A}}
    \;=&\;
    \Theta\Big( \mfrac{\gamma}{\|\mu\|_2^2} \exp\left(
        \mfrac{3d}{4\gamma^2} + o\left(\mfrac{d}{ \gamma^{2}}\right)
    \right) \Big)
    \;.
\end{align*}

\vspace{.2em}

\paragraph{Case 3: $\gamma = \omega(\| \bmu \|_2^2)$ and $\gamma = \omega(d^{1/2})$.} We first rewrite the expressions of $A$ and $B$ as
\begin{align*}
    A \;=&\; (A_1+A_2) - (1-(1-\alpha)^{d/2}) A_2\;, \tagaligneq \label{eqn:ksd:A:alt}
    \\
    B \;=&\; (B_1-B_2-B_3) + ((1+\beta)^{d/2}-1) B_1 + (1-(1-\alpha)^{d/2}) B_3\;. \tagaligneq \label{eqn:ksd:B:alt}
\end{align*}
Since $\gamma = \omega(d^{1/2})$, we can perform a further Taylor expansion on the expressions in \eqref{eq:mmd:alpha:expansion} and \eqref{eq:mmd:beta:expansion}:
\begin{align}
    (1 - \alpha)^{d/2}
    \;=&\;
    \exp\left(
        -\mfrac{d}{2(2+\gamma)^2} + O\Big(\mfrac{d}{ \gamma^4} \Big)
    \right)
    \;=\;
    1 - \mfrac{d}{2 (2 + \gamma)^2} + O\left(\mfrac{d}{ \gamma^{4} } \right)
    \;,
    \label{eq:mmd:pf:alpha:taylor}
    \\
    (1 + \beta)^{d/2}
    \;=&\;
    \exp\left( 
         \mfrac{3d}{2\gamma(4+\gamma)}  
         + O\Big(\mfrac{d}{ \gamma^{4}} \Big)
    \right)
    \;=\;
    1 + \mfrac{3d}{2\gamma(4 + \gamma)} 
    +
    O\left(\mfrac{d}{ \gamma^{4} } \right) \;,
    \label{eq:mmd:pf:beta:taylor}
\end{align}
On the other hand, since $\gamma^{-1}\|\mu\|^2_2$ is small, we can consider performing Taylor expansions on each exponential. By grouping the terms and extracting an appropriate exponential, we get that
\begin{align*}
    A_1 + A_2 
    \;=&\; 
    \exp\left( - \mfrac{1}{3 + \gamma} \| \bmu \|_2^2 \right)
    - \exp\Big( - \mfrac{1}{2 + \gamma} \| \bmu \|_2^2 \Big)
    \\
    &\;
    - 2 \exp\left( - \mfrac{2 + \gamma}{2(1 + \gamma)(3 + \gamma)} \| \mu \|_2^2 \right)
    +
    2\exp\Big( - \mfrac{1}{2(2 + \gamma)} \| \bmu \|_2^2 \Big)
    \\
    \;=&\;
    \exp\left( - \mfrac{1}{2 + \gamma} \| \bmu \|_2^2 \right)
    \Big( \exp\Big(\mfrac{1}{(3+\gamma)(2 + \gamma)} \| \bmu \|_2^2 \Big) - 1\Big)
    \\
    &\;
    -
    2\exp\Big( - \mfrac{1}{2(2 + \gamma)} \| \bmu \|_2^2 \Big)
    \Big(
    - \exp\left( - \mfrac{1}{2(6+11\gamma+6\gamma^2+\gamma^3)} \| \mu \|_2^2 \right)
    +
    1
    \Big)
    \\
    \;=&\;
    \mfrac{1}{(3+\gamma)(2 + \gamma)} \| \bmu \|_2^2 + o\Big( \mfrac{1}{(3+\gamma)(2 + \gamma)} \| \bmu \|_2^2 \Big) + O\Big( \mfrac{\| \bmu \|_2^2}{\gamma^{3}} \Big)
    \\
    \;=&\;
    \mfrac{\|\mu\|_2^2 }{(3+\gamma)(2 + \gamma)}
    +
    o\Big( \mfrac{\| \bmu \|_2^2}{ \gamma^{2}} \Big)
    \;. \tagaligneq \label{eqn:mmd:A1A2:taylor}
\end{align*}
In the last line, we have used that the dominating term is of the order $\| \bmu \|_2^2 / \gamma^2$. For $A_2$, we recall from \eqref{eqn:mmd:A2:taylor} that $A_2 = - \frac{\|\mu\|_2^4}{4\gamma^2} + o\big( \frac{\|\mu\|_2^4}{\gamma^2} \big)$. Substituting the computations into \eqref{eqn:ksd:A:alt} and using \eqref{eq:mmd:pf:alpha:taylor}, we obtain that
\begin{align*}
    A
    \;=&\; 
    (A_1+A_2)
    -
    (1-(1-\alpha)^{d/2}) A_2
    \\
    \;=&\;
    \mfrac{\|\mu\|_2^2 }{(3+\gamma)(2 + \gamma)}
    +
    o\Big( \mfrac{\| \bmu \|_2^2}{ \gamma^{2}} \Big)
    +
    \Big( \mfrac{d}
    {2(2+\gamma)^2} + O\Big( \mfrac{d}{\gamma^4} \Big) \Big)
    \Big(
    \mfrac{\|\mu\|_2^4}{4(2+\gamma)^2} + o\Big( \mfrac{\|\mu\|_2^4}{\gamma^2} \Big)
    \Big)
    \\
    \;=&\;
    \Theta\Big(  \mfrac{\|\mu\|_2^2}{\gamma^2}   
    + 
    \mfrac{d\|\mu\|_2^4 }{\gamma^4} \Big)
    \;.
\end{align*}
We use a similar argument to compute $B$. By grouping terms appropriately and performing Taylor expansions, we have
\begin{align*}
    B_1 - B_2 - B_3 
    \;=&\; \exp\Big( - \mfrac{1}{4+\gamma} \|\mu\|_2^2 \Big) - \exp\Big( - \mfrac{1}{2+\gamma} \|\mu\|_2^2 \Big)
    \\
    &\; - 4\exp \Big( - \mfrac{2+\gamma}{2(1+\gamma)(3+\gamma)}  \| \mu \|_2^2 \Big) +  4 \exp\Big( - \mfrac{1}{2(2+\gamma)} \|\mu\|_2^2 \Big) 
    \\
    \;=&\; \exp\Big( - \mfrac{1}{4+\gamma} \|\mu\|_2^2 \Big) \Big( 1 - \exp\Big( - \mfrac{2}{(4+\gamma)(2+\gamma)} \|\mu\|_2^2 \Big) \Big)
    \\
    &\;
    -
    4\exp\Big( - \mfrac{1}{2(2 + \gamma)} \| \bmu \|_2^2 \Big)
    \Big(
    - \exp\left( - \mfrac{1}{2(6+11\gamma+6\gamma^2+\gamma^3)} \| \mu \|_2^2 \right)
    +
    1
    \Big)\;,  
    \\
    \;=&\;
    \mfrac{2 \|\mu\|_2^2}{(4+\gamma)(2+\gamma)} 
    +
    o\Big( \mfrac{2 \|\mu\|_2^2}{(4+\gamma)(2+\gamma)}  \Big)
    + 
    O\left( \mfrac{\|\bmu\|_2^2}{\gamma^3} \right)
    \\
    \;=&\;
    \mfrac{2 \|\mu\|_2^2}{(4+\gamma)(2+\gamma)} 
    +
    o\Big( \mfrac{\|\bmu\|_2^2}{\gamma^2} \Big) \tagaligneq \label{eqn:mmd:B1B2B3:taylor}
    \;.
\end{align*}
By performing Taylor expansions again, we can control $B_1$ and $B_3$ as
\begin{align*}
    B_1 \;=&\;  1 + \exp\Big( - \mfrac{1}{4+\gamma} \|\mu\|_2^2\Big) \;=\; 2 + o(1)\;, \tagaligneq \label{eqn:mmd:B1:taylor}
    \\
    B_3 \;=&\; 
    1 + \exp\Big( - \mfrac{1}{2+\gamma} \|\mu\|_2^2 \Big)
    - 4\exp \Big( - \mfrac{1}{2(2+\gamma)}  \| \mu \|_2^2 \Big)
    \;=\;
    -2
    + o(1)\;. \tagaligneq \label{eqn:mmd:B3:taylor}
\end{align*}
Substituting the bounds into \eqref{eqn:ksd:B:alt} and using the bounds in \eqref{eq:mmd:pf:alpha:taylor} and \eqref{eq:mmd:pf:beta:taylor}, we obtain that
\begin{align*}
    B\;=&\;
    (B_1-B_2-B_3) + ((1+\beta)^{d/2}-1) B_1 + (1-(1-\alpha)^{d/2}) B_3
    \\
    \;=&\;
    \mfrac{2 \|\mu\|_2^2}{(4+\gamma)(2+\gamma)} +
    o\left( \mfrac{\|\bmu\|_2^2}{\gamma^2} \right)
    +
    \Big( \mfrac{3d}{2\gamma(4 + \gamma)} + O\left(\mfrac{d}{ \gamma^{4} } \right) \Big) \; (2+o(1))
    \\
    &\;+
    \Big( \mfrac{d}{2 (2 + \gamma)^2} + O\left(\mfrac{d}{ \gamma^{4} } \right) \Big)
    (-2+o(1))
    \\
    \;=&\;
    \mfrac{2 \|\mu\|_2^2}{(4+\gamma)(2+\gamma)} 
    +
    o\left( \mfrac{\|\bmu\|_2^2}{\gamma^2} \right)
    +
    \mfrac{2 d(6+4\gamma+\gamma^2)}{\gamma(2+\gamma)^2(4+\gamma)}
    +
    o\left( \mfrac{d}{\gamma^2} \right)
    \;=\;
    \Theta\Big( \mfrac{\|\mu\|_2^2}{\gamma^2} + \mfrac{d}{\gamma^2} \Big)
    \;.
\end{align*}
The variance ratio can therefore be bounded as
\begin{align*}
    \rho_d
    \;=\;
    \mfrac{\sigfull}{\sigcond}
    \;=\;
    \mfrac{\sqrt{B}}{\sqrt{A}}
    \;=\;
    \Theta\Big( \, \Big(\mfrac{ \gamma^{-2} \|\mu\|_2^2 + \gamma^{-2} d }{ \gamma^{-2} \|\mu\|_2^2 +  \gamma^{-4}  d\|\mu\|_2^4 } \Big)^{1/2} \, \Big)
    \;=&\;
    \Theta\Big( \,  \mfrac{ (\|\mu\|_2^2 + d)^{1/2} }{  ( \|\mu\|_2^2+  \gamma^{-2}  d\|\mu\|_2^4)^{1/2} } \, \Big)
    \\
    \;\overset{(a)}{=}&\;
    \Theta\Big( \,  \mfrac{ \|\mu\|_2 + d^{1/2} }{ \|\mu\|_2 +  \gamma^{-1}  d^{1/2} \|\mu\|_2^2 } \, \Big)
    \;.
\end{align*} 
In $(a)$, we have noted that for $a, b > 0$, $\sqrt{a+b} \leq \sqrt{a} + \sqrt{b}$ and, by the concavity of the square-root function, $\sqrt{a+b} \;=\; \sqrt{\frac{1}{2} (2a) + \frac{1}{2} (2b) } \geq \frac{1}{\sqrt{2}} (\sqrt{a} + \sqrt{b})$. 

\vspace{.2em}

\paragraph{Case 4:  $\gamma = \omega(\| \bmu \|_2^2)$ and $\gamma = \Theta(d^{1/2})$.} We can directly make use of the computations from Case 2 and 3 except that we control $(1-\alpha)^{d/2}$ and $(1+\beta)^{d/2}$ differently. Since 
\begin{align*}
    0 \;\leq\; (1-\alpha)^{d/2} \;=\; (1-\frac{1}{(2+\gamma)^2})^{d/2} \;\leq\; 1 \;,
\end{align*}
we see that $A=(A_1+A_2) - (1-(1-\alpha)^{d/2}) A_2$ takes value between $A_1+A_2$ and $A_1$, whose Taylor expansions under $\gamma=\omega(\|\mu\|_2^2)$ have been obtained in \eqref{eqn:mmd:A1A2:taylor} and \eqref{eqn:mmd:A1:taylor} respectively. Therefore,
\begin{align*}
    A \;=\; 
    \Theta\Big( (A_1+A_2) + A_1 \Big)
    \;=\; 
    \Theta\Big( \mfrac{\|\mu\|_2^4}{\gamma^2} \Big)\;.
\end{align*}
To compute $B$, we first recall the Taylor expansion from \eqref{eq:mmd:beta:expansion} using $\gamma = \omega(1)$ and additionally make use of $\gamma=\Theta(d^{1/2})$ to get
\begin{align*}
    (1 + \beta)^{d/2}
    \;=\;
    \exp\Big( 
        \mfrac{3d}{2\gamma^2} + o\left(\mfrac{d}{ \gamma^{2}} \right)
    \Big)
    \;=\;
    \exp\big( \Theta(1) \big)
    \;=\;
    O(1)
    \;.
\end{align*}
By using the expressions from \eqref{eqn:mmd:B1B2B3:taylor}, \eqref{eqn:mmd:B1:taylor} and \eqref{eqn:mmd:B3:taylor}, we get that
\begin{align*}
    B 
    \;=&\; 
    (B_1-B_2-B_3) + ((1+\beta)^{d/2}-1) B_1 + (1-(1-\alpha)^{d/2}) B_3
    \\
    \;=&\;
    \mfrac{2 \|\mu\|_2^2}{(4+\gamma)(2+\gamma)} 
    +
    o\Big( \mfrac{\|\bmu\|_2^2}{\gamma^2} \Big) 
    +
    ((1+\beta)^{d/2}-1)  (2+o(1)) 
    +
    (1-(1-\alpha)^{d/2}) ( - 2 + o(1))
    \\
    \;=&\;
    \Theta\Big( \mfrac{\|\bmu\|_2^2}{\gamma^2}
    +
    ((1+\beta)^{d/2}+(1-\alpha)^{d/2}-2) 
    \Big)
    \\
    \;=&\;
    O( \gamma^{-2} \|\bmu\|_2^2 + 1)
    \;=\; O(1)
    \;.
\end{align*}
In the last equality, we have noted that $\gamma^{-2} \| \mu\|_2^2 = o(\gamma^{-1}) = o(1) $ by assumption. By additionally noting that $\gamma=\Theta(d^{1/2})$, the variance ratio can therefore be bounded as
\begin{align*}
    \rho_d
    =
    \mfrac{\sigfull}{\sigcond}
    =
    \mfrac{\sqrt{B}}{\sqrt{A}}
    =
    O\Big( \,  \mfrac{ 1 }{ \gamma^{-1} \|\mu\|_2^2 } \, \Big)
    =
    O\Big( \,  \mfrac{ d^{1/2} }{ \|\mu\|_2^2 } \, \Big)
    \;.
\end{align*} 
This completes the proof.

\section{Proofs for \cref{appendix:gaussian:mean:shift}} \label{appendix:proof:gaussian:mean:shift} 

\subsection{Proofs for RBF decomposition and verifying \cref{assumption:L_nu}}

In this section, we prove Lemma \ref{lem:rbf:decomposition}, Lemma \ref{lem:ksd:assumption:L_nu} and Lemma \ref{lem:mmd:assumption:L_nu}.

\vspace{.2em}

\begin{proof}[Proof of Lemma \ref{lem:rbf:decomposition}] We first focus on the one-dimensional RBF kernel, denoted as $\kappa_1$, which can be expressed for $x, x' \in \R$ as
\begin{align*}
    | \kappa_1(x, x') |
    \;=\;
    \big| \exp(-(x - x')^2/(2 \gamma)) \big|
    \;=\; 
    \Big|
    \exp\Big( \mfrac{ x x'}{\gamma}  \Big)
    e^{- x^2 / (2\gamma) } e^{- (x')^2 / (2\gamma) }
    \Big|
    \;.
\end{align*}
By applying a Taylor expansion around $0$ to the infinitely differentiable function $z \mapsto \exp( \frac{z}{\gamma})$ for $z \in \R$, we obtain that for any $K \in \N$ and every $x, x' \in \R$.
\begin{align*}
    \Big| \kappa_1(x, x') - &\msum_{k=0}^K
    \mfrac{1}{k!} \Big( \mfrac{x x'}{\gamma} \Big)^k
    e^{- x^2 / (2\gamma) } e^{- (x')^2 / (2\gamma) }
    \Big|
    \\
    \;\leq&\;
    \msup_{z \in [0, x x']} \Big| \mfrac{1}{(K+1)!} \Big( \mfrac{x x'}{\gamma} \Big)^{K+1} e^{z/\gamma} \Big|
    \; e^{- x^2 / (2\gamma) } e^{- (x')^2 / (2\gamma) }\;.
\end{align*}
Fix $\nu \in (2,4]$. Consider two independent normal random variables $U \sim \cN(b_1,1)$ and $V \sim \cN(b_2,1)$ for some $b_1,b_2 \in \R$, and recall that $\phi^*_{k}(x) \coloneqq x^k
    e^{- x^2 / (2\gamma) }$ and $\lambda^*_{k} \coloneqq \frac{1}{k! \, \gamma^k} $. The above then implies that
\begin{align*}
    \mean\Big[
    \Big| \kappa_1(U, V) - \msum_{k=0}^K
    \lambda^*_{k} &\phi^*_{k}(U) \phi^*_{k}(V)
    \Big|^\nu
    \Big]
    \\
    \;\leq&\;
    \mean\Big[
    \msup_{z \in [0, UV]} \Big| \mfrac{1}{(K+1)!} \Big( \mfrac{UV}{\gamma} \Big)^{K+1} e^{z/\gamma} \Big|^\nu
    \; e^{- \nu U^2 / (2\gamma) } e^{- \nu V^2 / (2\gamma) }
    \Big]
    \\
    \;=&\;
    \mfrac{1}{((K+1)! \; \gamma^{K+1})^\nu } \, 
    \mean\big[
    |UV|^{\nu(K+1)}
    e^{- \nu U^2 / (2\gamma) - \nu V^2 / (2\gamma) + \msup_{z \in [0, UV]} \nu z/\gamma}
    \big]
    \\
    \;\leq&\;
    \mfrac{1}{((K+1)! \; \gamma^{K+1})^\nu} \, 
    \mean\big[
    |UV|^{\nu(K+1)}
    e^{- \nu (|U| - |V|)^2 / (2\gamma)}
    \big]
    \\
    \;\leq&\;
    \mfrac{1}{((K+1)! \; \gamma^{K+1})^\nu} \, 
    \mean\big[ |U|^{\nu(K+1)} \big] \; \mean\big[| V|^{\nu(K+1)}\big]\;.
\end{align*}
In the last inequality, we have noted that $U$ and $V$ are independent and bounded the exponential term from above by $1$. By the formula of absolute moments of a Gaussian, we get that 
\begin{align*}
    \mean\big[ |U - b_1 |^{\nu(K+1)} \big]  
    \;=\;
    \mean\big[ |V - b_2 |^{\nu(K+1)} \big]  
    \;=\; \mfrac{2^{(\nu K)/2}}{\sqrt{\pi}} \Gamma\Big( \mfrac{\nu K+1}{2}\Big)\;.
\end{align*}
By a Jensen's inequality applied to the convex function $x \mapsto |x|^{\nu(K+1)}$, we get that
\begin{align*}
    \mean\big[ & |U|^{\nu(K+1)} \big]  \;=\; \mean\big[ |b_1 + (U-b_1)|^{\nu(K+1)} \big] 
    \;=\;  
    \mean\Big[ \Big| \mfrac{1}{2} (2b_1) + \mfrac{1}{2} (2(U-b_1))\Big|^{\nu(K+1)} \Big]
    \\
    \;\leq&\; 2^{\nu(K+1)-1} \big( b^{\nu(K+1)} + \mean[ |U -b_1|^{\nu(K+1)}] \big) \;=\; \mfrac{(2b_1)^{\nu(K+1)}}{2} + \mfrac{2^{\frac{3}{2} \nu (K+1)}}{2 \sqrt{\pi}} \Gamma\Big( \mfrac{\nu (K+1) + 1}{2}\Big)\;.
\end{align*}
Similarly, we get that
\begin{align*}
    \mean\big[ |V|^{\nu(K+1)} \big]
    \;\leq&\;
    \mfrac{(2b_2)^{\nu(K+1)}}{2} + \mfrac{2^{\frac{3}{2} \nu (K+1)}}{2 \sqrt{\pi}} \Gamma\Big( \mfrac{\nu (K+1) + 1}{2} \Big)\;. \tagaligneq \label{eqn:rbf:decompose:gaussian:moment}
\end{align*}
Substituting these moment bounds and noting that $(K+1)! = \Gamma(K+2)$, we get that
\begin{align*}
    \mean\bigg[
    \bigg| \kappa_1(U, V) - \msum_{k=0}^K
    \lambda^*_{k} \phi^*_{k}(U) &\phi^*_{k}(V)
    \bigg|^\nu
    \bigg]
    \\
    \;\leq&\;
    \mfrac{ 1 }
    {\gamma^{\nu(K+1)} \big(\Gamma(K+2)\big)^\nu} 
    \Big( \mfrac{(2b_1)^{\nu(K+1)}}{2} 
    + 
    \mfrac{2^{\frac{3}{2} \nu (K+1)}}{2 \sqrt{\pi}} \Gamma\Big( \mfrac{\nu (K+1) + 1}{2}\Big) \Big)
    \\
    &\;\times 
    \Big( \mfrac{(2b_2)^{\nu(K+1)}}{2} 
    + 
    \mfrac{2^{\frac{3}{2} \nu (K+1)}}{2 \sqrt{\pi}} \Gamma\Big( \mfrac{\nu (K+1) + 1}{2}\Big) \Big)
    \\
    \;\eqqcolon&\; T \,(A_1+B)(A_2+B)
    \;.
\end{align*}
As $K$ grows, the dominating terms are the Gamma functions, so we only need to control their ratios. By Stirling's formula for the gamma function, we have $\Gamma(x) = \sqrt{2\pi} \,
x^{x-1/2} e^{-x} \big( 1 + O(x^{-1})\big)$ for $x > 0$. This immediately implies that
\begin{align*}
    T A_1 A_2
     \;=\;
     \Theta\Big(
     \mfrac{ (4b_1b_2 / \gamma)^{\nu(K+1)} }
    {(K+2)^{\nu(K+3/2)} e^{-\nu(K+2)}} 
     \Big)
     \;=\;o(1)\;
\end{align*}
as $K \rightarrow \infty$. Meanwhile,
\begin{align*}
    \mfrac{\Gamma\big( \frac{\nu (K+1) + 1}{2} \big)}{ \big(\Gamma(K+2)\big)^\nu}
    \;=&\;
    \Theta\Big( \mfrac{K^{\nu K/2}}{ K^{\nu K}}  \Big) \;=\; \Theta\Big( K^{-\nu K/2}\Big)\;,
\end{align*}
which implies that
\begin{align*}
    T A_1 B \;=\; \Theta \Big( ( 4\sqrt{2} b_1 / \gamma)^{\nu K}  K^{-\nu K/2} \Big) \;=\; o(1)\;,
\end{align*}
since the dominating term is $K^{-\nu K/2}$. Similarly, $T A_2 B = o(1)$. On the other hand, another application of Stirling's formula gives that
\begin{align*}
    \mfrac{\big(\Gamma\big( \frac{\nu (K+1) + 1}{2} \big)^2}{ \big(\Gamma(K+2)\big)^\nu}
    \;=&\;
    (2\pi)^{-(\nu-2)/2}
    \, 
    \mfrac{ \big( \frac{\nu (K+1)+1}{2} \big)^{\nu (K+1)} }{ (K+2)^{\nu(K+3/2)} }
    \,
    e^{ - \nu (K+1) - 1 + \nu(K+2) }
    \,
    \mfrac{\big(1+O(K^{-1})\big)^2}{\big( 1 + O(K^{-1}) \big)^\nu}
    \\
    \;=&\;
    \Theta\Big(  \mfrac{ (\nu/2)^{\nu K} K^{\nu K} }{ K^{\nu(K+3/2)} } \Big)
    \;=\;
    \Theta\big(  (\nu/2)^{\nu K} K^{-3\nu/2} \big)
    \;.
\end{align*}
This implies that
\begin{align*}
    T B^2 \;=\; 
    \Theta\big( (8/\gamma)^{\nu K} (\nu/2)^{\nu K} K^{-3\nu/2} \big)
    \;=\;
    \Theta\big( (2 \nu/\gamma)^{\nu K}  K^{-3\nu/2} \big)
    \;=\;o(1)\;,
\end{align*}
where we have recalled that $\nu \leq 4$ and used the assumption that $\gamma > 8$. In summary, we have proved that for $\nu \in (2,4]$ and \emph{any} fixed $b_1, b_2 \in \R$,
\begin{align*}
    \mean\Big[
    \Big| \kappa_1(U, V) - \msum_{k=0}^K
    \lambda^*_{k} &\phi^*_{k}(U) \phi^*_{k}(V)
    \Big|^\nu
    \Big] 
    \;\leq\; T(A_1+B)(A_2+B)
    \;\xrightarrow{K \rightarrow \infty}\; 0\;.
\end{align*}
To extend this to multiple dimensions, we note that for the vectors $\bx=(x_1,\ldots,x_d) \in \R^d$ and $\bx'=(x_1,\ldots,x_d) \in \R^d$, the multi-dimensional RBF kernel can then be expressed as
\begin{align*}
    \kappa(\bx,\bx') 
    \;=\; 
    \exp\big( - \|\bx-\bx'\|_2^2/(2\gamma)\big)
    \;=\; 
    \mprod_{l=1}^d \exp\big( - (x_l-x'_l)^2/(2\gamma)\big)
    \;=\;
    \mprod_{l=1}^d \kappa_1(x_l, x'_l)\;.
\end{align*}
Recall that we have defined the independent normal vectors $\bU \sim \cN(\bzero, I_d)$ and $\bV \sim \cN(\mu, I_d)$. Let $U_1, \ldots, U_d$ be the coordinates of $\bU$ and $V_1, \ldots, V_d$ be those of $\bV$, which are all independent since the covariance matrices are $I_d$. For $0 \leq l \leq d$ and $K \in \N$, define the random quantities
\begin{align*}
    S_{j;K} \;\coloneqq&\; \msum_{k=0}^K \lambda^*_{k} \phi^*_{k}(U_j) \phi^*_{k}(V_j)
    &\text{ and }&&
    W_{l;K} \;\coloneqq&\; 
    \Big( \mprod_{j=1}^l \kappa_1(U_j, V_j)\Big)
    \Big( \mprod_{j=l+1}^d S_{j;K} \Big)
    \;.
\end{align*}
In particular $\kappa(\bU,\bV) = W_{d;K}$. Now by expanding a telescoping sum and applying a triangle inequality followed by a Jensen's inequality, we have
\begin{align*}
    &\mean\big[ | \kappa(\bU,\bV) - W_{0;K} |^\nu \big]
    \;=\;
    \mean\Big[ \Big| \msum_{l=1}^d (W_{l;K} - W_{l-1;K}) \Big|^\nu \Big]
    \\
    \;\leq&\;
    \mean\Big[ \Big( \msum_{l=1}^d |W_{l;K} - W_{l-1;K}| \Big)^\nu \Big]
    \\
    \;\leq&\;
    d^{\nu-1} \msum_{l=1}^d \mean[ |W_{l;K} - W_{l-1;K}|^\nu ]
    \\
    \;=&\;
    d^{\nu-1} \msum_{l=1}^d
    \Big( \mprod_{j=1}^{l-1} \mean[ | \kappa_1(U_j,V_j)|^\nu]\Big)
    \mean[ | \kappa_1(U_l, V_l) - S_{l;K} |^\nu ]
    \Big( \mprod_{j=l+1}^d \mean[ | S_{j;K} |^\nu]\Big)
    \;. 
\end{align*}
In the last equality, we have used the independence of $U_j$'s and $V_j$'s. To bound the summands, we first note that $\kappa_1$ is uniformly bounded in norm by $1$, which implies that $\mean[|\kappa_1(U_j,V_j)|^\nu] \leq 1$. By the previous result, $\mean[ | \kappa_1(U_l, V_l) - S_{l;K} |^\nu ] = o(1)$ as $K \rightarrow \infty$. By a triangle inequality and a Jensen's inequality, we have that
\begin{align*}
    \mean[ | S_{j;K} |^\nu] 
    \;\leq&\; 
    \mean\big[ 
    \big| 
    | \kappa_1(U_j,V_j) | + | S_{j;K} - \kappa_1(U_j,V_j)| 
    \big|^\nu 
    \big]
    \\
    \;\leq&\;
    2^{\nu-1}
    \mean\big[ | \kappa_1(U_j,V_j) |^\nu \big] 
    + 
    2^{\nu-1}
    \mean\big[ | S_{j;K} - \kappa_1(U_j,V_j)|^\nu \big] \;\leq\; 2^{\nu - 1} + o(1)\;.
\end{align*}
This implies that each summand satisfies
\begin{align*}
    \Big( \mprod_{j=1}^{l-1} \mean[ | \kappa_1(U_j,V_j)|^\nu]\Big)
    \mean[ | \kappa_1(U_l, V_l) - S_{l;K} |^\nu ]
    \Big( \mprod_{j=l+1}^d \mean[ | S_{j;K} |^\nu]\Big)
    \;=\; o(1)\;
\end{align*}
as $K \rightarrow \infty$. Since $d$ is not affected by $K$, we have shown the desired result
\begin{align*}
    \mean\Big[ \Big| \kappa(\bU,\bV) - 
    \mprod_{j=1}^d \Big(
    \msum_{k=0}^K \lambda^*_{k} \phi^*_{k}(U_j) \phi^*_{k}(V_j)
    \Big)
    \Big|^\nu \Big]
    \;=\;
    \mean\big[ | \kappa(\bU,\bV) - W_{0;K} |^\nu \big]
    \;\xrightarrow{K \rightarrow \infty}\;
    0\;.
\end{align*}

\end{proof}

\vspace{.2em}

\begin{proof}[Proof of Lemma \ref{lem:ksd:assumption:L_nu}] We first rewrite $\uPksd$ as
\begin{align*}
    \uPksd(\bx, \bx')
    \;=&\;
    e^{ - \| \bx - \bx' \|_2^2 / (2\gamma) }
    \Big( 
        \bx^\top \bx' 
        - \mfrac{\gamma + 1}{\gamma^2} \| \bx - \bx' \|_2^2
        + \mfrac{d}{\gamma}
    \Big)
    \\
    \;=&\;
    e^{ - \| \bx - \bx' \|_2^2 / (2\gamma) }
    \Big( 
        - \mfrac{\gamma+1}{\gamma^2}(\|\bx\|_2^2 + \|\bx'\|_2^2)
        +
        \mfrac{\gamma^2+2\gamma+2}{\gamma^2}
        \bx^\top \bx'
        +
        \mfrac{d}{\gamma}
    \Big)
    \\
    \;=&\;
    e^{ - \| \bx - \bx' \|_2^2 / (2\gamma) }
    \Big( 
        - \mfrac{\gamma+1}{\gamma^2}
        (\|\bx\|_2^2 + 1)(\|\bx'\|_2^2 + 1)
        +
        \mfrac{\gamma+1}{\gamma^2} \|\bx\|_2^2 \|\bx'\|_2^2
    \\
    &\;\;\;\qquad\qquad\qquad+
        \mfrac{\gamma^2+2\gamma+2}{\gamma^2}
        \msum_{l=1}^d x_l x'_l
        +
        \Big( \mfrac{d}{\gamma} + \mfrac{\gamma+1}{\gamma^2} \Big)
    \Big)\;.
\end{align*}
For $K' \in \N$, write $S_{K'} \coloneqq \sum_{k'=1}^{K'} \alpha_{k'} \psi_{k'}(\bX_1)\psi_{k'}(\bX_2)$, and define the following random variables comparing each set of eigenvalue and eigenfunction to the corresponding term in $\uPksd$:
\begin{align*}
    T_{K';1} 
    \;=&\; 
    \msum_{k'=1}^{K'} \lambda_{(k'-1)(d+3)+1} \, \phi_{(k'-1)(d+3)+1}(\bX_1)  \, \phi_{(k'-1)(d+3)+1}(\bX_1)
    \\
    &\;\quad -
    e^{ - \| \bX_1 - \bX_2 \|_2^2 / (2\gamma) } \Big( - \mfrac{\gamma+1}{\gamma^2}
        (\|\bX_1\|_2^2 + 1)(\|\bX_2\|_2^2 + 1)\Big)
    \\
    \;=&\;
     - \mfrac{\gamma+1}{\gamma^2}
        (\|\bX_1\|_2^2 + 1)(\|\bX_2\|_2^2 + 1) S_{K'}\;,
    \\
    T_{K';2} 
    \;=&\; 
    \msum_{k'=1}^{K'} \lambda_{(k'-1)(d+3)+2} \, \phi_{(k'-1)(d+3)+2}(\bX_1)  \, \phi_{(k'-1)(d+3)+2}(\bX_1)
    \\
    &\;\quad -
    e^{ - \| \bX_1 - \bX_2 \|_2^2 / (2\gamma) } \Big( \mfrac{\gamma+1}{\gamma^2}
        \|\bX_1\|_2^2 \|\bX_2\|_2^2 \Big)
    \\
    \;=&\;
     \mfrac{\gamma+1}{\gamma^2}
        \|\bX_1\|_2^2 \|\bX_2\|_2^2 \, S_{K'}\;,
\end{align*}
\begin{align*}
    T_{K';3} 
    \;=&\; 
    \msum_{k'=1}^{K'} \lambda_{(k'-1)(d+3)+3} \, \phi_{(k'-1)(d+3)+3}(\bX_1)  \, \phi_{(k'-1)(d+3)+3}(\bX_1)
    \\
    &\;\quad -
     e^{ - \| \bX_1 - \bX_2 \|_2^2 / (2\gamma) } \Big( \mfrac{d}{\gamma} + \mfrac{\gamma+1}{\gamma^2} \Big)
    \\
    \;=&\;
    \Big( \mfrac{d}{\gamma} + \mfrac{\gamma+1}{\gamma^2} \Big) \, S_{K'}\;,
    \\
    T_{K';3+l} 
    \;=&\; 
    \msum_{k'=1}^{K'} \lambda_{(k'-1)(d+3)+3+l} \, \phi_{(k'-1)(d+3)+3+l}(\bX_1)  \, \phi_{(k'-1)(d+3)+3+l}(\bX_1)
    \\
    &\;\quad -
     e^{ - \| \bX_1 - \bX_2 \|_2^2 / (2\gamma) } \Big( \mfrac{\gamma^2+2\gamma+2}{\gamma^2} (\bX_1)_l (\bX_2)_l  \big)
    \\
    \;=&\;
    \Big( \mfrac{\gamma^2+2\gamma+2}{\gamma^2} (\bX_1)_l (\bX_2)_l \Big) \, S_{K'}\;
\end{align*}
for $l=1, \ldots, d$, where we have denoted the $l$-th coordinates of $\bX_1$ and $\bX_2$ by $(\bX_1)_l$ and $(\bX_2)_l$ respectively. We now bound the approximation error with $(d+3)K'$ summands for $K' \in \N$ and $\nu \in (2,3]$. Fix some $\nu_1 \in (\nu, 4]$ and let $\nu_2 = 1/(\nu^{-1} - \nu_1^{-1})$. By using the quantites defined above, a Jensen's inequality to the convex function $x \mapsto |x|^\nu$ and a H\"older's inequality to each $\mean[ | T_{K';l}  |^\nu ]$, we have
\begin{align*}
    &\;\mean\big[ \big| \msum_{k=1}^{(d+3)K'} \lambda_k \phi_k(\bX_1)\phi_k(\bX_2) - 
    \uPksd(\bX_1, \bX_2) \big|^\nu \big]
    \\
    \;=&\;
    \mean\big[ \big| \msum_{l=1}^{d+3}T_{K';l} \big|^\nu \big]
    \\
    \;\leq&\;
    (d+3)^{\nu - 1}
    \msum_{l=1}^{d+3}
    \mean[ | T_{K';l}  |^\nu ]
    \\
    \;\leq&\;
    (d+3)^{\nu - 1}
    \mean[ | S_{K'}  |^{\nu_1} ]^{\nu / \nu_1}
    \Big(
    \;
    \Big(\mfrac{\gamma+1}{\gamma^2}\Big)^\nu
     \mean[ (\|\bX_1\|_2^2 + 1)^{\nu_2}]^{\nu / \nu_2} \; \mean[ (\|\bX_2\|_2^2 + 1)^{\nu_2} ]^{\nu / \nu_2}
     \\
    &\;\qquad\qquad\qquad\qquad\qquad
    + \Big(\mfrac{\gamma+1}{\gamma^2}\Big)^\nu
     \mean\big[ \|\bX_1\|_2^{2\nu_2}\big]^{\nu / \nu_2} \mean\big[ \|\bX_2\|_2^{2\nu_2}\big]^{\nu / \nu_2}
     +
     \Big( \mfrac{d}{\gamma} + \mfrac{\gamma+1}{\gamma^2} \Big)^\nu
     \\
    &\;\qquad\qquad\qquad\qquad\qquad
    +
    \msum_{l=1}^d
     \Big( \mfrac{\gamma^2+2\gamma+2}{\gamma^2} \Big)^{\nu}
     \mean\big[ | (\bX_1)_l |^{\nu_2} \big]^{\nu / \nu_2} 
      \mean\big[ | (\bX_2)_l |^{\nu_2} \big]^{\nu / \nu_2}
    \Big)\;.
\end{align*}
The only $K'$-dependence above comes from $\mean[ | S_{K'}  |^{\nu_1} ]^{\nu / \nu_1} = \| S_{K'} \|_{L_{\nu_1}}^\nu$, which converges to 0 as $K'$ grows by Lemma \ref{lem:rbf:decomposition}. Therefore
\begin{align*}
    \mean\big[ \big| \msum_{k=1}^{(d+3)K'} \lambda_k \phi_k(\bX_1)\phi_k(\bX_2) - 
    \uPksd(\bX_1, \bX_2) \big|^\nu \big]
    \;\xrightarrow{K' \rightarrow \infty }0\;.
\end{align*}
Now for $K \in \N$ not necessarily divisible by $d+3$, we let $K'$ be the largest integer such that $dK' \leq K$. By a triangle inequality and a Jensen's inequality, we have
\begin{align*}
    &\;\mean\big[ \big| \msum_{k=1}^{K} \lambda_k \phi_k(\bX_1)\phi_k(\bX_2) - 
    \uPksd(\bX_1, \bX_2) \big|^\nu \big]
    \\
    \;\leq&\;
    \mean\Big[ \Big( 
    \big| \msum_{k=1}^{(d+3)K'} \lambda_k \phi_k(\bX_1)\phi_k(\bX_2) - 
    \uPksd(\bX_1, \bX_2) \big| 
    +
    \big| \msum_{k=(d+3)K'+1}^{K} \lambda_k \phi_k(\bX_1)\phi_k(\bX_2) \big|
    \Big)^\nu \Big]
    \\
    \;\leq&\;
    2^{\nu-1}
    \mean\big[
    \big| \msum_{k=1}^{(d+3)K'} \lambda_k \phi_k(\bX_1)\phi_k(\bX_2) - 
    \uPksd(\bX_1, \bX_2) \big|^\nu \big] 
    \\
    &\;
    +
    2^{\nu-1}
    \mean\big[
    \big| \msum_{k=(d+3)K'+1}^{K} \lambda_k \phi_k(\bX_1)\phi_k(\bX_2) \big|^\nu \big]\;.
\end{align*}
The goal is to show that the bound converges to $0$ as $K$ grows. We have already shown that the first term is $o(1)$, so we focus on the second term. The expectation in the second term can be bounded using a Jensen's inequality as
\begin{align*}
    \mean\big[
    \big| \msum_{k=(d+3)K'+1}^{K} \lambda_k &\phi_k(\bX_1)\phi_k(\bX_2) \big|^\nu \big]
    \;\leq\;
    \mean\big[
    \big( \msum_{k=(d+3)K'+1}^{K} | \lambda_k \phi_k(\bX_1)\phi_k(\bX_2) | \big)^\nu \big]
    \\
    \;\leq&\;
    (K-(d+3)K')^{\nu-1} \msum_{k=(d+3)K'+1}^{K} \mean\big[ \big(\lambda_k \phi_k(\bX_1)\phi_k(\bX_2) \big)^\nu \big]
    \\
    \;\leq&\;
    d^{\nu} \msup_{k \in \{(d+3)K'+1, \ldots, (d+3)K'+(d+3)\}} \mean\big[ \big(\lambda_k \phi_k(\bX_1)\phi_k(\bX_2) \big)^\nu \big]
    \\
    \;=&\;
    d^{\nu} \msup_{1 \leq l \leq d+3} \mean\big[ \big(\lambda_{(d+3)K'+l} \phi_{(d+3)K'+l}(\bX_1)\phi_{(d+3)K'+l}(\bX_2) \big)^\nu \big]\;.
\end{align*}
By observing the formula for $\lambda_k$ and $\phi_k$, we see that there exists some $K$-independent constant $C_{d,\gamma}$ such that for $1 \leq l \leq d+3$,
\begin{align*}
    |\lambda_{(d+3)K'+l}| \;&\leq\; C_{d,\gamma} \alpha_{K'+1}
    &\text{ and }&&
    |\phi_{(d+3)K'+l}| \;&\leq\; C_{d,\gamma} \psi_{K'+1}(\bx) ( \|\bx\|_2^2 + \|\bx\|_2 + 1 )\;.
\end{align*}
This allows us to obtain the bound
\begin{align*}
    &\;\mean\big[
    \big| \msum_{k=(d+3)K'+1}^{K}\lambda_k \phi_k(\bX_1)\phi_k(\bX_2) \big|^\nu \big]
    \\
    \;\leq&\;
    d^{\nu}
    C_{d,\gamma}^2
    \alpha_{K'+1}^\nu 
    \mean[ (\psi_{K'+1}(\bX_1) \psi_{K'+1}(\bX_2))^\nu \, (\|\bX_1\|_2^2 + \|\bX_1\|_2 + 1 )^\nu(\|\bX_2\|_2^2 + \|\bX_2\|_2 + 1 )^\nu] 
    \\
    \;\overset{(a)}{=}&\;
    d^{\nu}
    C'_{d,\gamma} 
    \Big( \mprod_{l=1}^d \lambda^*_{[g_d(K'+1)]_l} \Big)^\nu 
    \mean\Big[ \Big( \mprod_{l=1}^d \phi^*_{[g_d(K'+1)]_l} \big((\bX_1)_l\big) \; \phi^*_{[g_d(K'+1)]_l}\big((\bX_2)_l)\big) \Big)^\nu
    \\
    &\;\qquad\qquad\qquad\qquad \qquad\qquad \times (\|\bX_1\|_2^2 + \|\bX_1\|_2 + 1 )^\nu(\|\bX_2\|_2^2 + \|\bX_2\|_2 + 1 )^\nu\Big] 
    \\
    \;\overset{(b)}{\leq}&\;
    d^{\nu}
    C'_{d,\gamma} 
    \mean\big[ (\|\bX_1\|_2^2 + \|\bX_1\|_2 + 1 )^{2\nu}(\|\bX_2\|_2^2 + \|\bX_2\|_2 + 1 )^{2\nu} \big]^{1/2}
    \\
    &\;\; \times \Big( \mprod_{l=1}^d \lambda^*_{[g_d(K'+1)]_l} \Big)^\nu  \mean\Big[ \Big( \mprod_{l=1}^d \phi^*_{[g_d(K'+1)]_l} \big((\bX_1)_l\big) \; \phi^*_{[g_d(K'+1)]_l}\big((\bX_2)_l)\big) \Big)^{2\nu} \Big]^{1/2}
    \\
    \;\overset{(c)}{=}&\;
    d^{\nu}
    C'_{d,\gamma} 
    \mean\big[ (\|\bX_1\|_2^2 + \|\bX_1\|_2 + 1 )^{2\nu}(\|\bX_2\|_2^2 + \|\bX_2\|_2 + 1 )^{2\nu} \big]^{1/2}
    \\
    &\;\; \times \mprod_{l=1}^d \big(\lambda^*_{[g_d(K'+1)]_l}\big)^\nu  \; \Big( \mprod_{l=1}^d \mean\big[ \big(\phi^*_{[g_d(K'+1)]_l} \big((\bX_1)_l\big) \,\big)^{2\nu} \big] \mean\big[ \big(\phi^*_{[g_d(K'+1)]_l} \big((\bX_2)_l\big) \,\big)^{2\nu} \big] \Big)^{1/2}
    \\
    \;\overset{(d)}{=}&\;
    d^{\nu}
    C'_{d,\gamma} 
    \mean\big[ (\|\bX_1\|_2^2 + \|\bX_1\|_2 + 1 )^{2\nu}(\|\bX_2\|_2^2 + \|\bX_2\|_2 + 1 )^{2\nu} \big]^{1/2}
    \\
    &\;\; \times \mprod_{l=1}^d \Big( \big(\lambda^*_{[g_d(K'+1)]_l}\big)^\nu  \; \mean\big[ \big(\phi^*_{[g_d(K'+1)]_l} \big((\bX_1)_l\big) \,\big)^{2\nu} \big] \Big)
    \;,
\end{align*}
where we have used the definitions of $\alpha_k$ and $\psi_k$ from \eqref{eqn:rbf:decomposition:alt:defn} in $(a)$, a Cauchy-Schwarz inequality in $(b)$, the independence of $(\bX_1)_l$ and $(\bX_2)_l$ for $1 \leq l \leq d$ due to the identity covariance matrix in $(c)$ and finally the fact that $\bX_1$ and $\bX_2$ are identically distributed in $(d)$. The only quantity that depends on $K'$ now is 
\begin{align*}
    \big(\lambda^*_{[g_d(K'+1)]_l}\big)^\nu  \; \mean\big[ \big(\phi^*_{[g_d(K'+1)]_l} \big((\bX_1)_l\big) \,\big)^{2\nu} \big]
\end{align*}
for $1 \leq l \leq d$. We now seek to bound this quantity. Recall from Lemma \ref{lem:rbf:decomposition} that $ \lambda^*_{k} \coloneqq \mfrac{1}{k! \, \gamma^k}$, and for $V \sim \cN(b,1)$, we have
\begin{align*}
    \mean[ (\phi^*_{k}(U))^{2\nu} ]
    \;=\;
    \mean\big[ |U|^{2 \nu k} e^{- \nu U^2 / \gamma } \big] 
    \;\leq\;
    \mean\big[ |U|^{2 \nu k}\big]
    \;\leq\;
    \mfrac{(2b)^{2\nu k}}{2} + \mfrac{2^{3\nu k}}{2 \sqrt{\pi}} \, \Gamma\Big( \mfrac{2 \nu k +1}{2}\Big)\;.
\end{align*}
where we have used a bound similar to \eqref{eqn:rbf:decompose:gaussian:moment} in the proof of Lemma \ref{lem:rbf:decomposition}. By Stirling's formula for the gamma function, we have $\Gamma(x) = \sqrt{2\pi} \,
x^{x-1/2} e^{-x} \big( 1 + O(x^{-1})\big)$ for $x > 0$, which implies
\begin{align*}
    (\lambda^*_{k})^\nu \, \mean[ (\phi^*_{k}(U))^{2\nu} ]
    \;\leq&\;
    \mfrac{1}{ (k!)^\nu \, \gamma^{\nu k}} \Big(
    \mfrac{(2b)^{2\nu k}}{2} + \mfrac{2^{3\nu k}}{2 \sqrt{\pi}} \, \Gamma\Big( \mfrac{2 \nu k +1}{2}\Big)
    \Big)
    \\
    \;=&\;
    O\Big( \Big(\mfrac{8}{\gamma} \Big)^{\nu k} \mfrac{(\nu k)^{\nu k} e^{- \nu k} }{(k+1)^{\nu(k+1/2)} e^{-\nu(k+1)}} \Big)
    \\
    \;=&\;
    O\Big( \Big(\mfrac{8 \nu}{\gamma} \Big)^{\nu k}\Big)
    \;=\;
    O\Big( \Big(\mfrac{24}{\gamma} \Big)^{\nu k}\Big) \;=\; o(1)\;
\end{align*}
as $k \rightarrow \infty$, where we have used the assumption that $\gamma > 24$. By construction of $g_d$ in \eqref{eqn:rbf:decomposition:alt:defn}, as $K' \rightarrow \infty$, $\min_{1 \leq l \leq d} [g_d(K'+1)]_l \rightarrow \infty$, which implies that 
\begin{align*}
    \big(\lambda^*_{[g_d(K'+1)]_l}\big)^\nu  \; \mean\big[ \big(\phi^*_{[g_d(K'+1)]_l} \big((\bX_1)_l\big) \,\big)^{2\nu} \big]
    \;\xrightarrow{K' \rightarrow \infty}\; 0\;.
\end{align*}
Therefore
\begin{align*}
    \mean\big[
    \big| \msum_{k=(d+3)K'+1}^{K}\lambda_k \phi_k(\bX_1)\phi_k(\bX_2) \big|^\nu \big]
    \;\xrightarrow{K \rightarrow \infty}\; 0\;,
\end{align*}
which finishes the proof that
\begin{align*}
    \mean\big[ \big| \msum_{k=1}^{K} \lambda_k \phi_k(\bX_1)\phi_k(\bX_2) - 
    \uPksd(\bX_1, \bX_2) \big|^\nu \big]
    \;\xrightarrow{K \rightarrow \infty}\; 0\;.
\end{align*}
In other words, \cref{assumption:L_nu} holds.
\end{proof}

\vspace{.2em} 

\begin{proof}[Proof of Lemma \ref{lem:mmd:assumption:L_nu}] Fix $\nu \in (2,3]$. Consider the independent Gaussian vectors $\bX_1, \bX_2 \overset{i.i.d.}{\sim} P \equiv \cN(\bzero, I_d)$ and $\bY_1, \bY_2 \overset{i.i.d.}{\sim} Q \equiv \cN(\mu, I_d)$. Write $\bZ_1 = (\bX_1,\bY_1)$, $\bZ_2 = (\bX_2,\bY_2)$ and
\begin{align*}
    T_K(\bx,\bx') 
    \;\coloneqq\; 
    e^{-\|\bx-\bx'\|_2^2/(2\gamma)}
    -
    \msum_{k=1}^K \alpha_k \psi_k(\bx) \psi_k(\bx')
\end{align*}
for $K \in \N$, and recall that
\begin{align*}
    \ummd(\bZ_1,\bZ_2)
    \;=\; 
    e^{-\|\bX_1-\bX_2\|_2^2/(2\gamma)}
    -
    e^{-\|\bX_1-\bY_2\|_2^2/(2\gamma)}
    -
    e^{-\|\bX_2-\bY_1\|_2^2/(2\gamma)}
    +
    e^{-\|\bY_1-\bY_2\|_2^2/(2\gamma)}
    \;.
\end{align*}
Then by a triangle inequality and Jensen's inequality, we get that
\begin{align*}
    &\;\mean\Big[ \Big| \ummd(\bZ_1, \bZ_2) - \msum_{k=1}^K \lambda_k \phi_k(\bZ_1) \phi_k(\bZ_2) \Big|^\nu \Big] 
    \\
    \;=&\;\mean\Big[ \Big| \ummd(\bZ_1, \bZ_2) - \msum_{k=1}^K \alpha_k \big( \psi_k(\bX_1) - \psi_k(\bY_1)\big) \big( \psi_k(\bX_2) - \psi_k(\bY_2)\big)  \Big|^\nu \Big] 
    \\
    \;=&\;
    \mean\big[ \big| T_K(\bX_1,\bX_2) -  T_K(\bX_1,\bY_2)
    -
    T_K(\bX_2,\bY_1)
    +
    T_K(\bX_2,\bY_2)
    \big|^\nu \big]
    \\
    \;\leq&\;
    4^{\nu-1}
    \big(
    \mean[ | T_K(\bX_1,\bX_2) |^\nu ] 
    +
    \mean[ | T_K(\bX_1,\bY_2) |^\nu ]
    +
    \mean[ | T_K(\bX_2,\bY_1) |^\nu ]
    +
    \mean[ | T_K(\bY_1,\bY_2) |^\nu ]
    \big)
    \;.
\end{align*}
Since each expectation is taken with respect to a product of two Gaussian distributions with identity covariance matrices, by Lemma \ref{lem:rbf:decomposition} and \eqref{eqn:rbf:decomposition:alt}, they all decay to $0$ as $K \rightarrow \infty$. This proves that 
\begin{align*}
    \mean\Big[ \Big| \ummd(\bZ_1, \bZ_2) - \msum_{k=1}^K \lambda_k \phi_k(\bZ_1) \phi_k(\bZ_2) \Big|^\nu \Big] 
    \;\xrightarrow{K \rightarrow \infty}\; 0\;,
\end{align*}
and therefore \cref{assumption:L_nu} holds.
\end{proof}

\subsection{Proof for Lemma~\ref{lem:ksd:moments:analytical}}
\label{pf:ksd:moments:analytical}
We restate the KSD U-statistic for RBF under our Gaussian mean-shift setup from \eqref{eqn:defn:ksd:rbf}:
\begin{align}
    \uPksd(\bx, \bx')
    \;=&\;
    \exp\left( - \mfrac{1}{2\gamma} \| \bx - \bx' \|_2^2  \right)
    \Big( 
        \bx^\top \bx' 
        - \mfrac{\gamma + 1}{\gamma^2} \| \bx - \bx' \|_2^2
        + \mfrac{d}{\gamma}
    \Big)
    \;.
    \label{eq:uPksd:gaussian:simplified}
\end{align}

\subsubsection{Proof for $g(\bx)$}
Fix $\bx \in \R^d$. Taking expectation of $\uPksd(\bx, \bX')$ with respect to the distribution of $\bX'$,
\begin{align*}
    g(\bx)
    \;=&\;
    \E[ \uPksd(\bx, \bX') ]
    \\
    \;=&\;
    \E\Big[
    \exp\left( - \mfrac{1}{2\gamma} \| \bx - \bX' \|_2^2  \right)
        \left( 
            \bx^\top \bX' 
            - \mfrac{1 + \gamma}{\gamma^2} \| \bx - \bX' \|_2^2
            + \mfrac{d}{\gamma}
        \right)
    \Big]
    \\
    \;=&\;
    \left( \mfrac{\gamma}{1 + \gamma} \right)^{d/2}
    \exp\left( - \mfrac{1}{2(1 + \gamma)} \| \bx - \bmu \|_2^2 \right)
    \E\left[ 
        \bx^\top \bW' 
        - \mfrac{1 + \gamma}{\gamma^2} \| \bx - \bW' \|_2^2
        + \mfrac{d}{\gamma}
    \right]
    \;.
\end{align*}
where the third line follows by applying Lemma \ref{lem:gauss:integral:single}, and $\bW' \sim \cN\left( \frac{1}{1 + \gamma}\big(\mu + \frac{1}{\gamma}\bx\big), \frac{\gamma}{1 + \gamma} I_d \right)$. The proof is completed by calculating the expectation as
\begin{align*}
    &\;
    \E\left[ 
        \bx^\top \bW' 
        - \mfrac{1 + \gamma}{\gamma^2} \| \bx - \bW' \|_2^2
        + \mfrac{d}{\gamma}
    \right]
    \\
    \;=&\;
    \E\left[ 
        \bx^\top \bW' 
        - \mfrac{1 + \gamma}{\gamma^2} \left(
            \| \bW' \|_2^2 - 2\bx^\top \bW' + \| \bx \|_2^2
        \right)
        + \mfrac{d}{\gamma}
    \right]
    \\
    \;=&\;
    \mfrac{\gamma}{1 + \gamma} \left( \bmu + \mfrac{1}{\gamma} \bx \right)^\top \bx - \mfrac{\gamma + 1}{\gamma^2} \left( \mfrac{\gamma d}{1 + \gamma} + \mfrac{\gamma^2}{(1 + \gamma)^2} \big\| \bmu + \mfrac{1}{\gamma} \bx \big\|_2^2 
    - \mfrac{\gamma}{1 + \gamma} \left( \bmu + \mfrac{1}{\gamma} \bx \right)^\top \bx
    + \| \bx \|_2^2 \right) + \mfrac{d}{\gamma}
    \\
    \;=&\;
    \mfrac{2 + \gamma}{1 + \gamma} \bmu^\top \bx 
    - \mfrac{1}{1 + \gamma} \| \bmu \|_2^2
    \;.
\end{align*}

\subsubsection{Proof for $\ksd(Q,P)$}
Noting that $\ksd(Q, P) = \E[ \gksd(\bx) ]$, we can apply Lemma \ref{lem:gauss:integral:single} again to yield
\begin{align*}
    \ksd(Q, P)
    \;=&\;
    \left( \mfrac{\gamma}{1 + \gamma} \right)^{d/2}
    \E\Big[
        \exp\left( - \mfrac{1}{2(1 + \gamma)} \| \bX - \bmu \|_2^2 \right)
        \left(
            \mfrac{2 + \gamma}{1 + \gamma} \bmu^\top \bX 
            - \mfrac{1}{1 + \gamma} \| \bmu \|_2^2
        \right)
    \Big]
    \\
    \;=&\;
    \left( \mfrac{\gamma}{1 + \gamma} \right)^{d/2}
    \left( \mfrac{1 + \gamma}{2 + \gamma} \right)^{d/2}
    \E\Big[
        \mfrac{2 + \gamma}{1 + \gamma} \bmu^\top \bW - \mfrac{1}{1 + \gamma} \| \bmu \|_2^2
    \Big]
    \;,
\end{align*}
where $W \sim \cN\left( \frac{1 + \gamma}{2 + \gamma} \big( \bmu + \frac{1}{1 + \gamma} \bmu \big), \mfrac{1 + \gamma}{2 + \gamma} I_d \right)$. We then have 
\begin{align*}
    \ksd(Q, P) 
    \;=&\; 
    \left( \mfrac{\gamma}{1 + \gamma} \right)^{d/2}
    \left( \mfrac{1 + \gamma}{2 + \gamma} \right)^{d/2}
    \left( 
        \bmu^\top \left( \bmu + \mfrac{1}{1 + \gamma} \bmu \right)
        - \mfrac{1}{1 + \gamma}\|\bmu\|_2^2 
    \right)
    \\
    \;=&\;
    \left( \mfrac{\gamma}{2 + \gamma} \right)^{d/2} \| \bmu \|_2^2
    \;,
\end{align*}
as required.

\subsubsection{Proof for $\sigcond^2$}
We first calculate the second moment as
\begin{align*}
    \E[ g(\bX)^2 ]
    \;=&\;
    \left( \mfrac{\gamma}{1 + \gamma} \right)^{d}
    \E\Big[
        \exp\left( 
            - \mfrac{1}{1 + \gamma} \| \bX - \bmu \|_2^2 
        \right)
        \left( 
            \mfrac{2 + \gamma}{1 + \gamma} \bmu^\top \bX 
            - \mfrac{1}{1 + \gamma} \| \bmu \|_2^2 
        \right)^2
    \Big]
    \\
    \;=&\;
    \left( \mfrac{\gamma}{1 + \gamma} \right)^{d}
    \left( \mfrac{1 + \gamma}{3 + \gamma} \right)^{d/2}
    \E\Big[
        \left( 
            \mfrac{2 + \gamma}{1 + \gamma} \bmu^\top \bW 
            - \mfrac{1}{1 + \gamma} \| \bmu \|_2^2 
        \right)^2
    \Big]
\end{align*}
where in the last line we have applied Lemma \ref{lem:gauss:integral:single} while setting $\bW \sim \cN\big(\bm , \frac{1 + \gamma}{3 + \gamma} I_d \big)$ and $\bm \coloneqq \frac{1 + \gamma}{3 + \gamma} \big( \bmu + \frac{2}{1 + \gamma} \bmu \big) = \bmu$. This gives
\begin{align*}
    \E[ g(\bX)^2 ]
    \;=&\;
    \left( \mfrac{\gamma}{1 + \gamma} \right)^{d}
    \left( \mfrac{1 + \gamma}{3 + \gamma} \right)^{d/2}
    \E\Big[
            \left(\mfrac{2 + \gamma}{1 + \gamma} \right)^2 (\bmu^\top \bW)^2
            + \mfrac{1}{(1 + \gamma)^2}\| \bmu \|_2^4 
            - \mfrac{2(2 + \gamma)}{(1 + \gamma)^2} \| \bmu \|_2^2 \bmu^\top \bW
    \Big]
    \\
    \;=&\;
    \left( \mfrac{\gamma}{1 + \gamma} \right)^{d}
    \left( \mfrac{1 + \gamma}{3 + \gamma} \right)^{d/2}
    \Big(
        \left(\mfrac{2 + \gamma}{1 + \gamma} \right)^2 \bmu^\top \left( \mfrac{1 + \gamma}{3 + \gamma} I_d + \mu \mu^\top \right) \bmu
        + \mfrac{1}{(1 + \gamma)^2} \| \bmu \|_2^4
        \\
        &\qquad\qquad\qquad\qquad\qquad
        - \mfrac{2(2 + \gamma)}{(1 + \gamma)^2} \| \bmu \|_2^4
    \Big)
    \\
    \;=&\;
    \left( \mfrac{\gamma^2}{(1 + \gamma)(3 + \gamma)} \right)^{d/2}
    \Big(
        \mfrac{(2 + \gamma)^2}{(1 + \gamma)(3 + \gamma)} \| \bmu \|_2^2
        + \| \bmu \|_2^4
    \Big)
    \;.
\end{align*}
We hence obtain
\begin{align*}
    \sigcond^2
    \;=&\;
    \E[ g(\bX)^2 ] - \ksd(Q, P)^2
    \\
    \;=&\;
    \left( \mfrac{\gamma^2}{(1 + \gamma)(3 + \gamma)} \right)^{d/2}
    \Big(
        \mfrac{(2 + \gamma)^2}{(1 + \gamma)(3 + \gamma)} \| \bmu \|_2^2
        + \| \bmu \|_2^4
    \Big)
    - \left( \mfrac{\gamma}{2 + \gamma} \right)^d \| \bmu \|_2^4
    \\
    \;=&\;
    \left( \mfrac{\gamma^2}{(1 + \gamma)(3 + \gamma)} \right)^{d/2}
    \bigg(
        \mfrac{(2 + \gamma)^2}{(1 + \gamma)(3 + \gamma)} \| \bmu \|_2^2
        + \left( 1 - \left( \mfrac{(1 + \gamma)(3 + \gamma)}{(2 + \gamma)^2} \right)^{d/2} \right) \| \bmu \|_2^4
    \bigg)
    \;.
\end{align*}

\subsubsection{Proof for $\sigfull^2$}
For simplicity, we define $\bZ \coloneqq \bX - \bmu$ and $\bZ' \coloneqq \bX'z - \bmu$ so that $\bZ, \bZ'$ are independent copies from $\cN(0, I_d)$. By \eqref{eq:uPksd:gaussian:simplified}, the second moment can be simplified as
\begin{align*}
    &\;\E[ \uPksd(\bX, \bX')^2 ]
    \\
    \;=&\;
    \E\left[ 
        \exp\left( - \mfrac{1}{\gamma} \| \bX - \bX' \|_2^2  \right)
        \Big( 
            \bX^\top \bX' 
            - \mfrac{\gamma + 1}{\gamma^2} \| \bX - \bX' \|_2^2
            + \mfrac{d}{\gamma}
        \Big)^2
    \right]
    \\
    \;=&\;
    \E\left[ 
        \exp\left( - \mfrac{1}{\gamma} \| \bZ - \bZ' \|_2^2  \right)
        \Big( 
            (\bZ + \bmu)^\top (\bZ' + \bmu) 
            - \mfrac{\gamma + 1}{\gamma^2} \| \bZ - \bZ' \|_2^2
            + \mfrac{d}{\gamma}
        \Big)^2
    \right]
    \\
    \;=&\;
    \left( \mfrac{\gamma}{4 + \gamma} \right)^{d/2}
    \underbrace{
    \E\left[
        \Big( 
            (\bW + \bmu)^\top (\bW' + \alpha_1 \bW + \bmu) 
            - \alpha_2 \| (1-\alpha_1) \bW - \bW' \|_2^2
            + \mfrac{d}{\gamma}
        \Big)^2
    \right]
    }_{=: T}
    \;,
\end{align*}
where in the last line we have applied Lemma \ref{lem:gauss:integral:double}, and 
\begin{align*}
    \bW \sim \cN\left(0,\; \mfrac{1 + \gamma / 2}{2 + \gamma / 2} I_d \right)
    \;, 
    \qquad
    \bW' \sim \cN\left( 0, \; \mfrac{\gamma}{2 + \gamma } I_d \right)
    \;,
    \qquad 
    \alpha_1 \;\coloneqq\; \mfrac{1}{1 + \gamma / 2} \;,
    \qquad 
    \alpha_2 \;\coloneqq\; \mfrac{\gamma + 1}{\gamma^2}\;.
\end{align*}
We now aim to compute the expectation $T$ by first taking an expectation over $\bW'$:
\begin{align*}
    T
    \;=&\;
    \E\left[
    \left(
        (\bW + \bmu)^\top (\bW' + \alpha_1 \bW + \bmu)
        - \alpha_2 \| (1-\alpha_1) \bW - \bW' \|_2^2
        + \mfrac{d}{\gamma}
    \right)^2
    \right]
    \\
    \;=&\;
    \E\bigg[
    \Big(
        - \alpha_2 \big\| \bW' \big\|_2^2
        + \big(\bW + \bmu + 2\alpha_2(1-\alpha_1) \bW \big)^\top \bW'
        \\
        &\quad\quad
        + \underbrace{(\bW + \bmu)^\top (\alpha_1 \bW + \bmu)
        - \alpha_2 (1-\alpha_1)^2 \|  \bW \|_2^2
        + \mfrac{d}{\gamma}}_{=: \beta_\bW}
    \Big)^2
    \bigg]
    \\
    \;=&\;
    \E\Big[
        \alpha_2^2 \big\| \bW' \big\|_2^4
        + \Big( (\bW + \bmu + 2\alpha_2 (1-\alpha_1)\bW )^\top \bW' \Big)^2
        + \beta_\bW^2
        \\
        &\quad\;\;
        - 2\alpha_2 \big(\bW + \mu + 2 \alpha_2(1-\alpha_1)\bW \big)^\top \bW' \big\| \bW' \|_2^2
        - 2 \alpha_2 \beta_\bW \big\| \bW' \big\|_2^2 
        \\
        &\quad\;\;
        + 2 \beta_\bW ( \bW + \bmu + 2\alpha_2(1-\alpha_1)\bW )^\top \bW'
    \Big]
    \;.
\end{align*}
Since $\bW'$ is zero-mean, independent of $\bW$ and follows a distribution symmetric around zero, $\mean[\bW']=\mean[\bW'\|\bW'\|_2^2] = \bzero$. Since $\|\bW'\|_2^2 \sim \mfrac{\gamma}{2+\gamma} \chi^2_d$ where $\chi^2_d$ is a chi-squared distribution with $d$ degrees of freedom, we have $\mean[\|\bW'\|_2^2] = \frac{\gamma}{2+\gamma} d$ and $\mean[\|\bW'\|_2^4] = \frac{\gamma^2}{(2+\gamma)^2} (2d + d^2)$. We also have $\mean[ \bW' \bW'^\top] = \frac{\gamma}{2+\gamma} \bI_d$. Thus
\begin{align*}
    T
    \;=&\;
    \E\left[
        \alpha_2^2 \big\| \bW' \big\|_2^4
        + \Big( (\bW + \bmu + 2\alpha_2 (1-\alpha_1)\bW )^\top \bW' \Big)^2
        + \beta_\bW^2
        - 2 \alpha_2 \beta_\bw \big\| \bW' \big\|_2^2
    \right]
    \\
    \;=&\;
    \E\left[
        \alpha_2^2 \mfrac{\gamma^2}{(2+\gamma)^2} (2d + d^2)
        +\mfrac{\gamma}{2+\gamma}\left\| 2 \alpha_2(1-\alpha_1)\bW + \bW + \bmu \right\|_2^2
        + \beta_\bW^2
        - 2 \alpha_2 \beta_\bW \mfrac{\gamma}{2+\gamma} d
    \right]
    \\
    \;\overset{(a)}{=}&\;
    \E\left[
        \mfrac{(1+\gamma)^2}{\gamma^2(2+\gamma)^2} (2d + d^2)
        +\frac{\gamma}{2+\gamma}\left\| \mfrac{2(1+\gamma)}{\gamma(2+\gamma)}\bW + \bW + \bmu \right\|_2^2
        + \beta_\bW^2
        - 2 \beta_\bW  \mfrac{1+\gamma}{\gamma(2+\gamma)} d
    \right]
    \\
    \;=&\;
    \E\left[
        \mfrac{(1+\gamma)^2}{\gamma^2(2+\gamma)^2} (2d + d^2)
        +\mfrac{\gamma}{2+\gamma}\left\| \mfrac{2+4\gamma+\gamma^2}{\gamma(2+\gamma)}\bW + \bmu \right\|_2^2
        + \beta_\bW^2
        - 2 \beta_\bW  \mfrac{1+\gamma}{\gamma(2+\gamma)} d
    \right]
    \\
    \;\overset{(b)}{=}&\;
    \mfrac{(1+\gamma)^2}{\gamma^2(2+\gamma)^2} (2d + d^2)
    +
    \mfrac{(2+4\gamma+\gamma^2)^2}{\gamma(2+\gamma)^2(4+\gamma)}
    d
    +
    \mfrac{\gamma}{2+\gamma} \|\mu\|_2^2
    +
    \mean\Big[ \beta_\bW^2
        - 2 \beta_\bW  \mfrac{1+\gamma}{\gamma(2+\gamma)} d \Big] 
    \\
    \;=&\;
    \mfrac{(1+\gamma)^2}{\gamma^2(2+\gamma)^2} d^2 
    +
    \Big(\mfrac{2 (1+\gamma)^2}{\gamma^2(2+\gamma)^2}
    +
    \mfrac{(2+4\gamma+\gamma^2)^2}{\gamma(2+\gamma)^2(4+\gamma)}
    \Big) d
    +
    \mfrac{\gamma}{2+\gamma} \|\mu\|_2^2
    +
    \mean\Big[ \beta_\bW^2
        - 2 \beta_\bW  \mfrac{1+\gamma}{\gamma(2+\gamma)} d \Big] 
    \;.
\end{align*}
In $(a)$, we have substituted in $\alpha_1=2/(2+\gamma)$ and $\alpha_2 =(\gamma+1)/\gamma^2$, and in $(b)$ we have taken the expectation of the second term. Now re-express $\beta_\bW$ as
\begin{align*}
    \bbeta_\bW
    \;=&\;
    \mfrac{2}{2 + \gamma} \left\| \bW \right\|_2^2
    + \left( \mfrac{2}{2 + \gamma} + 1 \right) \bmu^\top \bW
    + \| \bmu \|_2^2
    - \mfrac{\gamma+1}{\gamma^2} \left( 1 - \mfrac{2}{2 + \gamma}  \right)^2 \left\| \bW \right\|_2^2
    + \mfrac{d}{\gamma}
    \\
    \;=&\;
    \left( \mfrac{2}{2 + \gamma} - \mfrac{\gamma+1}{(2+\gamma)^2} \right) \left\| \bW \right\|_2^2
    + \mfrac{4 + \gamma}{2 + \gamma} \bmu^\top \bW
    + \| \bmu \|_2^2
    + \mfrac{d}{\gamma}
    \\
    \;=&\;
    \mfrac{\gamma + 3}{(2+\gamma)^2} \left\| \bW \right\|_2^2
    + \mfrac{4 + \gamma}{2 + \gamma} \bmu^\top \bW
    + \| \bmu \|_2^2
    + \mfrac{d}{\gamma}
    \;.
\end{align*}
By noting that odd moments of $\bW$ vanish, we get that
\begin{align*}
    \mean[  - 2 \beta_\bW  \mfrac{1+\gamma}{\gamma(2+\gamma)} d ]
    \;=&\;
    - \mfrac{2(1+\gamma)}{\gamma(2+\gamma)} d 
    \Big(
    \mfrac{\gamma + 3}{(2+\gamma)(4+\gamma)} d
    + \| \bmu \|_2^2
    + \mfrac{d}{\gamma} \Big)
    \\
    \;=&\;
    - \mfrac{2(1+\gamma)}{\gamma(2+\gamma)} \Big( \mfrac{\gamma + 3}{(2+\gamma)(4+\gamma)}  + \mfrac{1}{\gamma} \Big) d^2 - \mfrac{2(1+\gamma)}{\gamma(2+\gamma)} d \| \bmu \|_2^2
    \;,
\end{align*}
and
\begin{align*}
    \mean[ \beta_\bW^2 ]
    \;=&\;
    \mfrac{(\gamma + 3)^2}{(2+\gamma)^2(4+\gamma)^2} (2d+d^2)
    +
    \mfrac{4+\gamma}{2+\gamma} \|\mu\|_2^2
    +
    \Big(\|\mu\|_2^2
    +
    \mfrac{d}{\gamma}\Big)^2
    +
    \mfrac{2(\gamma + 3)}{(2+\gamma)(4+\gamma)}  \Big(\|\mu\|_2^2
    +
    \mfrac{d}{\gamma}\Big) d
    \\
    \;=&\; 
    \Big( 
    \mfrac{(\gamma + 3)^2}{(2+\gamma)^2(4+\gamma)^2}
    +
    \mfrac{1}{\gamma^2}
    +
     \mfrac{2(\gamma + 3)}{(2+\gamma)(4+\gamma)\gamma} 
    \Big)
    d^2
    +
    \Big(
    \mfrac{2(\gamma + 3)^2}{(2+\gamma)^2(4+\gamma)^2}
    \Big) d
    \\
    &\;
    + 
    \Big( \mfrac{2}{\gamma} + \mfrac{2(\gamma + 3)}{(2+\gamma)(4+\gamma)} \Big) d \|\mu\|_2^2
    + 
    \mfrac{4+\gamma}{2+\gamma} 
    \|\mu\|_2^2
    +
    \|\mu\|_2^4
    \;.
\end{align*}
The coefficient of $d^2$ in $T$ can then be computed by noting $\gamma=\omega(1)$ as
\begin{align*}
    &\mfrac{(1+\gamma)^2}{\gamma^2(2+\gamma)^2} 
    - 
    \mfrac{2(1+\gamma)}{\gamma(2+\gamma)} \Big( \mfrac{\gamma + 3}{(2+\gamma)(4+\gamma)}  + \mfrac{1}{\gamma} \Big) 
    +
    \Big( 
    \mfrac{(\gamma + 3)^2}{(2+\gamma)^2(4+\gamma)^2}
    +
    \mfrac{1}{\gamma^2}
    +
     \mfrac{2(\gamma + 3)}{(2+\gamma)(4+\gamma)\gamma} 
    \Big)
    \\
    &\;=\; \mfrac{(2+\gamma)^2}{\gamma^2(4+\gamma)^2}
    \;=\; \mfrac{1}{\gamma^2} + o\Big(\mfrac{1}{\gamma^2}\Big)\;.
\end{align*}
Similarly, the coefficient of $d$ in $T$ can be computed as
\begin{align*}
    \Big(\mfrac{2 (1+\gamma)^2}{\gamma^2(2+\gamma)^2}
    +
    \mfrac{(2+4\gamma+\gamma^2)^2}{\gamma(2+\gamma)^2(4+\gamma)}\Big)
    +
    \Big(
    \mfrac{2(\gamma + 3)^2}{(2+\gamma)^2(4+\gamma)^2}
    \Big)
    \;=&\; 1 + o(1)\;,
\end{align*}
the coefficient of $d\|\mu\|_2^2$ in $T$ can be computed as
\begin{align*}
    -
    \mfrac{2(1+\gamma)}{\gamma(2+\gamma)}
    +
    \Big( \mfrac{2}{\gamma} + \mfrac{2(\gamma + 3)}{(2+\gamma)(4+\gamma)} \Big)
    \;=\;
    \mfrac{2(2+\gamma)}{\gamma(4+\gamma)}
    \;=\;
    \mfrac{2}{\gamma} + o\Big( \mfrac{1}{\gamma}\Big)\;,
\end{align*}
the coefficient of $\|\mu\|_2^2$ in $T$ can be computed as
\begin{align*}
    \mfrac{\gamma}{2+\gamma} + \mfrac{4+\gamma}{2+\gamma}
    \;=\;
    2 + o(1)\;,
\end{align*}
and finally the coefficient of $\|\mu\|_2^4$ in $T$ is $1$. Combining the five computations of coefficients, we get that
\begin{align*}
    T \;=\; 4 d + \mfrac{d^2}{\gamma^2} + \mfrac{2d\|\mu\|_2^2}{\gamma} + 2\|\mu\|_2^2 + \|\mu\|_2^4 + o\Big( d + \mfrac{d^2}{\gamma^2} +  \mfrac{d\|\mu\|_2^2}{\gamma} +  + \|\mu\|_2^2 \Big)\;,
\end{align*}
and therefore the desired quantity is given as
\begin{align*}
    \sigfull^2
    \;=&\;
    \E[ \uPksd(\bX, \bX')^2 ] - \ksd(Q,P)^2
    \\
    \;=&\;
    \left( \mfrac{\gamma}{4 + \gamma} \right)^{d/2}
    T
    -
    \Big( \mfrac{\gamma}{2+\gamma}\Big)^{d} \|\mu\|_2^4
    \\
    \;=&\;
     \left( \mfrac{\gamma}{4 + \gamma} \right)^{d/2}
    T - \left( \mfrac{\gamma}{4 + \gamma} \right)^{d/2} \left( \mfrac{\gamma (4 + \gamma)}{(2 + \gamma)^2} \right)^{d/2} \| \bmu \|_2^4
    \\
    \;=&\;
    \left( \mfrac{\gamma}{4 + \gamma} \right)^{d/2}
    \bigg(
        d 
        + \mfrac{d^2}{\gamma^2}
        + \mfrac{2d \| \bmu \|_2^2}{\gamma}
        + 2\| \bmu \|_2^2
        + \left(1 - \left( \mfrac{\gamma (4 + \gamma)}{(2 + \gamma)^2} \right)^{d/2} \right) \| \bmu \|_2^4
        \\
        &\qquad\qquad\quad\;
        + o\left( 
            d 
            + \mfrac{d^2}{\gamma^2}
            + \mfrac{d \| \bmu \|_2^2}{\gamma}
            +  \| \bmu \|_2^2
        \right)
    \bigg)\;.
\end{align*}

\subsubsection{Proof for upper bound on $\E[ | \gksd(\bX) |^\nu ]$}
Fix $\nu > 2$. We can apply Lemma \ref{lem:gauss:integral:single} to rewrite the $\nu$-th moment of $\gksd(\bZ)$ as
\begin{align*}
    \E[ | \gksd (\bX) |^\nu ]
    \;=&\;
    \left( \mfrac{\gamma}{1 + \gamma} \right)^{\nu d / 2 }
    \E\left[ 
        \exp\left( - \mfrac{\nu}{2(1 + \gamma)} \| \bX - \bmu \|_2^2 \right)
        \left| 
            \mfrac{2 + \gamma}{1 + \gamma} \bmu^\top \bX 
            - \mfrac{1}{1 + \gamma} \| \bmu \|_2^2
        \right|^\nu
    \right]
    \\
    \;=&\;
    \left( \mfrac{\gamma}{1 + \gamma} \right)^{\nu d / 2 }
    \left( \mfrac{(1 + \gamma) / \nu}{1 + (1 + \gamma) / \nu} \right)^{d/2}
    \underbrace{\E\left[ 
        \left| 
            \mfrac{2 + \gamma}{1 + \gamma} \bmu^\top \bW
            - \mfrac{1}{1 + \gamma} \| \bmu \|_2^2
        \right|^\nu
    \right]}_{\eqqcolon T}
    \;,
\end{align*}
where $\bW \sim \cN\left( \bm, a^2 I_d \right)$ with $\bm \coloneqq \frac{(1 + \gamma) / \nu}{1 + (1 + \gamma) / \nu} \big( \bmu + \mfrac{\nu}{1 + \gamma} \bmu \big) + \bmu $ and $a^2 \coloneqq \frac{(1 + \gamma) / \nu}{1 + (1 + \gamma) / \nu} = \frac{1 + \gamma}{1 + \nu = \gamma}$. Defining $\bV \coloneqq W - \bmu$ so that $\bV \sim \cN(0, a^2 I_d)$, we have
\begin{align*}
    T
    \;=&\;
    \E\left[ 
        \left| 
            \mfrac{2 + \gamma}{1 + \gamma} \bmu^\top (\bV + \bmu)
            - \mfrac{1}{1 + \gamma} \| \bmu \|_2^2
        \right|^\nu
    \right]
    \;=\;
    \E\left[ 
        \left| 
            \mfrac{2 + \gamma}{1 + \gamma} \bmu^\top \bV 
            - \| \bmu \|_2^2
        \right|^\nu
    \right]
    \\
    \;\stackrel{(i)}{\leq}&\;
    2^{\nu - 1} \left(
        \left( \mfrac{2 + \gamma}{1 + \gamma} \right)^\nu \E[ | \bmu^\top \bV |^\nu ]
        + \| \bmu \|_2^{2\nu}
    \right)
    \\
    \;=&\;
    2^{\nu - 1} \left(
        C_\nu \left(\mfrac{2 + \gamma}{1 + \gamma} \right)^\nu  a^\nu \| \bmu \|_2^\nu
        + \| \bmu \|_2^{2\nu}
    \right)
    \;,
\end{align*}
where \emph{(i)} follows by the fact that $|u + v|^\nu \leq 2^{\nu - 1} (|u|^\nu + |v|^\nu)$ for any $u, v \in \R$, and in the last line we have computed the expectation by noting that $\bmu^\top \bV $ follows a univariate Gaussian distribution $\cN(0, a^2 \| \bmu \|_2^2)$ and using its moment formula to yield $\E[ | \bmu^\top \bV |^\nu \leq C_\nu a^\nu \| \bmu \|_2^\nu$ for some constant $C_\nu$ that depends only on $\nu$. 
Combining these and substituting the definition of $a^2$ gives
\begin{align*}
    &\;
    \E[ | \gksd (\bX) |^\nu ]
    \\
    \;\leq&\;
    2^{\nu - 1}
    \left( \mfrac{\gamma}{1 + \gamma} \right)^{\nu d / 2 }
    \left( \mfrac{(1 + \gamma) / \nu}{1 + (1 + \gamma) / \nu} \right)^{d/2}
    \left(
        C_\nu \left(\mfrac{2 + \gamma}{1 + \gamma} \right)^\nu \left( \mfrac{1 + \gamma}{1 + \nu + \gamma} \right)^{\nu/2} \| \bmu \|_2^\nu
        + \| \bmu \|_2^{2\nu}
    \right)
    \\
    \;\leq&\;
    \left( \mfrac{\gamma}{1 + \gamma} \right)^{\nu d / 2 }
    \left( \mfrac{1 + \gamma}{1 + \nu + \gamma} \right)^{d/2}
    \left(
        2^{3\nu / 2 - 1} C_\nu \| \bmu \|_2^\nu
        + 2^{\nu - 1}\| \bmu \|_2^{2\nu}
    \right)
    \;,
\end{align*}
where in the last line we have used the assumption that $\nu > 2$ to yield the inequality 
\begin{align*}
    \left( \mfrac{2 + \gamma}{1 + \gamma} \right)^\nu \left( \mfrac{1 + \gamma}{1 + \nu + \gamma} \right)^{\nu/2}
    \;=\;
    \left( \mfrac{2 + \gamma}{1 + \gamma} \right)^{\nu} \left( \mfrac{2 + \gamma}{1 + \nu + \gamma} \right)^{\nu/2}
    \;=\;
    2^{\nu / 2} \times 1
    \;=\;
    2^{\nu / 2} 
    \;.
\end{align*}
Defining the constants $C_1 \coloneqq 2^{3\nu / 2 - 1} C_\nu $ and $C_2 \coloneqq 2^{\nu - 1}$ completes the proof.

\subsubsection{Proof for upper bound on $\E[ | \uPksd(\bX, \bX') |^\nu ]$}
Fix $\nu > 2$. Define $\bZ \coloneqq \bX - \bmu$ and $\bZ' \coloneqq \bX' - \bmu$ so that $\bZ, \bZ'$ are independent draws from $\cN(0, I_d)$. Using \eqref{eqn:defn:ksd:rbf}, we can write the $\nu$-th central moment as
\begin{align*}
    &\;
    \E[ | \uPksd(\bX, \bX') |^\nu ]
    \\
    \;=&\;
    \E\left[ 
        \Big| \exp\left( - \mfrac{1}{2\gamma} \| \bX - \bX' \|_2^2  \right)
        \Big( 
            \bX^\top \bX' 
            - \mfrac{\gamma + 1}{\gamma^2} \| \bX - \bX' \|_2^2
            + \mfrac{d}{\gamma}
        \Big) \Big|^\nu
    \right]
    \\
    \;=&\;
    \E\left[ 
        \exp\left( - \mfrac{\nu}{2\gamma} \| \bZ - \bZ' \|_2^2  \right)
        \Big| 
            (\bZ + \bmu)^\top (\bZ' + \bmu) 
            - \mfrac{\gamma + 1}{\gamma^2} \| \bZ - \bZ' \|_2^2
            + \mfrac{d}{\gamma}
        \Big|^\nu
    \right]
    \\
    \;=&\;
    \left( \mfrac{\gamma/ \nu}{2 + \gamma / \nu} \right)^{d/2}
    \underbrace{
    \E\left[ 
        \Big| (\bW + \bmu)^\top (\bW' + (1-\alpha_2) \bW + \bmu)
        - \alpha_1 \| \bW - \bW' - (1-\alpha_2) \bW \|_2^2
        + \mfrac{d}{\gamma}
        \Big|^{\nu}
    \right]}_{\eqqcolon T}
    \tagaligneq \label{eq:ksd:mfull:pf:eq1}
    \;,
\end{align*}
where the last line follows by using Lemma \ref{lem:gauss:integral:double} and defining the quantities
$\alpha_1 \coloneqq \frac{\gamma + 1}{\gamma^2}$, $\alpha_2 \coloneqq \frac{\gamma}{\nu + \gamma}$, $\alpha_3 \coloneqq \frac{\gamma}{2\nu + \gamma}$, and
\begin{align*}  
    &\bW' \sim \cN\left( \bzero,\; \mfrac{\gamma / \nu}{1 + \gamma / \nu} I_d \right) 
    = \cN\left( \bzero,\; \alpha_2 I_d \right)
    \;,
    &&
    \bW \sim \cN\left( \bzero,\; \mfrac{\gamma / \nu}{2 + \gamma / \nu} I_d \right)
    = \cN\left( \bzero,\; \alpha_3 I_d \right)
    \;,
\end{align*}
while also noting that $1 - \alpha_2 =
    \frac{\gamma / \nu}{1 + \gamma / \nu} \times \frac{\nu}{\gamma} 
    =
    \frac{\nu}{\nu + \gamma} = 1 - \alpha_2$.
By a Jensen's inequality, we get that
\begin{align*}
    T
    \;=&\;
    \E\left[ 
        \Big| (\bW + \bmu)^\top (\bW' + (1 - \alpha_2)\bW + \bmu)
        - \alpha_1 \left\| \alpha_2 \bW - \bW' \right\|_2^2
        + \mfrac{d}{\gamma}
        \Big|^{\nu}
    \right]
    \\
    \;=&\;
    \E\left[
        \Big| 
            (\bW + \bmu)^\top \bW' 
            + (1 - \alpha_2) \| \bW + \bmu \|_2^2
            + \alpha_2 (\bW + \bmu)^\top \bmu
            - \alpha_1 \| \alpha_2 \bW - \bW' \|_2^2
            + \mfrac{d}{\gamma}
        \Big|^\nu
    \right]
    \\
    \;\leq&\;
    5^{\nu - 1}
    \E\Big[
        \big| (\bW + \bmu)^\top \bW' \big|^\nu
        +  |1 - \alpha_2|^\nu \| \bW + \bmu \|_2^{2\nu}
        + \alpha_2^\nu \big| (\bW + \bmu)^\top \bmu \big|^\nu
        \\
        &\qquad\quad
        + \alpha_1^\nu \| \alpha_2 \bW - \bW' \|_2^{2\nu}
        + \Big( \mfrac{d}{\gamma} \Big)^\nu
    \Big]
    \;,
\end{align*}
where the last line follows from a Jensen's inequality applied to the convex function $x \mapsto | x |^\nu$. 

\vspace{.2em}

We next seek to bound the expectation of each term individually. To bound $\E\big[ \big| \big(\bW + \bmu)^\top \bW' \big|^\nu \big]$, we note that $(\bW - \bmu)^\top \bW'$ conditioning on $\bW$ follows a normal distribution $\cN(0, \alpha_2 \|\bW - \bmu \|_2^2 )$. Hence, using the moment formula of univariate Gaussians, we have
\begin{align*}
    \E\big[ \big| \big(\bW + \bmu)^\top \bW' \big|^\nu \big]
    \;=&\;
    \E\big[
        \E\big[ 
            \big| \big(\bW + \bmu)^\top \bW' \big|^\nu
            \big| \bW
        \big]
    \big]
    \;=\;
    \E\big[
        C_\nu \alpha_2^\nu \| \bW + \bmu \|_2^\nu
    \big]
    \;,
\end{align*}
for some constant $C_\nu$ constant depending only on $\nu$. By the convexity of the function $\bx \mapsto \| \bx \|_2^{\nu}$, we can bound the above term as
\begin{align*}
    \E\big[
        C_\nu \alpha_2^\nu \| \bW + \bmu \|_2^\nu
    \big]
    \;\leq&\;
    2^{\nu - 1} C_\nu \alpha_2^\nu 
    \big( \E\big[
        \| \bW \|_2^\nu
    \big]
    + \| \bmu \|_2^\nu
    \big)
    \\
    \;\stackrel{(i)}{\leq}&\;
    C_\nu \alpha_2^\nu \alpha_3^\nu  \big( \E\big[ \| \bW \|_2^{2\nu} \big] \big)^{1/2}
    + C_\nu \alpha_2^\nu \| \bmu \|_2^\nu
    \\
    \;\stackrel{(ii)}=&\;
    C_\nu \alpha_2^\nu \alpha_3^\nu (d^{\nu / 2} + o(d^{\nu / 2}))
    + C_\nu \alpha_2^\nu  \| \bmu \|_2^\nu
    \\
    \;\stackrel{(iii)}{\leq}&\;
    C_\nu d^{\nu / 2} + o( d^{\nu / 2} ) + C_\nu \| \bmu \|_2^{\nu}
    \;,
\end{align*}
where \emph{(i)} holds by a Jensen's inequality, in \emph{(ii)} we have noted that $\alpha_3^{-1} \| \bW \|_2^2$ follows a chi-squared distribution with $d$ degrees of freedom and used the formula for its $\nu$-th moment, and \emph{(iii)} follows since $\alpha_2 = \frac{\gamma}{\nu + \gamma} < 1$ and $\alpha_3 = \frac{\gamma}{2\nu + \gamma} < 1 $. The expectation of the second term can be bounded using a similar argument as
\begin{align*}
    \E\big[
        \| \bW + \bmu \|_2^{2\nu}
    \big]
    \;\leq\;
    2^{\nu - 1}\big(
        \E\big[ \| \bW \|_2^{2\nu} \big]
        + \| \bmu \|_2^{2\nu}
    \big)
    \;=&\;
    2^{\nu - 1} \big(
        \alpha_3^\nu d^{\nu} + o(d^\nu)
        + \| \bmu \|_2^{2\nu}
    \big)
    \\
    \;\leq&\;
    2^{\nu - 1} d^\nu + o(d^\nu) + 2^{\nu - 1} \| \bmu \|_2^{2\nu}
    \;.
\end{align*}
The expectation of the third term is
\begin{align*}
    \E\big[ 
        \big| (\bW + \bmu)^\top \bmu \big|^\nu
    \big]
    \;=\;
    \E\big[ 
        \big| \bW^\top \bmu + \| \bmu \|_2^2 \big|^\nu
    \big]
    \;\leq&\;
    2^{\nu - 1}
    \big(
    \E\big[ 
        \big| \bW^\top \bmu \big|^\nu
    \big]
    + \| \bmu \|_2^{2\nu}
    \big)
    \\
    \;=&\;
    2^{\nu - 1}\big(
        \alpha_3^{\nu/2} \| \bmu \|_2^\nu
        + \| \bmu \|_2^{2\nu}
    \big)
    \\
    \;\leq&\;
    2^{\nu - 1} \| \bmu \|_2^\nu
    + 2^{\nu - 1} \| \bmu \|_2^{2\nu}
    \;,
\end{align*}
where the second last line holds as $\bmu^\top \bW $ is a univariate Gaussian with zero-mean and variance $\alpha_3 \| \bmu \|_2^2$, and the last line holds again as $\alpha_3 \leq 1$. It then remains to bound $\E\big[ \| \alpha_2 \bW - \bW' \|_2^{2\nu} \big]$. Noting that $\alpha_2 \bW - \bW' \sim \cN(0, \alpha_2(\alpha_3 + 1)I_d)$, the random variable $\alpha_2^{-1}(\alpha_3 + 1)^{-1} \| \alpha_2 \bW - \bW' \|_2^{2}$ follows a chi-squared distribution with $d$ degrees of freedom. A similar argument as before gives
\begin{align*}
    \E\big[
        \| \alpha_2 \bW - \bW' \|_2^{2\nu}
    \big]
    \;\leq&\;
    \alpha_2^\nu(\alpha_3 + 1 )^{\nu} \big( d^\nu + o(d^\nu)) 
    \;\leq\;
    2^\nu d^\nu + o(d^\nu)
    \;,
\end{align*}
where in the last inequality we have used the fact that $\alpha_2(\alpha_3 + 1) < 2$. Combining these terms, we can bound $T$ as
\begin{align*}
    T
    \;\leq&\;
    5^{\nu - 1} \Big[
        \big( 
            C_\nu d^{\nu / 2} + o(d^{\nu / 2})
            + C_\nu \| \bmu \|_2^\nu
         \big)
        + | 1 - \alpha_2|^\nu 2^{\nu - 1} \big(
            d^{\nu} 
            + o(d^\nu)
            + \| \bmu \|_2^{2\nu}
        \big)
        \\
        &\;\qquad
        + 2^{\nu - 1}\big(
            \| \bmu \|_2^\nu
            + \| \bmu \|_2^{2\nu}
        \big)
        + \alpha_1^{\nu} 2^{\nu}  \big( d^\nu + o(d^\nu))
        + \Big( \mfrac{d}{\gamma} \Big)^\nu
    \Big]
    \;.
\end{align*}
To proceed, we note that $\alpha_1^\nu = \big(\frac{\gamma + 1}{\gamma^2} \big)^\nu = \big( \frac{1}{\gamma} + \frac{1}{\gamma^2} \big)^\nu = \frac{1}{\gamma^\nu} + o\big(\frac{1}{\gamma^\nu}\big)$ and that $(1 - \alpha_2)^\nu = \big( \frac{\nu}{\nu + \gamma} )^\nu = \frac{1}{\gamma^\nu} + o\big( \frac{1}{\gamma^\nu} \big)$, since $\gamma = \omega(1)$ by assumption. Therefore,
\begin{align*}
    | 1 - \alpha_2|^\nu 2^{\nu - 1} \big(
        d^{\nu} 
        + o(d^\nu)
        + \| \bmu \|_2^{2\nu}
    \big)
    \;=\;
    2^\nu \mfrac{d^\nu}{\gamma^\nu} 
    + 2^\nu \mfrac{\| \bmu \|_2^{2\nu}}{\gamma^\nu}
    + o\Big( 
        \mfrac{d^\nu}{\gamma^\nu} 
        + \mfrac{\| \bmu \|_2^{2\nu}}{\gamma^\nu} 
    \Big)
    \;,
\end{align*}
and 
\begin{align*}
    \alpha_1^{\nu} 2^{\nu}  \big( d^\nu + o(d^\nu))
    \;=\;
    \mfrac{d^\nu}{\gamma^\nu}
    + o\Big( \mfrac{d^\nu}{\gamma^\nu} \Big)
    \;.
\end{align*}
It then follows by grouping and rearranging that $T$ can be bounded as
\begin{align*}
    T
    \;\leq&\;
    5^{\nu - 1} \Big[
        C_\nu d^{\nu / 2} + o(d^{\nu / 2})
        + C_\nu \| \bmu \|_2^\nu
        + 2^\nu \mfrac{d^\nu}{\gamma^\nu} 
        + 2^\nu \mfrac{\| \bmu \|_2^{2\nu}}{\gamma^\nu}
        + o\Big( 
            \mfrac{d^\nu}{\gamma^\nu} 
            + \mfrac{\| \bmu \|_2^{2\nu}}{\gamma^\nu} 
        \Big)
        \\
        &\;\qquad
        + 2^{\nu - 1}\big(
            \| \bmu \|_2^\nu
            + \| \bmu \|_2^{2\nu}
        \big)
        + \mfrac{d^\nu}{\gamma^\nu}
        + o\Big( \mfrac{d^\nu}{\gamma^\nu} \Big)
        + \Big( \mfrac{d}{\gamma} \Big)^\nu
    \Big]
    \\
    \;=&\;
    5^{\nu - 1} \Big[
        C_\nu d^{\nu / 2} 
        + (2 + 2^\nu)\Big( \mfrac{d}{\gamma} \Big)^\nu
        + (C_\nu + 2^{\nu - 1}) \| \bmu \|_2^\nu
        + 2^{\nu - 1} \| \bmu \|_2^{2\nu}
        \\
        &\;\qquad
        + o\Big( 
            d^{\nu / 2}
            + \mfrac{d^\nu}{\gamma^\nu} 
            + \mfrac{\| \bmu \|_2^{2\nu}}{\gamma^\nu} 
        \Big)
    \Big]
    \\
    \;=&\;
    C_3 d^{\nu / 2} 
    + C_4 \Big( \mfrac{d}{\gamma} \Big)^\nu
    + C_5 \| \bmu \|_2^\nu
    + C_6 \| \bmu \|_2^{2\nu}
    + o\Big( 
        d^{\nu / 2}
        + \mfrac{d^\nu}{\gamma^\nu} 
        + \mfrac{\| \bmu \|_2^{2\nu}}{\gamma^\nu} 
    \Big)
    \;,
\end{align*}
where in the last line we have redefined the constants: $C_3 \coloneqq 5^{\nu - 1} C_\nu$, $C_4 \coloneqq 5^{\nu - 1} (2 + 2^\nu)$, $C_5 \coloneqq 5^{\nu - 1} (C_\nu + 2^{\nu - 1})$ and $C_6 \coloneqq 10^{\nu - 1}$. The proof is finished by substituting this bound into \eqref{eq:ksd:mfull:pf:eq1} to yield
\begin{align*}
    \E[ | \uPksd(\bX, \bX') |^\nu ]
    \;\leq&\;
    \left( \mfrac{\gamma}{2\nu + \gamma} \right)^{d/2}
    \Big(
        C_3 d^{\nu / 2} 
        + C_4 \Big( \mfrac{d}{\gamma} \Big)^\nu
        + C_5 \| \bmu \|_2^\nu
        + C_6 \| \bmu \|_2^{2\nu}
        \\
        &\qquad\qquad\qquad
        + o\Big( 
            d^{\nu / 2}
            + \mfrac{d^\nu}{\gamma^\nu} 
            + \mfrac{\| \bmu \|_2^{2\nu}}{\gamma^\nu} 
        \Big)
    \Big)
    \;.
\end{align*}

\subsubsection{Proof for verifying \cref{assumption:moment:ratio}}

First note that when $\gamma=\Omega(d)$, for any fixed $a, b, c > 0$, we have that by a Taylor expansion,
\begin{align*}
    \Big( \mfrac{a+\gamma}{b+\gamma}\Big)^{d/c}
    \;=\;
    \Big( 1 + \mfrac{a-b}{b+\gamma}\Big)^{d/c}
    \;=&\;
    \exp\Big( \mfrac{d}{c} \log\Big( 1 + \mfrac{a-b}{b+\gamma}\Big) \Big)
    \;=\;
    \exp\Big( \mfrac{d(a-b)}{c(b+\gamma)} + o\Big( \mfrac{d}{\gamma^2} \Big) \Big)
    \\
    \;=&\;
    \exp\Big( \mfrac{d(a-b)}{c(b+\gamma)}\Big) \Big( 1 + o\Big( \mfrac{d}{\gamma^2} \Big) \Big)
    \;=\; \Theta(1)
    \;.
\end{align*}
Using this together with the assumption $\|\mu\|_2=\Theta(1)$ and the moment bounds in Lemma \ref{lem:ksd:moments:analytical}(iii)-(vi), we get that 
\begin{align*}
    \sigcond^2 
    \;=&\;
    \Theta\bigg( \left( \mfrac{\gamma^2}{(1 + \gamma)(3 + \gamma)} \right)^{d/2}
    \bigg(
    \mfrac{(2 + \gamma)^2}{(1 + \gamma)(3 + \gamma)}
    + \left( 1 - \left( \mfrac{(1 + \gamma)(3 + \gamma)}{(2 + \gamma)^2} \right)^{d/2} \right)  \bigg) \bigg)
    \;=\;
    \Theta(1)
    \;,
    \\
    \sigfull^2
    \;=&\;
    \Theta\bigg(
            \left( \mfrac{\gamma}{4 + \gamma} \right)^{d/2}
            \bigg(
                d 
                + \mfrac{d^2}{\gamma^2}
                + \mfrac{2d}{\gamma}
                + \left(1 - \left( \mfrac{\gamma (4 + \gamma)}{(2 + \gamma)^2} \right)^{d/2} \right) 
                \;
                 + o\left( 
                    d 
                    + \mfrac{d^2}{\gamma^2}
                    + \mfrac{d}{\gamma}
                \right)
            \bigg) \bigg)
    \\
    \;=&\; \Theta \Big(d 
                + \mfrac{d^2}{\gamma^2}
                + \mfrac{2d}{\gamma}\Big) \;=\; \Theta \Big(d 
                + \mfrac{d^2}{\gamma^2} \Big) \;,
\end{align*}
and for $\nu \in (2,3]$,
\begin{align*}
    \Mcondnu^\nu
    \;\leq&\;  \E[ | \gksd(\bX) |^\nu ]
    \;=\; O\Big( \left( \mfrac{\gamma}{1 + \gamma} \right)^{\nu d / 2 }
            \left( \mfrac{1 + \gamma}{1 + \nu + \gamma} \right)^{d/2} \Big) \;=\;O(1)\;,
    \\
    \Mfullnu^\nu
    \;\leq&\; \E[ | \uPksd(\bX, \bX') |^\nu ]
    \;=\;
    O
    \Big(
    \left( \mfrac{\gamma}{2\nu + \gamma} \right)^{d/2}
            \Big(
                 d^{\nu / 2} 
                + \Big( \mfrac{d}{\gamma} \Big)^\nu
            \Big)
    \Big) \;=\; O\Big(
                 d^{\nu / 2} 
                + \mfrac{d^\nu}{\gamma^\nu} 
            \Big)\;.
\end{align*}
This implies that
\begin{align*}
    \mfrac{\Mcondnu}{\sigcond} \;&=\; O(1)\;,
    &\text{ and }&&
    \mfrac{\Mfullnu}{\sigfull} \;&=\; 
    O\bigg( \Big( d^{1/2} + \mfrac{d}{\gamma} \Big)^{-1} \Big( d^{1/2} + \mfrac{d}{\gamma} \Big) \bigg)
    \;=\; O(1)\;.
\end{align*}
In other words, $\frac{\Mcondnu}{\sigcond}$ and $\frac{\Mfullnu}{\sigfull}$ are both bounded by finite, $d$-independent constants, which verifies \cref{assumption:moment:ratio}.

\subsection{Proof for Lemma \ref{lem:mmd:moments:analytical}}
\label{pf:mmd:moments:analytical}

\subsubsection{Proof for $\gmmd(\bz)$}
Recall that the MMD U-statistic is $\ummd(\bz, \bz')=\kappa(\bx, \bx') + \kappa(\by, \by') - \kappa(\bx, \by') - \kappa(\bx', \by)$ for $\bz \coloneqq (\bx, \by)$ and $\bz' = (\bx', \by')$. Taking expectation with respect to the second argument, we have
\begin{align*}
    &\gmmd(\bz)
    \;\coloneqq\; 
    \E[ \ummd(z, Z') ] 
    \;=\;
    \E[ \kappa(\bx, \bX') + \kappa(\by, \bY') - \kappa(\bx, \bY') - \kappa(\bX', \by) ] \\
    \;=&\;
    \E\Big[ 
    \exp\Big( - \mfrac{\| \bx - \bX'\|_2^2}{2\gamma} \Big) 
    + \exp\Big( - \mfrac{\|\by - \bY'\|_2^2}{2\gamma} \Big) 
    - \exp\Big( - \mfrac{\| \bx - \bY'\|_2^2}{2\gamma} \Big) 
    - \exp\Big( - \mfrac{\| \bX' - \by \|_2^2}{2\gamma} \Big) 
    \Big] \;.
\end{align*}
We can apply Lemma \ref{lem:gauss:integral:single} to compute each term. For example, setting $\ba_1 = 1, \ba_2 = \gamma, \bm_1 = \bmu$ and $\bm_2 = \bx$ in Lemma \ref{lem:gauss:integral:single}, the first term simplifies to
\begin{align*}
    \E\left[ 
    \exp\left( - \mfrac{\| \bx - \bX'\|_2^2}{2\gamma} \right)
    \right]
    \;=&\;
    \left( \mfrac{\gamma}{1 + \gamma} \right)^{d / 2} \exp\left( - \mfrac{1}{2(1 + \gamma)} \| \bx - \bmu \|_2^2 \right) \;.
\end{align*}
Computing similarly the other terms yields the desired result:
\begin{align*}
    \gmmd(\bz) 
    \;&=\;  \left( \mfrac{\gamma}{1 + \gamma} \right)^{d/2} 
            \Big[
            e^{ -\frac{1}{2 (1 + \gamma) } \| \bx - \bmu \|_2^2}
            + e^{ -\frac{1}{2 (1 + \gamma) } \| \by\|_2^2} 
            - e^{ -\frac{1}{2 (1 + \gamma) } \| \bx \|_2^2}
            - e^{ -\frac{1}{2 (1 + \gamma) } \| \by - \bmu \|_2^2}
            \Big] \;.
\end{align*}

\subsubsection{Proof for $\mmd(Q,P)$}
This is a special case of \citet[Proposition 1]{ramdas2015decreasing} with $\bmu_1 = 0$, $\bmu_2 = \bmu$ and $\Sigma = I_d$. Alternatively, applying Lemma \ref{lem:gauss:integral:single} to compute each term in $\E[ g(Z) ]$ yields the same result.

\subsubsection{Proof for $\sigcond^2$}
For $\bZ=(\bX,\bY)$, the second moment of $g(\bZ)$ is
\begin{align*}
    \E[ g(\bZ)^2 ]
    \;=&\;
    \left( \mfrac{\gamma}{1 + \gamma} \right)^{d}
    \E\bigg[ \bigg( e^{ -\frac{1}{2 (1 + \gamma) } \| \bx - \bmu \|_2^2}
            + e^{ -\frac{1}{2 (1 + \gamma) } \| \by\|_2^2} 
            - e^{ -\frac{1}{2 (1 + \gamma) } \| \bx \|_2^2}
            - e^{ -\frac{1}{2 (1 + \gamma) } \| \by - \bmu \|_2^2}
    \bigg)^2 \bigg] \\
    \;=&\;
    \left( \mfrac{\gamma}{1 + \gamma} \right)^{d}
    \E\bigg[ 
    \exp\Big( -\mfrac{ \| \bX  - \bmu \|_2^2}{1 + \gamma} \Big) 
    + \exp\Big( -\mfrac{\| \bY \|_2^2}{1 + \gamma }  \Big)
    + \exp\Big( -\mfrac{\| \bX \|_2^2}{1 + \gamma}  \Big) 
    \\
    &\;\qquad\qquad\quad
    + \exp\Big( -\mfrac{\| \bY  - \bmu \|_2^2}{1 + \gamma}  \Big) 
    + 2 \exp\Big( -\mfrac{\| \bX - \bmu \|_2^2}{2(1 + \gamma)}  \Big) \exp\Big( -\mfrac{\| \bY \|_2^2 }{2(1 + \gamma)} \Big) \\
    &\;\qquad\qquad\quad
    - 2 \exp\Big( - \mfrac{ \| \bX - \bmu \|_2^2 + \| \bX \|_2^2 }{2(1 + \gamma)} \Big) 
    - 2 \exp\Big( - \mfrac{\| \bX - \bmu \|_2^2}{2(1 + \gamma)} \Big) \exp\Big( - \mfrac{\| \bY - \bmu \|_2^2 }{2(1 + \gamma)} \Big) \\
    &\;\qquad\qquad\quad
    - 2 \exp\Big( - \mfrac{\| \bX \|_2^2 }{2(1 + \gamma)} \Big) \exp\Big( - \mfrac{\| \bY \|_2^2 }{2(1 + \gamma)} \Big)
    -2 \exp\Big( - \mfrac{\| \bY \|_2^2 + \| \bY - \bmu \|_2^2 }{2(1 + \gamma)} \Big) \\
    &\;\qquad\qquad\quad
    + 2 \exp\left( - \mfrac{1}{2(1 + \gamma)} \| \bX \|_2^2 \right) \exp\left( - \mfrac{1}{2(1 + \gamma)} \| \bY - \bmu \|_2^2 \right)
    \bigg] \;.
\end{align*}
We can compute each term by applying Lemma \ref{lem:gauss:integral:single}. Noting that $\bY - \bmu$ and $\bX$ are equal in distribution, and also that Lemma \ref{lem:gauss:integral:single} depends on $\bm_1$ and $\bm_2$ only through their difference, we have
\begin{align*}
    \E\left[ \exp\left( - \mfrac{\| \bX - \bmu \|_2^2}{1 + \gamma}  \right) \right] 
    \;=&\; 
    \E\left[ \exp\left( - \mfrac{ \| \bY \|_2^2}{1 + \gamma} \right) \right] 
    \;=\;
    \left( \mfrac{1 + \gamma}{3 + \gamma} \right)^{d / 2} \;,
    \\
    \E\left[ \exp\left(  - \mfrac{\| \bX \|_2^2 }{1 + \gamma} \right) \right]
    \;=&\;
    \E\left[ \exp\left(  -\mfrac{\| \bY - \bmu \|_2^2 }{1 + \gamma} \right) \right] 
    \;=\;
    \left( \mfrac{1 + \gamma}{3 + \gamma} \right)^{d / 2} \exp\left( - \mfrac{ \| \bmu \|_2^2}{3 + \gamma} \right) \;,
    \\
    \E\left[ \exp\left( -\mfrac{\| \bX - \bmu \|_2^2}{2(1 + \gamma)} \right)  
    \right]
    \;=&\;
    \E\left[ \exp\left( -\mfrac{\| \bY \|_2^2}{2(1 + \gamma)} \right) 
    \right]
    \;=\;
    \left( \mfrac{1 + \gamma}{2 + \gamma} \right)^{d/2} \;,
    \\
    \E\left[ \exp\left( -\mfrac{\| \bX \|_2^2}{2(1 + \gamma)} \right)  
    \right]
    \;=&\;
    \E\left[ \exp\left( -\mfrac{\| \bY - \bmu \|_2^2}{2(1 + \gamma)} \right) 
    \right]
    \;=\;
    \left( \mfrac{1 + \gamma}{2 + \gamma} \right)^{d/2} \exp\left( - \mfrac{\| \mu \|_2^2}{2 (2 + \gamma)} \right) \;.
\end{align*}
It remains to calculate the expectations of the sixth and ninth terms, which involve two differently centred quadratic forms of $\bX$ and $\bY$ respectively. The sixth term simplifies to 
\begin{align}
    \E\left[ 
        \exp\left( - \mfrac{\| \bX - \bmu \|_2^2 + \| \bX \|_2^2}{2(1 + \gamma)} \right)
    \right]
    \;=&\;
    \E\left[
        \exp\left( - \mfrac{1}{2(1 + \gamma)} \left( 2\| \bX \|_2^2 - 2\bmu^\top \bX + \| \bmu \|_2^2 \right) \right)
    \right] \nonumber\\
    \;=&\;
    \E\left[
        \exp\left( - \mfrac{1}{1 + \gamma} \left\| \bX - \mfrac{\bmu}{2} \right\|_2^2 \right) 
    \right] \exp\left(  - \mfrac{ \| \bmu\|_2^2}{4(1 + \gamma)} \right) \nonumber\\
    \;=&\;
    \left( \mfrac{1 + \gamma}{3 + \gamma} \right)^{d/2} \exp\left( - \mfrac{\| \bmu \|_2^2}{4(3 + \gamma)} \right) \exp\left( - \mfrac{\| \bmu \|_2^2}{4(1 + \gamma)} \right) \nonumber\\
    \;=&\;
    \left( \mfrac{1 + \gamma}{3 + \gamma} \right)^{d/2} \exp\left( - \mfrac{2 + \gamma}{(2(3 + \gamma)(1 + \gamma)} \| \bmu \|_2^2 \right) 
    \label{eq:mmd:cond:var:cross:term}
    \;,
\end{align}
and a similar calculation gives,
\begin{align*}
    \E\left[
        \exp\left( - \mfrac{\| \bY \|_2^2 + \| \bY - \bmu \|_2^2}{2(1 + \gamma)} \right)
     \right]
     \;=&\;
     \left( \mfrac{1 + \gamma}{3 + \gamma} \right)^{d/2} \exp\left( - \mfrac{2 + \gamma}{(2(3 + \gamma)(1 + \gamma)} \| \bmu \|_2^2 \right) \;.
\end{align*}
Combining the above identities yields
\begin{align*}
    \E[ g(Z)^2 ]
    \;=&\;
    \left( \mfrac{\gamma}{1 + \gamma} \right)^{d}
    \bigg(
        2\left( \mfrac{1 + \gamma}{3 + \gamma} \right)^{d / 2} 
        + 2\left( \mfrac{1 + \gamma}{3 + \gamma} \right)^{d / 2} \exp\left( - \mfrac{ \| \bmu \|_2^2}{3 + \gamma} \right) \\
        &\qquad\qquad\quad
        + 2\left( \mfrac{1 + \gamma}{2 + \gamma} \right)^{d}
        - 4\left( \mfrac{1 + \gamma}{3 + \gamma} \right)^{d/2} \exp\left( - \mfrac{2 + \gamma}{2(3 + \gamma)(1 + \gamma)} \| \bmu \|_2^2 \right) \\
        &\qquad\qquad\quad
        - 4\left( \mfrac{1 + \gamma}{2 + \gamma} \right)^{d} \exp\left( - \mfrac{\| \mu \|_2^2}{2 (2 + \gamma)} \right)
        + 2 \left( \mfrac{1 + \gamma}{2 + \gamma} \right)^d \exp\left( - \mfrac{\| \bmu \|_2^2}{2 + \gamma} \right)
    \bigg) \\
    \;=&\;
    2 \left( \mfrac{\gamma}{1 + \gamma} \right)^{d/2} \left( \mfrac{\gamma}{3 + \gamma} \right)^{d/2}
    \\
    &\;\quad \times \bigg(
        1 
        + \exp\left( - \mfrac{ \| \bmu \|_2^2}{3 + \gamma} \right) 
        + \left( \mfrac{3 + \gamma}{2 + \gamma} \right)^{d/2} \left( \mfrac{1 + \gamma}{2 + \gamma} \right)^{d/2} - 2 \exp\left( - \mfrac{(2 + \gamma)\| \bmu \|_2^2}{2(3 + \gamma)(1 + \gamma)}  \right) \\
        &\;\qquad\;\;\;
        - 2 \left(\mfrac{1 + \gamma}{2 + \gamma} \right)^{d/2} \exp\left( - \mfrac{\| \mu \|_2^2}{2 (2 + \gamma)} \right) 
        + \left(\mfrac{1 + \gamma}{2 + \gamma} \right)^{d/2} \exp\left( - \mfrac{\| \bmu \|_2^2}{2 + \gamma} \right)
    \bigg) \;.
\end{align*}
By noting that $\mmd(Q, P)^2 = 4 \big( \frac{\gamma}{2 + \gamma} \big)^{d} \Big( 1 - \exp\big( -\frac{\| \bmu \|_2^2 }{2 (2 + \gamma) } \big) \Big)^2 $, we hence obtain
\begin{align*}
    \sigcond^2 
    \;=&\;
    \E[ g(Z)^2 ] - \mmd(Q, P)^2 
    \;=\;
    \E[ g(Z)^2 ] - 4 \left( \mfrac{\gamma}{2 + \gamma} \right)^{d} \left[ 1 - \exp\left( -\mfrac{\| \bmu \|_2^2 }{2 (2 + \gamma) } \right) \right]^2
    \\
    \;=&\;
    2 \left( \mfrac{\gamma}{1 + \gamma} \right)^{d/2} \left( \mfrac{\gamma}{3 + \gamma} \right)^{d/2}  \\
    &\; \times
    \bigg(
        1 + \exp\left( -\mfrac{\| \bmu \|_2^2 }{3 + \gamma} \right) 
        + 2 \left( \mfrac{3 + \gamma}{2 + \gamma} \right)^{d/2} \left( \mfrac{1 + \gamma}{2 + \gamma} \right)^{d/2} \exp\left( - \mfrac{\| \bmu \|_2^2 }{2(2 + \gamma)} \right) \\
        &\qquad 
        - 2 \exp\left( - \mfrac{7 + 5\gamma}{4 (1 + \gamma)(3 + \gamma)} \| \mu \|_2^2 \right) 
        - \left( \mfrac{3 + \gamma}{2 + \gamma} \right)^{d/2} \left( \mfrac{1 + \gamma}{2 + \gamma} \right)^{d/2} \\
        &\qquad 
        - \left( \mfrac{3 + \gamma}{2 + \gamma} \right)^{d/2} \left( \mfrac{1 + \gamma}{2 + \gamma} \right)^{d/2} \exp\left( - \mfrac{\| \bmu \|_2^2 }{2 + \gamma} \right)
    \bigg) \;,
\end{align*}
as required.

\subsubsection{Proof for $\sigfull^2$}
The second moment is
\begin{align*}
    &\;
    \E[\ummd(Z, Z')^2] \\
    \;=&\;
    \E\bigg[
        \exp\left( - \mfrac{\| \bX - \bX'\|_2^2}{\gamma} \right)
        + \exp\left( - \mfrac{\| \bY - \bY' \|_2^2}{\gamma} \right) 
        + \exp\left( - \mfrac{\| \bX - \bY' \|_2^2}{\gamma} \right)
        + \exp\left( - \mfrac{\| \bX' - \bY \|_2^2}{\gamma} \right) \\
        &\quad\;
        + 2 \exp\left( - \mfrac{\| \bX - \bX' \|_2^2}{2\gamma} - \mfrac{\| \bY - \bY' \|_2^2}{2\gamma} \right)
        - 2 \exp\left( - \mfrac{\| \bX - \bX' \|_2^2}{2\gamma} - \mfrac{\| \bX - \bY' \|_2^2}{2\gamma} \right) \\
        &\quad\;
        - 2 \exp\left( - \mfrac{\| \bX - \bX' \|_2^2}{2\gamma} - \mfrac{\| \bX' - \bY \|_2^2}{2\gamma} \right)
        - 2 \exp\left( - \mfrac{\| \bY - \bY' \|_2^2}{2\gamma} - \mfrac{\| \bX - \bY' \|_2^2}{2\gamma} \right) \\
        &\quad\;
        - 2 \exp\left( - \mfrac{\| \bY - \bY' \|_2^2}{2\gamma} - \mfrac{\| \bX' - \bY \|_2^2}{2\gamma} \right)
        + 2 \exp\left( - \mfrac{\| \bX - \bY' \|_2^2}{2\gamma} - \mfrac{\| \bX' - \bY \|_2^2}{2\gamma} \right)
    \bigg] \\
    \;=&\;
    \E\bigg[
        \exp\left( - \mfrac{\| \bX - \bX'\|_2^2}{\gamma} \right)
        + \exp\left( - \mfrac{\| \bY - \bY' \|_2^2}{\gamma} \right) 
    \bigg]
    + 2\E\bigg[
        \exp\left( - \mfrac{\| \bX - \bY' \|_2^2}{\gamma} \right)
    \bigg] \\
    &\;
    + 2\E\bigg[
        \exp\left( - \mfrac{\| \bX - \bX' \|_2^2}{2\gamma} - \mfrac{\| \bY - \bY' \|_2^2}{2\gamma} \right)
    \bigg]
    - 4\E\bigg[ 
        \exp\left( - \mfrac{\| \bX - \bX' \|_2^2}{2\gamma} - \mfrac{\| \bX - \bY' \|_2^2}{2\gamma} \right) 
    \bigg] \\
    &\;
    - 4\E\bigg[
        \exp\left( - \mfrac{\| \bY - \bY' \|_2^2}{2\gamma} - \mfrac{\| \bX - \bY' \|_2^2}{2\gamma} \right)
    \bigg]
    + 2\left( \E\bigg[
        \exp\left( - \mfrac{\| \bX - \bY' \|_2^2}{2\gamma} \right)
    \bigg] \right)^2 \;,
\end{align*}
where the second equality follows by the fact that $X, X'$ and $Y, Y'$ are respectively independent copies from $Q$ and $P$. To calculate each term, we apply Lemma \ref{lem:gauss:integral:double} to yield
\begin{align*}
    &\E\bigg[
        \exp\left( - \mfrac{\| \bX - \bX'\|_2^2}{\gamma} \right)
    \bigg]
    \;=\;
    \E\bigg[
        \exp\left( - \mfrac{\| \bY - \bY' \|_2^2}{\gamma} \right) 
    \bigg]
    \;=\;
    \left( \mfrac{\gamma/2}{2 + \gamma / 2} \right)^{d/2}
    \;=\;
    \left( \mfrac{\gamma}{4 + \gamma} \right)^{d/2} 
    \;,
    \\
    &\E\bigg[
        \exp\left( - \mfrac{\| \bX - \bX'\|_2^2}{2\gamma} \right)
    \bigg]
    \;=\;
    \E\bigg[
        \exp\left( - \mfrac{\| \bY - \bY' \|_2^2}{2\gamma} \right) 
    \bigg]
    \;=\;
    \left( \mfrac{\gamma}{2 + \gamma} \right)^{d/2}
    \;,
    \\
    &\E\bigg[
        e^{ - \frac{\| \bX - \bY' \|_2^2}{\gamma} }
    \bigg]
    \;=\;
    \left( \mfrac{\gamma}{4 + \gamma} \right)^{d/2} \exp\left( - \mfrac{\| \bmu \|_2^2}{4 + \gamma} \right) \;,
    \;\;
    \E\bigg[
        e^{- \frac{\| \bX - \bY' \|_2^2}{2\gamma}}
    \bigg]
    \;=\;
    \left( \mfrac{\gamma}{2 + \gamma} \right)^{d/2} \exp\left( - \mfrac{\| \bmu \|_2^2}{2(2 + \gamma)} \right) \;.
\end{align*}
Let $T_1\coloneqq\E\Big[ 
    \exp\Big( - \frac{\| \bX - \bX' \|_2^2}{2\gamma} - \frac{\| \bX - \bY' \|_2^2}{2\gamma} \Big) 
\Big]
$ and $T_2\coloneqq\E\Big[
    \exp\Big( - \frac{\| \bY - \bY' \|_2^2}{2\gamma} - \frac{\| \bX - \bY' \|_2^2}{2\gamma} \Big) 
\Big]$, which are the only remaining terms to compute. The first term can be simplified as
\begin{align*}
    T_1
    \;=&\;
    \E\bigg[ 
        \E\bigg[
            \exp\left( - \mfrac{\| \bX - \bX' \|_2^2}{2\gamma} \right) \Big| \bX
        \bigg]
        \E\bigg[
            \exp\left( - \mfrac{\| \bX - \bY' \|_2^2}{2\gamma} \right) \Big| \bX
        \bigg]
    \bigg] 
    \\
    \;\overset{(a)}{=}&\;
    \E\left[
        \left( \mfrac{\gamma}{1 + \gamma} \right)^{d/2} \exp\left( - \mfrac{\| \bX \|_2^2}{2 (1 + \gamma)} \right)
        \times
        \left( \mfrac{\gamma}{1 + \gamma} \right)^{d/2} \exp\left( - \mfrac{\| \bX - \bmu \|_2^2}{2 (1 + \gamma)} \right)
    \right]
    \\
    \;=&\;
    \left( \mfrac{\gamma}{1 + \gamma} \right)^{d}
    \E\left[
         \exp\left( - \mfrac{\| \bX \|_2^2}{2 (1 + \gamma)} \right)
         \exp\left( - \mfrac{\| \bX - \bmu \|_2^2}{2 (1 + \gamma)} \right)
    \right]
    \\
    \;\overset{(b)}{=}&\;
    \left( \mfrac{\gamma}{1 + \gamma} \right)^{d}
    \left( \mfrac{1 + \gamma}{3 + \gamma} \right)^{d/2} \exp\left( - \mfrac{2 + \gamma}{2(3 + \gamma)(1 + \gamma)} \| \bmu \|_2^2 \right)
    \\
    \;=&\;
    \left( \mfrac{\gamma}{1 + \gamma} \right)^{d/2}
    \left( \mfrac{\gamma}{3 + \gamma} \right)^{d/2} \exp\left( - \mfrac{2 + \gamma}{2(3 + \gamma)(1 + \gamma)} \| \bmu \|_2^2 \right) \;,
\end{align*}
where in $(a)$ we have applied Lemma~\ref{lem:gauss:integral:single} to compute the conditional expectations, and $(b)$ follows from substituting \eqref{eq:mmd:cond:var:cross:term}. The second term can be simplified with a similar calculation as,
\begin{align*}
    T_2
    \;=&\;
    \left( \mfrac{\gamma}{1 + \gamma} \right)^{d/2}
    \left( \mfrac{\gamma}{3 + \gamma} \right)^{d/2} \exp\left( - \mfrac{2 + \gamma}{2(3 + \gamma)(1 + \gamma)} \| \bmu \|_2^2 \right) \;.
\end{align*}
Collecting all terms gives
\begin{align*}
    \E[\ummd(Z, Z')^2] 
    \;=&\;
    2\left( \mfrac{\gamma}{4 + \gamma} \right)^{d/2}
    + 2\left( \mfrac{\gamma}{4 + \gamma} \right)^{d/2} \exp\left(- \mfrac{\| \bmu \|_2^2}{4 + \gamma} \right) 
    + 2 \left( \mfrac{\gamma}{2 + \gamma} \right)^d 
    \\
    &\;
    - 8 \left( \mfrac{\gamma}{1 + \gamma} \right)^{d/2}
    \left( \mfrac{\gamma}{3 + \gamma} \right)^{d/2} \exp\left( - \mfrac{2 + \gamma}{2(3 + \gamma)(1 + \gamma)} \| \bmu \|_2^2 \right)
    \\
    &\;
    + 2 \left( \mfrac{\gamma}{2 + \gamma} \right)^d \exp\left( - \mfrac{\| \mu\|_2^2}{2 + \gamma} \right) \;.
\end{align*}
By noting that 
\begin{align*}
    \mmd(Q, P)^2 
    \;=&\; 
    4 \Big( \mfrac{\gamma}{2 + \gamma} \Big)^{d} \Big( 1 - \exp\big( -\mfrac{\| \bmu \|_2^2 }{2 (2 + \gamma) } \big) \Big)^2
    \\
    \;=&\;
    4 \Big( \mfrac{\gamma}{2 + \gamma} \Big)^{d} \Big( 
        1
        - \exp\big( -\mfrac{\| \bmu \|_2^2 }{2 + \gamma } \big)
        + 2\exp\big( -\mfrac{\| \bmu \|_2^2 }{2 (2 + \gamma) } \big)
    \Big)
    \;,
\end{align*}
the variance takes the following form after subtracting $\mmd(Q, P)^2$ and collecting similar terms
\begin{align*}
    \sigfull^2
    \;=&\;
    \E[\ummd(Z, Z')^2] - \mmd(Q, P)^2
    \\
    \;=&\;
    2\Big( \mfrac{\gamma}{4 + \gamma} \Big)^{d/2} \Big(
    1 + \exp\Big(- \mfrac{\| \bmu \|_2^2}{4 + \gamma} \Big) \Big) 
    - 2 \Big( \mfrac{\gamma}{2 + \gamma} \Big)^d 
    \\
    &\;
    - 8 \Big( \mfrac{\gamma}{1 + \gamma} \Big)^{d/2}
    \Big( \mfrac{\gamma}{3 + \gamma} \Big)^{d/2} \exp\Big( - \mfrac{2 + \gamma}{2(3 + \gamma)(1 + \gamma)} \| \bmu \|_2^2 \Big)
    \\
    &\;
    - 2 \Big( \mfrac{\gamma}{2 + \gamma} \Big)^d \exp\Big( - \mfrac{\| \mu\|_2^2}{2 + \gamma} \Big) 
    + 8 \Big( \mfrac{\gamma}{2 + \gamma} \Big)^d \exp\Big( - \mfrac{\| \bmu \|_2^2}{2(2 + \gamma)} \Big) \;,
\end{align*}
which completes the proof.

\subsection{Proof of Lemma~\ref{lem:linear:mmd:moments:analytical}}
\label{pf:linear:mmd:moments:analytical}
With a linear kernel and under the stated assumption, the MMD statistic is
\begin{align*}
    \ummd( \bz, \bz' )
    \;=\;
    \bx^\top \bx' + \by^\top \by' - \bx^\top \by' - \by^\top \bx'
    \;,
    \quad 
    \text{ where }\;\;
    \bz = (\bx, \by), \bz' = (\bx', \by') \in \R^{2d}\;.
\end{align*}

\subsubsection{Proof for $\gmmd(\bz)$ and $\mmd(Q,P)$}
The expression for $\gmmd$ can be computed as
\begin{align*}
    \gmmd(\bz)
    \;=&\;
    \mean[ \ummd(\bz, \bZ')]
    \;=\;
    \mean[ \bx^\top \bX' + \by^\top \bY' - \bx^\top \bY' - \by^\top \bX' ]
    \;=\;
    \bmu^\top \bx - \bmu^\top \by
    \;.
\end{align*}
The formula for $\mmd(Q,P)$ then follows as
\begin{align*}
    \mmd(Q,P) \;=\; \mean[ \ummd(\bZ, \bZ')]
    \;=\; \mean[\gmmd(\bZ)] \;=\; \mean[\mu^\top \bX - \mu^\top \bY] \;=\; \mu^\top \mu \;=\; \|\mu\|_2^2\;.
\end{align*}

\subsubsection{Proof for $\sigcond^2$}
A direct computation gives
\begin{align*}
    \E[ \gmmd(\bZ)^2 ]
    \;=&\;
    \E\Big[  
        (\bmu^\top \bX)^2
        + (\bmu^\top \bY)^2
        - 2 \bmu^\top \bX \bmu^\top \bY
    \Big]
    \\
    \;=&\;
    \bmu^\top (\Sigma + \bmu \bmu^\top) \bmu
    + \bmu^\top \Sigma \bmu
    \;=\;
    2 \bmu^\top \Sigma \bmu + \| \bmu \|_2^4
    \;.
\end{align*} 
Therefore, $\sigcond^2 = \E[ \gmmd(\bZ)^2 ] - \mmd(Q, P)^2 = 2 \bmu^\top \Sigma \bmu$, as required.

\subsubsection{Proof for $\sigfull^2$}
The second moment is
\begin{align*}
    \E[\ummd(\bZ, \bZ')^2]
    \;=&\;
    \E\big[
        ( \bX^\top \bX' )^2
        + ( \bY^\top \bY' )^2
        + ( \bX^\top \bY' )^2
        + ( \bY^\top \bX' )^2
        \\
        &\quad\;\;
        + 2 \bX^\top \bX' \bY^\top \bY'
        - 2 \bX^\top \bX' \bX^\top \bY'
        - 2 \bX^\top \bX' \bY^\top \bX'
        \\
        &\quad\;\;
        - 2 \bY^\top \bY' \bX^\top \bY'
        - 2 \bY^\top \bY \bY^\top \bX'
        + 2 \bX^\top \bY' \bY^\top \bX'
    \big]
    \\
    \;=&\;
    \E\big[
        ( \bX^\top \bX' )^2
        + ( \bY^\top \bY' )^2
        + ( \bX^\top \bY' )^2
        + ( \bY^\top \bX' )^2
    \big]
    \;.
\end{align*}
In the last equality, we have noted that the cross-terms vanish since $\bX, \bX', \bY$ and $\bY'$ are mutually independent and $\bX, \bX'$ are zero-mean. A direct computation gives
\begin{align*}
    \E\big[ ( \bX^\top \bX' )^2 \big]
    \;=&\;
    \Tr\big( \E[ \bX \bX^\top \bX' (\bX')^\top  ] \big)
    \;=\;
    \Tr(\Sigma^2)
    \;,
    \\
    \E\big[ ( \bY^\top \bY' )^2 \big]
    \;=&\;
    \Tr\big( \E[ \bY \bY^\top \bY' (\bY')^\top  ] \big)
    \;=\;
    \Tr\big( (\Sigma + \bmu \bmu^\top)^2 \big)
    \;=\;
    \Tr(\Sigma^2) + 2 \bmu^\top \Sigma \bmu^\top + \| \bmu \|_2^4
    \;,
    \\
    \E\big[ (\bX^\top \bY')^2 \big]
    \;=&\;
    \E\big[ (\bY^\top \bX')^2 \big]
    \;=\;
    \Tr( \E[ \bY (\bY')^\top \bX (\bX')^\top  ] )
    \;=\;
    \Tr(\Sigma^2) + \bmu^\top \Sigma \bmu
    \;.
\end{align*}
Therefore, $\E[ \ummd(\bZ, \bZ')^2 ] = 4 \Tr(\Sigma^2) + 4  \bmu^\top \Sigma \bmu + \| \bmu \|_2^4$, and
\begin{align*}
    \sigfull^2
    \;=\;
    \E[ \ummd(\bZ, \bZ')^2 ] - \mmd(Q, P)^2
    \;=\;
    4 \Tr(\Sigma^2) + 4  \bmu^\top \Sigma \bmu
    \;,
\end{align*}
which completes the proof.

\subsubsection{Proof for upper bound on $ \Mcondthree^3$}
The 3rd absolute centred moment of $\gmmd(\bZ)$ satisfies
\begin{align*}
    \Mcondthree^3
    =
    \E[ | \gmmd(\bZ) - \mean[\gmmd(\bZ)] |^3 ]
    =
    \E[ | \bmu^\top \bY - \bmu^\top \bX - \bmu^\top \bmu |^3 ]
    =&
    \E[ | \bmu^\top \bY - \bmu^\top \bV |^3 ]
    \;,
\end{align*}
where we have defined $\bV \coloneqq \bX - \bmu$ so that $\bV \sim \cN(0, \Sigma)$. Noting that $| a + b |^3 \leq 2^{3-1} (|a|^3 + |b|^3 )$ for any $a, b\in \R$ by Jensen's inequality, we can bound the above as
\begin{align*}
    \Mcondthree^3
    =
    \E\big[ \big| \bmu^\top \bY - \bmu^\top \bV  \big|^m  \big]
    \leq
    4 \E[ 
        | \bmu^\top \bY |^3 
        + | \bmu^\top \bV|^3 
    ]
    \overset{(a)}{=}
    8 C' (\bmu^\top \Sigma \bmu)^{3/2}
    \overset{(b)}{=}
    C (\bmu^\top \Sigma \bmu)^{3/2}
    \;.
\end{align*}
In $(a)$ we have noted that the absolute 3rd moment of a univariate normal variable $\cN(0,\sigma^2)$ is given as $C'\sigma^3$ for some absolute constant $C'$. In $(b)$, we have defined $C \coloneqq 8 C' $. 

\subsubsection{Proof for upper bound on $\Mfullthree^3$}
For any $\bz = (\bx, \by), \bz' = (\bx', \bz') \in \R^{2d}$ we have
\begin{align*}
    \ummd(\bz, \bz')
    \;=&\;
    \bx^\top \bx' + \by^\top \by - \bx^\top \by' - \by^\top \bx'
    \;=\;
    ( \bx - \by )^\top( \bx' - \by' )
    \;.
\end{align*}
Write $\bV \coloneqq \bX - \mu$ and $\bV' \coloneqq \bX - \mu$ so that $\bV, \bV' \overset{i.i.d.}{\sim} \cN(\bzero, \Sigma)$. We can compute the 3rd absolute central moment as
\begin{align*}
    \Mfullthree^3
    \;=&\;
    \E[ | \ummd(\bZ, \bZ') - \mean[\ummd(\bZ, \bZ') ]|^3 ]
    \\
    \;=&\;
    \E[ | (\bX - \bY)^\top (\bX' - \bY') - \mu^\top \mu |^3 ]
    \\
    \;=&\;
    \E[ | (\bV + \mu - \bY)^\top (\bV' + \mu - \bY') - \mu^\top \mu |^3 ]
    \\
    \;=&\;
    \E[ | (\bV - \bY)^\top (\bV' - \bY') + \mu^\top (\bV'-\bY') + (\bV-\bY)^\top \mu |^3 ]
    \;.
\end{align*}
By a Jensen's inequality applied to the convex function $x \mapsto |x|^3$ and a H\"older's inequality, we get that
\begin{align*}
    \Mfullthree^3
    \;\leq&\;
    9 \big( 
        \E \big[  | (\bV-\bY)^\top (\bV'-\bY') |^3 \big]
        + \E \big[ | \mu^\top (\bV'-\bY') |^3 \big]
        + \E \big[ | (\bV-\bY)^\top \mu |^3 \big]
    \big)
    \\
    \;\leq&\;
    9 \big( 
        \E \big[  | \bU^\top \bU' |^3 \big]
        + 2 \E \big[ | \bU^\top \mu |^3 \big]
    \big)
    \;.
\end{align*}
In the last line, we have used that $\bU \coloneqq \bV-\bY$ and $\bU'\coloneqq \bV'-\bY'$ are identically distributed. In fact they are both $\cN(\bzero, 2\Sigma)$. The second expectation can be computed by the formula for the absolute 3rd moment of a univariate Gaussian as
\begin{align*}
    \E \big[  | \bU^\top \mu |^3 \big]
    \;=\;
    C' (\mu^\top \Sigma \mu)^{3/2}\;
\end{align*}
where $C'$ is some absolute constant. Similarly the first expectation can be computed first by noting that $\bU^\top \bU'$ conditioning on $\bU$ is a univariate Gaussian and secondly by using the moment formula for a Gaussian quadratic form Lemma \ref{lem:quadratic:gaussian:moments}:
\begin{align*}
     \E \big[  | \bU^\top \bU' |^3 &\big]
     \;=\;
     \E \big[ \mean[  | \bU^\top \bU' |^3 | \bU] \big]
     \;=\;
     C' \E \big[ (\bU^\top \Sigma \bU)^{3/2} \big]
     \\
     \;\leq&\; C' \E \big[ (\bU^\top \Sigma \bU)^3 \big]^{1/2}
     =
     C' \big(\Tr(\Sigma^2)^3 + 6 \Tr(\Sigma^2)\Tr(\Sigma^4) + 8 \Tr(\Sigma^6) \big)
     \leq 15 C' \Tr(\Sigma^2)^3 \;.
\end{align*}
In the last line, we have noted that $\Tr(A^m) \leq \Tr(A)^m$ for $m \in \N$ and positive semi-definite matrix $A$, which holds by expressing each trace as a sum of eigenvalues and applying the H\"older's inequality. Combining the two computations and redefining constants, we get that for some constant $C$,
\begin{align*}
    \Mfullthree^3 \;\leq&\; 
    C \big(  \Tr(\Sigma^2)^3 + (\mu^\top \Sigma \mu)^{3/2} \big)
    \;\leq\; C \big( \Tr(\Sigma^2) + \mu^\top \Sigma \mu \big)^{3/2}\;.
\end{align*}

\subsubsection{Proof for verifying \cref{assumption:moment:ratio}}

By the bounds from \emph{(iii)}-\emph{(vi)}, there exists absolute constants $C_1, C_2$ such that 
\begin{align*}
    &
    \mfrac{\Mcondthree}{\sigcond} 
    \;\leq\; 
    \mfrac{ C_1^{1/3} (\mu^\top \Sigma \mu)^{1/2}}{2^{1/2} (\bmu^\top \Sigma \bmu)^{1/2}}
    \;=\; 2^{-1/2} C_1^{1/3}\;,
    &&
    \mfrac{\Mfullthree}{\sigfull} 
    \;\leq\; 
    \mfrac{ C_2^{1/3} \big(\Tr(\Sigma^2) + \mu^\top \Sigma \mu \big)^{1/2}}{2 \big(\Tr(\Sigma^2) + \mu^\top \Sigma \mu \big)^{1/2}}
    \;=\; 2^{-1} C_2^{1/3}\;,
\end{align*}
which prove that \cref{assumption:moment:ratio} holds with $\nu=3$.

\section{Proofs for \cref{appendix:tools}} \label{appendix:proof:tools}

\subsection{Proofs for \cref{appendix:moments}}

The proof of Lemma \ref{lem:martingale:bound} combines the following two results:

\begin{lemma}[Theorem 2, \citet{von1965inequalities}] \label{lem:von:bahn:esseen} Fix $\nu \in [1,2]$. For a martingale difference sequence $Y_1,\ldots,Y_n$ taking values in $\R$,
\begin{align*}
    \mean \big[ \big| \msum_{i=1}^n Y_i \big|^\nu \big] \;\leq\; 2 \msum_{i=1}^n \mean[ |Y_i |^\nu ]\;.
\end{align*}
\end{lemma}

\begin{lemma}[\citet{dharmadhikari1968bounds}] \label{lem:martingale:bound:two} Fix $\nu \geq 2$. For a martingale difference sequence $Y_1,\ldots,Y_n$ taking values in $\R$,
\begin{align*}
    \mean \big[ \big| \msum_{i=1}^n Y_i \big|^\nu \big] \;\leq\; C_{\nu} n^{\nu/2 - 1} \msum_{i=1}^n \mean[ |Y_i |^\nu ]\;,
\end{align*}
where $C_{\nu} =  (8(\nu-1)\max\{1,2^{\nu-3}\})^{\nu}$.
\end{lemma}

\begin{proof}[Proof of Lemma \ref{lem:martingale:bound}] We first consider the upper bound. For $\nu \in [1,2]$, the result follows directly from the Von Bahn-Esseen inequality as stated below in Lemma \ref{lem:von:bahn:esseen}, and for $\nu > 1$, the result follows directly from Lemma \ref{lem:martingale:bound:two}. As for the lower bound, by Theorem 9 of \citet{burkholder1966martingale}, there exists an absolute constant $c_\nu > 0$ depending only on $\nu$ such that
\begin{align*}
    \mean \big[ \big| \msum_{i=1}^n Y_i \big|^\nu \big] \;\geq\; c_\nu \, \mean \big[ \big( \msum_{i=1}^n Y_i^2 \big)^{\nu/2} \big] \;.
\end{align*}
For $\nu \in [1,2]$, by applying Jensen's inequality on the concave function $x \mapsto x^{\nu/2}$, we get that
\begin{align*}
    \mean \big[ \big| \msum_{i=1}^n Y_i \big|^\nu \big] 
    \;\geq&\; 
    c_\nu \, \mean \big[ \big( \mfrac{1}{n} \msum_{i=1}^n n Y_i^2 \big)^{\nu/2} \big]
    \;\geq\; 
    c_\nu \, n^{\nu/2 -1} \msum_{i=1}^n \mean[ |Y_i|^\nu]
    \;.
\end{align*}
For $\nu > 2$, by noting that $(a+b)^{\nu/2} \geq a^{\nu/2} + b^{\nu/2}$ for $a,b \geq 0$, we get that
\begin{align*}
    \mean \big[ \big| \msum_{i=1}^n Y_i \big|^\nu \big] 
    \;\geq&\; 
    c_\nu \, \mean \big[ \msum_{i=1}^n \big(  Y_i^2 \big)^{\nu/2} \big]
    \;\geq\; 
    c_\nu  \msum_{i=1}^n \mean[ |Y_i|^\nu]
    \;.
\end{align*}
Combining the two results above give the desired bound.
\end{proof}

\vspace{1em}

\begin{proof}[Proof of Lemma \ref{lem:gauss:integral:single}]
    A direct computation gives
    \begin{align*}
        &\;
        \E\left[ f(\bX) \exp\left( - \mfrac{1}{2a_2^2} \| \bX - \bm_2 \|_2^2 \right) \right] \\
        \;&=\;
        \mfrac{1}{(2\pi)^{d/2} a_1^{d}} \int f(\bx) 
        \exp\left( - \mfrac{1}{2a_2^2} \| \bx - \bm_2 \|_2^2 \right)
        \exp\left( - \mfrac{1}{2a_1^2} \| \bx - \bm_1 \|_2^2 \right) d\bx \\
        \;&=\;
        \mfrac{1}{(2\pi)^{d/2} a_1^{d}} \int f(\bx) \exp\bigg( \underbrace{
        - \mfrac{1}{2} \left( \mfrac{\|\bx\|_2^2}{a_2^2} + \mfrac{\|\bm_2\|_2^2}{a_2^2} - \mfrac{2\bm_2^\top \bx}{a_2^2} + \mfrac{\|\bx\|_2^2}{a_1^2} + \mfrac{\|\bm_1\|_2^2}{a_1^2} - \mfrac{2\bm_1^\top \bx}{a_1^2} \right)
        }_{=: T} \bigg) d\bx \;.
    \end{align*}
    Simplifying $T$ by completing the square yields
    \begin{align*}
        T
        \;=&\; 
        -\mfrac{1}{2} 
        \bigg( 
        \mfrac{\|\bx\|_2^2}{a_2^2} 
        + \mfrac{\|\bx\|_2^2}{a_1^2}
        - \mfrac{2\bm_2^\top \bx}{a_2^2}
        - \mfrac{2\bm_1^\top \bx}{a_1^2}
        \bigg) 
        - \mfrac{1}{2}\left( \mfrac{\|\bm_2\|_2^2}{a_2^2} + \mfrac{\|\bm_1\|_2^2}{a_1^2}\right)
        \\
        \;=&\; 
        -\mfrac{a_1^2 + a_2^2}{2 a_1^2 a_2^2} 
        \Bigg( \| \bx \|_2^2 - \mfrac{2a_1^2 a_2^2}{a_1^2 + a_2^2} \left( \mfrac{\bm_2}{a_2^2} + \mfrac{\bm_1}{a_1^2} \right)^\top \bx
        + \mfrac{a_1^4 a_2^4}{(a_1^2 + a_2^2)^2} \left\| \mfrac{\bm_2}{a_2^2} + \mfrac{\bm_1}{a_1^2} \right\|_2^2 
        \Bigg) \\
        &\; 
        - \mfrac{1}{2}
        \underbrace{\Bigg(
        \mfrac{\| \bm_2 \|_2^2}{a_2^2} 
        + \mfrac{\| \bm_1 \|_2^2}{a_1^2} 
        - \mfrac{a_1^2 a_2^2}{a_1^2 + a_2^2} \left\| \mfrac{\bm_2}{a_2^2} + \mfrac{\bm_1}{a_1^2} \right\|_2^2
        \Bigg)}_{\eqqcolon T'}
        \\
        \;=&\;
        -\mfrac{a_1^2 + a_2^2}{2 a_1^2 a_2^2} \left\| \bx - \mfrac{a_1^2 a_2^2}{a_1^2 + a_2^2} \left( \mfrac{\bm_1}{a_1^2} + \mfrac{\bm_2}{a_2^2} \right) \right\|_2^2
        - \mfrac{1}{2(a_1^2 + a_2^2)} \| \bm_1 - \bm_2 \|_2^2 \;,
    \end{align*}
    where we have simplified $T'$ as
    \begin{align*}
        T'
        \;=&\;
        \mfrac{\| \bm_2 \|_2^2}{a_2^2} 
        + \mfrac{\| \bm_1 \|_2^2}{a_1^2}
        - \mfrac{a_1^2}{a_2^2( a_1^2 + a_2^2)} \| \bm_2 \|_2^2
        - \mfrac{a_2^2}{a_1^2( a_1^2 + a_2^2)} \| \bm_1 \|_2^2
        + \mfrac{2}{a_1^2 + a_2^2} \bm_1^\top \bm_2
        \\
        \;=&\;
        \mfrac{1}{a_1^2 + a_2^2} \| \bm_1 - \bm_2 \|_2^2
        \;.
    \end{align*}
    Substituting this into $\E\left[ f(\bX) \exp\left( - \frac{1}{2a_2^2} \| \bX - \bm_2 \|_2^2 \right) \right]$, we have
    \begin{align*}
        &\;
        \E\left[ f(\bX) \exp\left( - \mfrac{1}{2a_2^2} \| \bX - \bm_2 \|_2^2 \right) \right] \\
        \;=&\;
        \mfrac{1}{(2 \pi)^{d/2} a_1^d}
        \exp\left( - \mfrac{ \| \bm_1 - \bm_2 \|_2^2}{2 (a_1^2 + a_2^2)} \right)
        \int f(x)
        \exp\left( -\mfrac{a_1^2 + a_2^2}{2 a_1^2 a_2^2} \left\| \bx - \mfrac{a_1^2 a_2^2}{a_1^2 + a_2^2} \left( \mfrac{\bm_1}{a_1^2} + \mfrac{\bm_2}{a_2^2} \right) \right\|_2^2 \right)
        dx\\
        \;=&\;
        \left( \mfrac{a_2^2}{a_1^2 + a_2^2} \right)^{d/2}
        \exp\left( - \mfrac{\| \bm_1 - \bm_2 \|_2^2}{2 (a_1^2 + a_2^2)} \right)
        \E[ f(\bW) ] \;,
    \end{align*}
    where $\bW \sim \cN\left(\frac{a_1^2 a_2^2}{a_1^2 + a_2^2} \left( \mfrac{\bm_1}{a_1^2} + \mfrac{\bm_2}{a_2^2} \right), \mfrac{a_1^2 a_2^2}{a_1^2 + a_2^2} I_d \right)$, which completes the proof.
\end{proof}

\vspace{.2em}

\begin{proof}[Proof of Lemma \ref{lem:gauss:integral:double}]
    Rewriting by the tower rule,
    \begin{align*}
        \E\Big[ &f(\bX, \bX') \exp\Big( - \mfrac{1}{2a_3^2} \| \bX - \bX' \|_2^2 \Big) \Big]
        \;=\;
        \E\left[ \E\left[ f(\bX, \bX') \exp\left( - \mfrac{1}{2a_3^2} \| \bX - \bX'\|_2^2 \right) \bigg| \bX \right] \right] \\
        \;=&\;
        \E\left[ \left( \mfrac{a_3^2}{a_2^2 + a_3^2} \right)^{d/2} \exp\left( - \mfrac{1}{2(a_2^2 + a_3^2)} \| \bX - \bm_2 \|_2^2 \right) 
        \E\Big[ f\Big(\bX, \bW' + \mfrac{a_2^2 }{a_2^2+a_3^2} \bX \Big) \Big| \bX \Big] \right]
        \;,
    \end{align*}
    where the last line follows by applying Lemma \ref{lem:gauss:integral:single} to the inner expectation, and where $\bW'  \sim \cN\left( \frac{a_3^2}{a_2^2 + a_3^2} \bm_2, \frac{a_2^2 a_3^2}{a_2^2 + a_3^2} I_d \right)$. Applying Lemma \ref{lem:gauss:integral:single} again gives
    \begin{align*}
        &\;\E\Big[ f(\bX, \bX') \exp\Big( - \mfrac{1}{2a_3^2} \| \bX - \bX' \|_2^2 \Big) \Big]
        \\
        \;=&\;
        \left( \mfrac{a_3^2}{a_2^2 + a_3^2} \right)^{d/2} 
        \left( \mfrac{a_2^2 + a_3^2}{a_1^2 + a_2^2 + a_3^2} \right)^{d/2}
        \exp\left( - \mfrac{1}{2(a_1^2 + a_2^2 + a_3^2)} \| \bm_1 - \bm_2 \|_2^2 \right) 
        \\
        &\;\qquad \times \E\Big[ \E\Big[f(\bW, \bW' + \mfrac{a_2^2 }{a_2^2+a_3^2} \bW) \Big| \bW\Big] \Big] \\
        \;=&\;
        \left( \mfrac{a_3^2}{a_1^2 + a_2^2 + a_3^2} \right)^{d/2} 
        \exp\left( - \mfrac{1}{2(a_1^2 + a_2^2 + a_3^2)} \| \bm_1 - \bm_2 \|_2^2 \right) \E\Big[ f\Big(\bW, \bW' + \mfrac{a_2^2 }{a_2^2+a_3^2} \bW\Big) \Big]
        \;,
    \end{align*}
    where $\bW \sim \cN\left( \mfrac{a_1^2(a_2^2 + a_3^2)}{a_1^2 + a_2^2 + a_3^2} \left( \mfrac{1}{a_1^2} \bm_1 + \mfrac{1}{a_2^2 + a_3^2} \bm_2 \right), \mfrac{a_1^2(a_2^2 + a_3^2)}{a_1^2 + a_2^2 + a_3^2} I_d \right)$.
\end{proof}

\subsection{Proofs for \cref{appendix:u:moments}}

\begin{proof}[Proof of Lemma \ref{lem:u:moments}] Consider the sequence of sigma algebras with $\cF_0$ being the trivial sigma algebra and $\cF_i := \sigma(X_1, \ldots, X_i)$ for $i=1,\ldots,n$. This allows us to define a martingale difference sequence: For $i=1, \ldots, n$, let
\begin{align*}
    Y_i \;:=\; \mean[ D_n | \cF_i ] - \mean[D_n | \cF_{i-1}]\;. 
\end{align*}
This implies that $\mean [ | D_n - \mean D_n |^\nu ] = \mean \big[ \big| \sum_{i=1}^n Y_i \big|^\nu \big]$. By Lemma \ref{lem:martingale:bound}, we get that for some universal constants $c'_\nu, C'_\nu$,
\begin{align*}
    c'_{\nu} \, \msum_{i=1}^n \mean[ |Y_i |^\nu ]
    \;\leq\;
    \mean [ | D_n - \mean D_n |^\nu ] 
    \;\leq\;  C'_{\nu} \, n^{\nu/2 - 1} \msum_{i=1}^n \mean[ |Y_i |^\nu ]\;. \tagaligneq \label{eqn:u:intermediate}
\end{align*}
To compute the $\nu$-th moment of $Y_i$, recall that $D_n = \frac{1}{n(n-1)} \sum_{j,l \in [n], j \neq l} u(\bX_j,\bX_l)$, which implies
\begin{align*}
    \mean[ |Y_i |^\nu ] 
    \;=&\;
    \mean \big[ \big| \mean[ D_n | \cF_i ] - \mean[D_n | \cF_{i-1}] \big|^{\nu} \big]
    \\
    \;=&\;
    \mfrac{1}{n^\nu (n-1)^\nu} \mean \Big[ \big| \msum_{j,l \in [n], j \neq l} \big( \mean[ u(\bX_j, \bX_l) | \cF_i ] - \mean[ u(\bX_j, \bX_l) | \cF_{i-1} ] \big) \big|^\nu \Big]
    \\
    \;\overset{(a)}{=}&\;
    \mfrac{2}{n^\nu (n-1)^\nu} \mean \Big[ \big| \msum_{j \in [n], j \neq i} \big( \mean[ u(\bX_i, \bX_j) | \cF_i ] - \mean[ u(\bX_i, \bX_j) | \cF_{i-1} ] 
    \big) \big|^\nu \Big]
    \\
    \;=:&\;  \mfrac{2}{n^\nu (n-1)^\nu} \, \mean[ | S_i |^{\nu} ]
    \;.
\end{align*}
In $(a)$, we have used that each summand is zero if both $j$ and $l$ do not equal $i$, and that $u$ is symmetric. In the case $j<i$, we can compute each summand of $S_i$ as
\begin{align*}
    \mean[ u(\bX_i, \bX_j) | \cF_i ] - \mean[ u(\bX_i, \bX_j) | \cF_{i-1} ] 
    \;=&\; 
    u(\bX_i, \bX_j) - \mean[ u(\bX_1, \bX_j) | \bX_j]
    \\
    \;=&\; A_{ij} - B_j + B_i\;,
\end{align*}
where $A_{ij} := u(\bX_i, \bX_j) - \mean[u(\bX_i,\bX_1) | \bX_i ]$ and
\begin{align*}
    B_i \;:=\; \mean[u(\bX_i,\bX_1) | \bX_i ]-\mean[u(\bX_1,\bX_2)]
    \;=\; \mean[u(\bX_1,\bX_i) | \bX_i ]-\mean[u(\bX_1,\bX_2)]\;
\end{align*}
by symmetry of $u$. In the case $j>i$, we can compute each summand as
\begin{align*}
    \mean[ u(\bX_i, \bX_j) | \cF_i ] - \mean[ u(\bX_i, \bX_j) | \cF_{i-1} ] 
    \;=\; 
    \mean[ u(\bX_1, \bX_i) | \bX_i] - \mean[ u(\bX_1, \bX_2) ]
    \;=\;
    B_i \;.
\end{align*}
Therefore
\begin{align*}
    S_i
    \;=&\;
    \msum_{j < i} (A_{ij} - B_j) + n B_i
    \;. 
\end{align*}
Consider $R_1 := nB_i$ and $R_2 := \sum_{j < i} (A_{ij} - B_j)$, which forms a two-element martingale difference sequence with respect to the filtration $\sigma(\bX_i) \subseteq \sigma(\bX_i, \bX_1 \ldots, \bX_{i-1})$. By Lemma \ref{lem:martingale:bound} again, there exist constants $c^*_\nu$ and $C^*_\nu$ depending only on $\nu$ such that
\begin{align*}
    \mean[ |S_i|^{\nu} ]
    \;=\;
    \mean \big[ \big| \msum_{l=1}^2 R_l \big|^\nu \big] 
    \;&\leq\; 
    C^*_{\nu} \Big( \mean[ |n B_i |^\nu ] + \mean\big[ \big| \msum_{j < i} (A_{ij} - B_j) \big|^\nu \big] \Big)
    \\
    \;&=\;
    C^*_{\nu} \big( n^\nu \Mcondnu^\nu +  \mean\big[ \big| \msum_{j < i} (A_{ij} - B_j) \big|^\nu \big] \big)
    \;,
    \\
    \mean[ |S_i|^{\nu} ]
    \;=\;
    \mean \big[ \big| \msum_{l=1}^2 R_l \big|^\nu \big] 
    \;&\geq\; 
    c^*_\nu \Big( \mean[ |n B_i |^\nu ] +  \mean\big[ \big| \msum_{j < i} (A_{ij} - B_j) \big|^\nu \big] \Big)
    \\
    \;&=\;
    c^*_\nu  \big( n^\nu \Mcondnu^\nu +  \mean\big[ \big| \msum_{j < i} (A_{ij} - B_j) \big|^\nu \big] \big)
    \;.
\end{align*}
Now consider $T_{j} := A_{ij} - B_j$ for $j=1, \ldots, i-1$, which again forms a martingale difference sequence with respect to $\sigma(\bX_i, \bX_1), \ldots, \sigma(\bX_i, \bX_1 \ldots, \bX_{i-1})$. Then by Lemma \ref{lem:martingale:bound} again, there exist constants $c^\Delta_\nu$ and $C^\Delta_\nu$ depending only on $\nu$ such that
\begin{align*}
    \mean\big[ \big| \msum_{j < i} (A_{ij} - B_j) \big|^\nu \big] 
    \;&\leq\; 
    C^\Delta_{\nu} \, (i-1)^{\nu/2 - 1}  \msum_{j=1}^{i-1}  \mean[ |A_{ij} - B_j|^\nu ]
    \;=\;
    C^\Delta_{\nu} \, (i-1)^{\nu/2} \Mfullnu^\nu
    \;, 
    \\
    \mean\big[ \big| \msum_{j < i} (A_{ij} - B_j) \big|^\nu \big] 
    \;&\geq\; 
    c^\Delta_\nu \, \msum_{j=1}^{i-1}  \mean[ |A_{ij} - B_j|^\nu ] 
    \;=\;
    c^\Delta_\nu \, (i-1) \Mfullnu^\nu
    \;.
\end{align*}
Therefore
\begin{align*}
    \mean[ |S_i|^\nu ] 
    \;\leq&\;
    C^*_{\nu} n^\nu \Mcondnu^\nu 
    +  
    C^*_{\nu} C^\Delta_{\nu} \, (i-1)^{\nu/2} \Mfullnu^\nu
    \;,
    \\
    \mean[ |S_i|^\nu ] 
    \;\geq&\;
    c^*_\nu n^\nu \Mcondnu^\nu 
    +  
    c^*_\nu c^\Delta_\nu \, (i-1) \Mfullnu^\nu
    \;,
\end{align*}
which yield the following bounds on the $\nu$-th moment of $Y_i$:
\begin{align*}
    \mean[ |Y_i |^\nu ] 
    \;&\leq\;
    2 C^*_{\nu} \big(
    (n-1)^{-\nu} \Mcondnu^\nu 
    +  
    C^\Delta_{\nu} n^{-\nu} (n-1)^{-\nu} \, (i-1)^{\nu/2} \Mfullnu^\nu
    \big)
    \;,
    \\
    \mean[ |Y_i |^\nu ] 
    \;&\geq\;
    2 c^*_{\nu} \big(
    (n-1)^{-\nu} \Mcondnu^\nu 
    +  
    c^\Delta_{\nu} n^{-\nu} (n-1)^{-\nu} \, (i-1) \Mfullnu^\nu
    \big)
    \;,
\end{align*}
To sum these terms over $i=1, \ldots, n$, we note that since $\nu/2 > 0$,
\begin{align*}
    &
    \msum_{i=1}^n (i-1)^{\nu/2} 
    \;\leq\; 
    \mint_0^{n} x^{\nu/2} dx \;=\;
    \mfrac{n^{1+\nu/2}}{1+\nu/2}\;,
    &&
    \msum_{i=1}^n (i-1)
    \;=\;
    \mfrac{n(n-1)}{2}\;.
\end{align*}
Define $C_\nu := \frac{2 C'_\nu C^*_\nu  \mmax\{1, C^\Delta_\nu\}}{1+\nu/2}$ and $c_\nu :=  c'_\nu c^*_\nu \min\{1,c^\Delta_\nu\}$. By summing the bounds on $\mean[|Y_i|^\nu]$ and substituting into \eqref{eqn:u:intermediate}, we get the desired bounds
\begin{align*}
    \mean [ | D_n - \mean D_n |^\nu ] 
    \;&\leq\;
    C_{\nu} \, n^{\nu/2 - 1} 
    \, \big( 
    n (n-1)^{-\nu} \Mcondnu^\nu 
    +  
    n^{-\nu} (n-1)^{-\nu} n^{1+\nu/2} \Mfullnu^\nu
    \big)
    \\
    \;&=\;
    C_{\nu} \, n^{\nu/2} 
    (n-1)^{-\nu} \Mcondnu^\nu 
    +  
    C_{\nu} \, (n-1)^{-\nu} \Mfullnu^\nu
    \;,
    \\
    \mean [ | D_n - \mean D_n |^\nu ] 
    \;&\geq\;
    c_{\nu}
    n(n-1)^{-\nu} \Mcondnu^\nu 
    +  
    c_{\nu}
    n^{-(\nu-1)} (n-1)^{-(\nu-1)} \Mfullnu^\nu
    \;.
\end{align*}
\end{proof}

\ 

\begin{proof}[Proof of Lemma \ref{lem:factorisable:moment}] The first result is directly obtained from linearity of expectation and Jensen's inequality:
\begin{align*}
    \Big| 
    D
    -
    \msum_{k=1}^K \lambda_k \mu_k^2
    \Big|
    \;=&\;
    \Big| 
    \mean[ u(\bX_1, \bX_2) ] 
    -
    \msum_{k=1}^K \lambda_k \mean[\phi_k(\bX_1)] \mean[\phi_k(\bX_2)] 
    \Big|
    \\
    \;=&\;
    \Big| 
    \mean\Big[ 
    u(\bX_1, \bX_2) 
    -
    \msum_{k=1}^K \lambda_k \phi_k(\bX_1) \phi_k(\bX_2)
    \Big]
    \Big|
    \\
    \;\leq&\;
    \mean
    \Big| 
    u(\bX_1, \bX_2) 
    -
    \msum_{k=1}^K \lambda_k \phi_k(\bX_1) \phi_k(\bX_2)
    \Big|
    \;=\; \varepsilon_{K;1}
    \;.
\end{align*}

To prove the next few bounds, we first derive a useful inequality: For $a, b \in \R$ and $\nu \geq 1$, by Jensen's inequality, we have
\begin{align*}
    | a + b |^\nu 
    \;=\; 
    \big| \mfrac{1}{2} (2a) + \mfrac{1}{2} (2b) \big|^\nu 
    \;\leq\; 
    \mfrac{1}{2} |2a|^\nu + \mfrac{1}{2} |2b|^\nu
    \;=\;
    2^{\nu-1} ( |a|^\nu + |b|^\nu )\;.
\end{align*}
By a triangle inequality followed by applying the above inequality again with $a$ replaced by $|a|-|b|$ and $b$ replaced by $|b|$, we have
\begin{align*}
     |a + b |^\nu \;\geq\; |  |a| - |b| |^\nu \;\geq\; 2^{-(\nu-1)} |a|^\nu - |b|^\nu \;.
\end{align*}
Since $\nu \in [1,3]$, we have $2^{\nu - 1} \in [1, 4]$. Therefore
\begin{align*}
    \mfrac{1}{4} |a|^\nu - |b|^\nu \;\leq\; |a + b |^\nu \;\leq\; 4 ( |a|^\nu + |b|^\nu )\;. \tagaligneq \label{eqn:factorisable:useful:ineq}
\end{align*}
Now to prove the conditional bound, we make use of the fact that $\bX_1, \bX_2$ are i.i.d.~to see that
\begin{align*}
    &\;
    \mean\Big[ \Big| \msum_{k=1}^K \lambda_k (\phi_k(\bX_1) - \mu_k) \mu_k \Big|^\nu \Big]
    \\
    \;=&\;
    \mean\Big[ \Big| 
    \msum_{k=1}^K \lambda_k \big(
    \mean[ \phi_k(\bX_1) \phi_k(\bX_2) | \bX_1]
    - \mean[ \phi_k(\bX_1) \phi_k(\bX_2)]\big) \Big|^\nu \Big] 
    \\
    \;=&\;
    \mean \Big[ \Big| 
    \mean[ u(\bX_1, \bX_2) | \bX_1] -  \mean[ u(\bX_1, \bX_2) ]
    +
    \Delta_{K;1}
    -
    \Delta_{K;2}
    \Big|^\nu \Big]
    \;, \tagaligneq \label{eqn:factorisable:moment:cond:intermediate}
\end{align*}
where
\begin{align*}
    \Delta_{K;1} 
    \;\coloneqq&\; 
    \msum_{k=1}^K \lambda_k 
    \mean[ \phi_k(\bX_1) \phi_k(\bX_2) | \bX_1]
    - 
    \mean[ u(\bX_1, \bX_2) | \bX_1]\;,
    \\
    \Delta_{K;2} 
    \;\coloneqq&\; 
    \msum_{k=1}^K \lambda_k 
    \mean[ \phi_k(\bX_1) \phi_k(\bX_2) ]
    - 
    \mean[ u(\bX_1, \bX_2) ]\;.
\end{align*}
Moments of the two error terms can be bounded by Jensen's inequality applied to $x \mapsto |x|^\nu$ with respect to the conditional expectation $\mean[ \argdot | \bX_2]$ and the expectation $\mean[ \argdot ]$:
\begin{align*}
    \mean[ | \Delta_{K;1}|^\nu], \mean[ | \Delta_{K;2}|^\nu]
    \;\leq&\;
    \mean\Big[ \Big|
    u(\bX_1, \bX_2) 
    -
    \msum_{k=1}^K \lambda_k \phi_k(\bX_1) \phi_k(\bX_2)
    \Big|^\nu \Big] 
    \\
    \;=&\;
    \Big\|
    u(\bX_1, \bX_2) 
    -
    \msum_{k=1}^K \lambda_k \phi_k(\bX_1) \phi_k(\bX_2)
    \Big\|_{L_\nu}^\nu 
    \;=\; \varepsilon_{K;\nu}^{\nu}\;.
\end{align*}
On the other hand,
\begin{align*}
    (\Mcondnu)^\nu \;=\; 
    \mean \big[ \big| 
    \mean[ u(\bX_1, \bX_2) | \bX_1] -  \mean[ u(\bX_1, \bX_2) ] \big|^\nu \big]\;.
\end{align*}
Therefore applying \eqref{eqn:factorisable:useful:ineq} gives
\begin{align*}
    \mfrac{1}{4} (\Mcondnu)^\nu - \varepsilon_{K;\nu}^{\nu}
    \;\leq\; 
    \mean\Big[ \Big| \msum_{k=1}^K \lambda_k (\phi_k(\bX_1) - \mu_k) \mu_k \Big|^\nu \Big] 
    \;\leq\; 
    4 ( (\Mcondnu)^\nu + \varepsilon_{K;\nu}^{\nu})
\end{align*}

For the last bound, we start by considering the following quantity, which can be thought of as the truncated version of $\Mfullnu^\nu$:
\begin{align*}
    m_K\;\coloneqq&\;\mean \Big[ \Big| 
    \msum_{k=1}^K
    \lambda_k ( \phi_k(\bX_1)\phi_k(\bX_2) - \mu_k^2 )
    \Big|^\nu \Big]
    \\
    \;=&\;
    \mean \Big[ \Big| 
    \msum_{k=1}^K
    \lambda_k (\phi_k(\bX_1) - \mu_k)\phi_k(\bX_2)
    +
    \msum_{k=1}^K
    \lambda_k \mu_k(\phi_k(\bX_2) - \mu_k)
    \Big|^\nu \Big]
    \\
    \;\eqqcolon&\;
    \mean[| T_2 + T_1 |^\nu]\;.
\end{align*}
Since $\{T_1, T_2\}$ forms a two-element martingale difference sequence with respect to $\sigma(\bX_2) \subseteq \sigma(\bX_1,\bX_2)$, by Lemma \ref{lem:martingale:bound}, there exists absolute constants $c'_\nu, C'_\nu > 0$ depending only on $\nu$ such that
\begin{align*}
    c'_\nu \big( \mean[|T_1|^\nu]  + \mean[|T_2|^\nu]  \big) 
    \;\leq\; 
    m_K
    \;\leq\; 
    C'_\nu \big( \mean[|T_1|^\nu]  + \mean[|T_2|^\nu] \big)\;.
\end{align*}
Similarly, by writing
\begin{align*}
    \mean[|T_2|^\nu] \;=&\; 
    \mean \Big[ \Big| 
    \msum_{k=1}^K
    \lambda_k (\phi_k(\bX_1) - \mu_k)\phi_k(\bX_2) \Big|^\nu \Big]
    \\
    \;=&\;
    \mean \Big[ \Big| 
    \msum_{k=1}^K
    \lambda_k (\phi_k(\bX_1) - \mu_k)(\phi_k(\bX_2) - \mu_k) +  
    \msum_{k=1}^K
    \lambda_k (\phi_k(\bX_1) - \mu_k)\mu_k
    \Big|^\nu \Big]
    \\
    \;=&\;
    \mean[ | R_2 + R_1|^\nu ]\;,
\end{align*}
and noting that $\{R_1,R_2\}$ forms a two-element martingale difference sequence with respect to $\sigma(\bX_1) \subseteq \sigma(\bX_1,\bX_2)$, by Lemma \ref{lem:martingale:bound}, there exists absolute constants $c''_\nu, C''_\nu > 0$ depending only on $\nu$ such that
\begin{align*}
    c''_\nu \big( \mean[|R_1|^\nu]  + \mean[|R_2|^\nu]  \big) 
    \;\leq\; 
    \mean[|T_2|^\nu]
    \;\leq\; 
    C''_\nu \big( \mean[|R_1|^\nu]  + \mean[|R_2|^\nu] \big)\;.
\end{align*}
Combining the results and setting $A = \msup_{\nu \in [1,3]} C'_\nu \max\{ C''_\nu, 1\}$ and $a =  \minf_{\nu \in [1,3]}  c'_\nu \min\{ c''_\nu, 1\}$, we have shown that
\begin{align*}
    a \big( \mean[|T_1|^\nu] +  \mean[|R_1|^\nu] +  \mean[|R_2|^\nu] \big)
    \;\leq\;
    m_K
    \;\leq&\; 
    A \big( \mean[|T_1|^\nu] +  \mean[|R_1|^\nu] +  \mean[|R_2|^\nu] \big)
    \;.
\end{align*}
Notice that the quantity we would like to control is exactly
\begin{align*}
    \mean[|R_2|^\nu] \;=&\; 
    \mean\Big[ \Big|
    \msum_{k=1}^K
    \lambda_k (\phi_k(\bX_1) - \mu_k)(\phi_k(\bX_2) - \mu_k)
    \Big|^\nu \Big]\;,
\end{align*}
and that $\mean[|T_1|^\nu]=\mean[|R_1|^\nu]$. By setting $c = A^{-1}$ and $C = a^{-1}$, this allows us to obtain a bound about $\mean[|R_2|^\nu]$ as
\begin{align*}
    c m_K
    - 2 \mean[|T_1|^\nu]
    \leq
    \mean\Big[ \Big|
    \msum_{k=1}^K
    \lambda_k (\phi_k(\bX_1) - \mu_k)(\phi_k(\bX_2) - \mu_k)
    \Big|^\nu \Big]
    \leq
    C m_K
    - 2 \mean[|T_1|^\nu]\;.
\end{align*}
Now notice that 
\begin{align*}
    \mean[|T_1|^\nu] \;=\;
    \mean\Big[
    \Big|\msum_{k=1}^K
    \lambda_k (\phi_k(\bX_1) - \mu_k)\mu_k
    \Big|^\nu \Big]\;,
\end{align*}
which has already been controlled by the second result of the lemma as 
\begin{align*}
    \mfrac{1}{4} (\Mcondnu)^\nu - \varepsilon_{K;\nu}^{\nu}
    \;\leq\; 
    \mean[|T_1|^\nu]
    \;\leq\; 
    4 ( (\Mcondnu)^\nu + \varepsilon_{K;\nu}^{\nu} )\;.
\end{align*}
On the other hand, we can use an exactly analogous argument by using \eqref{eqn:factorisable:useful:ineq} and applying Jensen's inequality to control the errors to show that
\begin{align*}
    \mfrac{1}{4} (\Mfullnu)^\nu - \varepsilon_{K;\nu}^{\nu}
    \;\leq\; 
    m_K
    \;\leq\; 
    4 ( (\Mfullnu)^\nu + \varepsilon_{K;\nu}^{\nu} )\;.
\end{align*}
Applying these two results to the previous bound gives the desired bounds:
\begin{align*}
    \mean\Big[ \Big|
    \msum_{k=1}^K
    \lambda_k (\phi_k(\bX_1) - \mu_k)(\phi_k(\bX_2) - \mu_k)
    \Big|^\nu \Big]
    \leq&\;
    4 C (\Mfullnu)^\nu    
    - 
    \mfrac{1}{2} (\Mcondnu)^\nu
    + 
    (4 C + 2) \varepsilon_{K;\nu}^{\nu}\;,
    \\
    \mean\Big[ \Big|
    \msum_{k=1}^K
    \lambda_k (\phi_k(\bX_1) - \mu_k)(\phi_k(\bX_2) - \mu_k)
    \Big|^\nu \Big]
    \geq&\;
    \mfrac{c}{4}  (\Mfullnu)^\nu    
    - 
    8 (\Mcondnu)^\nu
    -
    (c + 8) \varepsilon_{K;\nu}^{\nu}\;.
\end{align*}
\end{proof}

\

\begin{proof}[Proof of Lemma \ref{lem:factorisable:moment:gaussian}] To compute the first bound, we rewrite the expression of interest as a quantity that we have already considered in the proof of Lemma \ref{lem:factorisable:moment}:
\begin{align*}
    (\mu^K)^\top \Lambda^K \Sigma^K \Lambda^K (\mu^K) 
    \;=&\; 
     (\mu^K)^\top  \Lambda^K  \mean\Big[ \big(  \phi^K(\bX_1) -  \mu^K \big) \big(  \phi^K(\bX_1) -  \mu^K \big)^\top \Big] \Lambda^K  (\mu^K)
    \\
    \;=&\;
    \mean\Big[ \Big( \big(  \phi^K(\bX_1) -  \mu^K \big)^\top \Lambda^K  \mu^K \Big)^2 \Big]
    \\
    \;=&\;
    \mean \Big[ \Big( \msum_{k=1}^K \lambda_k(\phi_k(\bX_1) - \mu_k) \mu_k \Big)^2 \Big]
    \\
    \;=&\;
    \mean \Big[ \Big(
    \mean[ u(\bX_1, \bX_2) | \bX_2] -  \mean[ u(\bX_1, \bX_2) ]
    +
    \Delta_{K;1}
    -
    \Delta_{K;2}
    \Big)^2 \Big]
    \;,
\end{align*}
where we have used the calculation in \eqref{eqn:factorisable:moment:cond:intermediate} with $\nu=2$ and defined the same error terms
\begin{align*}
    \Delta_{K;1} 
    \;\coloneqq&\; 
    \msum_{k=1}^K \lambda_k 
    \mean[ \phi_k(\bX_1) \phi_k(\bX_2) | \bX_2]
    - 
    \mean[ u(\bX_1, \bX_2) | \bX_2]\;,
    \\
    \Delta_{K;2} 
    \;\coloneqq&\; 
    \msum_{k=1}^K \lambda_k 
    \mean[ \phi_k(\bX_1) \phi_k(\bX_2) ]
    - 
    \mean[ u(\bX_1, \bX_2) ]\;.
\end{align*}
Since we are dealing with the second moment, we can provide a finer bound by expanding the square explicitly: 
\begin{align*}
    (\mu^K)^\top \Lambda^K \Sigma^K \Lambda^K (\mu^K) 
    \;=&\; 
    \Var \mean[ u(\bX_1, \bX_2) | \bX_2]
    +
    \mean[ (\Delta_{K;1}
    -
    \Delta_{K;2}
    )^2 ]
    \\
    &\;
    +
    2 
    \mean\big[ \big( \mean[ u(\bX_1, \bX_2) | \bX_2] -  \mean[ u(\bX_1, \bX_2) ] \big) ( 
    \Delta_{K;1}
    -
    \Delta_{K;2}
    )^2 \big]
    \;.
\end{align*}
Then by a Cauchy-Schwartz inequality, we get that
\begin{align*}
    &\;
    \Big| (\mu^K)^\top \Lambda^K \Sigma^K \Lambda^K (\mu^K)  - \Var \mean[ u(\bX_1, \bX_2) | \bX_2] \Big| 
    \\
    \;=&\; 
    2 \big|
    \mean\big[ \big( \mean[ u(\bX_1, \bX_2) | \bX_2] -  \mean[ u(\bX_1, \bX_2) ] \big) ( 
    \Delta_{K;1}
    -
    \Delta_{K;2}
    )^2 \big]
    \big|
    +
    \mean[ (\Delta_{K;1}
    -
    \Delta_{K;2}
    )^2 ]
    \\
    \;\leq&\;
    2 \sqrt{ \Var \mean[ u(\bX_1, \bX_2) | \bX_2] } \sqrt{\mean[ (\Delta_{K;1}
    -
    \Delta_{K;2}
    )^2 ] } 
    +
    \mean[ (\Delta_{K;1}
    -
    \Delta_{K;2}
    )^2 ]\;.
\end{align*}
The variance term is exactly $\sigcond^2$. Since the individual error terms have already been bounded in the proof of Lemma \ref{lem:factorisable:moment} as $\mean[\Delta_{K;1}^2], \mean[\Delta_{K;2}^2] \leq \varepsilon_{K;2}^2$, by a triangle inequality and a Cauchy-Schwarz inequality, we have
\begin{align*}
    | \mean[ (\Delta_{K;1}
    -
    \Delta_{K;2}
    )^2 ] | 
    \;=&\;
    | \mean[ \Delta_{K;1}^2 ] - 2 \mean[ \Delta_{K;1} \Delta_{K;2} ] + \mean [ \Delta_{K;2}^2 ] |
    \\
    \;\leq&\;
    | \mean[ \Delta_{K;1}^2 ] | + 2 \sqrt{ | \mean[ \Delta_{K;1}^2 ] | | \mean[ \Delta_{K;2}^2 ] | } + | \mean[ \Delta_{K;2}^2 ] |
    \;\leq\;
    4 \varepsilon_{K;2}^2\;.
\end{align*}
Combining the bounds gives
\begin{align*}
    \big| (\mu^K)^\top \Lambda^K \Sigma^K \Lambda^K (\mu^K) - (\sigcond)^2 \big|
    \;\leq\;
    4\varepsilon_{K;2}^2 + 4 \sigcond \varepsilon_{K;2}
    \;,
\end{align*}
which rearranges to give
\begin{align*}
    \sigcond^2 - 4 \sigcond \varepsilon_{K;2} -4\varepsilon_{K;2}^2
    \;\leq\;
    (\mu^K)^\top \Lambda^K \Sigma^K \Lambda^K (\mu^K) 
    \;\leq&\; 
    \sigcond^2 + 4 \sigcond \varepsilon_{K;2} + 4\varepsilon_{K;2}^2
    \\
    \;\leq&\;
    (\sigcond + 2\varepsilon_{K;2})^2\;.
\end{align*}
 
The second bound is obtained similarly by giving a finer control than the bound in Lemma \ref{lem:factorisable:moment}. We first rewrite the expression of interest by using linearity of expectation and the cyclic property of trace:
\begin{align*}
    \Tr( ( \Lambda^K \Sigma^K)^2 )
    \;=&\;
    \Tr\big( \Lambda^K \mean\big[ \phi^K(\bX_1) \phi^K(\bX_1)^\top \big]  \Lambda^K
     \mean\big[  \phi^K(\bX_2) \phi^K(\bX_2)^\top \big]
    \big)
    \\
    \;=&\;
    \mean\Big[ \big(  \phi^K(\bX_1)^\top \Lambda^K \phi^K(\bX_2) \big)^2 \Big]
    \;=\;
    \mean\Big[ \Big( \msum_{k=1}^K \lambda_k \phi_k(\bX_1) \phi_k(\bX_2)  \Big)^2 \Big]\;.
\end{align*}
Again by expanding the square explicitly, we get that
\begin{align*}
    \Tr( ( \Lambda^K \Sigma^K)^2 )
    \;=&\;
    \mean\Big[ \Big( \msum_{k=1}^K \lambda_k \phi_k(\bX_1) \phi_k(\bX_2) - \bar u(\bX_1,\bX_2) + \bar u(\bX_1,\bX_2) \Big)^2 \Big]
    \\
    \;=&\;
    \mean\Big[ \Big( \msum_{k=1}^K \lambda_k \phi_k(\bX_1) \phi_k(\bX_2) - \bar u(\bX_1,\bX_2) \Big)^2 \Big] 
    + \mean\big[ \bar u(\bX_1,\bX_2)^2 \big]
    + 2 \Delta_{K;3}
    \\
    \;=&\;
    \varepsilon_{K;2}^2 + \sigfull^2 + 2 \Delta_{K;3}
    \;,
\end{align*}
where we have defined the additional error term as
\begin{align*}
    \Delta_{K;3} 
    \;\coloneqq\;
    \mean\Big[ \Big( \msum_{k=1}^K \lambda_k \phi_k(\bX_1) \phi_k(\bX_2) - \bar u(\bX_1,\bX_2) \Big) \bar u(\bX_1,\bX_2) \Big] \;.
\end{align*}
By a Cauchy-Schwarz inequality, we get that
\begin{align*}
    \big| \Tr( & ( \Lambda^K \Sigma^K)^2 ) - \sigfull^2 - \varepsilon_{K;2}^2 \big| 
    \;=\; 2 | \Delta_{K;3} |
    \\
    \;\leq&\; 
    2 \sqrt{\mean\big[ \big( \msum_{k=1}^K \lambda_k \phi_k(\bX_1) \phi_k(\bX_2) - \bar u(\bX_1,\bX_2) \big)^2 \big] }\sqrt{\mean\big[ \bar u(\bX_1,\bX_2)^2 \big]}
    \;=\;
    2 \varepsilon_{K;2} \sigfull\;.
\end{align*}
Combining the above two bounds yields the desired inequality that
\begin{align*}
    (\sigfull - \varepsilon_{K;2})^2
    \;\leq\;
    \Tr( (\Lambda^K \Sigma^K)^2 ) \;\leq\; (\sigfull + \varepsilon_{K;2})^2\;.
\end{align*}

\vspace{.5em}

To prove the third bound, note that $(\mu^K)^\top \Lambda^K \bZ_1$ is a zero-mean normal random variable with variance given by $(\mu^K)^\top \Lambda^K  \Sigma^K \mu^K$\;, which is already bounded above. By applying the formula of the $\nu$-th absolute moment of a normal distribution and noting that $\nu \leq 3$, we obtain
\begin{align*}
    \mean [| (\mu^K)^\top& \Lambda^K  \bZ_1  |^\nu ] 
    \;=\; 
    \mfrac{2^{\nu/2}}{\sqrt{\pi}} \Gamma\Big( \mfrac{\nu+1}{2} \Big) \big( (\mu^K)^\top  \Lambda^K \Sigma^K \Lambda^K \mu^K \big)^{\nu/2}
    \\
    \;\leq&\;
    \mfrac{2^{\nu/2}}{\sqrt{\pi}} (\sigfull + 2\varepsilon_{K;2})^\nu
    \;\overset{(a)}{\leq}\;
    \mfrac{2^{\nu/2}}{\sqrt{\pi}} \mmax\{1, 2^{\nu-1}\} \big(
    \sigcond^\nu + 2^\nu \varepsilon_{K;2}^\nu \big)
    \;\overset{(b)}{\leq}\;
    7 (
    \sigcond^\nu + 8 \varepsilon_{K;2}^\nu )
    \;.
\end{align*}
In $(a)$, we have noted that given $a,b > 0$, for $\nu/2 \in (0,1]$, $(a+b)^{\nu/2} \leq a^{\nu/2} + b^{\nu/2}$ and for $\nu/2 > 1$, the bound follows from Jensen's inequality. In $(b)$, we have noted that $\nu \leq 3$. This finishes the proof for the third bound.

\vspace{.5em} 

To prove the fourth bound, we can first condition on $\bZ_2$:
\begin{align*}
     \mean [| \bZ_1^\top \Lambda^K \bZ_2  |^\nu ]
     \;=\;
     \mean\big[ \,\mean [| \bZ_1^\top \Lambda^K \bZ_2  |^\nu |\, \bZ_2 ] \, \big]\;.
\end{align*}
The inner expectation is again the $\nu$-th absolute moment of a conditionally Gaussian random variable with variance $\bZ_2^\top  \Lambda^K  \Sigma^K  \Lambda^K  \bZ_2$, so again by the formula of the $\nu$-th absolute moment of a normal distribution, we get that
\begin{align*}
    \mean [| \bZ_1^\top \Lambda^K  \bZ_2  |^\nu ]
    \;\leq\;
    \mfrac{2^{\nu/2}}{\sqrt{\pi}} 
    \,\mean \Big[
    \big( \bZ_2^\top  \Lambda^K \Sigma^K  \Lambda^K \bZ_2 \big)^{\nu/2}
     \Big]
     \;\leq\;
    \mfrac{2^{\nu/2}}{\sqrt{\pi}} 
    \,\mean \Big[
    \big( \bZ_2^\top  \Lambda^K \Sigma^K  \Lambda^K  \bZ_2 \big)^2
     \Big]^{\nu / 4}
     \;.
\end{align*}
We have noted that $\nu \leq 3$ and used a H\"older's inequality. The remaining expectation is taken over a quadratic form of normal variables. Writing $\Sigma_* = (\Sigma^K)^{1/2} \Lambda^K (\Sigma^K)^{1/2}$ for short, the second moment can be computed by the formula from Lemma \ref{lem:quadratic:gaussian:moments} as
\begin{align*}
    \mean \Big[
    \big( \bZ_2^\top  \Lambda^K \Sigma^K  \Lambda^K \bZ_2 \big)^2
     \Big]
     \;=&\;
     \Tr(\Sigma_*^2)^2
     +
     2
     \Tr\big( \Sigma_*^4 \big)
     \;\overset{(a)}{\leq}\;
     3  \Tr(\Sigma_*^2)^2
     \;=\;
     3  \Tr\big( ( \Lambda^K \Sigma^K )^2 \big)^2
     \;.
\end{align*}
Note that in $(a)$, we have used the fact that the square of a symmetric matrix, $\Sigma_*^2$, has non-negative eigenvalues, and therefore $\Tr(\Sigma_*^4) \leq \Tr(\Sigma_*^2)^2$. Since we have already bounded $ \Tr( ( \Lambda^K \Sigma^K )^2 )$ earlier, substituting the above result into the previous bound, we get that
\begin{align*}
    \mean [| \bZ_1^\top \Lambda^K \bZ_2  |^\nu ]
    \;\leq&\;
    \mfrac{2^{\nu/2}}{\sqrt{\pi}} 
    \,\mean \Big[
    \big( \bZ_2^\top  \Lambda^K \Sigma^K  \Lambda^K  \bZ_2 \big)^2
     \Big]^{\nu / 4}
    \leq
    \mfrac{2^{\nu/2} 3^{\nu/4}}{\sqrt{\pi}} 
    \, \Tr\big( ( \Lambda^K \Sigma^K )^2 \big)^{\nu/2}
    \\
    \;\leq&\;
    \mfrac{2^{\nu/2} 3^{\nu/4}}{\sqrt{\pi}} (
    \sigfull^2
    +\varepsilon_{K;2}^2
    )^{\nu/2}
    \\
    \;\leq&\;
     \mfrac{2^{\nu/2} 3^{\nu/4}}{\sqrt{\pi}} 
     \mmax\{1, 2^{\nu/2-1}\} 
     \big( 
    \sigfull^\nu
    + 
     \varepsilon_{K;2}^\nu
     \big)
     \;\leq\;
     6
     \big( 
    \sigfull^\nu
    + 
     \varepsilon_{K;2}^\nu
     \big)
     \;.
\end{align*}
In the last two inequalities, we have used the same argument as in the proof for the third bound to expand the term with $\nu$-th power. This gives the desired bound.

\vspace{.5em}

To prove the final bound, we first condition on $\bX_1$:
\begin{align*}
    \mean \big[ \big| (\phi^K(\bX_1) - \mu^K)^\top \Lambda^K\bZ_1  \big|^\nu \big]
    \;=\;
    \mean \big[ \mean \big[ \big| (\phi^K(\bX_1) - \mu^K)^\top \Lambda^K\bZ_1  \big|^\nu \big| \bX_1 \big] \big]\;.
\end{align*}
The inner expectation is the $\nu$-th absolute moment of a conditionally Gaussian random variable with variance $(\phi^K(\bX_1) - \mu^K)^\top \Lambda^K  \Sigma^K  \Lambda^K (\phi^K(\bX_1) - \mu^K)$, so by the formula of the $\nu$-th absolute moment of a normal distribution with $\nu \leq 3$, we get that
\begin{align*}
    &\;\mean [| (\phi^K(\bX_1) - \mu^K)^\top \Lambda^K  \bZ_2  |^\nu ]
    \\
    \;\leq&\;
    \mfrac{2^{\nu/2}}{\sqrt{\pi}} 
    \,\mean \Big[
    \big( (\phi^K(\bX_1) - \mu^K)^\top \Lambda^K  \Sigma^K  \Lambda^K (\phi^K(\bX_1) - \mu^K) \big)^{\nu/2}
     \Big]
     \\
     \;=&\;
     \mfrac{2^{\nu/2}}{\sqrt{\pi}} 
    \,
    \mean \Big[
    \big( (\phi^K(\bX_1) - \mu^K)^\top \Lambda^K \mean\big[ (\phi^K(\bX_2) - \mu^K) (\phi^K(\bX_2) - \mu^K)^\top \big] \Lambda^K (\phi^K(\bX_1) - \mu^K)  \big)^{\nu/2}
     \Big]
     \\
     \;\overset{(a)}{\leq}&\;
     \mfrac{2^{\nu/2}}{\sqrt{\pi}} 
    \,
    \mean \Big[
    \big| (\phi^K(\bX_1) - \mu^K)^\top \Lambda^K (\phi^K(\bX_2) - \mu^K) \big|^{\nu}
     \Big]
     \\
     \;=&\;
     \mfrac{2^{\nu/2}}{\sqrt{\pi}} 
    \,
    \mean \Big[
    \big| \msum_{k=1}^K \lambda_k \phi_k(\bX_1)\phi_k(\bX_2) \big|^{\nu}
     \Big]
     \\
     \;\overset{(b)}{\leq}&\; 
    8 C (\Mfullnu)^\nu    
    -
     (\Mcondnu)^\nu
    + 
    (8 C + 4) \varepsilon_{K;\nu}^\nu\;.
\end{align*}
In $(a)$, we have applied Jensen's inequality to the convex function $x \mapsto |x|^{\nu/2}$ to move the inner expectation outside the norm. In $(b)$, we have applied the bound in Lemma \ref{lem:factorisable:moment} and noted that $\frac{2^{\nu/2}}{\sqrt{\pi}} < 2$ for $\nu \in [1, 3]$. This gives the desired result.
\end{proof}

\ 

\begin{proof}[Proof of Lemma \ref{lem:WnK:moments}] For the first equality in distribution, we recall that $\{\tau_{k;d}\}_{k=1}^K$ are the eigenvalues of $(\Sigma^K)^{1/2} \Lambda^K (\Sigma^K)^{1/2}$ and $\{\xi_k\}_{k=1}^K$ are a sequence of i.i.d.~standard Gaussian variables. Let $\{\eta_{ik}\}_{i \in [n], k \in[K]}$ be a set of i.i.d.~standard Gaussian variables. Since Gaussianity is preserved under orthogonal transformation, we have
\begin{align*}
    &\;\mfrac{1}{n^{3/2}(n-1)^{1/2}} 
    \Big( \msum_{i,j=1}^n (\eta^K_i)^\top (\Sigma^K)^{1/2} \Lambda^K (\Sigma^K)^{1/2} \eta^K_j
    -
    n \Tr( \Sigma^K \Lambda^K )
    \Big)    
    \\
    \;\overset{d}{=}&\;
    \mfrac{1}{n^{3/2}(n-1)^{1/2}} 
    \Big(
    \msum_{k=1}^K \msum_{i,j=1}^n \tau_{k;d} \eta_{ik} \eta_{jk}
    -
    n \Tr( (\Sigma^K)^{1/2} \Lambda^K (\Sigma^K)^{1/2}  ) \Big)
    \\
    \;=&\;
    \mfrac{1}{n^{3/2}(n-1)^{1/2}} 
     \msum_{k=1}^K 
     \tau_{k;d}
     \Big(
    \big(\msum_{i=1}^n \eta_{ik}\big) \big( \msum_{j=1}^n \eta_{jk}\big)
    -
    n
    \Big)    
    \\
    \;\overset{d}{=}&\;
    \mfrac{1}{n^{1/2}(n-1)^{1/2}} 
    \msum_{k=1}^K 
     \tau_{k;d}
     (\xi_k^2-1)
     \;=\; W_n^K - D\;,
\end{align*}
which proves the desired statement. 

\vspace{1em}

We now use the expression above for moment computation. The expectation is given by $\mean[ W_n^K] = D$ for every $K \in \N$. The variance can be computed by noting that the quantity is a quadratic form in Gaussian, applying Lemma \ref{lem:quadratic:gaussian:moments} and using the cyclic property of trace:
\begin{align*}
    \Var[ W_n^K ] 
    \;=&\; 
    \mfrac{1}{n(n-1)}
    \Var\big[ (\eta^K_1)^\top (\Sigma^K)^{1/2} \Lambda^K (\Sigma^K)^{1/2} \eta^K_1 \big]
    \\
    \;=&\;
    \mfrac{2}{n(n-1)}
    \Tr\big( (\Lambda^K \Sigma^K )^2 \big)\;.
\end{align*}
By Lemma \ref{lem:factorisable:moment:gaussian}, we get the desired bound that
\begin{align*}
    \mfrac{2}{n(n-1)} (\sigfull - \varepsilon_{K;2})^2
    \;\leq\; 
    \Var[ W_n^K ] 
    \;\leq\; 
    \mfrac{2}{n(n-1)} (\sigfull + \varepsilon_{K;2})^2\;.
\end{align*}

\vspace{1em}

The third central moment can be expanded using a binomial expansion and noting that each summand is zero-mean:
\begin{align*}
    \mean\big[ (W_n^K - D)^3 \big] \;=&\; 
    \mfrac{1}{n^{3/2}(n-1)^{3/2}} 
    \mean\Big[ \Big( \msum_{k=1}^K \tau_{k;d} (\xi_k^2 - 1) \Big)^3 \Big]
    \\
    \;=&\;
    \mfrac{1}{n^{3/2}(n-1)^{3/2}} 
    \mean\big[ \msum_{k=1}^K \tau_{k;d}^3 (\xi_k^2 - 1)^3 \big]
    \\
    \;=&\;
    \mfrac{8}{n^{3/2}(n-1)^{3/2}} \msum_{k=1}^K \tau_{k;d}^3\;.
\end{align*}
Meanwhile, the sum can be further expressed as
\begin{align*}
    &\;\msum_{k=1}^K \tau_{k;d}^3
    \\
    \;=&\;
    \Tr\Big( \big( (\Sigma^K)^{1/2} \Lambda^K (\Sigma^K)^{1/2} \big)^3 \Big)
    \;=\;
    \Tr\Big( \big( \Sigma^K \Lambda^K \big)^3 \Big)
    \\
    \;=&\;
    \Tr\Big( \big( \mean\big[ \phi^K(\bX_1) (\phi^K(\bX_1))^\top \big] \Lambda^K \big)^3 \Big)
    \\
    \;=&\;
    \mean \Big[  (\phi^K(\bX_1))^\top \Lambda^K \phi^K(\bX_2) \, (\phi^K(\bX_2))^\top \Lambda^K \phi^K(\bX_3) \, (\phi^K(\bX_3))^\top \Lambda^K \phi^K(\bX_1)  \Big]
    \\
    \;=&\; \mean\Big[ 
    \big(\msum_{k=1}^K \lambda_k \phi_k(\bX_1)\phi_k(\bX_2)\big)
    \big(\msum_{k=1}^K \lambda_k \phi_k(\bX_2)\phi_k(\bX_3)\big)
    \big(\msum_{k=1}^K \lambda_k \phi_k(\bX_3)\phi_k(\bX_1)\big)
    \Big]
    \\
    \;\eqqcolon&\; \mean[ S_{12} S_{23} S_{31} ]\;.
\end{align*}
We now approximate each $S_{ij}$ term by $u(\bX_i,\bX_j)$. For convenience, denote $U_{ij} = u(\bX_i, \bX_j)$ and $\Delta_{ij} = S_{ij} - U_{ij}$. Then
\begin{align*}
    \msum_{k=1}^K \tau_{k;d}^3
    \;=&\;
    \mean\big[ ( U_{12} + \Delta_{12} ) ( U_{23} + \Delta_{23} ) ( U_{31} + \Delta_{31} ) \big]
    \\
    \;=&\;
    \mean[ U_{12} U_{23} U_{31}] + \mean[U_{12}U_{23}\Delta_{31}] +  \mean[ U_{12} \Delta_{23} U_{31}] + \mean[ U_{12} \Delta_{23} \Delta_{31}]
    \\
    &\;
    +
    \mean[ \Delta_{12} U_{23} U_{31}] + \mean[\Delta_{12}U_{23}\Delta_{31}] +  \mean[ \Delta_{12} \Delta_{23} U_{31}] + \mean[ \Delta_{12} \Delta_{23} \Delta_{31}]\;.
\end{align*}
Recall that $\varepsilon_{K;3} = \mean[ |\Delta_{ij}|^3 ]^{1/3}$ for $i \neq j$ by definition. Then by a triangle inequality followed by a H\"older's inequality, we get that
\begin{align*}
    &\;\Big| \msum_{k=1}^K \tau_{k;d}^3 
    -  \mean[ u(\bX_1,\bX_2) u(\bX_2,\bX_3) u(\bX_3,\bX_1) ]
    \Big|
    \\
    \;\leq&\;
    \big| \mean[U_{12}U_{23}\Delta_{31}] \big| +  \big| \mean[ U_{12} \Delta_{23} U_{31}] \big| + \big| \mean[ U_{12} \Delta_{23} \Delta_{31}] \big|
    \\
    &\;
    +
   \big| \mean[ \Delta_{12} U_{23} U_{31}] \big|
   + 
   \big|\mean[\Delta_{12}U_{23}\Delta_{31}]\big| 
   + 
   \big|\mean[ \Delta_{12} \Delta_{23} U_{31}] \big|
   + 
   \big| \mean[ \Delta_{12} \Delta_{23} \Delta_{31}] \big|
   \\
   \;\leq&\;
   3 \mean[ |u(\bX_1,\bX_2)|^3 ]^{2/3} \varepsilon_{K;3}
   +
   3 \mean[ |u(\bX_1,\bX_2)|^3 ]^{1/3} \varepsilon_{K;3}^2
   +
   \varepsilon_{K;3}^3
   \\
   \;=&\;
   3 \Mfullthree^2 \varepsilon_{K;3}
   +
   3 \Mfullthree \varepsilon_{K;3}^2
   +
   \varepsilon_{K;3}^3
   \;.
\end{align*}
This implies that
\begin{align*}
    \msum_{k=1}^K \tau_{k;d}^3 
    \;\leq&\; \mean[u(\bX_1,\bX_2)u(\bX_2,\bX_3)u(\bX_3,\bX_1)] - \Mfullthree^3 + ( \Mfullthree + \varepsilon_{K;3})^3\;,
    \\
    \msum_{k=1}^K \tau_{k;d}^3 
    \;\geq&\;
    \mean[u(\bX_1,\bX_2)u(\bX_2,\bX_3)u(\bX_3,\bX_1)] + \Mfullthree^3 - ( \Mfullthree + \varepsilon_{K;3})^3 \;,
\end{align*}
which gives the desired bounds:
\begin{align*}
    \mean\big[ (W_n^K - D)^3 \big] \;\leq&\; 
    \mfrac{8 \big( \mean[u(\bX_1,\bX_2)u(\bX_2,\bX_3)u(\bX_3,\bX_1)] - \Mfullthree^3 + ( \Mfullthree + \varepsilon_{K;3})^3 \big)}{n^{3/2}(n-1)^{3/2}} \;,
    \\
    \mean\big[ (W_n^K - D)^3 \big] \;\geq&\; 
    \mfrac{8 \big( \mean[u(\bX_1,\bX_2)u(\bX_2,\bX_3)u(\bX_3,\bX_1)] + \Mfullthree^3 - ( \Mfullthree + \varepsilon_{K;3})^3 \big)}{n^{3/2}(n-1)^{3/2}} \;.
\end{align*}

\vspace{1em}

The fourth central moment can again be expanded using a binomial expansion and noting that each summand is zero-mean:
\begin{align*}
    &\;\mean\big[ (W_n^K - D)^4 \big] 
    \\
    \;=&\; 
    \mfrac{1}{n^2(n-1)^2} 
    \mean\Big[ \Big( \msum_{k=1}^K \tau_{k;d} (\xi_k^2 - 1) \Big)^4 \Big]
    \\
    \;=&\;
    \mfrac{1}{n^2(n-1)^2} 
    \Big(
    \mean\Big[ \msum_{k=1}^K \tau_{k;d}^4 (\xi_k^2 - 1)^4 \Big]
    +
    3
    \mean\Big[ \msum_{1 \leq k \neq k' \leq K} \tau_{k;d}^2 (\xi_k^2 - 1)^2 \tau_{k';d}^2 (\xi_{k'}^2 - 1)^2 \Big]
    \Big)
    \\
    \;=&\;
    \mfrac{1}{n^2(n-1)^2} 
    \Big(
    60 \msum_{k=1}^K \tau_{k;d}^4 
    +
    12 \msum_{1 \leq k \neq k' \leq K} \tau_{k;d}^2 \tau_{k';d}^2
    \Big)
    \\
    \;=&\;
    \mfrac{1}{n^2(n-1)^2} 
    \Big(
    48 \msum_{k=1}^K \tau_{k;d}^4 
    +
    12 \msum_{1 \leq k, k' \leq K} \tau_{k;d}^2 \tau_{k';d}^2
    \Big)
    \\
    \;=&\;
    \mfrac{12}{n^2(n-1)^2} 
    \Big(
    4 \msum_{k=1}^K \tau_{k;d}^4 
    +
    \big(\msum_{k=1}^K \tau_{k;d}^2\big)^2
    \Big)
    \;.
\end{align*}
Since we have already controlled $\sum_{k=1}^K \tau_{k;d}^2 = \Tr\big( (\Sigma^K \Lambda^K)^2 \big)$, we focus on bounding the first sum. Using notations from the third moment, we can express the sum as
\begin{align*}
    \msum_{k=1}^K \tau_{k;d}^4 
    \;=&\; \mean\Big[ 
    \big(\msum_{k=1}^K \lambda_k \phi_k(\bX_1)\phi_k(\bX_2)\big)
    \big(\msum_{k=1}^K \lambda_k \phi_k(\bX_2)\phi_k(\bX_3)\big)
    \\
    &\;\quad\;
    \big(\msum_{k=1}^K \lambda_k \phi_k(\bX_3)\phi_k(\bX_4)\big)
    \big(\msum_{k=1}^K \lambda_k \phi_k(\bX_4)\phi_k(\bX_1)\big)
    \Big]
    \\
    \;=&\; \mean[ S_{12} S_{23} S_{34} S_{41} ]
    \\
    \;=&\; \mean\big[ (U_{12} + \Delta_{12}) (U_{23} + \Delta_{23})  (U_{34} + \Delta_{34})   (U_{41} + \Delta_{41}) \big]
    \;.
\end{align*}
A similar argument as before shows that
\begin{align*}
    \Big| 
    \msum_{k=1}^K \tau_{k;d}^4 
    -  
    \mean[ u(\bX_1,\bX_2) u(\bX_2,\bX_3) & u(\bX_3,\bX_4)  u(\bX_4,\bX_1) ]
    \Big|
    \\
    \;\leq&\;
    4 \Mfullfour^3 \varepsilon_{K;4}
    +
    6 \Mfullfour^2 \varepsilon_{K;4}^2
    +
    4 \Mfullfour \varepsilon_{K;4}^3
    +
    \varepsilon_{K;4}^4
    \;.
\end{align*}
This implies that
\begin{align*}
    \msum_{k=1}^K \tau_{k;d}^4
    \;\leq&\; 
    \mean[ u(\bX_1,\bX_2) u(\bX_2,\bX_3) u(\bX_3,\bX_4)  u(\bX_4,\bX_1) ]
    -
    \Mfullfour^4
    + 
    (\Mfullfour + \varepsilon_{K;4})^4
    \;,
    \\
    \msum_{k=1}^K \tau_{k;d}^4 
    \;\geq&\; 
    \mean[ u(\bX_1,\bX_2) u(\bX_2,\bX_3) u(\bX_3,\bX_4)  u(\bX_4,\bX_1) ]
    +
    \Mfullfour^4
    - 
    (\Mfullfour + \varepsilon_{K;4})^4
    \;.
\end{align*}
On the other hand, by Lemma \ref{lem:factorisable:moment:gaussian}, we have 
\begin{align*}
    (\sigfull - \varepsilon_{K;2})^2
    \;\leq\;
    \msum_{k=1}^K \tau_{k;d}^2
    \;=\;
    \Tr( (\Lambda^K \Sigma^K)^2 ) 
    \;\leq\; (\sigfull + \varepsilon_{K;2})^2\;.
\end{align*}
Combining the results give the desired bounds:
\begin{align*}
    \mean\big[ (W_n^K - D)^4 \big] 
    \;\leq&\;
    \mfrac{12}{n^2(n-1)^2} 
    \Big(
    4 \,\mean[ u(\bX_1,\bX_2) u(\bX_2,\bX_3) u(\bX_3,\bX_4)  u(\bX_4,\bX_1) ]
    \\
    &\;\qquad\qquad\quad
    -
    4 \Mfullfour^4
    + 
    4 (\Mfullfour + \varepsilon_{K;4})^4
    +
    (\sigfull + \varepsilon_{K;2})^4
    \Big)
    \;,
    \\
    \mean\big[ (W_n^K - D)^4 \big] 
    \;\geq&\;
    \mfrac{12}{n^2(n-1)^2} 
    \Big(
    4 \,\mean[ u(\bX_1,\bX_2) u(\bX_2,\bX_3) u(\bX_3,\bX_4)  u(\bX_4,\bX_1) ]
    \\
    &\;\qquad\qquad\quad
    +
    4 \Mfullfour^4
    - 
    4 (\Mfullfour + \varepsilon_{K;4})^4
    +
    (\sigfull - \varepsilon_{K;2})^4
    \Big)
    \;.
\end{align*}

\vspace{1em}

For the generic moment bound, we first use a Jensen's inequality to get that
\begin{align*}
    \mean \big[ (W_n^K)^{2m} \big] 
    \;=&\;
    \mean\Big[ \Big(
    \mfrac{1}{n^{1/2}(n-1)^{1/2}} 
    \msum_{k=1}^K 
     \tau_{k;d}
     (\xi_k^2-1) + D \Big)^{2m} \Big]
     \\
     \;\leq&\;
     \mfrac{2^{2m-1}}{n^{m}(n-1)^{m} }
     \mean\Big[\Big(\msum_{k=1}^K 
     \tau_{k;d}
     (\xi_k^2-1)\Big)^{2m} \Big] 
     +
     2^{2m-1} \, D^{2m}
     \;. 
\end{align*}
Denote the set of all possible orderings of a length-$2m$ sequence consisting of elements from $[K]$ by $\cP(K,2m)$ and denote its elements by $p$. Consider the subset
\begin{align*}
    \cP'(K,2m) \;\coloneqq\; \{ p \in \cP(K,2m)\,:\,\text{ every element in } p \text{ appears at least twice } \}\;.
\end{align*}
By noting that $\xi_k-1$ is zero-mean and $\{\xi_k\}_{k=1}^K$ are independent, we can re-express the sum first as a sum over $\cP(K,2m)$ and then as a sum over $\cP'(K,2m)$:
\begin{align*}
     \mean\Big[\Big(\msum_{k=1}^K 
     \tau_{k;d}
     (\xi_k^2-1)\Big)^{2m} \Big] 
     \;=&\;
     \msum_{ p \in \cP(K,2m)} \big(\mprod_{k \in p} \tau_{k;d} \big) \mean\big[ \mprod_{k \in p} (\xi_k^2-1) \big]
     \\
     \;=&\;
     \msum_{ p \in \cP'(K,2m)} \big(\mprod_{k \in p} \tau_{k;d} \big) \mean\big[ \mprod_{k \in p} (\xi_k^2-1) \big]
     \\
     &\;+
     \msum_{ p \in \big( \cP(K,2m) \setminus \cP'(K,2m) \big)} \big(\mprod_{k \in p} \tau_{k;d} \big) \mean\big[ \mprod_{k \in p} (\xi_k^2-1) \big]
     \\
     \;=&\;
     \msum_{ p \in \cP'(K,2m)} \big(\mprod_{k \in p} \tau_{k;d} \big) \mean\big[ \mprod_{k \in p} (\xi_k^2-1) \big]\;.
\end{align*}
Write $C'_{m}$ as the $2m$-th central moment of a chi-squared random variable with degree $1$, which depends only on $m$ and not on $K$ or $\tau_{k;d}$. By a H\"older's inequality and the bound from Lemma \ref{lem:factorisable:moment:gaussian}, we get that
\begin{align*}
    \mean\Big[\Big(\msum_{k=1}^K 
     \tau_{k;d}
     (\xi_k^2-1)\Big)^{2m} \Big] 
     \;\leq&\;
     C'_m \, \msum_{ p \in \cP'(K,2m)} \big(\mprod_{k \in p} \tau_{k;d} \big)
     \\
     \;\leq&\;
     C'_m \, \mbinom{2m}{m} \big(\msum_{k=1}^K \tau_{k;d}^2\big)^{m}
     \\
     \;=&\;
     C'_m \, \mbinom{2m}{m} \, \Tr\big( (\Lambda^K \Sigma^K)^2 \big)^m
     \;\leq\;
      C'_m \, \mbinom{2m}{m} \, (\sigfull + \varepsilon_{K;2})^{2m}
     \;.
\end{align*}
Writing $C_m \coloneqq 2^{2m-1} \max\{1,  C'_m \, \binom{2m}{m} \}$, we get the desired bound that
\begin{align*}
    \mean \big[ (W_n^K)^{2m} \big] 
     \;\leq&\;
     \mfrac{C_m}{n^{m}(n-1)^{m} }
     (\sigfull + \varepsilon_{K;2})^{2m}
     +
     C_m \, D^{2m}\;.
\end{align*}

\vspace{1em}

Finally, if \cref{assumption:L_nu} is true for some $\nu \geq 2$, we have $\varepsilon_{K;2} \rightarrow 0$ as $K$ grows. Taking $K \rightarrow \infty$ in the bound for second moment gives
\begin{align*}
    \lim_{K \rightarrow \infty} \Var[ W_n^K ] 
    \;=\; 
    \mfrac{2}{n(n-1)} \sigfull^2 \;.
\end{align*}
If \cref{assumption:L_nu} holds for $\nu \geq 3$, similarly we have
\begin{align*}
    \lim_{K \rightarrow \infty} 
    \mean\big[ (W_n^K - D)^3 \big] \;=&\; 
    \mfrac{8 \mean[u(\bX_1,\bX_2)u(\bX_2,\bX_3)u(\bX_3,\bX_1)]}{n^{3/2}(n-1)^{3/2}} \;.
\end{align*}
If \cref{assumption:L_nu} holds for $\nu \geq 4$, we have
\begin{align*}
    \lim_{K \rightarrow \infty} 
    \mean\big[ (W_n^K - D)^4 \big] \;=&\; 
    \mfrac{12 (4 \mean[u(\bX_1,\bX_2)u(\bX_2,\bX_3)u(\bX_3,\bX_4)u(\bX_4,\bX_1)] + \sigfull^4)}{n^2(n-1)^2} \;.
\end{align*}
\end{proof}

\subsection{Proofs for \cref{appendix:distribution}}

\begin{proof}[Proof of Lemma \ref{lem:smooth:approx:ind}] Write $\delta' := \delta/(m+1)$ for convenience. Define the $m$-times differentiable function
\begin{align*}
    h_{m;\tau;\delta}(x) 
    \;\coloneqq\; 
    (\delta')^{-(m+1)} 
    \mint_{x}^{x + \delta'} \mint_{y_1}^{y_1 + \delta'} \ldots \mint_{y_{m-1}}^{y_{m-1} + \delta'} \mint_{y_m}^{y_m + \delta'} 
    \ind_{\{ y > \tau \}} \; dy \, dy_m \ldots dy_1
    \;.
\end{align*}
In the case $m=0$, the function is $h_{0;\tau;\delta}(x) \;\coloneqq\; \delta^{-1} \int_{x}^{x + \delta} \ind_{\{ y > \tau \}} \, dy$. By construction, $h_{m;\tau;\delta}(x) = 0$ for $x \leq \tau-\delta$, $h_{m;\tau;\delta}(x) \in [0,1]$ for $x \in (\tau-\delta, \tau]$ and $h_{m;\tau;\delta}(x) = 1$ for $x > \tau$. This implies $\ind_{\{x > \tau\}} \leq h_{m;\tau;\delta}(x) \leq \ind_{\{x > \tau-\delta\}}$ and therefore the desired inequality
$$
    h_{m;\tau+\delta;\delta}(x) \;\leq\; \ind_{\{x > \tau\}} \;\leq\; h_{m;\tau;\delta}(x)\;.
$$
Next, we prove the properties of the derivatives of $h_{m;\tau;\delta}$. Denote recursively
\begin{align*}
    &
    J_{m+1}(x) \;:=\; \mint_x^{x + \delta'}\ind_{\{ y > \tau \}} dy\;,
    &&
    J_r(x) \;:=\; \mint_x^{x + \delta'} J_{r+1}(y) \, dy\;
    \;\text{ for }\; 0 \leq r \leq m\;.
\end{align*}
Since $h_{m;\tau;\delta}(x) = (\delta')^{-(m+1)} J_0(x)$ and $\frac{\partial}{\partial x} J_i(x) = J_{i+1}(x+\delta') - J_{i+1}(x)$ for $0 \leq i \leq m$, by induction, we have that for $0 \leq r \leq m$,
\begin{align*}
    h^{(r)}_{m;\tau;\delta}(x) 
    \;=&\;
    (\delta')^{-(m+1)} \mfrac{\partial^{r}}{\partial x^r} J_0(x) 
    \;=\;
    (\delta')^{-(m+1)} \msum_{i=0}^r \mbinom{r}{i} (-1)^i \, J_{r+1}\big(x+ (r-i) \delta' \big) \;. \tagaligneq \label{eqn:h:m:sum}
\end{align*}
Note that $J_{m+1}$ is continuous, uniformly bounded above by $\delta'$, and satisfies that $J_{m+1}(x)=0$ for $x$ outside $[\tau-\delta', \tau]$. By induction, we get that for $0 \leq r \leq m$, $J_{r+1}$ is continuous, bounded above by $(\delta')^{m+1-r}$ and satisfies that $J_{r+1}(x)=0$ for $x$ outside $[\tau- (m+1-r) \delta', \tau]$. This shows that $h^{(r)}_{m;\tau;\delta}$ is continuous and $h^{(r)}_{m;\tau;\delta}(x)=0$ for $x$ outside $[\tau-\delta, \tau]$, and the uniform bound
\begin{align*}
    \big| h^{(r)}_{m;\tau;\delta}(x)  \big| 
    \;\leq\;
    (\delta')^{-r} \msum_{i=0}^r \mbinom{r}{i} \;=\; \big( \mfrac{2}{m+1} \big)^r \delta^{-r} \;\leq\; \delta^{-r} \;.
\end{align*}
Finally to prove the H\"older property of $h^{(m)}_{m;\tau;\delta}(x)$, we first note that $J_{m+1}$ is constant outside $x \in [\tau-\delta', \tau]$ and linear within the interval with Lipschitz constant $1$. The formula in \eqref{eqn:h:m:sum} suggests that $h^{(m)}_{m;\tau;\delta}(x)$ is piecewise linear and the Lipschitz constant in the interval $[\tau - (m-i+1)\delta', \tau - (m-i)\delta']$ is given by the Lipschitz constant of the $i$-th summand. Therefore, $h^{(m)}_{m;\tau;\delta}$ is also Lipschitz with Lipschitz constant 
\begin{align*}
     L_m \;\coloneqq\;(\delta')^{-(m+1)} \mmax_{0 \leq i \leq m} \mbinom{m}{i}  \;=\; (\delta')^{-(m+1)}\mbinom{m}{\lfloor m/2 \rfloor}\;.
\end{align*}
For $x, y \in [\tau-\delta, \tau]$, we then have
\begin{align*}
    | h^{(m)}_{m;\tau;\delta}(x) - h^{(m)}_{m;\tau;\delta}(y) | 
    \;\leq\; 
    L_m | x - y | 
    \;=&\; 
    L_m \delta \, \big| \mfrac{x-y}{\delta} \big| 
    \\ 
    \;\leq&\; 
    L_m \delta \, \big| \mfrac{x-y}{\delta} \big|^{\epsilon}
    \;=\;
    L_m \delta^{1-\epsilon} | x - y |^{\epsilon} 
    \;, \tagaligneq \label{eqn:h:m:lipschitz}
\end{align*}
where we have noted that $\big|\frac{x-y}{\delta}\big| \leq 1$ and $\epsilon \in [0,1]$. \eqref{eqn:h:m:lipschitz} is trivially true for $x, y$ both outside $[\tau-\delta, \tau]$ since $h^{(m)}_{m;\tau;\delta}$ evaluates to zero. Now consider $x \in [\tau-\delta, \tau]$ and $y < \tau-\delta$. We have that
\begin{align*}
    | h^{(m)}_{m;\tau;\delta}(x) - h^{(m)}_{m;\tau;\delta}(y) |
    \;=\;
    | h^{(m)}_{m;\tau;\delta}(x) - h^{(m)}_{m;\tau;\delta}(\tau-\delta) |
    \;\overset{\eqref{eqn:h:m:lipschitz}}{\leq}&\; 
    L_m \delta^{1-\epsilon} (x - \tau + \delta)^{\epsilon}
    \\
    \;\leq&\;
    L_m \delta^{1-\epsilon} |x - y|^{\epsilon}\;.
\end{align*}
Similarly for $x \in [\tau-\delta, \tau]$ and $y > \tau$, we have that
\begin{align*}
    | h^{(m)}_{m;\tau;\delta}(x) - h^{(m)}_{m;\tau;\delta}(y) |
    \;=\;
    | h^{(m)}_{m;\tau;\delta}(x) - h^{(m)}_{m;\tau;\delta}(\tau) |
    \;\overset{\eqref{eqn:h:m:lipschitz}}{\leq}&\; 
    L_m \delta^{1-\epsilon} (\tau - x)^{\epsilon}
    \;\leq\;
    L_m \delta^{1-\epsilon} |x - y|^{\epsilon}\;.
\end{align*}
Therefore \eqref{eqn:h:m:lipschitz} holds for all $x,y$. The proof for the derivative bound is complete by computing the constant explicitly as
\begin{align*}
    L_m \delta^{1-\epsilon} 
    \;=\;  
    (\delta')^{-(m+\epsilon)}\mbinom{m}{\lfloor m/2 \rfloor}
    \;=\;
    \delta^{-(m+\epsilon)} \mbinom{m}{\lfloor m/2 \rfloor}  (m+1)^{m+\epsilon}\;,
\end{align*}
and therefore
\begin{align*}
    | h^{(m)}_{m;\tau;\delta}(x) - h^{(m)}_{m;\tau;\delta}(y) |\;\leq\; C_{m,\epsilon} \, \delta^{-(m+\epsilon)}  \, |x - y|^{\epsilon}\;, \tagaligneq \label{eqn:h:m:lipschitz:final}
\end{align*}
with respect to the constant $C_{m,\epsilon} = \binom{m}{\lfloor m/2 \rfloor} (m+1)^{m+\epsilon}$.
\end{proof}

\ 

\begin{proof}[Proof of Lemma \ref{lem:approx:XplusY:by:Y}] By conditioning on the size of $Y$, we have that for any $a, b \in \R$ and $\epsilon > 0$,
\begin{align*}
    \P( a \leq X+Y \leq b) 
    \;=&\; 
    \P( a \leq X+Y \leq b \,,\, |Y| \leq \epsilon ) + \P(  a \leq X \leq b \,,\, |Y| \geq \epsilon ) 
    \\
    \;\leq&\;
    \P( a - \epsilon \leq X \leq b+\epsilon ) + \P(  |Y| \geq \epsilon )\;,
\end{align*}
and by using the order of inclusion of events, we have the lower bound
\begin{align*}
    \P( a \leq X+Y \leq b) 
    \;\geq&\;
    \P( a + \epsilon \leq X \leq b-\epsilon \,,\,|Y| \leq \epsilon  )
    \\
    \;=&\;
    \P( a + \epsilon \leq X \leq b-\epsilon ) - \P(  |Y| \geq \epsilon )\;.
\end{align*}

\end{proof}

\subsection{Proof for \cref{appendix:weak:mercer}}

\begin{proof}[Proof of Lemma \ref{lem:mercer}] By Lemma 2.3 of \cite{steinwart2012mercer}, the assumption that $\kappa^*$ is measurable and $\mean[ \kappa^*(\bV_1,\bV_1) ] < \infty$ implies the RKHS $\cH$ associated with $\kappa^*$ is compactly embedded into $L_2(\R^d, R)$. By Lemma 2.12 and Corollary 3.2 of \cite{steinwart2012mercer}, for some index set $\cI \subseteq \N$, there exists a sequence of non-negative, bounded values $\{\lambda_k\}_{k \in \cI}$ that converges to $0$ and a sequence of functions $\{\phi_k\}_{k \in \cI}$ that form an orthonormal basis of $L_2(\R^d,R)$ such that 
\begin{align*}
    \msum_{k \in \cI} \lambda_k \psi_k(\bV_1) \psi_k(\bV_2) \;=\; \kappa^*(\bV_1,\bV_2)\;,
\end{align*}
where the equality holds almost surely when $\cI$ is finite and the convergence holds almost surely when $\cI$ is infinite. We can extend $\cI$ to $\N$ by adding zero values of $\lambda_k$ and $\phi_k$ whenever necessary and drop the requirement that $\{\phi_k\}_{k=1}^\infty$ forms a basis, which gives the desired statement.
\end{proof}

\end{document}